%% file: pA_autoequivalence.tex
\documentclass[11pt]{amsart}
\usepackage{amsmath}
\usepackage{amsfonts}
\usepackage{amssymb}
\usepackage{graphicx}
\usepackage{hyperref}
\usepackage{mathrsfs}
\usepackage{pb-diagram}
\usepackage{epstopdf}
\usepackage{amsmath,amsfonts,amsthm,enumerate,amscd,latexsym,curves}
\usepackage{bbm}
\usepackage[mathscr]{eucal}
\usepackage{epsfig,epsf,epic}
\usepackage{caption}
\usepackage{subcaption}
\usepackage{multirow}
\usepackage{color}
\usepackage[all]{xy}
\usepackage{fourier}
\usepackage{tikz-cd}
\usepackage{calc}
\usepackage{float}
\usepackage{enumitem}
\usepackage{comment}
\usepackage{mathtools}

\newcommand{\Tw}{\mathrm{Tw}}
\newcommand{\lhom}{\mathrm{Hom}}
\newcommand{\shom}{\mathrm{hom}}

\newcommand{\Fuk}{\mathcal{F}_T}
\newcommand{\wrapped}{\mathcal{W}_T}

\newcommand{\length}{\mathrm{leng}}
\newcommand{\cl}{\mathrm{cl}}
\newcommand{\vect}{\vec{V}}
\newcommand{\stab}{\mathrm{Stab}}
\newcommand{\imagin}{\sqrt{-1}}
\newcommand{\Aut}{\operatorname{Aut}}
\newcommand{\Sim}{\operatorname{Sim}}
\newcommand{\Cone}{\operatorname{Cone}}
\newcommand{\Z}{\mathbb{Z}}

\newtheorem{thm}{Theorem}[section]
\newtheorem{lem}[thm]{Lemma}

\newtheorem{prop}[thm]{Proposition}
\newtheorem{definition}[thm]{Definition}

\newtheorem{remark}[thm]{Remark}

\newtheorem*{thm*}{Theorem}
\newtheorem{property}{Property}
\newtheorem{notation}{Notation}

\numberwithin{equation}{section}

\begin{document}

\title[pseudo-Anosov autoequivalences and their hyperbolic actions]{Pseudo-Anosov autoequivalances arising from Symplectic topology and their hyperbolic actions on stability conditions}

\author{Hanwool Bae}
\address[Hanwool Bae]{Center for Quantum Structures in Modules and Spaces, Seoul National University, Seoul, South Korea}
\email{hanwoolb@gmail.com}

\author{Sangjin Lee}
\address[Sangjin Lee]{Korea Institute for Advanced Study, 85 Hoegiro Dongdaemun-gu, Seoul 02455, Republic of Korea.}
\email{sangjinlee@kias.re.kr}

\begin{abstract}
	Within $N$-Calabi-Yau categories associated with quivers whose base graphs form trees, we delve into the study of the asymptotic behaviors of autoequivalences of a specific type.
	These autoequivalences, which we call "Penner type," exhibit straightforward asymptotic characteristics, making them noteworthy exemplars of "pseudo-Anosov" autoequivalences  in the sense of \cite{Fan-Filip-Haiden-Katzarkov-Liu21}, and also in a stronger sense that we define in the present paper. 
	
	In addition, we provide a practical methodology for calculating the stretching factors of Penner type autoequivalences. 
	We expect that this computational approach can have applications.
	As an example, we establish a positive lower bound on the translation length of the induced action these autoequivalences have on the space of stability conditions. 
	Our anticipation is that this lower bound is, in fact, exact. 
	Notably, we have observed instances of Penner type $\Phi$ where the induced actions align precisely with this lower bound.
	In other words, these examples induce hyperbolic actions on the space of stability conditions.	
\end{abstract}

\maketitle

\section{Introduction}
\label{section introduction}

\subsection{Background and results}
\label{subsection background and results}
Let $\Sigma$ be an oriented closed surface with genus $\geq 2$. 
The {\em mapping class group} of $\Sigma$ $\mathrm{Mod}(\Sigma)$ is defined to be the group of isotopy classes of orientation preserving diffeomorphisms of $\Sigma$.
The Nielsen--Thurston classification theorem \cite{Nielsen44, Thurston88} classifies all elements of $\mathrm{Mod}(\Sigma)$ into three classes; {\em periodic, reducible}, and {\em pseudo-Anosov}.
The last class, pseudo-Anosov, is not only generic in $\mathrm{Mod}(\Sigma)$, but also most interesting among those three classes because the other classes are simpler in the following sense: 
A power of a periodic class is the identity class, and every reducible class can be seen as a combination of periodic and pseudo-Anosov classes. 
The most interesting class has attracted many attentions, for example, see \cite{Fathi-Shub79, Marlies-Anatole82, Rivin08}, etc.

For a given surface $\Sigma$, the {\em Teichmüller space of $\Sigma$}, denoted as $\mathrm{Teich}(\Sigma)$, is defined as the space of hyperbolic metrics on $\Sigma$ up to isotopy.
It is well-known that $\mathrm{Mod}(\Sigma)$ and $\mathrm{Teich}(\Sigma)$ are closely interconnected.
For instance, every element of $\mathrm{Mod}(\Sigma)$ acts on $\mathrm{Teich}(\Sigma)$ as an isometry with respect to the Teichmüller metric.
Furthermore, the three types of mapping classes can be characterized based on their actions on $\mathrm{Teich}(\Sigma)$.
For more insights into the connections between $\mathrm{Mod}(\Sigma)$ and $\mathrm{Teich}(\Sigma)$, please refer to \cite{Farb-Margalit12} and the references therein.

Recently, people have started to recognize analogies between Teichm\"{u}ller theory on surfaces $\Sigma$ and the theory of Bridgeland stability conditions on triangulated categories $\mathcal{D}$.
Table \ref{table connections}, which is almost the same as \cite[Table 1]{Fan-Filip-Haiden-Katzarkov-Liu21}, summarizes the connections.
See \cite{Gaiotto-Moore-Neitzke13, Dimitrov-Haiden-Katzarkov-Kontsevich14, Bridgeland-Smith15, Haiden-Katzarkov-Kontsevich17} for more details.
\begin{table}[h!]
	\caption{Connections between surface/Teichm\"{u}ller theory and the theory of Bridgeland stability conditions on triangulated categories}
	\begin{tabular}{ c | c }
		\label{table connections}
		On $\Sigma$  & On $\mathcal{D}$ \\ \hline\hline
		diffeomorphisms of $\Sigma$ & autoequivalences  \\ \hline
		closed curves $C$ & objects $E$ of $\mathcal{D}$ \\ \hline
		intersections $C_1 \cap C_2$ & morphisms $\lhom_\mathcal{D}(E_1, E_2)$ \\ \hline
		flat metrics & stability conditions \\ \hline 
		geodesics & stable objects \\ \hline
		length of $C$ & mass of $E$ \\ \hline
		slope of $C$ & phase of $E$ \\ \hline 
		entropy & behavior of the mass of generators \\  
	\end{tabular}
\end{table}

Given the connections between surface/Teichm\"{u}ller theory and the theory of Bridgeland stability conditions, as well as between $\mathrm{Mod}(\Sigma)$ and $\mathrm{Teich}(\Sigma)$, one might wonder if there exists an analogy to the generic mapping class, i.e., pseudo-Anosov mapping class, in category theory. 
More precisely, one can ask the following questions: 
\begin{enumerate}
	\item[Question 1.] {\em Is there a categorical counterpart to a generic mapping class, i.e., pseudo-Anosov maps? If such an analogue exists, do we have examples of the counterpart?}
	\item[Question 2.] {\em Assuming the existence of a categorical analogue for pseudo-Anosov maps, what property of pseudo-Anosov maps does this analogue satisfy?}
\end{enumerate}
These questions have been extensively investigated by numerous researchers, including Dimitrov, Haiden, Katzarkov, and Kontsevich in \cite{Dimitrov-Haiden-Katzarkov-Kontsevich14}, as well as Fan, Filip, Haiden, Katzarkov, and Liu in \cite{Fan-Filip-Haiden-Katzarkov-Liu21}, and Kikuta in \cite{Kikuta22}.
This paper contributes to the study of the above questions and provides partial answers to them. 

In the rest of Section \ref{section introduction}, we will explain the known-answers to the questions and our results. 
To do that, let us briefly introduce some properties of pseudo-Anosov maps first. 

The definition of pseudo-Anosov maps is the following:
\begin{definition}
	\label{def pseudo-Anosov surface automorphism}
	A surface automorphism $\phi : \Sigma \to \Sigma$ is {\bf pseudo-Anosov} if there exists a transverse pair of measured singular foliations on $\Sigma$, $(\mathcal{F}^s, \mu_s)$(stable) and $(\mathcal{F}^u, \mu_u)$(unstable), and a real number $\lambda>1$, such that 
	\begin{itemize}
		\item the stable and unstable foliations $\mathcal{F}^s$ and $\mathcal{F}^u$ are preserved by $\phi$, and 
		\item their transverse measures are multiplied by $\lambda^{-1}$ and $\lambda$, respectively.
	\end{itemize}
	Or equivalently, 
	\[\phi(\mathcal{F}^s,\mu_s) = \left(\mathcal{F}^s, \lambda^{-1}\mu_s\right), \phi(\mathcal{F}^u,\mu_u) = \left(\mathcal{F}^u, \lambda \mu_u\right)\]
\end{definition}

When presented with a pseudo-Anosov map $\phi: \Sigma \to \Sigma$, a natural question arises regarding how to obtain the corresponding stable and unstable foliations. 
In this context, a well-known property emerges as a tool for acquiring the desired pair of foliations and characterizing the pseudo-Anosov mapping class:
\begin{property}[Corollary 14.25 of \cite{Farb-Margalit12}]
	\label{property 1}
	Let $\phi:\Sigma \to \Sigma$ be a pseudo-Anosov map, and let $(\mathcal{F}^u,\mu_u)$ denote the unstable foliation for $\phi$. 
	It is known that for any simple closed curve $C \subset \Sigma$ not homotopic to $0 \in H_1(\Sigma)$, we have 
	\[\lim_{n \to \infty}[\phi^n(C)] = [(\mathcal{F}^u,\mu_u)],\]
	in the space of projective classes.
\end{property}

Property \ref{property 1} is a corollary of Property \ref{property 2} given below.
We refer the reader to \cite[Chapter 14]{Farb-Margalit12} for the relation between Properties \ref{property 1} and \ref{property 2}.
\begin{property}[Theorem 5 of \cite{Thurston88}]
	\label{property 2}
	For any diffeomorphism $\phi: \Sigma \to \Sigma$, there is a finite set of algebraic integers $1 \leq \lambda_1 < \lambda_2 < \dots < \lambda_k$ such that for any homotopically nontrivial simple closed curve $C$, there is a real number $\lambda_i$ such that for any Riemannian metric $g$ on $\Sigma$, 
	\[\lim_{n \to \infty}\length_g(\phi^nC)^{\tfrac{1}{n}}=\lambda_i,\]
	where $\length_g$ denotes the length of a shortest representative in the homotopy class. 
	Moreover, $\phi$ represents a pseudo-Anosov mapping class if and only if $k=1$ and $\lambda_1 >1$. 
	In this case, $\lambda_1$ is called the {\em stretch factor} of $\phi$. 
\end{property}
See \cite[Expos\'{e} 12, Section IV]{Fathi-Laudenbach-Poenaru79} for the proof of Property \ref{property 2}.

We also note that the actions of elements of $\mathrm{Mod}(\Sigma)$ on $\mathrm{Teich}(\Sigma)$ characterize pseudo-Anosov mapping classes.
To be more precise, let us define the notion of {\em hyperbolic action} first. 
\begin{definition}
	\label{def hyperbolic action} 
	Let $(X,d)$ be a metric space, and let $f: X \to X$ be an isometry. 
	\begin{enumerate}
		\item The {\bf translation length} of $f$ is defined by 
		\[\ell(f):=\inf_{x \in X} \left\{d\left(x,f(x)\right)\right\}.\]
		\item The isometry $f$ is said to be {\bf hyperbolic} if $\ell(f)>0$ and there exists $x \in X$ such that $\ell(f) = d \left(x, f(x)\right)$.
	\end{enumerate}
\end{definition}

Bers \cite{Bers78} (resp.\ Daskalopoulos and Wentworth \cite{Daskalopoulos-Wentworth03}) proved the following property when $\mathrm{Teich}(\Sigma)$ is equipped with the Teichm\"{u}ller metric (resp.\ Weil--Petersson metric).
\begin{property}
	\label{property 3}
	A mapping class $\phi \in \mathrm{Mod}(\Sigma)$ induces a hyperbolic action on $\mathrm{Teich}(\Sigma)$ if and only if $\phi$ is of pseudo-Anosov type. 
\end{property}

The above characterizations of pseudo-Anosov mapping class are generalized in category theory as follows:
Motivated by Definition \ref{def pseudo-Anosov surface automorphism}, Dimitrov, Haiden, Katzarkov, and Kontsevich\cite{Dimitrov-Haiden-Katzarkov-Kontsevich14} suggested a definition of ``pseudo-Anosov autoequivalence" on a triangulated category by replacing a pair of foliations with a stability condition.
Let us recall that in Definition \ref{def pseudo-Anosov surface automorphism}, a pseudo-Anosov map $\phi$ has the stable/unstable foliations such that $\phi$ stretches them by the factors $\lambda^{-1}$ and $\lambda$, respectively, with $\lambda >1$.
In \cite{Dimitrov-Haiden-Katzarkov-Kontsevich14}, an autoequivalence $\Phi$ on a triangulated category is pseudo-Anosov if $\Phi$ admits a stability condition $\sigma$ such that $\Phi$ stretches the real/imaginary axes by $\lambda^{-1}$ and $\lambda$.

Some properties of pseudo-Anosov autoequivalences are studied by Kikuta in \cite{Kikuta22}.
Especially, \cite[Theorem 4.9]{Kikuta22} proves that every pseudo-Anosov autoequivalence acts hyperbolically on the space of stability conditions up to $\mathbb{C}$-action. 
In other words, \cite{Kikuta22} proved an analogy of Property \ref{property 3}.
For the details on the space of stability conditions (up to $\mathbb{C}$-action) and the metric on it, see Bridgeland \cite{Bridgeland07}, and also see Sections \ref{section pseudo-Anosov autoequivalences and stability conditions} and \ref{section actions on the space of stability conditions}.

In \cite{Fan-Filip-Haiden-Katzarkov-Liu21}, Fan, Filip, Haiden, Katzarkov, and Liu proposed another definition of pseudo-Anosov autoequivalence. 
Let us refer the notion of pseudo-Anosov given in \cite{Dimitrov-Haiden-Katzarkov-Kontsevich14} as ``DHKK pseudo-Anosov". 
According to \cite{Fan-Filip-Haiden-Katzarkov-Liu21}, the notion of DHKK pseudo-Anosov is too restrictive in some cases, for example, see \cite[Remark 2.16 and Proposition 3.7]{Fan-Filip-Haiden-Katzarkov-Liu21}.
Thus, \cite{Fan-Filip-Haiden-Katzarkov-Liu21} gave a less-restrictive definition by generalizing Property \ref{property 2}, and proved that DHKK pseudo-Anosov autoequivalences are pseudo-Anosov in the sense of \cite{Fan-Filip-Haiden-Katzarkov-Liu21}.

The new notion of pseudo-Anosov autoequivalence is defined as follows: 
First, \cite{Fan-Filip-Haiden-Katzarkov-Liu21} fixed a stability condition $\sigma$ of a triangulated category $\mathcal{C}$, then they can measure a {\em mass} of objects $E \in \mathcal{C}$ with respect to $\sigma$. 
Let $m_\sigma(E)$ denote the mass of $E$. 
We note that, as mentioned in Table \ref{table connections}, in the analogy between Teichmüller theory and category theory, an object $E$, the chosen stability condition $\sigma$, and the mass $m_\sigma(E)$ correspond to a closed curve $C$, a flat metric $g$, and the length of $C$ with respect to $g$ $\length_g(C)$. 
Thus, a proper generalization of Property \ref{property 2} is the following: 
an autoequivalence $\Phi: \mathcal{C}\to \mathcal{C}$ is pseudo-Anosov if and only if, for any nonzero object $E \in \mathcal{C}$, $m_\sigma(\Phi^n E)$ has the same exponential growth bigger than $1$, or equivalently there exists $\lambda$ such that
\[\lim_{n \to \infty} \left(m_\sigma(\Phi^n E)\right)^{\tfrac{1}{n}}= \lambda >1.\]
See Section \ref{section pseudo-Anosov autoequivalences and stability conditions}, especially Definition \ref{def growth filtration}, for the formal definition. 

In the current paper, we suggest a stronger notion of pseudo-Anosov autoequivalence.
In order to be define the stronger version, we observe the asymptotic behavior of $\phi^\pm_\sigma(\Phi E)$ where $\phi^\pm_\sigma$ means the maximal/minimal phase with respect to $\sigma$. 
We then define $\Phi$ as a {\em strong pseudo-Anosov} autoequivalence if it satisfies the following conditions:
\begin{itemize}
	\item $\Phi$ qualifies as a pseudo-Anosov autoequivalence in the sense described in \cite{Fan-Filip-Haiden-Katzarkov-Liu21}.
	\item For any nonzero $E$, $\phi^\pm_\sigma(\Phi^n E)$ exhibits consistent asymptotic behavior.
\end{itemize}
For more details, see Section \ref{subsection strong pseudo-Anosov autoequivalence}. 

One can see the notion of pseudo-Anosov (and strong pseudo-Anosov) autoequivalences as answers to the first half of Question 1, and the second half is to find examples of pseudo-Anosov autoequivalences. 
To the best of our knowledge, there are not many known examples of pseudo-Anosov autoequivalences; 
There are a few examples of DHKK pseudo-Anosov, and most of the examples in \cite{Fan-Filip-Haiden-Katzarkov-Liu21} are ``weak" pseudo-Anosov.

In the main part of the paper, we give a construction of (strong) pseudo-Anosov autoequivalences, and we prove that the constructed functors satisfy a generalized version of Property \ref{property 3} under an assumption. 

Before going further, we would like to note that the second-named author generalized a weaker version of Property \ref{property 1} in symplectic setting. 
More precisely, Lee \cite{Lee19} gave a construction of symplectic automorphisms $\Phi$ on symplectic manifolds of general dimension, such that $\Phi^n(L)$ converges as $n \to \infty$ if $L$ is a Lagrangian submanifold.
The constructed symplectic automorphism is called {\em Penner type} since the construction in \cite{Lee19} is a higher dimensional generalization of {\em Penner construction} \cite{Penner88} of pseudo-Anosov surface maps. 

A symplectic manifold has a triangulated category as an invariant. 
The invariant triangulated category is the (derived) Fukaya category. 
Moreover, a symplectic automorphism induces an autoequivalence on the Fukaya category. 
In this situation, it would be natural to expect that a symplectic automorphism of Penner type induces a pseudo-Anosov autoequivalence on the corresponding Fukaya category.

Motivated by the expectation, we define the notion of Penner type autoequivalences on some triangulated categories in a purely categorical way, and prove that every Penner type autoequivalence is strong pseudo-Anosov.
More detailed statement is the following theorem:
\begin{thm}[= Theorems \ref{thm pseudo-Anosov} and \ref{thm strong pseudo-Anosov}]
	\label{thm pseudo-Anosov intro} 
	For a given tree $T$ and an integer $N \geq 3$, let $\mathcal{D}$ denote the finite dimensional derived category of the Ginzburg dg-algebra $\Gamma_N T$. 
	If $\Phi: \mathcal{D} \to \mathcal{D}$ is of Penner type defined in Definition \ref{def Penner type}, then we have the following:
	\begin{enumerate}
		\item (Theorem \ref{thm pseudo-Anosov}.) $\Phi$ is a pseudo-Anosov autoequivalence in the sense of \cite{Fan-Filip-Haiden-Katzarkov-Liu21}, or equivalently, Definition \ref{def growth filtration}.
		Moreover, there exists an algorithmic computation of the exponential growth of mass of $\Phi^n E$ for any nonzero $E \in \mathcal{D}$.
		\item (Theorem \ref{thm strong pseudo-Anosov}.) $\Phi$ is strong pseudo-Anosov in the sense of Definition \ref{def strong pesudo-Anosov}. 
		Moreover, there exists an algorithmic computation of the growth of the maximal/minimal phase of $\Phi^n E$ for any nonzero $E \in \mathcal{D}$. 
	\end{enumerate} 
\end{thm}

We also demonstrate that a Penner-type autoequivalence exhibits properties that generalize those found in pseudo-Anosov surface mapping class.
For example, in surface theory, it is well-known that the logarithm of the stretch factor of a pseudo-Anosov surface automorphism is the same as its entropy. 
The generalization of the property is a simple corollary of the algorithmic computation given in Theorem \ref{thm pseudo-Anosov intro}. 
\begin{thm}[= Theorem \ref{thm stretch factor and entropy}]
	\label{thm stretch factor and entropy intro}
	Under the same conditions of Theorem \ref{thm pseudo-Anosov intro}, the categorical entropy of $\Phi$ (see Definition \ref{def categorical entropy}) is the same as the logarithm of the stretch factor, defined in Definition \ref{def growth filtration}, of $\Phi$. 
\end{thm}

More interestingly, we prove that the Penner type autoequivalences satisfy a generalization of Property \ref{property 3} under an assumption.
To prove that, we first show that every Penner type autoequivalence, without any assumption, induces an action on the space of stability condition having {\em positive translation length}. 
Then, by adding an assumption, we can show that the induced action is hyperbolic. 
More formally, we prove Theorem \ref{thm hyperbolic action intro}.

\begin{thm}[= Theorems \ref{thm shifting number} and \ref{thm hyperbolic action}]
	\label{thm hyperbolic action intro}
	In the same setting as Theorem \ref{thm pseudo-Anosov intro}, the following hold:
	\begin{itemize}
		\item[(i)] (Theorem \ref{thm shifting number}.) $\Phi$ induces an action with positive translation length on the space of stability conditions of $\mathcal{D}$. Moreover, $\Phi$ also induces an action with positive translation length on the space of stability conditions up to $\mathbb{C}$-action.
		\item[(ii)] (Theorem \ref{thm hyperbolic action}.) Under an assumption mentioned in Section \ref{subsection hyperbolic actions of Penner type autoequivalences on stab under an assumption}, the above mentioned actions are hyperbolic. 
	\end{itemize}
\end{thm}
We note that we need an extra assumption in Theorem \ref{thm hyperbolic action intro} (ii) since the techniques employed in the current paper are suitable to observe the asymptotic behaviors of actions, but not efficient to observe one time behaviors. 
However, if a Penner type autoequivalence $\Phi$ satisfies the added assumption, then the techniques of the paper can show that $\Phi$ induces hyperbolic actions.
And, we expect that every Penner type $\Phi$ will induce hyperbolic actions even though we do not prove it in the present paper.

\subsection{Further directions}
\label{subsection further directions}
In this subsection, we briefly suggest a list of further questions and possible applications. 
It is important to note that the listed avenues represent just a few examples, and there are likely numerous additional applications to be explored. 
We welcome suggestions and comments on other potential applications.

One can succinctly summarize our results as a generalization of Penner's pseudo-Anosov construction within finite-dimensional derived categories of Ginzburg dg-algebras.
Thus, it becomes pertinent to inquire whether this generalization remains applicable in broader categorical contexts.
It's worth noting that within the realm of surface theory, Penner \cite{Penner88} employed the concept of (measured) train tracks and the space of measured train tracks to prove the generation of pseudo-Anosov maps through his methodology. 
Therefore, when contemplating the extension of our findings to encompass general categories, one viable avenue involves the exploration of generalized notions of measured train tracks within the domain of category theory.
It could be one direction for further study. 

Based on the connection between surface/Teichm\"{u}ller theory and the theory of stability conditions, one could expect that some results in surface theory could be generalized in the category theory. 
The results presented in this article offer infinitely many examples of pseudo-Anosov autoequivalences. 
We hope that these examples will prove valuable for exploring such generalizations.
Before ending the introduction section, we briefly mention about the possible applications. 

The first application is to study the {\em systole}. 
In Teichm\"{u}ller theory, it is well-known, for example by \cite{Fathi-Laudenbach-Poenaru79}, that 
\[\min \left\{\log \lambda_f | f \in \mathrm{Mod}(\Sigma) \text{  is pseudo-Anosov and  } \lambda_f \text{  is the stretching factor of  } f \right\}\]
equals the length of the shortest geodesic (systole) of the moduli space of $\Sigma$.
In other words, the stretch factor of a pseudo-Anosov surface automorphism is an upper bound of the systole. 
We note that Fan \cite{Fan22} defined the notion of {\em categorical systole}.
Moreover, \cite{Fan22} also proved that the categorical systole has an upper bound if the underlying category is $\mathrm{D}^b\mathrm{Coh}(X)$ where $X$ is a projective K3 surface. 
From these background story, we hope that a pseudo-Anosov autoequivalence and its stretching factor give an upper bound for categorical systole.

Another possible application is to study the group of autoequivalences. 
For example, in surface theory, algebraic properties of the mapping class group $\mathrm{Mod}(\Sigma)$ have been studied extensively, and sometimes, pseudo-Anosov surface maps play a role in the research.
One example is a result of Bestvina and Fujiwara \cite{Bestvina-Fujiwara02}, which proved that if a subgroup $G \subset \mathrm{Mod}(\Sigma)$ contains two independent pseudo-Anosov mapping classes, then $G$ an its action on the curve complex satisfy the weak proper discontinuity condition. 
See \cite{Bestvina-Fujiwara02, McCarthy-Papadopoulos89} and references therein. 
Similarly, we hope that pseudo-Anosov autoequivalences could play a role in the research of algebraic properties of the group of autoequivalences.

We would like to note that our result not only constructs examples of pseudo-Anosov, but also provides a practical tool of computing the stretching factors of the constructed examples. 
Moreover, in the later sections, it will be turned out that our computation tool also computes the {\em shifting numbers} of examples. 
We note that the notion of shifting numbers is defined in \cite{Fan-Filip23} as an analogue of Poin\'{c}are translation numbers.
One conjecture presented and studied in \cite{Fan-Filip23, Fan23} is that the shifting numbers provide a quasimorphism on the group of autoequivalences. 
We hope that our computation of shifting numbers could be useful to study the conjecture. 

The paper consists of 11 sections and the organization is the following: 
The current section and Section \ref{section settings and the main idea} are introductory sections.
In the next section, we explain the setting of the current paper and the main idea. 
We note that everything in the paper could be written in purely algebraic language, but the main idea arises from symplectic topology. 
Thus, in Section \ref{section settings and the main idea}, we explain both of algebraic and symplectic settings as well as their equivalence. 
Sections \ref{section twisted complex}--\ref{section actions on the space of stability conditions} are preliminary sections explaining the notion of twisted complexes, stability conditions, pseudo-Anosov autoequivalences, etc. 
In each section in the preliminary part, we explain preliminaries in the first few subsections, then we apply the preliminaries to our setting.
But, for an expert, it would be not a bad idea to skip this preliminary part. 
After the preliminary part, we study the asymptotic behaviors of autoequivalences in Sections \ref{section preparations for the proof of pseudo-Anosov theorem}--\ref{section strong pseudo-Anosov} and prove Theorem \ref{thm pseudo-Anosov intro}. 
In Section \ref{section translation of Penner type autoequivalences}, we prove that the actions on the space of stability condition, induced from Penner type autoequivalences, have positive translation lengths by giving their positive lower bound. 
Moreover, we can see that under an assumption, the action is hyperbolic, i.e., we prove Theorem \ref{thm hyperbolic action intro}.
In the last section, we provide examples explaining three topics, the computational aspect of the paper, the reason why we define a stronger notion of pseudo-Anosov, and hyperbolic actions on the space of stability conditions. 

\subsection{Acknowledgment}
\label{subsection acknowledgement}
We express our gratitude to Yu-Wei Fan and Fabian Haiden for kindly providing answers to our questions regarding their article \cite{Fan-Filip-Haiden-Katzarkov-Liu21}.

During this work, the first-named author was supported by the National Research Foundation of Korea(NRF) grant funded by the Korea government(MSIT) (No.2020R1A5A1016126). The second-named author was supported by a KIAS Individual Grant (MG094401) at Korea Institute for Advanced Study.

\section{Settings and the main idea}
\label{section settings and the main idea}
We first set our algebraic setting in Section \ref{subsection setting}.
In Section \ref{subsection symplectic topological setting}, we explain the setting from symplectic topological viewpoint. 
And in the last subsection of Section \ref{section settings and the main idea}, we give the main idea of Theorem \ref{thm pseudo-Anosov intro}.

\subsection{Setting}
\label{subsection setting}
In this subsection, we describe our setting.

Let $T$ denote a tree, i.e., a graph without cycles. 
We let $V(T)$ and $E(T)$ denote the set of vertices and the set of edges, respectively.

Since $T$ is a tree, there exists a decomposition of $V(T)$ into two disjoint subsets $V_+(T)$ and $V_-(T)$ such that if two vertices $v_1, v_2 \in V(T)$ are adjacent to each other, then either 
\begin{itemize}
	\item $v_1 \in V_+(T)$ and $v_2 \in V_-(T)$, or 
	\item $v_1 \in V_-(T)$ and $v_2 \in V_+(T)$. 
\end{itemize}
Or equivalently, if $v_1, v_2 \in V(T)$ are connected by even number of edges, then either 
\begin{itemize}
	\item $v_1, v_2 \in V_+(T)$, or 
	\item $v_1, v_2 \in V_-(T)$. 
\end{itemize}
We say that a vertex $v \in V(T)$ is {\em positive} (resp.\ {\em negative}) if a vertex $v$ is in $V_+(T)$ (resp.\ $V_-(T)$).

\begin{remark}
	\label{rmk sign conventions} 
	We note that the labeling of $V_+(T)$ and $V_-(T)$, or signs in the subscriptions of $V_+(T)$ and $V_-(T)$, are randomly given.
	There is no natural way of giving the signs. 
	In the rest of the paper, for any tree $T$, we will assume that the choice of labeling/signs is a part of given data for convenience. 
	And, the other choice will be called {\em the other sign convention}.
	 
	We would like to point out that if one choose the other sign convention, it does not effect on the contents of the paper.
	The change exchanges the names/roles of two disjoint subsets of $V(T)$, but one can prove every claim in the paper with the other sign convention.  
	See also Remarks \ref{rmk sign is not important 0}, \ref{rmk sign is not important 1}, \ref{rmk sign is not important 2}, etc.
\end{remark}
 
We note that every edge $e \in E(T)$ connects one positive vertex and one negative vertex. 
Thus, one can direct edges of $T$ so that every edge starts from a positive vertex and ends at a negative vertex. 

Since the edges are directed, $T$ is a directed graph, i.e., a {\em quiver}. 
Let $N \geq 3$ be a given integer. 
Then, one can define the {\em Ginzburg $N$-Calabi--Yau dg-algebra} of the quiver $T$, as defined in \cite{Ginzburg06}. 
Let $\Gamma_N T$ denote the Ginzburg $N$-Calabi--Yau dg-algebra. 
The triangulated category we are interested in the paper is the {\em finite dimensional derived category} of $\Gamma_N T$, i.e., the full subcategory of the derived category of dg-modules over $\Gamma_N T$ whose homology is of finite total dimension. 
Let $\mathcal{D}(\Gamma_N T)$ denote the triangulated category.
For more details about $\Gamma_N T$ and $\mathcal{D}(\Gamma_N T)$, we refer the reader to \cite[Sections 7.2 and 7.3]{Keller12}.
See also \cite{Ginzburg06, Keller-Yang11, Keller11, Keller12, King-Qiu15}.

\begin{remark}
	\label{rmk sign is not important 0}
	\mbox{}
	\begin{enumerate}
		\item We note that if one chooses the other sign convention for $V(T)$, then every edge of $T$ has the reverse direction.
		Since the construction of Ginzburg dg-algebra depends on the quiver, not the base graph, the different sign convention could affect on the resulting Ginzburg dg-algebra.
		However, it does not change the equivalence class of the triangulated category $\mathcal{D}(\Gamma_N T)$. 
		Here, the equivalence relation we are considering is the Morita equivalence, in other words, the triangulated closures of categories of modules of finite total dimension over two Ginzburg dg-algebras are quasi-equivalent. 
		We skip the proof since it is straightforward.  
		\item Moreover, we also note that even if one directs the edges of $T$ randomly, the resulting Ginzburg dg-algebra does not change up to Morita-equivalence. 
		In other words, if two quiver $Q_1$ and $Q_2$ have the same tree as their base graph, then their Ginzburg dg-algebra are Morita equivalent to each other. 
	\end{enumerate}
\end{remark}

We briefly describe some properties of $\mathcal{D}(\Gamma_N T)$ without proofs. 
For the details, we refer \cite[Section 7]{King-Qiu15} and references therein. 

First, we note that, as a triangulated category, $\mathcal{D}(\Gamma_N T)$ admits a generating set 
\[\{S_v | v \in V(T)\}.\]
Each generator $S_v$ can be understood as the simple module of the dg algebra $\Gamma_N T$ associated to a vertex $v \in V(T)$. 
Thus, by taking proper shifts of $\{S_v |v \in V(T)\}$, one can assume that the morphism spaces between $\{S_v | v \in V(T)\}$ satisfies the following: 
\begin{itemize}
	\item For all $v \in V(T)$, $\lhom^i_{\mathcal{D}(\Gamma_N T)}(S_v, S_v) = 
	\begin{cases} \mathbb{k} \text{  if  } i = 0, N, \\ 0 \text{  otherwise}.
	\end{cases}$
	\item For $u \in V_+(T), w \in V_-(T)$ such that $u$ and $w$ are connected by an edge, 
	\[\lhom^i_{\mathcal{D}(\Gamma_N T)}(S_u,S_w) = \begin{cases}
		\mathbb{k} \text{  if  } i =1, \\ 0 \text{  otherwise}, 
		\end{cases} 
		\lhom^i_{\mathcal{D}(\Gamma_N T)}(S_w,S_u) = \begin{cases}
		\mathbb{k} \text{  if  } i =N-1, \\ 0 \text{  otherwise}.
	\end{cases}\]
	\item for all $u, w \in V(T)$ which do not satisfy any of the above conditions, 
	\[\lhom^i_{\mathcal{D}(\Gamma_N T)} (S_u, S_w) = 0 \text{  for all  } i.\]
\end{itemize}
We note that $\lhom_{\mathcal{D}}^i(X,Y)$ denotes the morphism space in $\mathcal{D}(\Gamma_N T)$ from $X$ to $Y$. 

\begin{definition}
	\label{def sim}
	For simplicity, we say $\boldsymbol{v \sim w}$ for $v, w \in V(T)$ if $v$ and $w$ are connected by an edge in $T$.  
\end{definition}

Also, for all $v \in V(T)$, one can see that $S_v$ is a spherical object of $\mathcal{D}(\Gamma_N T)$.
Thus, there exists the Seidel-Thomas twist functor \cite{Seidel-Thomas01}, or equivalently the spherical twist, along $S_v$ for all $v \in V(T)$. 
Let $\tau_v : \mathcal{D}(\Gamma_N T) \to \mathcal{D}(\Gamma_N T)$ denote the spherical twist along $S_v$.
The spherical twists and their inverses construct the autoequivalences of Penner type which appeared in the statement of Theorem \ref{thm pseudo-Anosov intro} as follows:
\begin{definition}
	\label{def Penner type}
	An autoequivalence $\Phi: \mathcal{D}(\Gamma_N T) \to \mathcal{D}(\Gamma_N T)$ is of {\bf Penner type} if 
	\begin{enumerate}
		\item[(I)]  either $\Phi$ or $\Phi^{-1}$ is a product of the elements of the following set:
		\[\{\tau_u, \tau_w^{-1} | u \in V_+(T), w \in V_-(T)\},\]
		\item[(II)] for every $v \in V(T)$, either $\tau_v$ or $\tau_v^{-1}$ should appear in the product at least once.
	\end{enumerate}
\end{definition}

\begin{remark}
	\label{rmk sign is not important 1}
	We note that Definition \ref{def Penner type} is independent of the sign convention described in Remark \ref{rmk sign conventions}, even thought the statement uses a specific choice of signs.
	We note that for convenience, in the most part of the paper, we will assume that a Penner type autoequivalence $\Phi$ is a product of 
	\[\{\tau_u, \tau_w^{-1} | u \in V_+(T), w \in V_-(T)\},\]
	and we will prove the contents of the paper under the assumption.
	If $\Phi^{-1}$ is a product of the above spherical twists and the inverses, one can imagine that, by choosing the other sign convention and by applying the arguments given in the paper, the same results can be proven for such $\Phi$. 
	See Remark \ref{rmk sign is not important 2}.
\end{remark}

\subsection{Symplectic topological setting}
\label{subsection symplectic topological setting}
As mentioned in the introduction section, we can state our main results and their proof without mentioning any of symplectic topology.
However, since we are motivated by symplectic topology, in the subsection, we introduce the symplectic topological setting that corresponds to the setting given in Section \ref{subsection setting}.
 
As we did in Section \ref{subsection setting}, let $T$ denote a given tree, and let $N \geq 3$ be a given integer. 
We consider the plumbing of multiple copies of $T^*S^N$ along $T$. 
Let $P_N(T)$ denote the plumbing space.
Then, it is well-known that $P_N(T)$ admits a Weinstein structure.

Since $P_N(T)$ is a plumbing of multiple copies of $T^*S^N$, the zero section of each copies of $T^*S^N$ can be seen as a Lagrangian sphere in $P_N(T)$. 
In other words, we have a Lagrangian sphere in $P_N(T)$ for each vertex $v \in V(T)$. 
Let $S_v$ be the Lagrangian sphere corresponding to $v \in V$. 

Let $\Fuk$ denote the triangulated envelope of the compact Fukaya category.
We note that Abouzaid and Smith \cite{Abouzaid-Smith12} proved a generation result for $\Fuk$. 
According to the generation result, $\Fuk$ is generated by a set of Lagrangian spheres
\[\{S_v | v \in V(T)\}.\]
Moreover, if $\hom_{\Fuk}^i(X,Y)$ denotes the degree $i$ morphism space of $\Fuk$ from $X$ to $Y$, and if $\lhom_{\Fuk}^i(X,Y)$ denotes the homology of $\hom^i_{\Fuk}(X,Y)$, the following holds:
\begin{lem}
	\label{lem grading}
	\mbox{}
	\begin{enumerate}
		\item It is possible to assume that $\{S_v\}_{v \in V}$ are properly graded so that 
		\begin{itemize}
			\item for all $v \in V(T)$, $\shom^*_{\Fuk}(S_v, S_v) = 
			\begin{cases} \mathbb{k} \text{  if  } * = 0, N, \\ 0 \text{  otherwise},
			\end{cases}$
			\item for $u \in V_+(T), w \in V_-(T)$ such that $u$ and $w$ are connected by an edge, 
			\[\shom^*_{\Fuk}(S_u,S_w) = \begin{cases}
				\mathbb{k} \text{  if  } * =1, \\ 0 \text{  otherwise}, 
			\end{cases} \shom^*_{\Fuk}(S_w,S_u) = \begin{cases}
				\mathbb{k} \text{  if  } * =N-1, \\ 0 \text{  otherwise},
			\end{cases}\]
			\item for all $u, w \in V(T)$ which do not satisfy any of the above conditions, 
			\[\shom^*_{\Fuk} (S_u, S_w) = 0.\]
		\end{itemize}
		\item Moreover, because of the degree reason, the differential map of $\hom_{\Fuk}^*(S_v, S_w)$ is the zero map for all $v, w \in V(T)$.
		Thus, $\hom_{\Fuk}^*(S_v,S_w) = \lhom_{\Fuk}^*(S_v,S_w)$. 
	\end{enumerate}
\end{lem}
Throughout the paper, we assume that the generating set $\{S_v\}_{v \in V}$ of $\Fuk$ is graded as given in Lemma \ref{lem grading}.

\begin{remark}
	\label{rmk sign is not important 2}
	\mbox{}
	\begin{enumerate}
		\item We note that after introducing the notion of twisted complex in Section \ref{section twisted complex}, $\Fuk$ can be seen as the category of twisted complexes generated from $\{S_v| v \in V(T)\}$. 
		The generation result of \cite{Abouzaid-Smith12} proves the equivalence between $\Fuk$ and the category of twisted complexes. 
		\item By Section \ref{subsection setting} and the argument above, one can observe that $\mathcal{D}(\Gamma_N T)$ and $\Fuk$ have the same generating set $\{S_v | v \in V(T)\}$. 
		Moreover, 
		\[\lhom^i_{\mathcal{D}(\Gamma_N T)}(S_v,S_w) = \lhom_{\Fuk}^i(S_v,S_w),\]
		for all $i \in \mathbb{Z}$ and $v, w \in V(T)$. 
		It proves the equivalence between $\mathcal{D}(\Gamma_N T)$ and $\Fuk$.  
		\item 	If we have the other sign convention for $T$(see Remark \ref{rmk sign conventions}) for 
		\[V(T) = V_+(T) \sqcup V_-(T),\]
		then $S_v$ in Lemma \ref{lem grading} should be the same Lagrangian sphere with different shift. 
		For convenience, let $\{R_v\}_{v \in V(T)}$ be the Lagrangian spheres defined as follows:
		\[R_v = \begin{cases}
			S_v \text{  if  } v \in V_+(T), \\ S_v[2-N] \text{  if  } v \in V_-(T).
		\end{cases}\]
		Then, Lemma \ref{lem grading} holds for another collection of Lagrangian spheres $\{R_v\}_{v \in V(T)}$.
		Thus, if we choose the other sign collection, then $\{R_v\}_{v \in V(T)}$ can replace $\{S_v\}_{v \in V(T)}$.
	\end{enumerate}
\end{remark}

For the later sections, we set the following notation for generating morphisms of $\shom_{\Fuk}^*(S_u, S_w)$:
\begin{definition}
	\label{def xyz maps}
	\mbox{}
	\begin{enumerate}
		\item For all $v \in V(T)$, let {\bf $\boldsymbol{e_v}$} denote the identity morphisms in $\shom_{\Fuk}^*(S_v, S_v)$. 
		Thus, the grading of $e_v$ is $0$, i.e., $|e_v|=0$. 
		\item For all $v \in V(T)$, let the {\bf $\boldsymbol{z}$-morphism $\boldsymbol{z_v}$} denote a nonzero morphism in $\shom_{\Fuk}^*(S_v, S_V)$ such that $|z_v|=N$. 
		\item For all $u \in V_+(T), w \in V_-(T)$ such that $u$ and $w$ are connected by an edge, let the {\bf $\boldsymbol{x}$-morphism $\boldsymbol{x_{u,w}}$} (resp.\ {\bf $\boldsymbol{y}$-morphism $\boldsymbol{y_{w,u}}$}) denote a fixed nonzero morphism in $\shom_{\Fuk}^*(S_u,S_w)$ (resp.\ $\shom_{\Fuk}^*(S_w,S_u)$). 
		In particular, $|x_{u,w}| = 1, |y_{w,u}|=N-1$. 
	\end{enumerate}
\end{definition}

With the notation in Definition \ref{def xyz maps}, Lemma \ref{lem grading} can be summarized as follows:
\begin{itemize}
	\item For all $v \in V(T)$, $\shom^*_{\Fuk}(S_v, S_v) = \mathbb{k} \langle e_v, z_v \rangle$.
	\item For $u \in V_+(T), w \in V_-(T)$ such that $u$ and $w$ are connected by an edge, 
	\[\shom^*_{\Fuk}(S_u,S_w) = \mathbb{k} \langle x_{u,w} \rangle, \shom^*_{\Fuk}(S_w,S_u) = \mathbb{k} \langle y_{w,u} \rangle.\]
	\item for all $u, w \in V(T)$ which do not satisfy any of the above conditions, $\shom^*_{\Fuk} (S_u, S_w) = 0$.
\end{itemize}

Now, we define a counterpart of Definition \ref{def Penner type} in symplectic setting. 
Since $S_v$ is a Lagrangian sphere, there is a well-known symplectic automorphism of $P_N(T)$. 
The symplectic automorphism is the {\em generalized Dehn twist along $S_v$}.
Let $\tau_v$ denote the generalized Dehn twist along $S_v$.

\begin{remark}
	We note that the Lagrangian sphere $S_v$ gives only a Hamiltonian isotopy class of $\tau_v$. 
	In order to define a specific representative of the class, one needs extra choices, for example, a Dehn twist profile. 
	For more detail, we refer the reader to \cite[Section 2.1]{Mak-Wu18}.
	For convenience, we assume that $\tau_v$ is a specific symplectic automorphism without mentioning the extra choices. 
	Since we are interested in the induced functors on $\Fuk$, it is okay not to specify a representative of Dehn twists. 
\end{remark}

One can define a class of symplectic automorphisms, which corresponds to the {\em Penner type} defined in Definition \ref{def Penner type}.
\begin{definition}
	\label{def Penner type symplectic automoprhism}
	A symplectic automorphism $\Phi: P_N(T) \to P_N(T)$ is of {\bf Penner type} if 
	\begin{enumerate}
		\item[(I)]  either $\Phi$ or $\Phi^{-1}$ is a product of the elements of the following set:
		\[\{\tau_u, \tau_w^{-1} | u \in V_+(T), w \in V_-(T)\},\]
		\item[(II)] for every $v \in V(T)$, either $\tau_v$ or $\tau_v^{-1}$ should appear in the product at least once.
	\end{enumerate}
\end{definition}
We note that a symplectic automorphism $\Phi: P_n(T) \to P_n(T)$ induces an exact autoequivalence on $\Fuk$. 
If $\Phi$ is a symplectic automorphism of Penner type, then the induced autoequivalence on $\Fuk$ is also of Penner type in the sense of Definition \ref{def Penner type}, under the equivalence between $\mathcal{D}(\Gamma_N T)$ and $\Fuk$.

Even though we study autoequivalences on $\Fuk$ in the current paper, we also will use the wrapped Fukaya category of $P_N(T)$.
Let $\wrapped$ denote (the derived category of) the wrapped Fukaya category of $P_N(T)$. 
Then, it is well-known that $\Fuk$ is a fully faithful subcategory of $\wrapped$. 
Moreover, the following facts in Lemma \ref{lem wrapped} are also well-known.
\begin{lem}
	\label{lem wrapped}
	\mbox{}
	\begin{enumerate}
		\item For each $v \in V(T)$, let $L_v$ be a cotangent fiber of $T^*S_v \subset P_N(T)$. 
		Then, $\wrapped$ is generated by
		\[\{L_v | v\in V(T)\}.\]
		See \cite{Chantraine-Rizell-Ghiggini-Golovko17, Ganatra-Pardon-Shende18a} for the proof. 
		\item For each $v \in V(T)$, one can assume that the above Lagrangian submanifold $L_v$ is properly graded so that 
		\[\shom^i_{\wrapped}(L_v,S_v)  = \begin{cases}
			\mathbb{k} \text{  if  } i=0, \\
			0 \text{  otherwise},
		\end{cases} \text{  or equivalently,    } \shom^i_{\wrapped}(S_v,L_v)  = \begin{cases}
		\mathbb{k} \text{  if  } i=N, \\
		0 \text{  otherwise},
		\end{cases}\]
		where $\shom_{\wrapped}^i$ denote the degree $i$ morphism space in $\wrapped$. 
		\item If $v, w \in V(T)$ satisfy $v \neq w$, then 
		\[\shom^i_{\wrapped}(L_v, S_w) = \shom^i_{\wrapped}(S_w,L_v) =0.\]
		\item Because of the degree reason, for any $v, w \in V(T)$, $\hom_{\wrapped}^*(L_v,S_w)$ has the zero differential map. 
		Especially, $\hom_{\wrapped}^i(L_v,S_w) = \lhom_{\wrapped}^i(L_v,S_w)$ for all $i \in \mathbb{Z}$ and $v, w \in V(T)$, where $\lhom_{\wrapped}^*(L_v,S_w)$ denotes the homology of $\hom_{\wrapped}^*(L_v,S_w)$.
	\end{enumerate}
\end{lem}

We summarize the equivalence between two settings given in Sections \ref{subsection setting} and \ref{subsection symplectic topological setting} in the following table:
\begin{table}[h!]
	\caption{Equivalences between the settings given in Sections \ref{subsection setting} and \ref{subsection symplectic topological setting}}
	\begin{tabular}{ c | c }
		\label{table category=symplectic}
		Categorical setting  & Symplectic setting \\ \hline\hline
		\shortstack{Ginzburg dg-algebra \\ $\Gamma_N T$} & \shortstack{The endomorphism space of a generator of $\wrapped$\\ $\hom_{\wrapped}^*(\bigoplus_{v \in V(T)}L_v, \bigoplus_{v \in V(T)}L_v)$}  \\ \hline
		$\mathcal{D}(\Gamma_N T)$ & $\Fuk$ \\ \hline
		The perfect derived category of $\Gamma_N T$ & $\wrapped$ \\ \hline	
		Spherical twists & Dehn twists \\ \hline
		Penner type autoequivalences on $\mathcal{D}(\Gamma_N T)$ & \shortstack{(The autoequivalences on $\Fuk$ induced by) \\ Penner type symplectic automorphisms} \\ \hline 
	\end{tabular}
\end{table}

We note that {\em in the rest of the paper, we will work with the notation given in Section \ref{subsection symplectic topological setting}}, based on the equivalence given in Table \ref{table category=symplectic}.
Especially, our category will be denoted by $\Fuk$, and we will heavily use the $x, y, z$-morphisms defined in Definition \ref{def xyz maps}.

\subsection{The main idea}
\label{subsection the main idea}
In this subsection, we explain the motivating idea for Theorem \ref{thm pseudo-Anosov intro}. 
Since the proof of Theorem \ref{thm pseudo-Anosov intro} given in Sections \ref{section preparations for the proof of pseudo-Anosov theorem}--\ref{section strong pseudo-Anosov} is technical, we would like to give a background idea arising from symplectic topology before starting the technical parts.
However, because the goal is to introduce the idea, the contents in this subsection do not contain all the details. 
It would not be a bad idea to skip this subsection for the first read, and come back after reading preliminary sections, i.e., Sections \ref{section twisted complex}-\ref{section actions on the space of stability conditions}.

Let $T$ be a tree and let $\Phi: \Fuk \to \Fuk$ be an autoequivalence of Penner type. 
To prove Theorem \ref{thm pseudo-Anosov intro}, we need to keep track of the changes in the sequence $\Phi^n E$ as $n$ increases. 
The changes we would like to keep track are those of mass and maximal/minimal phases with respect to a stability condition. 

Let us recall that our category $\Fuk$ could be seen as the Fukaya category of $P_N(T)$. 
In \cite{Lee19}, the second-named author studied the changes of the sequence $\Phi^m(L)$ and proved that $\Phi$ satisfies a generalized version of Property \ref{property 1}, where $\Phi$ is a symplectic automorphism of Penner type and $L \subset P_N(T)$ is a Lagrangian submanifold. 
The main idea of \cite{Lee19} is to generalize Thurston's idea of proving Property \ref{property 1}. 
Moreover, by \cite{Lee19}, for sufficiently large $m$, it would be natural to expect that the following holds:
\begin{enumerate}
	\item[($\star$)] As an object of $\Fuk$, $\Phi^m(L)$ is generated by the zero sections $\{S_v|v \in V(T)\}$ by taking direct sums, shifts, and cones of $x$ and $z$-morphisms. 
	Especially, one does not need to take a cone of $y$-morphism to generate $\Phi^m(L)$.
\end{enumerate}
The main idea of proving Theorem \ref{thm pseudo-Anosov intro} originated from the expectation ($\star$).
In the rest of the subsection, we explain how ($\star$) motivates Theorem \ref{thm pseudo-Anosov intro}, as well as why it is natural to expect ($\star$).

First, let us introduce the idea of Thurston for proving Property \ref{property 1} briefly. 
For more details, we refer the reader to \cite[Chapter 15]{Farb-Margalit12}. 
If a pseudo-Anosov surface map $\Phi: \Sigma \to \Sigma$ is given, the idea is to use a combinatorial tool called {\em measured train track} in order to keep track the change of $\Phi^m(C)$, where $C$ is a simple closed curve in $\Sigma$.
The definition of measured train track will appear after introducing the definitions of trains track and weight.
\begin{definition}
	\label{def train track} 
		A {\bf train track} in a surface $\Sigma$ is a smoothly embedded trivalent graph in $\Sigma$. 
		The term ``smoothly embedded" means that at all points, including vertices, of a train track, the tangent line is well-defined.  
\end{definition}
An example and non-example of train track are given in Figure \ref{figure train track} (a) and (b). 
We note that at each vertex of a train track, the three (half-)edges meeting at the vertex are divided into two sets, one on each side of the vertex. 

\begin{figure}
	\centering
	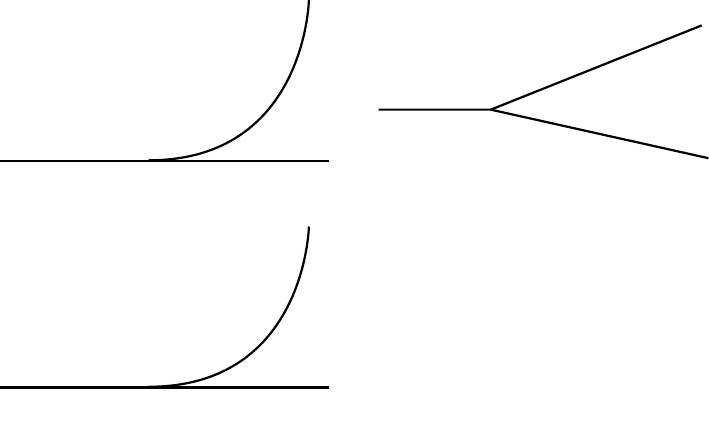		
	\caption{(a) is an example of train track, (b) is a non-example of train track,
	(c) is an example of weight, and (d) shows how to construct an isotopy class of curves from weights on a train track.}
	\label{figure train track}
\end{figure}

\begin{definition}
	\label{def measured train track}
	\mbox{}
	\begin{enumerate}
		\item A {\bf weight} of a train track is a set of non-negative integers assigned to edges of the train track. 
		\item A weight satisfies the {\bf switch condition} if the sums of the weights on each side of each vertex are equal to each other.
		\item A {\bf measured train track} is a pair of a train track and a weight on the train track satisfying the switch condition. 
	\end{enumerate}
\end{definition}
\begin{remark}
	We note that in Definition \ref{def measured train track} (1), we follow the convention of \cite{Farb-Margalit12}. 
	Sometimes, a weight is defined to be a collection of non-negative {\em real numbers}, instead of integers.
\end{remark}

We would like to point out that if one has a measured train track, then one can construct an isotopy class of a curve from the given measured train track, as follows: 
In a small neighborhood of the train track, one can consider parallel copies of each edges. 
The number of copies of an edge is the same as the weight assigned to the edge. 
Then, because of the switch condition, there is a unique way to connect the parallel copies of edges. 
By connecting them, one has an isotopy class corresponding to the given measured train track. 
See Figure \ref{figure train track} (c) and (d) for an example. 

If an isotopy class of a curve is obtained from a measured train track, i.e., a pair of a train track $\mathcal{R}$ and a weight on $\mathcal{R}$, by the above procedure, then we can roughly think that the isotopy class can be {\em encoded on} the train track $\mathcal{R}$.
More formally, we define the notion of {\em carried by}.
\begin{definition}
	\label{def carried by}
	A curve $C \subset \Sigma$ is {\bf carried by} a train track $\mathcal{R}$ if there is a weight $w_C$ on $\mathcal{R}$ such that the measured train track $(\mathcal{R}, w_C)$ gives the isotopy class of $C$ through the above process. 
\end{definition}

If $\Phi: \Sigma \to \Sigma$ is a pseudo-Anosov surface map, it is known that there is a train track $\mathcal{R}_\Phi$ and a matrix $M_\Phi$ such that 
\begin{enumerate}
	\item[$(\square)$] if a curve $C$ is carried by the train track $\mathcal{R}_\Phi$ with a weight $w_C$, then $\Phi(C)$ is also carried by $\mathcal{R}_\Phi$ with the weight $M_\Phi \cdot w_C$. 
\end{enumerate}
We note that since the weight $w_C$ is a set of non-negative integers, we can see $w_C$ as a vector in a Euclidean space $\mathbb{R}^K$ where $K$ is the number of edges in $\mathcal{R}_\Phi$.
Then, the matrix multiplication $M_\Phi \cdot w_C$ makes sense, and to keep track of changes of $\Phi^n(C)$ as $n$ increases, it is enough to study the changes of $M_\Phi^n \cdot w_C$ as $n \to \infty$. 
Moreover, $\Phi$ also satisfies that,
\begin{enumerate}
	\item[$(\blacksquare)$] for any simple closed curve $C$, there exists a natural number $m$ such that $\Phi^m(C)$ is carried by $\mathcal{R}_\Phi$. 
\end{enumerate}
Thus, for any simple closed curve $C$, the matrix $M_\Phi$ could be a tool of keeping track of the asymptotic behavior of $\Phi^n(C)$ as $n \to \infty$.

In \cite{Lee19}, the second-named author generalized the above arguments in a symplectic setting. 
More precisely, if $\Phi: P_N(T) \to P_N(T)$ is a symplectic automorphism of Penner type, \cite{Lee19} constructed a generalization of the train track $\mathcal{R}_\Phi$ associated to $\Phi$, called {\em Lagrangian branched submanifold}, such that if a Lagrangian $L$ is carried by the Lagrangian branched submanifold associated to $\Phi$, then $\Phi(L)$ is also carried by it. 
We note that the notion of ``carried by" in symplectic setting is a generalized one, and is slightly different from the notion in Definition \ref{def carried by}.

For convenience, let $\mathcal{L}_\Phi$ denote the Lagrangian branched submanifold associated to $\Phi$. 
Let us recall that on a surface $\Sigma$, if a curve $C$ is carried by a train track $\mathcal{R}$, then the isotopy class of $C$ can be recovered for the corresponding weight and $\mathcal{R}$.
And, at least locally, $C$ looks like a parallel copy of $\mathcal{R}$. 
Similarly, if $L$ is carried by $\mathcal{L}_\Phi$, then we expect that $L$ seems {\em similar} to $\mathcal{L}_\Phi$.

Let us briefly review the construction of $\mathcal{L}_\Phi$.
To do that, let us fix a simple example.
The example we consider is the plumbing of two $T^*S^N$ along the Dynkin diagram of $A_2$-type, or simply $P_N(A_2)$. 
Because $A_2$ has two vertices, there exist one positive and one negative vertex. 
Let $S_u$ (resp.\ $S_w$) denote the zero section of $T^*S^N$ corresponding to the positive (resp.\ negative) vertex.
And, let $\tau_u$ (resp.\ $\tau_w$) denote the Dehn twist along $S_u$ (resp.\ $S_w$).
In this setting, 
\begin{itemize}
	\item if a Penner type symplectic automorphism $\Phi: P_N(A_2) \to P_N(A_2)$ is a product of $\tau_u$ and $\tau_w^{-1}$, $\mathcal{L}_\Phi$ is either $\left(S_u \cup \tau_u(S_w)\right)$ or $\left(S_w \cup \tau_u(S_w)\right)$, and 
	\item if a Penner type symplectic automorphism $\Phi^{-1}$ is a product of $\tau_u$ and $\tau_w^{-1}$, $\mathcal{L}_\Phi$ is either $\left(S_w \cup \tau_w(S_u)\right)$ or $\left(S_u \cup \tau_w(S_u)\right)$.
\end{itemize}

Let assume that $\Phi$ is the first case, i.e., $\Phi$ is a product of positive Dehn twists and the inverses of negative Dehn twists.
We would like to emphasize that $\tau_u(S_w) = \tau_w^{-1}(S_u)$.  
We also note that $\tau_u(S_w) = \tau_w^{-1}(S_u)$ is equivalent to a cone of the $x$-morphism $\in \lhom_{\Fuk}^*(S_u,S_w)$. 
Thus, if $L$ is {\em similar} to $\mathcal{L}_\Phi$ in $\Fuk$ {\em categorically} as we expected, then $L$ is generated from $\{S_v\}_{v \in V(T)}$ by taking direct sums, shifts, and 
\begin{itemize}
	\item cones of $z$-morphisms, which corresponds to $S_u$ or $S_w$ part of the Lagrangian branched submanifold, and
	\item cones of $x$-morphisms, which corresponds to $\tau_u(S_w)=\tau_w^{-1}(S_u)$ part of the Lagrangian branched submanifold. 
\end{itemize}
In other words, one can expect that $(\star)$ holds. 

We note that for a general tree $T$, not just $A_2$, a similar construction of $\mathcal{L}_\Phi$ holds. 
For more details, see \cite[Section 3]{Lee19}.
Thus, for a general $T$ and a Penner type $\Phi$, we can expect that if $L$ is carried-by $\mathcal{L}_\Phi$, the $(\star)$ holds. 

For convenience, we define the following:
\begin{definition}
	\label{def categorically carried by}
	Let $T$ be a tree, and let $E$ be a nonzero object of $\Fuk$. 
	We say that $E$ is {\bf categorically carried-by} if $E$ satisfies the condition $(\star)$, i.e., $E$ is generated from $\{S_v\}_{v \in V(T)}$ by taking the direct sums, shifts, and cones of $x$- and $z$-morphisms. 
\end{definition}

\begin{remark}
	\label{rmk categorically carried by} 
	\mbox{}
	\begin{enumerate}
		\item We note that Definition \ref{def categorically carried by} is a brief introduction to the notion of categorically carried-by and the notion is more rigorously defined in Section \ref{subsection sketch of the proof}.
		See Definition \ref{def categorically carried by 2}.
		\item We would like to point out that it would be more natural to define the term {\em categorically carried-by ``something.''}
		However, we remark that to the best of the author's knowledge, the categorical counter part of Lagrangian branched submanifold is not defined/studied in the literature. 
		In order to stress out the similarity, we use the expression ``$E$ is categorically carried-by'' without specifying what carries the object $E$.
		\item We also would like to point out that the condition ($\star$) is dependent on the choice of the sign convention. 
		By taking it into account, one can use the following expression:
		$E$ is categorically carried-by the {\em sign convention}. 
	\end{enumerate}
\end{remark}

We can also expect that if $L$ is categorically carried-by, a categorical generalization of $(\square)$ holds.
To give a formal statement of the expected generalization of $(\square)$, let us recall that every object of $\Fuk$ is generated by $\{S_v | v \in V(T)\}$.
Let $a_v$ be the minimal number of $S_v$ we need to generate $E$, and let $\vect(E) = \left(a_v\right)_{v \in V(T)}$. 
Then, the expected generalization is the following:
\begin{enumerate}
	\item[$(\square')$] If $E$ is categorically carried-by, then so is $\Phi E$.
	Moreover, there exists a matrix $M_\Phi$ such that $\vect\left(\Phi(L)\right) = M_\Phi \cdot \vect\left(L\right)$. 
\end{enumerate}
If the above expectation holds, then at least for a categorically carried-by $E$, it is enough to analyze $M_\Phi$ to study the changes in the sequence $\Phi^n E$ as $n$ increases.
It explains why $(\star)$ motivates Theorem \ref{thm pseudo-Anosov intro}.

In order to complete the proof of Theorem \ref{thm pseudo-Anosov intro}, we also should consider an object $E$ which is {\em not} categorically carried-by. 
For such a $E$, one can expect that $(\blacksquare)$ could be generalized, i.e., 
\begin{enumerate}
	\item[$(\blacksquare')$] for any object $E$, there exists a natural number $m$ such that $\Phi^n E$ is categorically carried-by, or equivalently, for generating $\Phi^n E$, taking cone of a $y$-morphism is not needed.
\end{enumerate}

However, one can easily find an object $E$ such that $(\blacksquare')$ does not hold. 
We expect that the problem happens when we convert the geometric setting to the categorical one. 
To be more precise, let us note that our triangulated category $\Fuk$ can have {\em nongeometric} objects, i.e., objects which {\em cannot} be represented by a Lagrangian submanifold.
For those nongeometric objects, the geometric motivation from the surface theory and \cite{Lee19} could not work.
We also note that the example of $E$ not satisfying $(\blacksquare')$ is a nongeometric object. 

In order to handle the objects $E$ such that $(\blacksquare')$ does not hold, we focus on the number of (cones of) $y$-morphisms in the generation of $\Phi^n E$. 
Since $\Phi^n E$ is not categorically carried-by, the number of $y$-morphisms is not zero, but we can show that the number {\em decreases} as $n$ increases. 
For more details, see Section \ref{section proof of thm pseudo-Anosov}. 
Than, based on the observation, we will prove that for any object $E$, if $n$ is large enough, $\Phi^n E$ can be divided into two parts, one containing all $y$-morphisms and the other which is independent of (the effects of) $y$-morphisms, or equivalently, categorically carried-by. 
See Lemma \ref{lem partially carried-by}.

\begin{remark}
	\label{rmk weak pseudo-Anosov}
	In \cite[Section 3.1]{Fan-Filip-Haiden-Katzarkov-Liu21}, the authors proved that if $T$ is the $A_2$-quiver, and if $N \geq 3$ is an odd number, then every Penner type autoequivalence is {\em weak pseudo-Anosov}. 
	The idea in \cite{Fan-Filip-Haiden-Katzarkov-Liu21} is to keep track of the asymptotic behavior of Penner type autoequivalence $\Phi$, i.e., $\Phi^n E$, on the Grothendieck group. 
	However, if the equivalence class of $E$ is zero in the Grothendieck group, then one cannot keep track of the asymptotic behavior. 
	Because of the difficulty, \cite{Fan-Filip-Haiden-Katzarkov-Liu21} considered odd $N$ only, and even for an odd $N$, $\Phi$ is {\em weak} pseudo-Anosov. 
\end{remark}

Based on the idea given in the current subsection, we prove Theorem \ref{thm pseudo-Anosov intro} in Sections \ref{section preparations for the proof of pseudo-Anosov theorem}--\ref{section proof of thm pseudo-Anosov}. 
Before that, we explain preliminaries in Sections \ref{section twisted complex}--\ref{section actions on the space of stability conditions}.

\section{Twisted complex}
\label{section twisted complex}
Section \ref{section twisted complex} consists of two subsections. 
Section \ref{subsection definition of twisted complex} is a preliminary section in where we will define the notion of twisted complex, set notation, and introduce well-known properties of twisted complexes. 
In Section \ref{subsection twisted complexes in the Fukaya category of a plumbing space}, we will apply the contents of Section \ref{subsection definition of twisted complex} to our setting, i.e., the category $\Fuk$ associated to a tree $T$. 

\subsection{Definition of Twisted complex} 
\label{subsection definition of twisted complex}
In this subsection, we briefly review the notion of twisted complexes. 
Since there are many references for the notion, we omit details. 
However, we state Lemma \ref{lem induced functor} that plays an important role in the paper. 
We note that most of this subsection follows \cite[Chapter 3]{Seidel08}.

We fix an $\mathcal{A}_\infty$-category $\mathcal{C}$ over a field $\mathbb{k}$ in this subsection. 
\begin{definition}
	\label{def twisted complex}
	A {\bf twisted complex} $E$ in $\mathcal{C}$ is a pair 
	\[E= \left(\bigoplus_{i\in I} V \otimes X_i, F = (F_{j,i})\right),\]
	such that 
	\begin{itemize}
		\item the index set $I$ is an ordered finite set (and, for simplicity, we always assume that $I =\{1, 2, \dots, k\}$ for some $k \in \mathbb{N}$),
		\item $V_i$ is a finite-dimensional $\mathbb{k}$-vector space, 
		\item $X_i$ is an object of $\mathcal{C}$, 
		\item for every $j, i \in I$, $F_{j,i} \in \mathrm{Mor}(V_i, V_j) \otimes \shom^1_\mathcal{C}(X_i,X_j)$ where $\mathrm{Mor}(V_i,V_j)$ is the space of $\mathbb{k}$-linear maps from $V_i$ to $V_j$, 
		\item $F$ is {\em strictly lower-triangular}, i.e., for every $j \leq i$, $F_{j,i} = 0 \in \mathrm{Mor}(V_i,V_j) \otimes \shom^1_\mathcal{C}(X_i,X_j)$,  
		\item $F$ satisfies the {\em generalized Maurer--Cartan equation}.
	\end{itemize}
\end{definition}
For more details, for example, the generalized Maurer--Cartan equation, see \cite[Chapter 3]{Seidel08}. 
For simplicity, we assume that if 
\[\psi_{j,i} \otimes f_{j,i} = 0 \in \mathrm{Mor}(V_i, V_j) \otimes \hom_\mathcal{C}(X_i, X_j),\] then $\psi_{j,i}=0$ and $f_{j,i}=0$. 

Let $\Tw(\mathcal{C})$ denote the category of twisted complexes in $\mathcal{C}$.
Then, it is well-known that $\Tw(\mathcal{C})$ also admits an $\mathcal{A}_\infty$-category structure.
Let $\left\{\mu_{\Tw(\mathcal{C})}^d\right\}_{d \in \mathbb{N}}$ denote the higher structure maps of $\Tw(\mathcal{C})$.

\begin{definition}
	\label{def chain of arrows}
	Let $E = \left(\bigoplus_{i \in I} V_i \otimes X_i, (F_{j,i})\right)$ be a twisted complex in $\mathcal{C}$.
	A {\bf chain of arrows of $\boldsymbol{E}$ from $\boldsymbol{V_i \otimes X_i}$ to $\boldsymbol{V_j \otimes X_j}$} is a subcollection 
	\[\left\{F_{i_1,i_0}, F_{i_2, i_1}, \dots, F_{i_k,i_{k-1}}\right\} \subset \left\{F_{j,i} | i, j \in I\right\}.\]
	We note that $V_i \otimes X_i$ is regarded a twisted complex in $\mathcal{C}$ with the index set $\{i\}$.
\end{definition}

Before going further, we note that we can express a twisted complex $E= \left(\bigoplus_{i\in I} V \otimes X_i, F = (F_{j,i})\right)$ as a quiver equipped with extra data, satisfying the following: 
\begin{itemize}
	\item The quiver has $|I|$-many vertices. Let $\{v_i\}_{i \in I}$ denote the set of vertices. 
	\item For each $i <j$, there is an arrow from $v_i$ to $v_j$. Let $e_{j,i}$ denote the arrow from $v_i$ to $v_j$. 
	\item The extra data are assignments of $V_i \otimes X_i$ and $F_{j,i}$ to $v_i$ and $e_{j,i}$.
\end{itemize}
An example is  
\[\left(X_1 \oplus X_2 \oplus X_3, F_{j,i}\right) = \begin{tikzcd}
	X_1 \ar[r, "F_{2,1}"'] \ar[rr, bend left, "F_{3,1}"] & X_2 \ar[r, "F_{3,2}"'] & X_3.
\end{tikzcd}\]

The main goal of Section \ref{subsection definition of twisted complex} is to state Lemma \ref{lem induced functor}.
\begin{lem}
	\label{lem induced functor}
	Let $\Phi: \mathcal{C} \to \mathcal{C}$ be an autoequivalence.
	Then, there is an induced functor $Tw\Phi: Tw\mathcal{C} \to Tw\mathcal{C}$ where $Tw\mathcal{C}$ is the category of twisted complexes in $\mathcal{C}$. 
	Moreover, if $E =\left(\bigoplus_i V_i \otimes X_i, (F_{j,i})\right)$ is a twisted complex in $\mathcal{C}$, then there exists $G_{j,i}$ such that $Tw\Phi(E) = \left(\bigoplus_i V_i \otimes \mathcal{F}(X_i), G=(G_{j,i})\right)$.
\end{lem}
\begin{proof}
	See \cite[Chapter (3m)]{Seidel08}.
\end{proof}

\begin{remark}
	We note that for simplicity, if the induced functor $Tw \Phi$ on $Tw \mathcal{C}$ from a functor $\Phi: \mathcal{C} \to \mathcal{C}$ will be denoted $\Phi$ again. 
	We also note that for a twisted complex $E \in Tw(\mathcal{C})$, $\Phi(E)$ denotes the twisted complex obtained by applying Lemma \ref{lem induced functor}.
	In other words, if $E = \left(\bigoplus_{i \in I} V_i \otimes X_i, (F_{j,i})\right)$, then $\Phi(E) =\left(\bigoplus_i V_i \otimes \mathcal{F}(X_i), G=(G_{j,i})\right)$.
\end{remark}

We also define the notion of sub-twisted complex.
\begin{definition}
	\label{def subcomplex}
	Let $E = \left(\bigoplus_{i=1}^K V_i \otimes X_i, (F_{j,i}) \right)$ be a twisted complex. 
	A twisted complex $E' = \left(\bigoplus_{j=1}^{K'} W_j \otimes Y_j, G_{j,i} \right)$ is a {\bf sub-twisted complex of $\boldsymbol{E}$} if there exists $i_0$ such that 
	\begin{itemize}
		\item for all $1 \leq j \leq K'$, $W_j = V_{i_0 + j}$ and $Y_j = X_{i_0 + j}$, and
		\item for all pair $1 \leq i <j \leq K'$, $G_{j,i} =F_{i_0+j, i_0 +i}$.  
	\end{itemize} 
\end{definition}

\begin{remark}
	\label{rmk definition of subtwisted complex}
	\mbox{}
	\begin{enumerate}
		\item By Lemma \ref{lem induced functor} and Definition \ref{def subcomplex}, one can check that if $E'$ is a sub-twisted complex of $E$, then $\Phi E'$ is a sub-twisted complex of $\Phi E$. 
		\item We also remark that the notion of sub-twisted complex is {\em different} from the notion of {\em subcomplex}.
		One can find the definition of subcomplex in \cite[Section (3l)]{Seidel08}.
	\end{enumerate}
\end{remark}

\subsection{Twisted complexes in $\Fuk$}
\label{subsection twisted complexes in the Fukaya category of a plumbing space}
Throughout Section \ref{subsection twisted complexes in the Fukaya category of a plumbing space}, let $N \geq 3$ be an integer, and let $T$ denote a tree. 
We also assume that the sign convention for $V_+(T)$ and $V_-(T)$ is given. 
In the setting, we will prove Lemma \ref{lem minimal twisted complex}, i.e., every twisted complex in $\Fuk$ is equivalent to a twisted complex of a specific form. 

First, we recall that since $\Fuk$ is generated by $\{S_v|v \in V(T)\}$ and because of Lemma \ref{lem grading}, every object of $\Fuk$ is equivalent to a {\em good} twisted complex defined as follows:
\begin{definition}
	\label{def good twisted complex}
	A twisted complex $E=\left(\bigoplus_{i \in I} V_i \otimes X_i, F_{j,i} \right)$ in $\Fuk$ is {\bf good} if, for every $i$, $X_i = S_{v_i}[d_i]$ for some $v_i \in V(T)$ and for some $d_i \in \mathbb{Z}$.
\end{definition} 

In the most part of the current article, {\em a twisted complex in $\Fuk$} means a {\em good} twisted complex. 
The advantage of considering good twisted complexes is that the arrows of good twisted complexes can be written in a simpler form. 
See Proposition \ref{prop good twisted complex}. 
\begin{prop}
	\label{prop good twisted complex}
	Let $E=\left(\bigoplus_{i \in I} V_i \otimes X_i, F_{j,i} \right)$ be a good twisted complex in $\Fuk$. 
	Then, for all $i < j \in I$, there exist $\psi_{j,i} \in \mathrm{Mor}(V_i,V_j)$ and $f_{j,i} \in \shom_{\Fuk}^1(S_{v_i}[d_i], S_{v_j}[d_j])$ such that
	\begin{itemize}
		\item $F_{j,i} = \psi_{j,i} \otimes f_{j,i}$, and
		\item $f_{j,i}$ is one of $x, y, z$-morphisms, $e_v$, or the zero morphism.    
	\end{itemize}
\end{prop}
\begin{proof}
	Since $E$ is a good twisted complex, for every $i \in I$, $X_i$ is a shift of $S_v$ for some $v \in V(T)$. 
	Thus, Lemma \ref{lem grading} implies that 
	\[\dim\hom_{\Fuk}^1\left(X_i, X_j\right) = 1 \text{  or  } 0.\]
	Thus, every element of $\mathrm{Mor}(V_i, V_j) \otimes \hom_{\Fuk}(X_i, X_j)$ can be written in the desired way. 
\end{proof}
In the rest of the paper, we assume that every good twisted complex in $\Fuk$ satisfies the two conditions of Proposition \ref{prop good twisted complex}.
Then, for a good twisted complex, we can define the notion of {\em nonzero chain of arrows}.
\begin{definition}
	\label{def nonzero chain of arrows}
	Let $E = \left(\bigoplus_{i \in I} V_i \otimes S_{v_i} [d_i], F_{j,i} = \psi_{j,i} \otimes f_{j,i}\right)$. 
	Then, a chain of arrows $\left\{F_{i_1,i_0}, \dots, F_{i_k,i_{k-1}}\right\}$ is {\bf nonzero} if the composition of linear parts of the chain of arrows is nonzero, i.e., 
	\[\psi_{i_k,i_{k-1}} \circ \dots \circ \psi_{i_1,i_0} \neq 0.\]
\end{definition}

\begin{remark}
	\label{rmk nonzero chain of arrows}
	The importance of Definition \ref{def nonzero chain of arrows} arises when we compute $\Phi(E)$ where 
	\[E = \left(\bigoplus_{i \in I} X_i, F_{j,i} = \psi_{j,i} \otimes f_{j,i}\right) \in \Fuk\]
	is a good twisted complex and $\Phi$ is an autoequivalence on $\Fuk$, by applying Lemma \ref{lem induced functor}. 
	By applying Lemma \ref{lem induced functor}, we know that 
	\[\Phi(E) \simeq \left(\bigoplus_{i \in I} \Phi(X_i), G_{j,i}\right),\]
	for some $G_{j,i}$. 
	Moreover, \cite[Chapter (3m)]{Seidel08} provides a formula for $G_{j,i}$ and the formula is 
	\begin{gather}
		\label{eqn induced functor}
		G_{j,i} = \sum_{i_0 =i < i_1 < \dots < i_k=j} (-1)^* \psi_{i_k,i_{k-1}} \circ \dots \circ \psi_{i_1, i_0} \Phi^k\left(f_{i_k,i_{k-1}}, \dots, f_{i_1,i_0}\right).
	\end{gather}
	We note that in \cite{Seidel08}, $\Phi$ is a $\mathcal{A}_\infty$-functor, thus $\Phi^k$ is defined as a part of the $\mathcal{A}_\infty$-functor.
	We also note that we are omitting the details on the sign $(-1)^*$. 
	By the formula \eqref{eqn induced functor}, $G_{j,i}$ is determined by the nonzero chain of arrows from $X_i$ to $X_j$. 
	Especially, if there exists no nonzero chain of arrows from $X_i$ to $X_j$, then $G_{j,i}$ should be the zero map. 
\end{remark}

We also define a {\em minimal} twisted complex.
\begin{definition}
	\label{def minimal twisted complex}
	A good twisted complex $E=\left(\bigoplus_{i \in I} V_i \otimes S_{v_i}[d_i], (\psi_{j,i} \otimes f_{j,i}) \right)$ in $\Fuk$ is {\bf minimal} if there exists no $j >i$ such that $f_{j,i} = e_v$.
\end{definition}

Lemma \ref{lem zero differential for minimal} explains why we need minimal twisted complexes. 
Before that, let us remark that $\Fuk \subset \wrapped$. 
Thus, $\Tw(\Fuk) \subset \Tw(\wrapped)$ and if $E$ is a twisted complex in $\Fuk$, $E$ is also a twisted complex in $\wrapped$. 

\begin{lem}
	\label{lem zero differential for minimal} 
	Let $E$ be a minimal twisted complex in $\Fuk$. 
	Then, the morphism space \[\shom^*_{\Tw(\wrapped)}\left(E, L_v\right)\] has the zero differential for all $v \in V(T)$. 
\end{lem}
\begin{proof}
	In this proof, we will work with the higher product structures of $\mathcal{A}_\infty$-categories $\wrapped$ and $\Tw(\wrapped)$. 
	Let $\{\mu^d_{\wrapped}\}_{d \in \mathbb{N}}$ (resp.\ $\{\mu^d_{\Tw(\wrapped)}\}_{d \in \mathbb{N}}$) denote the higher structure maps on $\wrapped$ (resp.\ $\Tw(\wrapped)$).
		
	We note that for a fixed $v \in V(T)$, $L_v \simeq \mathbb{k} \otimes L_v$ is a twisted complex in $\wrapped$.  
	Let 
	\[E=\left(\bigoplus_{i \in I} X_i=(V_i \otimes S_{v_i}[d_i]), (F_{j,i}=\psi_{j,i} \otimes f_{j,i}) \right)\] 
	be a minimal twisted complex in $\Fuk$.
	Then, the morphism space in $\Tw(\wrapped)$  
	\[\shom^*_{\Tw(\wrapped)}(E, L_v) = \bigoplus_{v_{i_s} =v} \otimes \shom^*_{\wrapped}(S_{v_{i_s} = v}[d_{i_s}],L_v) = \bigoplus_{v_{i_s} =v} V_{i_s} \otimes \mathbb{k} [-d_{i_s} - N]\]
	is a chain complex of vector spaces with the differential map $\mu^1_{\Tw(\wrapped)}$. 
	
	Now let us assume that $\mu^1_{\Tw(\wrapped)}(\vec{v} \otimes f) \neq 0$ for some $\vec{v} \otimes f \in V_{i_s} \otimes \shom^*_{\wrapped}(S_v[d_{i_s}],L_v)$. 
	Since $\mu^1_{\wrapped}(f) =0$ for any $f \in \shom^*_{\wrapped}(S_v[d_{i_s}],L_v)$,  there should be an increasing sequence $j_1, j_2, \dots, j_m \in \{i_1, \dots, i_k\}$ of length $m \geq 2$ such that
	\begin{itemize}
		\item $j_m = i_s$,
		\item $0 \neq \mu^m_{\wrapped}(f, f_{j_m,j_{m-1}}, \dots, f_{j_2,j_1}) \in \hom^*_{\wrapped}\left(S_v[d_{j_m}],L_v\right)$. 
	\end{itemize} 
	Then, $\mu^1_{\Tw(\wrapped)}\left(V_{i_s = j_m} \otimes \shom^*_{\wrapped}(S_v[d_{i_s =j_m}],L_v)\right)$ will hit $V_{j_1} \otimes \shom^*_{\wrapped}\left(S_v[d_{j_1}],L_v\right)$.
	
	Since $\mu^1_{\Tw(\wrapped)}$ is a degree $1$ map, and since 
	\[\shom^*_{\wrapped}(S_v[d_{j_m}],L_v) = \mathbb{k}[-d_{j_m}-N] \text{  and  } \shom^*_{\wrapped}(S_v[d_{j_1}],L_v) = \mathbb{k}[-d_{j_1}-N],\]
	we have an equation $-d_{j_m} -N +1 = -d_{j_1} -N$, or equivalently, $d_{j_m}  = d_{j_1} +1$.
	
	On the other hand, because $E$ is a minimal twisted complex, $f_{j_{t+1},j_t} \in \shom^*_{\wrapped}(S_v[d_{j_{t+1}}],S_v[d_{j_t}])$ is a $z$-morphism.
	Also, since all arrows in the twisted complex have degree $1$, $f_{j_{t+1},j_t}$ has degree $1$. 
	These two facts implies that 
	\[f_{j_{t+1},j_t} \in \shom_{\wrapped}^N(S_v,S_v) = \shom_{\wrapped}^{N+d_{j_t} - d_{j_{t+1}}}(S_v[d_{j_t}], S_v[d_{j_{t+1}}])=\shom_{\wrapped}^1(S_v[d_{j_t}], S_v[d_{j_{t+1}}]).\]
	
	As a consequence of the above arguments, we have 
	\[d_{j_m} = d_{j_1} + (m-1)(N-1).\]
	Since $N \geq 3, m \geq 2$, it is a contradiction to $d_{j_m} = d_{j_1} +1$.
	Thus, $\mu^1_{\Tw(\wrapped)}(\vec{v} \otimes f) = 0$ for any $\vec{v} \otimes f \in V_{i_s} \otimes \shom^*_{\wrapped}(S_v[d_{i_s}],L_v)$.
	It implies that the differential map of $\shom^*_{\Tw(\wrapped)}(E, L_v)$ is the zero map for any $v \in V(T)$. 
\end{proof}

Lemma \ref{lem minimal twisted complex} is the main Lemma of the current subsection.
Before discussing the main Lemma, we define Definition \ref{def well-ordered} which appears in the statement of Lemma \ref{lem minimal twisted complex}.

\begin{definition}
	\label{def well-ordered}
	A minimal twisted complex $E=\left(\bigoplus_{i \in I} V_i \otimes S_{v_i}[d_i], (F_{j,i}) \right)$ is {\bf well-ordered} if
	\begin{itemize}
		\item $d_i \leq d_{i+1}$ for all $i$, and 
		\item if $d_i = d_{i+1}$ and $v_i \in V_-(T)$, then $v_{i+1} \in V_-(T)$.
	\end{itemize}
\end{definition}

\begin{lem}
	\label{lem minimal twisted complex}
	\mbox{}
	\begin{enumerate}
		\item[(i)] Every minimal twisted complex $E$ in $\Fuk$ is equivalent to a minimal, well-ordered twisted complex $E'$.
		Moreover, one can choose $E'$ in the way that $E$ and $E'$ have the same nonzero arrows.
		\item[(ii)] For all object $X \in \Fuk$, there exists a minimal twisted complex $E$ such that $X$ and $E$ are equivalent.  
	\end{enumerate}
\end{lem}
\begin{proof}
	By Lemma \ref{lem grading}, (i) is easy to prove. 
	To be more precise, let $E=\left(\bigoplus_{i \in I} V_i \otimes S_{v_i}[d_i], (F_{j,i}) \right)$ be a minimal twisted complex with $i$ such that $d_i > d_{i+1}$. 
	Then, by Lemma \ref{lem grading}, $f_{i+1,i}$ should be the zero map. 
	Thus, one can switch the position of $V_i \otimes S_{v_i}[d_i]$ and $V_i \otimes S_{v_{i+1}}[d_{i+1}]$ in $E$ without affecting the quasi-equivalent class of $E$. 
	Since the index set $I$ is a finite set, by repeating the above procedure finitely many times, we can obtain a minimal well-ordered twisted complex $E'$ such that $E \simeq E'$.
	We remark that when we convert $E$ to $E'$, we only change the positions of objects and we do not change any nonzero arrows. 
	
	In order to prove (ii), we recall that every $X \in \Fuk$ is equivalent to a good twisted complex. 
	Thus, without loss of generality, let us assume that $X$ is a good twisted complex, i.e., 
	\[X=\left(\bigoplus_{i \in I} V_i \otimes S_{v_i}[d_i], (F_{j,i}=\psi_{j,i} \otimes f_{j,i}) \right).\]

	We will prove (ii) by the induction on the cardinality of the index set $I$. 
	If $|I|= 1$, then it is easy to show that (ii) holds. 
	
	Now, let assume that $|I| = k$ and the induction hypothesis holds.  
	We apply the following two-step algorithm finding a minimal twisted complex that is equivalent to $X$. 
	
	The first step is to check the minimality of $X$.
	If $X$ is minimal, the algorithm stops immediately. 
	If $X$ is not minimal, then there is a pair of integer $i_0, j_0 \in I$ such that 
	\begin{itemize}
		\item $f_{j_0,i_0} = e_v$ for some $v \in V(T)$, and
		\item $j_0 - i_0 = \min \{j-i | f_{j,i} = e_v \text{  for some  } v \in V(T)\}$.
	\end{itemize}

	The second step consists of two cases. 
	The first case is that the pair $(i_0,j_0)$ satisfies $j_0 > i_0 +1$, and the second case is that the pair $(i_0,j_0)$ satisfies $j_0 = i_0 +1$. 
	
	If the first case happens, i.e., $j_0 - i_0 >1$, then, we consider the following subcomplex 
	\[X_0 = \left(\bigoplus_{i_0 < i < j_0} V_i \otimes S_{v_i}[d_i], (\psi_{j,i} \otimes f_{j,i})\right) \subset X.\]
	Then, by definition, $X_0$ is a minimal twisted complex. 
	
	By applying (i), we know that $X_0$ is equivalent to a well-ordered minimal twisted complex 
	\[X'_0 = \left(\bigoplus_{i_0 < i < j_0} V'_i \otimes S_{v'_i}[d'_i], (F'_{j,i}) \right).\]
	We note that we can convert $X_0$ to $X'_0$ by changing the positions of objects, without changing the nonzero arrows. 
	
	Now, we consider a slightly bigger subcomplex 
	\[X_1 = \left(\bigoplus_{i_0 \leq i \leq j_0} V_i \otimes S_{v_i}[d_i], (\psi_{j,i} \otimes f_{j,i})\right) \subset X.\]
	In $X_1$, it is easy to check that 
	\begin{equation}
		\label{eqn X_1} 
		X_1 \simeq \left(\begin{tikzcd}
			V_{i_0} \otimes S_v[d_{i_0}] \ar[r] \ar[rr, blue, bend left, "\psi_{j_0,i_0} \otimes e_v"] & X'_0 \ar[r] & V_{j_0} \otimes S_v[d_{j_0}]
		\end{tikzcd}\right).
	\end{equation}
	
	Since $X'_0$ is well-ordered, without changing the quasi-equivalence class of $X_1$, one can move $V_{i_0} \otimes S_v[d_{i_0}]$ (resp.\ $V_{j_0} \otimes S_v[d_{j_0}]$) backward (resp.\ forward) in the right-hand side of Equation \eqref{eqn X_1}, {\em up until} $V_{i_0} \otimes S_v[d_{i_0}]$ and $V_{j_0} \otimes S_v[d_{j_0}]$ are adjacent.
	Let $X'_1$ be the twisted complex equivalent to $X_1$, obtained from the above procedure. 
	
	Now, we replace the subcomplex $X_1$ in $X$ with $X'_1$, and let the new twisted complex be denoted by $X$ again. 
	Then, the new $X$ is equivalent to the original twisted complex, and the second case of the second step happens.
	Thus, it is enough to consider the second case of the second step only. 
	
	Let us assume that the second case happens, i.e., $j_0 = i_0 +1$. 
	For simplicity, we let $v$ denote the vertex $v = v_{i_0} = v_{i_0+1}$. 
	We note that $d_{j_0} =d_{i_0} -1$ because $f_{j_0,i_0}=e_v$ is degree $1$ map. 
	We can observe that 
	\[\mathrm{Cone} \left(V_{i_0} \otimes S_v[d_{i_0}] \xrightarrow{\psi_{i_0+1,i_0} \otimes e_v} V_{i_0+1} \otimes S_v[d_{i_0+1}]\right) \simeq \mathrm{Cone} \left(V'_{i_0} \otimes S_v[d_{i_0}] \xrightarrow{\phi'_{i_0+1,i_0} \otimes 0\text{-morphism}} V'_{i_0+1} \otimes S_v[d_{i_0+1}]\right),\]
	with \[V'_{i_0} = \mathrm{Ker} \psi_{i_0+1,i_0} \text{  and  } V'_{i_0+1} = V_{i_0+1} / \mathrm{Im}\psi_{i_0+1,i_0}.\]
	Thus, 
	\[X=\left(\bigoplus_{i \in I} V_i \otimes S_{v_i}[d_i], (F_{j,i}=\psi_{j,i} \otimes f_{j,i}) \right)\simeq X' = \left(\bigoplus_{i \in I} V'_i \otimes S_{v_i}[d_i], (F'_{j,i}=\phi'_{j,i} \otimes f'_{j,i}) \right),\]
	with 
	\[V'_i = \begin{cases}
		V_i \text{  if  } i \neq i_0, i_0+1, \\
		\mathrm{Ker} \psi_{i_0+1,i_0} \text{  if  } i = i_0, \\
		 V_{i_0+1} / \mathrm{Im}\psi_{i_0+1,i_0} \text{  if  } i = i_0 +1.
	\end{cases}\]
	We note that $\psi_{i_0+1,i_0}$ is nonzero. 
	Thus, 
	\[\sum_{i \in I} \mathrm{dim}V_i > \sum_{i \in I} \mathrm{dim}V'_i.\] 
	
	If $\psi_{i_0+1,i_0}$ is injective (resp.\ surjective), $V'_{i_0}$ (resp.\ $V'_{i_0+1}$) is $0$-dimensional.
	Thus, for that case, the induction hypothesis completes the proof. 
		
	If $\psi_{i_0+1, i_0}$ is not injective or surjective, then we repeat the algorithm for the new twisted complex $X'$.  
	Since the algorithm decreases $\left(\sum_{i \in I} \mathrm{dim}V_i\right)$, by applying the algorithm finitely many times, the algorithm should stop, i.e., one can obtain a minimal twisted complex $X'$ such that $X \simeq X'$.  
\end{proof}

\section{Pseudo-Anosov autoequivalences and stability conditions}
\label{section pseudo-Anosov autoequivalences and stability conditions}
Section \ref{section pseudo-Anosov autoequivalences and stability conditions} introduces the definition of {\em pseudo-Anosov autoequivalence} originally given in \cite{Fan-Filip-Haiden-Katzarkov-Liu21}. 
The definition requires for the triangulated category that a pseudo-Anosov autoequivalence is defined on to have a stability condition. 
Thus, we also review the notion of stability conditions in Section \ref{subsection definition of pseudo-Anosov autoequivalence}, and we will discuss the existence of stability conditions on our category $\Fuk$ in the following subsections. 
In the last subsection, we suggest another notion of pseudo-Anosov autoequivalence, which is stronger than that of \cite{Fan-Filip-Haiden-Katzarkov-Liu21}.

\subsection{Definition of pseudo-Anosov autoequivalence}
\label{subsection definition of pseudo-Anosov autoequivalence}
In this subsection, we introduce the main character of this paper, {\em pseudo-Anosov autoequivalences}. 
We note that the notion is originally defined in \cite{Fan-Filip-Haiden-Katzarkov-Liu21}.
Thus, we briefly review the notion and, for more details, we refer the reader to \cite[Section 2]{Fan-Filip-Haiden-Katzarkov-Liu21} and references therein.

Fan, Filip, Haiden, Katzarkov, and Liu \cite{Fan-Filip-Haiden-Katzarkov-Liu21} defined the notion of pseudo-Anosov autoequivalence on a triangulated category $\mathcal{D}$ under the assumption that the space of stability conditions on $\mathcal{D}$ is not empty. 
We start this subsection by reviewing the definition of {\em stability conditions}.

\begin{definition}[Bridgeland \cite{Bridgeland07}]
	\label{def stability condition}
	Let $\mathcal{D}$ be a triangulated category, let $K(\mathcal{D})$ be the Grothendieck group of $\mathcal{D}$, and let $\cl: K(\mathcal{D}) \to \Gamma$ be a group homomorphism to a finite rank abelian group $\Gamma$. 
	A {\bf Bridgeland stability condition $\boldsymbol{\sigma = (Z_\sigma, P_\sigma)}$} consists of 
		\begin{itemize}
			\item a group homomorphism $Z_\sigma: \Gamma \to \mathbb{C}$, which is called the {\bf central charge}, and 
			\item a collection of full additive subcategories $P_\sigma= \{P_\sigma(\phi)\}_{\phi \in \mathbb{R}}$ of $\mathcal{D}$, which is called the {\bf $\boldsymbol{\sigma}$-semistable object of phase $\boldsymbol{\phi}$},
		\end{itemize}
	such that 
	\begin{enumerate}
		\item[(i)] $Z_\sigma(E):=Z_\sigma(\cl([E])) \in \mathbb{R}_{>0} \cdot e^{i\pi\phi}$ for any $0 \neq E \in P(\phi)$,
		\item[(ii)] $P_\sigma(\phi+1) = P_\sigma(\phi)[1]$, 
		\item[(iii)] if $\phi_1 > \phi_2$ and $A_i \in P(\phi_i)$, then $\lhom_\mathcal{D}(A_1,A_2) = 0$, 
		\item[(iv)]  for any $0 \neq E \in \mathcal{D}$, there exists a (unique) collection of exact triangles, which is called {\bf Harder-Narasimhan filtration of $\boldsymbol{E}$} 
		\begin{equation*}
			\begin{tikzcd}
				0 =E_0\arrow{r} & E_1 \arrow{d}\arrow{r} & E_2 \arrow{d} \ar[r] & \dots \ar[r] & E_{n-1} \ar[r] \ar[d] & E \ar[d]\\
			              	& A_1 \arrow[ul, dashed]         & A_2 \arrow[ul, dashed]       & \dots        & A_{n-1} \ar[ul, dashed]       & A_n \ar[ul, dashed]
			\end{tikzcd} 
		\end{equation*}
		such that for each $i =1, \dots, n-1$, $E_{i-1} \to E_i \to A_i \dashrightarrow$ is an exact triangle, $A_i \in P_\sigma(\phi_i)$, and $\phi_1 > \phi_2 > \dots > \phi_n$, and
		\item[(v)] (Support property \cite{Kontsevich-Soibelman08}) there exists a constant $C > 0$ and a norm $\|\cdot \|$ on $\Gamma \otimes_\mathbb{Z} \mathbb{R}$ such that 
		\[\|\cl([E])\| \leq C|Z_\sigma(E)|\]
		for any $E \in P_\sigma(\phi)$ and any $\phi \in \mathbb{R}$.
	\end{enumerate}
\end{definition}

\begin{remark}
	\label{rmk two different categories}
	We note that in the previous sections, we considered {\em $\mathcal{A}_\infty$-categories}.
	In this section, we consider {\em triangulated categories}.
	To be more precise, we need to distinguish $\mathcal{A}_\infty$-categories and triangulated categories.
	However, in the rest of the paper, if there is no chance of confusions, then we will abuse notation for convenience. 
	For example, our category $\Fuk$ is an $\mathcal{A}_\infty$-category and the corresponding triangulated category is $H^0\Fuk$. 
	Especially, the morphism space in $H^0\Fuk$ is the space of degree $0$ morphisms in $\Fuk$. 
	Thus, to be more precise, we need to distinguish $\lhom_{\Fuk}^*$ and $\lhom_{H^0\Fuk}$. 
	However, for convenience, we use $\lhom_{\Fuk}^0$, instead of $\lhom_{H^0\Fuk}$.
\end{remark}

Throughout this subsection, we assume that the space of stability conditions on a triangulated category $\mathcal{D}$, denoted by $\stab(\mathcal{D})$, is not empty. 
Then, a stability condition on $\mathcal{D}$ defines the {\em mass} function as follows:
\begin{definition}
	\label{def stability mass function}
	\mbox{}
	\begin{enumerate}
		\item Let $\sigma=\left(Z_\sigma,P_\sigma\right)$ be a stability condition on $\mathcal{D}$. 
		For a given $E \in \mathcal{D}$, let us assume that the Harder-Narasimhan filtration of $E$ with respect to $\sigma$ is 
		\begin{equation*}
			\begin{tikzcd}
				0 =E_0 \arrow{r}  & E_1 \arrow{d}\arrow{r} & E_2 \arrow{d} \ar[r] & \dots \ar[r] & E_{n-1} \ar[r] \ar[d] & E \ar[d]\\
	                       	 & A_1 \arrow[ul, dashed]         & A_2 \arrow[ul, dashed]       & \dots        & A_{n-1} \ar[ul, dashed]       & A_n \ar[ul, dashed]
			\end{tikzcd}.
		\end{equation*}
		The semistable objects $A_i$ in the above diagram are called the {\bf $\boldsymbol{\sigma}$-semistable components of $\boldsymbol{E}$}. 
		\item The {\bf mass of $\boldsymbol{E}$ with respect to $\boldsymbol{\sigma}$} is defined to be
		\begin{gather}
			\label{eqn mass function with respect to sigma}
			m_\sigma(E):= \sum_{i=1}^n |Z_\sigma(A_i)|,
		\end{gather}
		where $A_i$s are $\sigma$-semistable components of $E$. 
		\item The function $m_\sigma: \mathcal{D} \to \mathbb{R}$ in Equation \eqref{eqn mass function with respect to sigma} is called the {\bf mass function with respect to $\boldsymbol{\sigma}$}.
	\end{enumerate}
\end{definition}

The notion of {\em pseudo-Anosov autoequivalence} is defined by measuring the exponential growth rate of $m_\sigma(\Phi^m E)$ where $\Phi: \mathcal{D}\to \mathcal{D}$ is an autoequivalence. 
To be more precise, we state Proposition \ref{prop growth filtration}.
\begin{prop}[Lemma 2.10 of \cite{Fan-Filip-Haiden-Katzarkov-Liu21}]
	\label{prop growth filtration}
	Let $\sigma$ be a stability condition on $\mathcal{D}$, and let $\Phi: \mathcal{D} \to \mathcal{D}$ be an autoequivalence on $\mathcal{D}$. 
	Then, for any $\lambda \in \mathbb{R}_{>0}$,
	\begin{gather}
		\label{eqn mass filtration}
		\mathcal{D}_{\sigma, \lambda}(\Phi):= \left\{E \in \mathcal{D} \Big\vert \limsup_{n \to \infty} \tfrac{1}{n} \log m_\sigma(\Phi^n E) \leq \log \lambda\right\}
	\end{gather}
	is a $\Phi$-invariant thick triangulated subcategory of $\mathcal{D}$.
\end{prop}

We note that $\mathcal{D}_{\sigma,\lambda}(\Phi)$ in Equation \eqref{eqn mass filtration} depends on the choice of a stability condition $\sigma$.
However, $\mathcal{D}_{\sigma, \lambda}(\Phi)$ is independent of the choice of $\sigma$ in the following sense: 
The space of stability conditions $\stab(\mathcal{D})$ admits a topology which is induced by a metric.
(The definition of the metric is originally defined in \cite[Proposition 8.1]{Bridgeland07} and one can find the definition in Definition \ref{def metric}).
If two stability conditions $\sigma$ and $\tau$ lie in the same connected component of $\stab(\mathcal{D})$, Proposition \ref{prop independence of the choice of stability conditions} holds.
\begin{prop}[Proposition 2.11 of \cite{Fan-Filip-Haiden-Katzarkov-Liu21}]
	\label{prop independence of the choice of stability conditions} 
	For any autoequivalence $\Phi$ and for any $\lambda \in \mathbb{R}_{>0}$,
	\[\mathcal{D}_{\sigma,\lambda}(\Phi) = \mathcal{D}_{\tau,\lambda}(\Phi).\]
\end{prop}

Based on Proposition \ref{prop independence of the choice of stability conditions}, we define the following:
\begin{definition}[Definitions 2.12 and 2.13 of \cite{Fan-Filip-Haiden-Katzarkov-Liu21}]
	\label{def growth filtration}
	Let $\mathcal{D}$ be a triangulated category, and let $\Phi: \mathcal{D} \to \mathcal{D}$ be an autoequivalence of $\mathcal{D}$.
	Let us assume that 
	\begin{enumerate}
		\item[(A)] $\stab(\mathcal{D}) \neq \varnothing$, and 
		\item[(B)] we fix a connected component $\stab^\dagger(\mathcal{D}) \subset \stab(\mathcal{D})$. 
	\end{enumerate}
	\begin{enumerate}
		\item The {\bf growth filtration associated to an autoequivalence $\boldsymbol{\Phi}$} is defined to be the filtration 
		\[\left\{\mathcal{D}_\lambda(\Phi):=\mathcal{D}_{\sigma, \lambda}(\Phi)\right\}_{\lambda \in \mathbb{R}_{\geq 0}},\]
		for any $\sigma \in \stab^\dagger (\mathcal{D})$.
		\item An autoequivalence $\Phi$ is said to be {\bf pseudo-Anosov} if the associated growth filtration has only one step:
		\begin{gather}
			\label{eqn stretching factor}
			0 \subset \mathcal{D}_\lambda(\Phi) = \mathcal{D},
		\end{gather}
		with $\lambda>1$. 
		\item If $\Phi$ is a pseudo-Anosov, we call the $\lambda$ in \eqref{eqn stretching factor} {\bf stretch factor of $\boldsymbol{\Phi}$}.
	\end{enumerate}
\end{definition}

\subsection{Stability conditions on $\Fuk$}
\label{subsection stability conditions on Fuk}
In the body of the paper, we give a construction of pseudo-Anosov autoequivalences on $\Fuk$ where $T$ is an arbitrary tree. 
Thus, we should discuss the two assumptions that are necessarily for defining the notion of pseudo-Anosov autoequivalence, i.e., the assumptions (A) and (B) in Definition \ref{def growth filtration}. 
In Section \ref{subsection stability conditions on Fuk}, we show that there exists a stability conditions on $\Fuk$ and we discuss which connected component of $\stab(\Fuk)$ is fixed throughout the paper. 

We remark that in this subsection, we are seeing $\Fuk$ as a triangulated category. 
Thus, the morphism space in which we are interested should be the {\em Lagrangian Floer cohomology at degree $0$}.
See Remark \ref{rmk two different categories}.
Moreover, we note that the Grothendieck group $K\left(\Fuk\right)$ is finitely generated by 
\[\left\{[S_v] | v \in V(T)\right\},\]
where $[S_v]$ denotes the equivalence class of $S_v$ in $K\left(\Fuk\right)$. 
Based on this, we let $\cl$ in Definition \ref{def stability condition} be the identity morphism of $K\left(\Fuk\right)$.

First, we discuss the assumption (A).
In order to show that $\Fuk$ satisfies the assumption (A), it is enough to construct a stability condition of $\Fuk$. 
Our tool for constructing a stability condition is the following proposition:
\begin{prop}[Proposition 5.3 of \cite{Bridgeland07}]
	\label{prop construction of a stability condition}
	To give a stability condition on a triangulated category $\mathcal{D}$ is equivalent to giving a bounded $t$-structure on $\mathcal{D}$ and a stability function on its heart with the Harder-Narasimhan property.
\end{prop}

The undefined terms in Proposition \ref{prop construction of a stability condition}, for example, bounded $t$-structure, stability function, etc, are defined below.
For more details on them, we refer the reader to \cite{Bridgeland07}.

\begin{definition}
	\mbox{}
	\begin{enumerate}
		\item A {\bf stability function} on an abelian category $\mathcal{A}$ is a group homomorphism $Z: K(\mathcal{A}) \to \mathbb{C}$ such that for all nonzero object $E \in \mathcal{A}$, the complex number $Z(E)$ lies in the strict upper half-plane
		\[H:= \left\{r \exp(i\pi \phi) | r \in \mathbb{R}_{>0} \text{  and  } \phi \in (0,1]\right\}.\]
		\item Let $Z$ be a stability function on an abelian category $\mathcal{A}$.
		A nonzero object $E \in \mathcal{A}$ is said to be {\bf semistable} if every subobject $0 \neq E' \subset E$ satisfies that 
		\[\operatorname{arg}Z(E') \leq \operatorname{arg}Z(E).\]
		\item Let $Z$ be a stability function on an abelian category $\mathcal{A}$.
		For a nonzero object $E \in \mathcal{A}$, a {\bf Harder-Narasimhan filtration of $\boldsymbol{E}$ with respect to $\boldsymbol{Z}$} is a finite chain of subobjects 
		\[0=E_0 \subset E_1 \subset \dots \subset E_{m-1} \subset E_m = E,\]
		such that whose factors $F_i = E_i / E_{i-1}$ are semistable object of $\mathcal{A}$ with 
		\[\operatorname{arg}Z(F_1) > \operatorname{arg}Z(F_2) > \dots > \operatorname{arg}Z(F_m).\]
		\item A stability function $Z$ is said to have the {\bf Harder-Narasimhan property} if every nonzero object of $\mathcal{A}$ has a Harder-Narasimhan filtration with respect to $Z$. 
		\item A {\bf $\boldsymbol{t}$-structure} on a triangulated category $\mathcal{D}$ is a full subcategory $\mathcal{T} \subset \mathcal{D}$ such that 
		\begin{itemize}
			\item $\mathcal{T}[1] \subset \mathcal{T}$, and 
			\item if one defines $\mathcal{T}^\perp$ as 
			\[\mathcal{T}^\perp :=\left\{G \in \mathcal{D}| \lhom_\mathcal{D}^0\left(F,G\right) =0 \text{  for all  } F \in \mathcal{T}\right\},\]
			then for every object $E \in \mathcal{D}$, there is an exact triangle 
				\[\begin{tikzcd}
				F \ar[rr] & & E \ar[dl] \\
				& G \ar[ul, dashed] &
			\end{tikzcd},\]
			such that $F \in \mathcal{T}$ and $G \in \mathcal{T}^\perp$.
		\end{itemize}
		\item The {\bf heart} of a $t$-structure $\mathcal{T}\subset \mathcal{D}$ is the full subcategory 
		\[\mathcal{A} := \mathcal{T} \cap \mathcal{T}^\perp[1] \subset \mathcal{D}.\]
		\item A $t$-structure $\mathcal{T} \subset \mathcal{D}$ is {\bf bounded} if 
		\[\mathcal{D} = \bigcup_{i,j \in \mathbb{Z}} \mathcal{T}[i] \cap \mathcal{T}^\perp[j].\]
	\end{enumerate}
\end{definition}
It was proved in \cite{Beilinson-Bernstein-Deligne81} that the heart of a $t$-structure is an abelian category. 
Thus, the expression ``stability condition on its heart with the Harder-Narasimhan property" in Proposition \ref{prop construction of a stability condition} makes sense. 

Based on Proposition \ref{prop construction of a stability condition}, we prove that $\Fuk$ satisfies the assumption (A).
\begin{lem}
	\label{lem existence of a stability condition}
	The triangulated category $\Fuk$ satisfies the assumption (A) of Definition \ref{def growth filtration}.
	In other words, 
	\[\stab(\Fuk) \neq \varnothing.\]
\end{lem}
\begin{proof}
	Let $\mathcal{T}$ denote the full subcategory $\Fuk$ such that the objects of $\mathcal{T}$ consist of 
	\begin{itemize}
		\item the zero object and
		\item the objects equivalent to a good, minimal twisted complex 
		\[F =\left(\bigoplus_{i \in I} V_i \otimes S_{v_i}[d_i], (F_{j,i}=\psi_{j,i} \otimes f_{j,i}) \right)\]
		satisfying that $d_i \geq 0$ for all $i \in I$. 
	\end{itemize}
	
	Then, thanks to Lemmas \ref{lem grading} and \ref{lem minimal twisted complex}, one can observe that the objects of $\mathcal{T}^\perp$ consist of 
	\begin{itemize}
		\item the zero object, and
		\item the objects equivalent to a good, minimal twisted complex  
		\[E =\left(\bigoplus_{j \in J} V_j \otimes S_{v_j}[d_j], (G_{j,i}=\psi_{j,i} \otimes g_{j,i}) \right)\]
		satisfying $d_i < 0$ for all $j \in J$.
	\end{itemize}
	
	Now, it is easy to check the following two:
	\begin{enumerate}
		\item[(i)] $\mathcal{T}[1] \subset \mathcal{T}$ because of the definition of $\mathcal{T}$.
		\item[(ii)] For every nonzero object $E \in \mathcal{F}_T$, there exist $F \in \mathcal{T}, G \in \mathcal{T}^\perp$ such that $E$ is equivalent to a twisted complex in the form of $\left(G \oplus F, \Psi \in \lhom^1_{\Tw(\Fuk)}(G, F)\right)$.
		One can find $F$ and $G$ satisfying the above conditions from a minimal, well-ordered twisted complex equivalent to $E$.  
		As a cone, we can write $E \simeq \mathrm{Cone}(G[-1] \to F)$. 
		In other words, there exists an exact triangle
		\[\begin{tikzcd}
			F \ar[rr] & & E \ar[dl] \\
			& G \ar[ul, dashed] &
		\end{tikzcd}.\]
	\end{enumerate} 
	The above two items (i) and (ii) mean that $\mathcal{T}$ is a $t$-structure. 
	
	Moreover, one can easily observe that the heart $\mathcal{A} = \mathcal{T} \cap \mathcal{T}^\perp[1]$ consists of 
	\begin{itemize}
		\item the zero object and 
		\item the objects equivalent to a good, minimal twisted complex 
		\[E = \left(\bigoplus_{i \in I} V_i \otimes S_{v_i}, (F_{j,i} = \psi_{j,i} \otimes f_{j,i})\right).\]
	\end{itemize}
	Also, it is easy to see that 
	\[\Fuk = \cup_{i \in \mathbb{Z}} \mathcal{A}[i].\]
	It implies that $\mathcal{T}$ is a bounded $t$-structure. 
	
	Let us assume that $Z: K(\mathcal{A}) \to \mathbb{C}$ is a group homomorphism such that 
	\[Z(S_v) = \begin{cases}
		\imagin \text{  if  } v \in V_+(T), \\
		-1 \text{  if  } v \in V_-(T).
	\end{cases}\]
	Then, one can easily check that every semistable object $E$ of $Z$ with $\operatorname{arg}Z(E) = \tfrac{1}{2}$ (resp.\ $1$) satisfies that 
	\[E \simeq \bigoplus_{v_i \in V_-(T)} V_i \otimes S_{v_i} \text{  (resp.\  }\bigoplus_{v_i \in V_+(T)} V_i \otimes S_{v_i}),\]
	and there exists no other semistable object.
		
	Moreover, for any nonzero object $E \simeq \left(\bigoplus_{i \in I} V_i \otimes S_{v_i}, (F_{j,i} = \psi_{j,i} \otimes f_{j,i})\right)\in \mathcal{A}$, we have  
	\[\begin{tikzcd}
		0\ar[r] & \bigoplus_{v_i \in V_-(T)} V_i \otimes S_{v_i} \ar[r] \ar[d] & E \ar[d] \\
		&\bigoplus_{v_i \in V_-(T)} V_i \otimes S_{v_i} \ar[ul, dashed] & \bigoplus_{v_j \in V_+(T)} V_j \otimes S_{v_j} \ar[ul, dashed] 
	\end{tikzcd}. \]
	The above filtration is the Harder-Narasimhan filtration of $E$ with respect to $Z$. 
	Then, Proposition \ref{prop construction of a stability condition} completes the proof.
\end{proof}

In the proof of Lemma \ref{lem existence of a stability condition}, we constructed a stability condition. 
The fixed connected component of $\stab(\Fuk)$ which we will consider for the assumption (B) of Definition \ref{def growth filtration} is the connected component containing the above constructed stability condition. 
We give names to the constructed stability condition and the fixed connected component in Definition \ref{def fixed component of the space of stability conditions}.

\begin{definition}
	\label{def fixed component of the space of stability conditions}
	\mbox{}
	\begin{enumerate}
		\item Let $\boldsymbol{\sigma_0} = \left(Z_0, P_0\right)$ denote the stability condition on $\Fuk$ constructed in the proof of Lemma \ref{lem existence of a stability condition}. 
		\item Let $\boldsymbol{\stab^\dagger(\Fuk)}$ be the connected component of $\stab(\Fuk)$ containing $\sigma_0$.
	\end{enumerate}
\end{definition}
We note that as mentioned in the proof of Lemma \ref{lem existence of a stability condition}, every element of $P_0(\tfrac{1}{2})$ (resp.\ $P_0(1)$) is a direct sum of 
\[\{S_w| w \in V_-(T)\} \text{  (resp.\  }\{S_u | u \in V_+(T)\}),\]
with respectively, and $P_0(\phi) = 0$ if $\phi$ is neither $k$ or $k +\tfrac{1}{2}$ for some $k \in \mathbb{Z}$. 

Now, it is easy to show that $m_{\sigma_0}(X)$ is easily computed from the minimal twisted complex 
\[E =\left(\bigoplus_{i \in I} V_i \otimes S_{v_i}[d_i], (F_{j,i}=\psi_{j,i} \otimes f_{j,i}) \right)\]
such that $X \simeq E$. 
It is because one can obtain the Harder-Narasimhan filtration of $X$ by rearranging $V_i \otimes S_i[d_i]$ in $E$. 
As a result, we have 
\begin{gather}
	\label{eqn m_sigma(X)}
	m_{\sigma_0}(X) = \sum_{i\in I} \mathrm{dim} V_i.
\end{gather}

\subsection{Stability condition for the other sign convention}
\label{subsection stability condition for the other sign convention}
We recall that, as mentioned in Remark \ref{rmk sign conventions} the sign convention for 
\[V(T) = V_+(T) \sqcup V_-(T)\]
is a part of the given data.
Under the assumption, in the Sections \ref{section preparations for the proof of pseudo-Anosov theorem}--\ref{section strong pseudo-Anosov}, we will prove that $\Phi$ is a pseudo-Anosov autoequivalence with respect to $\stab^\dagger(\Fuk)$ if $\Phi$ is a product of 
\[\left\{\tau_u, \tau_w^{-1} | u \in V_+(T), w \in V_-(T)\right\}.\]

In order to show that if $\Phi^{-1}$ is a product of the above Dehn twists and inverses of Dehn twists, $\Phi$ is also a pseudo-Anosov autoequivalence, we can choose the other sign convention, as mentioned in Remark \ref{rmk sign is not important 1}.
When we apply the arguments in Section \ref{subsection stability conditions on Fuk} to the other choice, we have a stability condition $\sigma_0' = (Z_0', P_0')$, especially satisfying that 
\[Z_0'(R_v) = \begin{cases}
	Z_0'(R_v = S_v[2-N]) = \imagin \text{  if  } v \in V_-(T), \\
	Z_0'(R_v = S_v) = -1 \text{  if  } v \in V_+(T).
\end{cases}\]
We note that $\left\{R_v\right\}_{v \in V(T)}$ is defined in Remark \ref{rmk sign is not important 2}.
Thus, if one applies the arguments of the later sections, then one can prove that $\Phi$ is a pseudo-Anosov autoequivalence {\em with respect to $\sigma_0'$, not $\sigma_0$}. 

However, if one can prove that $\sigma_0$ and $\sigma_0'$ are in the same connected component $\stab^\dagger(\Fuk)$, Proposition \ref{prop independence of the choice of stability conditions} guarantees that $\Phi$ is pseudo-Anosov with respect to $\sigma_0$.
In Section \ref{subsection stability condition for the other sign convention}, we prove that $\sigma_0'$ and $\sigma_0$ are contained in the same connected component $\stab^\dagger(\Fuk)$. 
More precisely, we will give a path connecting $\sigma_0$ and $\sigma_0'$ in $\stab(\Fuk)$. 

First of all, we introduce the metric structure on the space of stability conditions, which induces a topology on the space. 
We note that Definition \ref{def metric} is defined on a general triangulated category $\mathcal{D}$.
\begin{definition}
	\label{def metric}
	\mbox{}
	\begin{enumerate}
		\item Let $\sigma_i = (Z_i, P_i)$ be a stability condition on a triangulated category $\mathcal{D}$ for $i=1,2$. 
		For a nonzero object $E \in \mathcal{D}$, let the Harder-Narasimhan filtration of $E$ with respect to $\sigma_i$ be 
		\[\begin{tikzcd}
			0 =E_0\arrow{r} & E_1 \arrow{d}\arrow{r} & E_2 \arrow{d} \ar[r] & \dots \ar[r] & E_{n-1} \ar[r] \ar[d] & E \ar[d]\\
			& A_1 \arrow[ul, dashed]         & A_2 \arrow[ul, dashed]       & \dots        & A_{n-1} \ar[ul, dashed]       & A_n \ar[ul, dashed]
		\end{tikzcd}.\]
		The {\bf maximal} (resp.\ {\bf minimal}) {\bf phase of $E$ with respect to $\sigma_i$} $\phi_{\sigma_i}^+(E)$ (resp.\ $\phi_{\sigma_i}^-(E)$) is the phase of $A_1$ (resp.\ $A_n$).
		\item The {\bf generalized metric} on $\stab(\mathcal{D})$ is defined by the following distance function:
		\[d(\sigma_1, \sigma_2) := \sup_{0 \neq E \in \mathcal{D}} \left\{|\phi^-_{\sigma_2}(E) - \phi^-_{\sigma_1}(E)|, |\phi^+_{\sigma_2}(E)- \phi^+_{\sigma_1}(E)|, |\log \frac{m_{\sigma_2}(E)}{m_{\sigma_1}(E)}|\right\}.\]
	\end{enumerate}
\end{definition}
For more details about the generalized metric structure on $\stab(\mathcal{D})$ and the topology induced from the generalized metric, we refer the reader to \cite[Section 8]{Bridgeland07}.

Now, we return to our triangulated category $\Fuk$. 
\begin{prop}
	\label{prop same connected component}
	Two stability conditions $\sigma_0$ and $\sigma_0'$ of $\Fuk$ are in the same connected component $\stab^\dagger(\Fuk)$ of $\stab(\Fuk)$.
\end{prop}
\begin{proof}
	We prove Proposition \ref{prop same connected component} by giving a continuous path 
	\[\sigma(t) : [0,N-2] \to \stab(\Fuk),\]
	such that $\sigma(0) = \sigma_0$ and $\sigma(N-2)=\sigma_0'$.
	Thanks to Proposition \ref{prop construction of a stability condition}, it is enough to construct a bounded $t$-structure on $\Fuk$ and a stability function on its heart for each $t \in [0,N-2]$.
	
	First, we construct a bounded $t$-structure $\mathcal{T}_k$ for each $k = 1, \dots, N-2$. 
	Let $\mathcal{T}_k$ be the full subcategory of $\Fuk$ whose objects consist of 
	\begin{itemize}
		\item the zero object, and 
		\item the objects equivalent to a twisted complex consisting of 
		\[\left\{S_u[d_u] | u \in V_+(T), d_u \geq 0\right\} \cup \left\{S_w[d_w]| w \in V_-(T), d_w \geq -k\right\}.\]
	\end{itemize}

	In order to prove that $\mathcal{T}_k$ is a $t$-structure, one needs to show that 
	\begin{enumerate}
		\item[(i)] $\mathcal{T}_k[1] \subset \mathcal{T}_k$, and
		\item[(ii)] for any nonzero $E \in \Fuk$, there exist $F \in \mathcal{T}_k$ and $G \in \mathcal{T}_k^\perp$ such that there exists an exact triangle
		\[\begin{tikzcd}
			F \ar[rr] & & E \ar[dl] \\
			& G \ar[ul, dashed] &
		\end{tikzcd}.\]
	\end{enumerate}
	
	By the definition of $\mathcal{T}_k$, (i) is trivial. 
	To prove (ii), we define a full subcategory $\mathcal{G}_k \subset \Fuk$ consisting of 
	\begin{itemize}
		\item the zero object, and 
		\item the objects equivalent to a twisted complex consisting of 
		\[\left\{S_u[d_u] | u \in V_+(T), d_u < 0\right\} \cup \left\{S_w[d_w]| w \in V_-(T), d_w < -k\right\}.\]
	\end{itemize}
	Then, by definition, $\mathcal{G}_k \subset \mathcal{T}_k^\perp$.
	
	For any nonzero object $E \in \Fuk$, we know that $E$ is equivalent to a minimal, well-ordered twisted complex.
	Thus, without loss of generality, we may assume that
	\[E=\left(\bigoplus_{i \in I} V_i \otimes S_{v_i}[d_i], (F_{j,i}) \right),\]
	with $d_i \leq d_{i+1}$. 
	Moreover, we observe that if $i \in [0,k]$ is an integer, then for any $u \in V_+(T), w \in V_-(T)$ and for any $m \in \mathbb{Z}$, 
	\[\lhom_{\Fuk}^1(S_w[-i+m],S_u[m]) = \lhom_{\Fuk}^{i+1}(S_w,S_u) = 0.\]
	We note that one can switch the order of $V_i \otimes S_{v_i}[d_i]$ and $V_{i+1} \otimes S_{v_{i+1}}[d_{i+1}]$ if $F_{i+1,i}$ is the zero map. 
	Thus, one can reorder the index set $I$ of a given twisted complex $E$ so that after the reordering, the following holds: 
	For any integer $m \in \mathbb{Z}$, $S_w[-k+m]$ appears after $S_u[m]$ in the twisted complex $E$. 
	It proves that for any nonzero $E \in \Fuk$, there exist $F \in \mathcal{T}_k$ and $G \in \mathcal{G}_k$ such that there exists an exact triangle
	\[\begin{tikzcd}
		F \ar[rr] & & E \ar[dl] \\
		& G \ar[ul, dashed] &
	\end{tikzcd}.\]
	The above arguments also prove that $\mathcal{G}_k = \mathcal{T}_k^\perp$. 
	Moreover, it proves that $\mathcal{T}_k$ is a $t$-structure of $\Fuk$. 
		
	Since $\mathcal{T}_k$ is a $t$-structure, $\mathcal{T}_k$ gives us the heart 
	\[\mathcal{A}_k := \mathcal{T}_k \cap \mathcal{T}_k^\perp[1].\]
	From the definition, every object of $\mathcal{A}_k$ is a direct sum of 
	\[\left\{S_u, S_w[-k] | u \in V_+(T), w \in V_-(T)\right\}.\]
	Now, it is easy to observe that 
	\[\Fuk = \cup_{i \in \mathbb{Z}} \mathcal{A}_k[i].\]
	In other words, $\mathcal{T}_k$ is a bounded $t$-structure of $\Fuk$.
	
	Since $\mathcal{T}_k$ is a bounded $t$-structure, if $Z_t: K\left(\mathcal{A}_k\right) \to \mathbb{C}$ is a stability function satisfying Harder-Narasimhan property, then the pair $(\mathcal{T}_k, Z_t)$ defines a stability condition on $\Fuk$. 
	Keeping it in mind, we define $Z_t: K\left(\mathcal{A}_k\right) \to \mathbb{C}$ for any $ t \in (k-1,k]$ as follows:
	\begin{gather*}
		Z_t(S_u) = \imagin \text{  if  } u \in V_+(T),\\
		Z_t(S_w[-k]) = \cos\big((t-k+1)\pi\big) + \imagin \sin\big((t-k+1)\pi\big) \text{  if  } w \in V_-(T).
	\end{gather*}
	Then, it is easy to show that $Z_t$ satisfies the Harder-Narasimhan property since 
	\[\lhom_{\Fuk}^0(S_u,S_w[-k]) = \lhom_{\Fuk}^0(S_w[k],S_u) = 0,\]
	for any $u \in V_+(T), w \in V_-(T)$. 
	
	For any $t \in (k-1,k]$, let $\sigma_t$ denote the stability condition obtained from $\mathcal{T}_k$ and $Z_t$. 
	Then, it is easy to check that $\sigma_t$ is a continuous path on $\stab(\Fuk)$ with respect to the generalized metric defined in Definition \ref{def metric}.
	Moreover, $\sigma_{N-2}=\sigma_0'$. 
	Thus, it completes the proof.
\end{proof}

\subsection{Strong pseudo-Anosov autoequivalence}
\label{subsection strong pseudo-Anosov autoequivalence}
Let $\Phi$ be an autoequivalence of a triangulated category $\mathcal{D}$.
When a stability condition $\sigma$ is given, (or more precisely, when a connected component of $\stab(\mathcal{D})$ is given,) $\Phi$ is pseudo-Anosov with respect to $\sigma$ if the asymptotic behavior of $m_\sigma(\Phi^n E)$ is independent of the choice of nonzero $E \in \mathcal{D}$.
In this definition, the choice of stability condition $\sigma$ is necessarily because $\sigma$ defines the invariant $m_\sigma(\Phi^n E)$. 

We would like to point out that for a given $\sigma$, one can define other invariants of $\Phi^n E$. 
The other invariants are the maximal/minimal phases of $\Phi^n E$, i.e., $\phi^\pm_\sigma(\Phi^n E)$. 
In this subsection, similar to the point of view of Definition \ref{def growth filtration}, we suggest a stronger notion of pseudo-Anosov by considering the asymptotic behavior of $\phi^\pm_\sigma(\Phi^n E)$ as well as that of $m_\sigma(\Phi^n E)$.

To define the stronger version, we need to show an analogue of Proposition \ref{prop growth filtration}, i.e., Lemma \ref{lem growth filtration 2}.
\begin{lem}
	\label{lem growth filtration 2}
	Let $\sigma$ be a stability conditions on $\mathcal{D}$, and let $\Phi: \mathcal{D} \to \mathcal{D}$ be an autoequivalence on $\mathcal{D}$. 
	And, for any $\lambda \in \mathbb{R}_{>0}$, let $\mathcal{P}^+_{\sigma, \lambda}(\Phi)$ (resp.\ $\mathcal{P}^-_{\sigma, \lambda}(\Phi)$) be the collection of objects $E$ such that 
	\[\mathcal{P}^+_{\sigma, \lambda}(\Phi) := \left\{E \in \mathcal{D} \Big\vert \limsup_{n \to \infty} \tfrac{1}{n} \phi^+_\sigma (\Phi^n E) \leq \lambda \right\}, \mathcal{P}^-_{\sigma, \lambda}(\Phi) := \left\{E \in \mathcal{D} \Big\vert \liminf_{n \to \infty} \tfrac{1}{n} \phi^-_\sigma (\Phi^n E) \geq - \lambda \right\}.\] 
	Then, $\mathcal{P}^+_{\sigma, \lambda}(\Phi)$ and $\mathcal{P}^-_{\sigma, \lambda}(\Phi)$ are $\Phi$-invariant thick triangulated subcategories of $\mathcal{D}$.
\end{lem}
\begin{proof}
	We note that $\mathcal{P}^\pm_{\sigma, \lambda}$ is $\Phi$-invariant by definition. 
	Thus, it is enough to show that $\mathcal{P}^\pm_{\sigma, \lambda}$ is a triangulated subcategory.
	
	By the definition of $\phi^\pm_\sigma$, it is easy to check that if $E \simeq \mathrm{Cone}(E_1 \to E_2)$, then 
	\[\phi^+_\sigma(E) \leq \max \left\{\phi^+_\sigma(E_1), \phi^+_\sigma(E_2)\right\}, \phi^-_\sigma(E) \geq \min \left\{\phi^-_\sigma(E_1), \phi^-_\sigma(E_2)\right\}.\]
	We also note that if $E = \mathrm{Cone}(E_1 \to E_2)$, then $\Phi E = \mathrm{Cone}(\Phi E_1 \to \Phi E_2)$ since $\Phi$ is an autoequivalence. 
	
	Let us assume that $E_i \in \mathcal{P}^+_{\sigma, \lambda}(\Phi)$ (resp.\ $E_i \in \mathcal{P}^-_{\sigma, \lambda}(\Phi)$) for $i =1, 2$, i.e., 
	\[\limsup_{n \to \infty} \tfrac{1}{n} \phi^+_\sigma (\Phi^n E_i) \leq \lambda, \liminf_{n \to \infty} \tfrac{1}{n} \phi^-_\sigma (\Phi^n E_i) \geq - \lambda.\]
	Then, the above arguments show that if $E \simeq \mathrm{Cone}(E_1 \to E_2)$, 
	\begin{gather*}
		\limsup_{n \to \infty} \tfrac{1}{n} \phi^+_\sigma (\Phi^n E) \leq \limsup_{n \to \infty} \max \left\{\tfrac{1}{n} \phi^+_\sigma (\Phi^n E_1), \tfrac{1}{n} \phi^+_\sigma (\Phi^n E_2)\right\} \leq \lambda, \\
		\liminf_{n \to \infty} \tfrac{1}{n} \phi^-_\sigma (\Phi^n E) \geq \limsup_{n \to \infty} \min \left\{\tfrac{1}{n} \phi^-_\sigma (\Phi^n E_1), \tfrac{1}{n} \phi^-_\sigma (\Phi^n E_2)\right\} \geq -\lambda.
	\end{gather*}
	It means that $E = \Cone(E_1 \to E_2) \in \mathcal{P}^+_{\sigma, \lambda}(\Phi)$ (resp.\ $\mathcal{P}^-_{\sigma, \lambda}(\Phi)$).
	It completes the proof.
\end{proof}
 
Moreover, we can prove an analogy of Proposition \ref{prop independence of the choice of stability conditions}.
\begin{lem}
	\label{lem independence of the choice of stability conditions}
	If two stability conditions $\sigma$ and $\tau$ lie in the same connected component of $\stab(\mathcal{D})$, for any autoequivalence $\Phi$ and for any $\lambda \in \mathbb{R}_{>0}$, 
	\[\mathcal{P}^\pm_{\sigma,\lambda}(\Phi) = \mathcal{P}^\pm_{\tau, \lambda}(\Phi).\]
\end{lem}
\begin{proof}
	If $\sigma$ and $\tau$ are in the same connected component of $\stab(\mathcal{D})$, then the distance with respect to the generalized metric is finite, i.e., 
	\[d(\sigma, \tau) := \sup_{0 \neq E \in \mathcal{D}} \left\{|\phi^-_\sigma(E) - \phi^-_\tau(E)|, |\phi^+_\sigma(E)- \phi^+_\tau(E)|, |\log \frac{m_\sigma(E)}{m_\tau(E)}|\right\} < \text{  a constant  } C.\]
	Thus, for any $n \in \mathbb{N}$ and for any $E \in \mathcal{D}$, $|\phi^\pm_\sigma(\Phi^n E) - \phi^\pm_\tau(\Phi^n E)| < C$. 
	It implies that $\mathcal{P}^\pm_{\sigma, \lambda} (\Phi) = \mathcal{P}^\pm_{\tau,\lambda}(\Phi)$. 
\end{proof}

Now, we define the notion of {\em strong pseudo-Anosov}.
\begin{definition}
	\label{def strong pesudo-Anosov}
	Let $\mathcal{D}$ be a triangulated category, and let $\Phi: \mathcal{D} \to \mathcal{D}$ be an autoequivalence of $\mathcal{D}$.
	Let us assume that 
	\begin{enumerate}
		\item[(A)] $\stab(\mathcal{D}) \neq \varnothing$, and 
		\item[(B)] we fix a connected component $\stab^\dagger(\mathcal{D}) \subset \stab(\mathcal{D})$. 
	\end{enumerate}
	\begin{enumerate}
		\item The {\bf maximal} (resp.\ {\bf minimal}) {\bf phase filtration associated to an autoequivalence $\boldsymbol{\Phi}$} is defined to be the filtration 
		\[\left\{\mathcal{P}^+_\lambda(\Phi):=\mathcal{P}^+_{\sigma, \lambda}(\Phi)\right\}_{\lambda \in \mathbb{R}_{\geq 0}} \text{  (resp.\  } \left\{\mathcal{P}^-_\lambda(\Phi):=\mathcal{P}^-_{\sigma, \lambda}(\Phi)\right\}_{\lambda \in \mathbb{R}_{\geq 0}} ),\]
		for any $\sigma \in \stab^\dagger (\mathcal{D})$.
		\item An autoequivalence $\Phi$ is said to be {\bf strong pseudo-Anosov} if the following two conditions hold:
		\begin{itemize}
			\item $\Phi$ is a pseudo-Anosov autoequivalence in the sense of Definition \ref{def growth filtration}, and
			\item the associated growth filtration and the associated maximal/minimal phase filtrations have only one step:
			\[0 \subset \mathcal{D}_\lambda(\Phi) = \mathcal{D}, 0 \subset \mathcal{P}^\pm_{\lambda_\pm}(\Phi) = \mathcal{D},\]
			with $\lambda, \lambda_\pm >1$.
		\end{itemize}
		\item For a strong pseudo-Anosov autoequivalence $\Phi$, the {\bf positive} (resp.\ {\bf negative}) {\bf phase stretching factor} is the number $\lambda_+$ (resp.\ $\lambda_-$) giving the one step filtration  
		\[0 \subset \mathcal{P}^\pm_{\lambda_\pm}(\Phi) = \mathcal{D}.\]
	\end{enumerate}
\end{definition}

\begin{remark}
	\label{rmk the shifting number 1}
	\mbox{}
	\begin{enumerate}
		\item We would like to note that the asymptotic behavior of $\phi^\pm_\sigma(\Phi^n G)$ has been studied by Fan and Filip in \cite{Fan-Filip23}, where $G$ is a generator of $\mathcal{D}$ and $\mathcal{D}$ is a saturated, finite-type triangulated category. 
		\item We did not provide an explicit rationale for introducing the concept of {\em strong pseudo-Anosov}, and we would like to remark the motivation before ending the present section. 
		Our initial motivation stemmed from an artificial example, denoted as $\Phi$, which is pseudo-Anosov but does not qualify as strong pseudo-Anosov. 
		We will present this specific example in Section \ref{subsection example of pseudo-Anosov but not strong pseudo-Anosov}.
		However, it is worth noting that this example was artificially constructed, and the underlying category in which it resides lacks a sense of naturalness. 
		Consequently, we remain uncertain whether the distinctions between pseudo-Anosov and strong pseudo-Anosov hold under reasonable and natural assumptions. 
	\end{enumerate}
\end{remark}

\section{Actions on the space of stability conditions}
\label{section actions on the space of stability conditions}
In Section \ref{section actions on the space of stability conditions}, we will discuss the actions on the space of stability conditions.
Since the contents of the current section are preparations for Section \ref{section translation of Penner type autoequivalences}, it would be not a bad idea to skip the section for the first reading, and to come back when it is needed. 

Section \ref{subsection two actions on the space of stability conditions} is a preliminary section reviewing two well-known actions on stability conditions. 
In Section \ref{subsection Dehn twist actions on stab}, we will study the actions on our setting, i.e., $\stab(\Fuk)$.
Especially, we will show that every Penner type autoequivalence induces an action preserving the connected component $\stab^\dagger(\Fuk)$.

\subsection{Two actions on the space of stability conditions}
\label{subsection two actions on the space of stability conditions}
Let $\mathcal{D}$ be an arbitrary triangulated category. 
In this subsection, we briefly review the two actions on $\stab(\mathcal{D})$, which are introduced by Bridgeland.
For more details, see \cite[Lemma 8.2]{Bridgeland07}.

The first action is the right-action of $\widetilde{GL}^+(2,\mathbb{R})$, where $\widetilde{GL}^+(2,\mathbb{R})$ denotes the universal cover of $GL^+(2,\mathbb{R})$.
We note that the group $\widetilde{GL}^+(2,\mathbb{R})$ can be seen as the set of pairs $\left(T: \mathbb{R}^2 \to \mathbb{R}^2, f:\mathbb{R} \to \mathbb{R}\right)$ such that 
\begin{itemize}
	\item $T$ is an orientation-preserving linear map, 
	\item for any $r \in \mathbb{R}$, $f(r) = f(r+1)$, and
	\item $T$ and $f$ induce the same map on $S^1 = \mathbb{R}/\mathbb{Z} = \left(\mathbb{R}^2 \setminus \{0\}\right)/\mathbb{R}_{>0}$. 
\end{itemize}
Then, an element $(T,f) \in \widetilde{GL}^+(2,\mathbb{R})$ sends a stability condition $\sigma = (Z, P) \in \stab(\mathcal{D})$ to another stability condition $\sigma' = (Z', P')$ defined as follows: 
\[Z' := T^{-1} \circ Z, P'(\phi) := P\left(f(\phi)\right) \text{  for all  } \phi \in \mathbb{R}.\]

We note that the right $\widetilde{GL}^+(2,\mathbb{R})$-action induces a right $\mathbb{C}$-action on $\stab(\mathcal{D})$ because one can see $\mathbb{C}$ as a subgroup of $\widetilde{GL}^+(2,\mathbb{R})$. 
More precisely, a complex number $z \in \mathbb{C}$ acts as follows: 
\begin{gather}
	\label{eqn C cation}
	(Z,P) \cdot z = \left(\exp(-i\pi z)\cdot Z, \left\{P\left(\phi+\operatorname{Re}(z)\right)\right\}_{\phi \in \mathbb{R}}\right).
\end{gather}

The second action is the left-action of $\Aut(\mathcal{D})$, where $\Aut(\mathcal{D})$ means the group of exact autoequivalences of $\mathcal{D}$. 
We note that an element $\Phi \in \Aut(\mathcal{D})$ induces a linear map on $K(\mathcal{D})$, which is denoted by $\Phi$ again. 
Then, the left-action of $\Aut(\mathcal{D})$ on $\stab(\mathcal{D})$ is given as follows:
\[\Phi \cdot (Z,P) = \left(Z \circ \Phi^{-1}, \left\{\Phi\left(P(\phi)\right)\right\}_{\phi \in \mathbb{R}}\right).\]

It is well-known that the above actions satisfy Proposition \ref{prop actions on stability conditions}.
\begin{prop}[Lemma 8.2 of \cite{Bridgeland07}, Lemma 2.8 of \cite{Kikuta22}]
	\label{prop actions on stability conditions}
	\mbox{}
	\begin{enumerate}
		\item The group $\Aut(\mathcal{D})$ acts on $\stab(\mathcal{D})$ isometrically.
		\item The right-action of $\widetilde{GL}^+(2,\mathbb{R})$ and the left-action $\Aut(\mathcal{D})$ on $\stab(\mathcal{D})$ commute. 
		\item The subgroup $\mathbb{C} \subset \widetilde{GL}^+(2,\mathbb{R})$ acts on $\stab(\mathcal{D})$ freely and isometrically.
	\end{enumerate}
\end{prop}

\subsection{Dehn twist actions on $\stab^\dagger(\Fuk)$}
\label{subsection Dehn twist actions on stab}
The goal of Section \ref{section translation of Penner type autoequivalences} is to show that an equivalence $\Phi$ of Penner type induces an action on $\stab^\dagger(\Fuk)$ having positive translation length.
We note that the left-action of $\Aut(\Fuk)$ implies that $\Phi$ acts on $\stab(\Fuk)$, not $\stab^\dagger(\Fuk)$. 
In this subsection, we show that the action of $\Phi$ preserves the connected component $\stab^\dagger(\Fuk)$, i.e., $\Phi$ acts on $\stab^\dagger(\Fuk)$.

Our strategy is to prove Lemma \ref{lem dehn twist action}.
\begin{lem}
	\label{lem dehn twist action}
	For any $v \in V(T)$, the action of $\tau_v$ on $\stab(\Fuk)$ preserves the connected component $\stab^\dagger(\Fuk)$.   	
\end{lem}  
Since $\Phi$ is a product of Dehn twists and inverses of Dehn twists, Lemma \ref{lem dehn twist action} is enough to prove that $\Phi$ acts on $\stab^\dagger(\Fuk)$. 

To prove Lemma \ref{lem dehn twist action}, we apply \cite[Lemma 5.2]{Bridgeland09} and \cite[Proposition 5.2]{King-Qiu15}. 
Before applying the known facts, we need to define some terminologies first. 

\begin{definition}
	\label{def heart of a stability conditions}
	Let $\mathcal{D}$ be a triangulated category. 
	\begin{enumerate}
		\item According to Proposition \ref{prop construction of a stability condition}, giving a stability condition $\sigma$ on $\mathcal{D}$ is equivalent to giving a bounded $t$-structure on $\mathcal{D}$ and a stability function on the heart $\mathcal{H}$ of the bounded $t$-structure. 
		We say that $\mathcal{H}$ is the {\bf heart of a stability condition $\boldsymbol{\sigma}$}.
		\item For a heart $\mathcal{H}$ of a bounded $t$-structure of $\Fuk$, let $\boldsymbol{\stab(\mathcal{H})}$ denote the set of stability conditions whose heart is $\mathcal{H}$.
	\end{enumerate}
\end{definition}

It is proven in \cite{Beilinson-Bernstein-Deligne81} that if $\mathcal{H}$ is a heart of a bounded $t$-structure of a triangulated category $\mathcal{D}$, then $\mathcal{H}$ is an abelian category. 
For an object of $\mathcal{H}$, we can consider the following terminologies:
\begin{definition}
	\label{def simple rigid}
	\mbox{}
	\begin{enumerate}
		\item An object of $\mathcal{H}$ is {\bf simple} if it is not the middle term of any (non-trivial) short exact sequence. 
		\item The set of simple objects of $\mathcal{H}$ is denoted by $\boldsymbol{\Sim(\mathcal{H})}$.
		\item A heart $\mathcal{H}$ of a bounded $t$-structure on $\mathcal{D}$ is {\bf finite} if $\mathcal{H}$ has only finitely many simple objects and any object of $\mathcal{H}$ is generated by finitely many extensions of simple objects. 
		\item An object $M$ of $\mathcal{H}$ is {\bf rigid} if $\operatorname{Ext}^1_\mathcal{D}(M,M)=0$. 
	\end{enumerate}
\end{definition}

We can state \cite[Lemma 5.2]{Bridgeland09}.
\begin{prop}[Lemma 5.2 of \cite{Bridgeland09}]
	\label{prop fixed heart}
	Let $\mathcal{H}$ be the heart of a bounded $t$-structure on $\mathcal{D}$. 
	If $\mathcal{H}$ is finite with $m$-many simple objects, then $\stab(\mathcal{H}) \subset \stab(\Fuk)$ is isomorphic to $H^m$, where 
	\[H:= \left\{r \exp(i\pi \phi) | r \in \mathbb{R}_{>0} \text{  and  } \phi \in (0,1]\right\}.\]
\end{prop}

Let $\mathcal{H}$ is a finite heart, and let $S$ be a rigid simple object of $\mathcal{H}$. 
Then, \cite[Proposition 5.4]{King-Qiu15} gives us another heart of a bounded $t$-structure of $\Fuk$ from $\mathcal{H}$.
\begin{prop}[Proposition 5.4 of \cite{King-Qiu15}]
	\label{prop forward/backward tilting 1}
	Let $\mathcal{D}$ be a triangulated category, and let $\mathcal{H}$ be a finite heart of a bounded $t$-structure of $\mathcal{D}$. 
	If $S$ is a rigid simple object of $\mathcal{H}$, then there exist two $t$-structures of $\mathcal{D}$ whose hearts $\mathcal{H}_S^\#, \mathcal{H}_S^\flat$ satisfy that 
	\begin{gather*}
		\Sim(\mathcal{H}^\#_S) = \left\{S[1]\right\} \cup \left\{\Cone\left(X \to S[1]\otimes \lhom_{\mathcal{D}}^1(X,S)\right) | X \in \Sim(\mathcal{H}), X \neq S\right\}, \\
		\Sim(\mathcal{H}^\flat_S) = \left\{S[-1]\right\} \cup \left\{\Cone\left(S[-1]\otimes \lhom_{\mathcal{D}}^1(S,X) \to X\right) | X \in \Sim(\mathcal{H}), X \neq S\right\}.
	\end{gather*}
\end{prop}

\begin{remark}
	The new hearts $\mathcal{H}^\#_S, \mathcal{H}^\flat_S$ in Proposition \ref{prop forward/backward tilting 1} are called the forward/backward tilts of $\mathcal{H}$ in \cite{King-Qiu15}. 
\end{remark}

Moreover, it is known that if $\mathcal{H}, \mathcal{H}_S^\#, \mathcal{H}_S^\flat$ are the hearts given in Proposition \ref{prop forward/backward tilting 1}, then, the closure of $\stab(\mathcal{H})$ intersects the closures of $\stab(\mathcal{H}^\#_S), \stab(\mathcal{H}^\flat_S)$. 
It is proven in \cite[Lemma 7.9]{Bridgeland-Smith15}.
\begin{prop}[Lemma 7.9 of \cite{Bridgeland-Smith15}]
	\label{prop forward/backward tilting 2}
	Let $\mathcal{D}$ be a triangulated category, and let $\mathcal{H} \subset \mathcal{D}$ be a finite heart. 
	Let us assume that 
	\begin{itemize}
		\item a stability condition $\sigma=(Z,P) \in \stab(\mathcal{D})$ lies on a unique boundary component of the region $\stab(\mathcal{H})$, so that $\operatorname{Im} Z(S) =0$ for a unique simple object $S$, and 
		\item $\mathcal{H}_S^\#, \mathcal{H}_S^\flat$ are also finite. 
	\end{itemize}
	Then, there exists a neighborhood $\sigma \in U \subset \stab(\mathcal{D})$ such that one of the following holds
	\begin{enumerate}
		\item[(i)] $Z(S) \in \mathbb{R}_{<0}$, and $U \subset \stab(\mathcal{H}) \sqcup \stab(\mathcal{H}_S^\flat)$,
		\item[(ii)] $Z(S) \in \mathbb{R}_{>0}$, and $U \subset \stab(\mathcal{H}) \sqcup \stab(\mathcal{H}_S^\#)$.
	\end{enumerate}
\end{prop}

Now, we are ready to prove Lemma \ref{lem dehn twist action}.
\begin{proof}[Proof of Lemma \ref{lem dehn twist action}]
	We recall that there exists a fixed stability condition $\sigma_0 \in \stab^\dagger(\Fuk)$ (Definition \ref{def fixed component of the space of stability conditions}).
	Since $\Aut(\Fuk)$ acts on $\stab(\Fuk)$ isometrically, for any $v \in V(T)$, the $\tau_v$-action gives a continuous map on $\stab(\Fuk)$. 
	Thus, it is enough to show that $\tau_v\cdot \sigma_0 \in \stab^\dagger(\Fuk)$.
	
	Before starting the proof, let us remind the notation. 
	In the proof of Lemma \ref{lem existence of a stability condition}, we constructed a bounded $t$-structure $\mathcal{T}$ of $\Fuk$. 
	The heart $\mathcal{A}$ of $\mathcal{T}$ consists of 
	\begin{itemize}
		\item the zero object and 
		\item the objects equivalent to a good, minimal twisted complex 
		\[E = \left(\bigoplus_{i \in I} V_i \otimes S_{v_i}, (F_{j,i} = \psi_{j,i} \otimes f_{j,i})\right).\]
	\end{itemize}
	Because every nonzero object of $\mathcal{A}$ is a direct sum of $\{S_v|v \in V(T)\}$, it is not hard to observe that $\mathcal{A}$ is a finite abelian category and
	\[\Sim(\mathcal{A}) = \left\{S_v | v \in V(T)\right\}.\]
	Moreover, for any $v \in V(T)$, $S_v$ is rigid in $\mathcal{A}$ because $N \geq 3$. 
	See Lemma \ref{lem grading}.
	
	We would like to point out that $\stab(\mathcal{A})$ is connected because $\stab(\mathcal{A})$ is isomorphic to $H^{|V(T)|}$ by Proposition \ref{prop fixed heart}.
	We also point out that $\sigma_0 \in \stab(\mathcal{A})$.
	Thus, $\stab(\mathcal{A}) \subset \stab^\dagger(\Fuk)$. 
	
	We would like to find a heart $\mathcal{A}_k$ for $k =1, \dots, N-1$ such that $\Sim(\mathcal{A}_k) \subset \stab^\dagger(\Fuk)$ and
	\begin{gather}
		\label{eqn induction hypothesis}
		\Sim(\mathcal{A}_k) = \left\{S_v[-k]\right\} \cup \left\{\tau_v(S_w) | w \in V(T), w \neq v\right\}.
	\end{gather}
	
	In order to construct such $\mathcal{A}_k$ for $k=1, \dots, N-1$ inductively, we first apply Proposition \ref{prop forward/backward tilting 1} to $\mathcal{A}$ and we can obtain another heart $\mathcal{A}_1$ of $\Fuk$ as follows:
	\[\mathcal{A}_1:= \mathcal{A}_{S_v}^\flat.\]
	Moreover, it is easy to observe that 
	\[\Sim(\mathcal{A}_1) = \left\{S_v[-1]\right\} \cup \left\{\tau_v(S_w) | w \in V(T), w \neq v\right\}.\]
	
	Now, in order to prove the inductive step, for an integer $1 \leq i \leq N-2$, let us assume that there exists a heart $\mathcal{A}_i$ satisfying the induction hypotheses. 
	Then, by applying Proposition \ref{prop forward/backward tilting 1}, we can set $\mathcal{A}_{i+1} := \left(\mathcal{A}_i\right)_{S_v[-i]}^\flat$. 
	Then, 
	\[\Sim(\mathcal{A}_{i+1}) = \left\{S_v[-(i+1)]\right\} \cup \left\{\Cone\left(S_v[-(i+1)]\otimes \lhom_{\Fuk}^1(S_v[-i],\tau_v(S_w)) \to \tau_v(S_w)\right) | w \in V(T), w \neq v\right\}.\]
	Moreover, we observe that 
	\begin{align*}
		\lhom_{\Fuk}^1\left(S_v[-i],\tau_v(S_w)\right) &\simeq \lhom_{\Fuk}^1(S_v[-i+N-1], S_w) \\
		&= \lhom_{\Fuk}^{2 +i -N}(S_v, S_w) \\
		&= 0.
	\end{align*}
	The last equality holds because $2 +i -N \leq 0$ and $w \neq v$.
	Thus, $\mathcal{A}_{i+1}$ satisfies Equation \eqref{eqn induction hypothesis}.
	
	Now, we would like to point out the following:
	\begin{itemize}
		\item For any $k = 1, \dots, N-1$, $\mathcal{A}_k$ is a finite abelian category. 
		In order to show this, one can employ the same logic we used to prove that $\mathcal{A}$ is finite. 
		Then, Proposition \ref{prop fixed heart} implies that $\stab(\mathcal{A}_k)$ is connected for all $k =1, \dots, N-1$.
		\item The closures of $\stab(\mathcal{A}_{k-1})$ and $\stab(\mathcal{A}_k)$ intersect for any $k =1, \dots, M-1$. 
		It means that for any $k=1, \dots, N-1$, $\mathcal{A}_k$ is contained in the same connected component. 
		Thus, $\stab(\mathcal{A}_k) \subset \stab^\dagger(\Fuk)$ for all $k$.
	\end{itemize}
	Thus, the inductive construction works. 
	
	At the final step of the inductive construction, we have $\mathcal{A}_{N-1} \subset \stab^\dagger(\Fuk)$ such that 
	\[\Sim(\mathcal{A}_{N-1}) = \left\{S_v[-(N-1)]=\tau_v(S_v)\right\} \cup \left\{\tau_v(S_w) | w \in V(T), w \neq v\right\}.\]
	Finally, we would like to show that $\tau_v \cdot \sigma_0 \in \stab(\mathcal{A}_{N-1})$. 
	In other words, we would like to show that the heart of $\tau_v \cdot \sigma_0$ is the abelian category $\mathcal{A}_{N-1}$.
	
	As mentioned in the proof of \cite[Lemma 5.3]{Bridgeland07}, the heart of $\tau_v \cdot \sigma_0 = \left(Z_0 \circ \tau_v^{-1}, \left\{\tau_v\left(P_0(\phi)\right)\right\}_{\phi \in \mathbb{R}}\right)$ is given as the subcategory generated by 
	\[\left\{\tau_v\left(P_0(\phi)\right)\right\}_{\phi \in (0,1]}.\]
	Thus, it implies that the heart of $\tau_v \cdot \sigma_0$ is $\mathcal{A}_{N-1}$, i.e., $\tau_v \cdot \sigma_0 \in \stab(\mathcal{A}_{N-1}) \subset \stab^\dagger(\Fuk)$. 
	It completes the proof.
\end{proof}

\section{Preparations for the proof of Theorem \ref{thm pseudo-Anosov intro}} 
\label{section preparations for the proof of pseudo-Anosov theorem} 
The main goal of the current article is to introduce a construction of pseudo-Anosov and strong pseudo-Anosov autoequivalences, i.e., to prove Theorem \ref{thm pseudo-Anosov intro}. 
In Sections \ref{section preparations for the proof of pseudo-Anosov theorem}--\ref{section proof of thm pseudo-Anosov}, we prove the first half of Theorem \ref{thm pseudo-Anosov intro}.
In other words, we prove Theorem \ref{thm pseudo-Anosov}. 
The second half of Theorem \ref{thm pseudo-Anosov intro} will be proven in Section \ref{section strong pseudo-Anosov} by employing the techniques introduced in Sections \ref{section preparations for the proof of pseudo-Anosov theorem}--\ref{section proof of thm pseudo-Anosov}.

\begin{thm}[= Theorem \ref{thm pseudo-Anosov intro} (1)]
	\label{thm pseudo-Anosov}
	Let $T$ be a tree, and let $\Phi: \Fuk \to \Fuk$ be an autoequivalence of Penner type. 
	Then, $\Phi$ is pseudo-Anosov with respect to any $\sigma \in \stab^\dagger(\Fuk)$.
	In other words, there exists a real number $\lambda_\Phi >1$ such that for any nonzero $E \in \Fuk$ and any $\sigma \in \stab^\dagger(\Fuk)$, 
	\begin{gather}
		\label{eqn pseduo-Anosov}
		\lim_{n \to \infty} \tfrac{1}{n} m_\sigma(\Phi^nE) = \log \lambda_\Phi.
	\end{gather}
\end{thm}

\begin{remark}
	Rigorously speaking, Equation \eqref{eqn pseduo-Anosov} is slightly stronger than that $\Phi$ is pseudo-Anosov. 
	In order to prove that $\Phi$ is pseudo-Anosov, it is enough to show that 
	\[\limsup_{n \to \infty} \tfrac{1}{n} m_\sigma(\Phi^n E) = \log \lambda_\Phi.\]
\end{remark}

The proof of Theorem \ref{thm pseudo-Anosov} will be given through Sections \ref{section preparations for the proof of pseudo-Anosov theorem}-\ref{section proof of thm pseudo-Anosov}.
In the current section, we will sketch the proof of Theorem \ref{thm pseudo-Anosov} and prove Lemmas \ref{lem x-morphism}--\ref{lem simple computation} to prepare Sections \ref{section the case of categorically carried-by} and \ref{section proof of thm pseudo-Anosov}.

\subsection{Sketch of the proof of Theorem \ref{thm pseudo-Anosov}}
\label{subsection sketch of the proof}
First of all, we would like to point out that we will prove Theorem \ref{thm pseudo-Anosov} in Sections \ref{section preparations for the proof of pseudo-Anosov theorem}-\ref{section proof of thm pseudo-Anosov} only if $\Phi$ is a product of 
\[\left\{\tau_u, \tau_w^{-1} | u \in V_+(T), w \in V_-(T)\right\}.\]
For the other case, i.e., the case that $\Phi^{-1}$ is a product of the above autoequivalences, we can prove Theorem \ref{thm pseudo-Anosov} by the same arguments after choosing the other sign convention. 
See Remark \ref{rmk sign is not important 1}.

Let us recall that, as mentioned in Section \ref{subsection the main idea}, the main idea is motivated by the notion of train track and the result of \cite{Lee19}.
To use the main idea, we defined the notion of {\em categorically carried-by}.
For the readers' convenience, we recall the notion. 
\begin{definition}[= Definition \ref{def categorically carried by}]
	\label{def categorically carried by 2}
	Let $T$ be a tree, and let $E$ be a nonzero object of $\Fuk$. 
	We say that $E$ is {\bf categorically carried-by} if $E$ is generated from $\{S_v\}_{v \in V(T)}$ by taking the direct sums, shifts, and cones of $x$- and $z$-morphisms. 
	More precisely, $E$ is categorically carried-by if $E$ is equivalent to a minimal good twisted complex	$\left(\bigoplus_{i \in I} V_i \otimes S_{v_i}[d_i], (F_{j,i}=\psi_{j,i} \otimes f_{j,i}) \right)$ such that $f_{j,i}$ is one of the zero morphism, $x$-morphism, or $z$-morphism.  
	By Lemma \ref{lem minimal twisted complex}, it is equivalent to saying that one can generate $E$ from $\{S_v\}_{v \in V(T)}$ without taking a cone of $y$-morphism.	
\end{definition}

We first show that if $E$ is categorically carried-by, then, for any $\Phi$ of Penner type, there exists a real number $\lambda_\Phi>1$ satisfying Equation \eqref{eqn pseduo-Anosov}.
It will be proven in Section \ref{section the case of categorically carried-by} through the following tree steps: 
\begin{itemize}
	\item Step 1 (= Section \ref{subsection step 1 of the sketch of proof}): We first show that if $E$ is categorically carried-by, then so is $\Phi E$. 
	\item Step 2 (= Section \ref{subsection step 2 of the sketch of proof}): After the first step, we will assign a vector $\vect(E) \in \mathbb{Z}_{\geq 0}^{|V(T)|}$ for any object $E \in \Fuk$. Then, we will prove that for any object $E$, the following holds: 
	\[\parallel \vect(E) \parallel_1 = m_{\sigma_0}(E).\] 
	Moreover, we will show that there exists a matrix $M_\Phi : \mathbb{R}^{|V(T)|} \to \mathbb{R}^{|V(T)|}$ satisfying that 
	\[\vect(\Phi E) = M_\Phi \cdot \vect(E),\]
	if $E$ is categorically carried-by.
	\item Step 3 (= Section \ref{subsection step 3 of the sketch of proof}): Finally, we will show that the spectral radius of $M_\Phi$ is the exponential growth of $m_{\sigma_0}(\Phi^n E)$ as $n \to \infty$ if $E$ is categorically carried-by. 
	Moreover, we also show that the spectral radius of $M_\Phi$ is bigger than $1$. 
	It completes the proof of Theorem \ref{thm pseudo-Anosov} for a categorically carried-by $E$. 
\end{itemize}
\begin{remark}
	\label{rmk hidden step}
	In step 3, we will also prove that for any nonzero $E \in \Fuk$,
	\[\parallel \vect(E) \parallel_1 \leq \parallel M_\Phi \cdot \vect(E) \parallel_1.\]
	Thus, one can observe that for any $E$, 
	\[\lim_{n \to \infty} \tfrac{1}{n} \log m_{\sigma_0}(\Phi^n E) \leq \log \lambda_\Phi.\]
\end{remark}

We note that not every nonzero object $E \in \Fuk$ is categorically carried-by. 
To discuss the case, we define the notion of {\em partially carried-by}.
Since the notion is complicated, we omit stating the formal definition of ``partially carried-by", but roughly $E$ is partially carried-by if 
$E$ is equivalent to the cone of a morphism $f: A \to B$ such that the triple $(A, B, f)$ satisfying {\em some conditions}. 
One of the conditions which the above triple $(A,B,f)$ should satisfy is that $B$ is categorically carried-by.
It is the reason why we use the term ``partially carried-by."

Then, we prove the following steps in Section \ref{section proof of thm pseudo-Anosov}:
\begin{itemize}
	\item Step 4: We note that if $E \simeq \Cone\left(A \stackrel{f}{\to} B\right)$, then $\Phi E = \Cone \left(\Phi A \stackrel{g}{\to} \Phi B\right)$ for some $g$. 
	We prove that if the triple $(A,B,f)$ satisfies some conditions, then $(\Phi A, \Phi B, g)$ also satisfies the same conditions. 
	If the triple $(A, B, f)$ satisfies the conditions, we will say that $E$ is {\em partially carried-by with a triple $(A,B,f)$}. 
	See Definition \ref{definition partially carried-by 2}. 
	\item Step 5: We will show that for any nonzero $E \in \Fuk$, there exists a positive integer $K$ such that $\Phi^KE$ is partially carried-by. 
	\item Step 6: We will show that if $E$ is partially carried-by with a triple $(A,B,f)$, then $m_{\sigma_0}(E) \geq m_{\sigma_0}(B)$. 
	Thus, by steps 4 and 5, for any nonzero object $E$, we have 
	\[\lim_{n \to \infty} \tfrac{1}{n} \log m_{\sigma_0}\left(\Phi^{n-k}(\Phi^K E)\right)  \geq \lim_{n \to \infty} \tfrac{1}{n} \log m_{\sigma_0}(\Phi^n B) = \log \lambda_\Phi,\]
	for an object $B$ such that $B$ is categorically carried-by.
	The last equality holds since $B$ is categorically carried-by. 
	Moreover, by Remark \ref{rmk hidden step}, we have 
	\[\lim_{n \to \infty} \tfrac{1}{n} \log m_{\sigma_0}(\Phi^n E) = \log \lambda_\Phi.\]
	It completes the proof of Theorem \ref{thm pseudo-Anosov}.
\end{itemize}

As mentioned above, Section \ref{section the case of categorically carried-by} will prove steps 1-3, and Section \ref{section proof of thm pseudo-Anosov} will prove steps 4-6. 
The rest of Section \ref{section preparations for the proof of pseudo-Anosov theorem} is devoted to proving lemmas in order to prepare Sections \ref{section the case of categorically carried-by} and \ref{section proof of thm pseudo-Anosov}.

\subsection{Setting for Lemmas} 
\label{subsection setting for lemmas} 
As we did in the previous sections, let $T$ be a fixed tree, let $N \geq 3$ be a fixed integer, and let $\Fuk, \wrapped$ denote the $\mathcal{D}\left(\Gamma_N T\right)$ and the perfect derived category of $\Gamma_NT$, respectively.
Or equivalently, $\Fuk$ and $\wrapped$ can be seen as the compact and wrapped Fukaya categories of $P_N(T)$, as mentioned in Table \ref{table category=symplectic}.
Also, as mentioned in Sections \ref{section settings and the main idea} and \ref{section twisted complex}, we know that $\Fuk$ is equivalent to the category of twisted complexes generated by $\{S_v\}_{v \in V}$. 

Throughout Section \ref{section preparations for the proof of pseudo-Anosov theorem}, we let 
\[E = \left(\bigoplus_{i} (X_i:= V_i \otimes S_{v_i}[d_i]), (F_{j,i} = \psi_{j,i} \otimes f_{j,i}) \ \right)\]
be a {\em minimal, well-ordered} twisted complex, or equivalently, 
\begin{itemize}
	\item $V_i$ is a finite dimensional $\mathbb{k}$-vector space for all $i$, 
	\item $v_i \in V(T)$ and $d_i \in \mathbb{Z}$ for all $i$, 
	\item $\psi_{j,i}: V_i \to V_j$ is a linear map for all $i < j$, 
	\item $f_{j,i} \in \shom_{\Fuk}(S_{v_i}[d_i],S_{v_j}[d_j])$ is one of $x, y, z$-morphisms or the zero morphism. 
\end{itemize}

We are interested in $\tau(E)$ for 
\[\tau = \begin{cases}
	\tau_u \text{  for some  } u \in V_+, \\
	\tau_w^{-1} \text{  for some  } w \in V_-.
\end{cases}\]

By Lemma \ref{lem induced functor}, $\tau(E)$ is equivalent to the following twisted complex:
\[\left(\bigoplus_i V_i \otimes \tau(S_{v_i}[d_i]), G_{j,i} : \tau(X_i) \to \tau(X_j) \right).\]
We note that $G_{j,i}$ is determined by the functor $\tau$ and the nonzero chains of arrows from $X_i$ to $X_j$, see Remark \ref{rmk nonzero chain of arrows}.
Moreover, since $\tau(X_i)$ and $\tau(X_j)$ can be seen as good twisted complexes (See Lemma \ref{lem simple computation}), $G_{j,i}$ can be seen as a combination of the $x, y, z$-morphisms, identity morphisms, and zero morphisms.

The main goal of Section \ref{section preparations for the proof of pseudo-Anosov theorem} is to prove the following Lemmas \ref{lem x-morphism}--\ref{lem the zero morphism}.
\begin{lem}
	\label{lem x-morphism}
	If $f_{j,i}$ is an $x$-morphism, then $G_{j,i}$ does not contain neither the identity morphism $e_v$ for any $v \in V$ nor a $y$-morphism.  
\end{lem}

\begin{lem}
	\label{lem z-morphism}
	If $f_{j,i}$ is a $z$-morphism, then $G_{j,i}$ does not contain neither the identity morphism $e_v$ for any $v \in V$ nor a $y$-morphism.  
\end{lem}

\begin{lem}
	\label{lem the zero morphism}
	If $f_{j,i}$ is the zero morphism, then $G_{j,i}$ does not contain the identity morphism $e_v$ for any $v \in V$.  
\end{lem}

In Sections \ref{subsection proof of lemma 1}--\ref{subsection proof of lemma 3}, we prove the above Lemmas \ref{lem x-morphism}--\ref{lem the zero morphism}.
In the proof, we will repeatedly use Lemma \ref{lem simple computation} that is a consequence of simple computation. 

\begin{lem}
	\label{lem simple computation}
	Let $u, u' \in V_+, w, w' \in V_-$.
	The following equations hold.
	\begin{gather}
		\label{eqn 1}
		\tau_u(S_{u'}) = \begin{cases}
			S_u[1-N] \text{  if  } u= u', \\
			S_{u'} \text{  if  } u \neq u'.
		\end{cases} \\
		\label{eqn 2}
		\tau_u(S_w) = \begin{cases}
			\left(S_u \oplus S_w, x_{u,w} \in \shom_{\Fuk}^1(S_u, S_w)\right) \text{  if  } u \sim w, \\
			S_w \text{  if  } u \nsim w.
		\end{cases} \\
		\label{eqn 3}
		\tau_w^{-1}(S_u) = \begin{cases}
			\left(S_u \oplus S_w, x_{u,w} \in \shom_{\Fuk}^1(S_u, S_w)\right) \text{  if  } w \sim u,  \\
			S_u \text{  if  } w \nsim u.
		\end{cases} \\
		\label{eqn 4}
		\tau_w^{-1}(S_{w'}) = \begin{cases}
			S_w[N-1] \text{  if  } w = w', \\
			S_{w'} \text{  if  } w \neq w'.
		\end{cases}
	\end{gather} 
\end{lem}
We recall that two vertices $u, w \in V(T)$ are said to be $u \sim w$ if $u$ and $w$ are connected by an edge in $T$. 

\subsection{Proof of Lemma \ref{lem x-morphism}}
\label{subsection proof of lemma 1}
We note that Lemma \ref{lem x-morphism} is a simple corollary of Lemma \ref{lem simple computation}.
We write details below for an eager reader, but the basic idea is to prove by cases and to apply Lemma \ref{lem simple computation} for each case. 
 
In Lemma \ref{lem x-morphism}, $f_{j,i}$ is an $x$-morphism, thus, there exists a positive vertex $u_1 \in V_+$ and a negative vertex $w_1 \in V_-$ such that 
\[X_i = V_i \otimes S_{u_1}[d_i], X_j= V_j \otimes S_{w_1}[d_j], \text{  and  } u_1 \sim w_1.\]
Moreover, since the grading of an $x$-morphism is always $1$, we know that $d_i = d_j$. 
By shifting $E$ by $d_i$, we can assume that $d_i= d_j =0$. 

There are following six cases:
\begin{enumerate}
	\item[(i)] $\tau = \tau_{u_1}$.
	\item[(ii)] $\tau = \tau_{u_2}$ such that $u_1 \neq u_2 \sim w_1$.
	\item[(iii)] $\tau = \tau_{u_3}$ such that $u_1 \neq u_3 \nsim w_1$. 
	\item[(iv)] $\tau = \tau_{w_1}^{-1}$.
	\item[(v)] $\tau = \tau_{w_2}^{-1}$ such that $u_1 \sim w_2 \neq w_1$.
	\item[(vi)] $\tau = \tau_{w_3}^{-1}$ such that $u_1 \nsim w_3 \neq w_1$.
\end{enumerate}
\vskip0.2in

\noindent {\em Case (i).\ }$\tau = \tau_{u_1}$ : By Equations \eqref{eqn 1} and \eqref{eqn 2} in Lemma \ref{lem simple computation}, 
\[\tau(S_{u_1}) = S_{u_1}[1-N], \tau(S_{w_1}) = \left( S_{u_1} \oplus S_{w_1}, x_{u_1,w_1} \in \shom_{\Fuk}^1(S_{u_1}, S_{w_1}) \right).\]
Thus, $G_{j,i} : V_i \otimes \tau(X_i = S_{u_1}) \to V_j \otimes \tau(S_{w_1})$ should consist of a linear map from $V_i$ to $V_j$ and morphism in
\[\shom^1_{\Fuk}(S_{u_1}[1-N],S_{u_1}) \text{  and  } \shom^1_{\Fuk}(S_{u_1}[1-N], S_{w_1}).\]
We note that $\shom^1_{\Fuk}(S_{u_1}[1-N],S_{u_1})$ is generated by $z_{u_1}$ and $\shom^1_{\Fuk}(S_{u_1}[1-N], S_{w_1}) = 0$.
Thus, $G_{j,i}$ cannot contain neither an identity morphism or a $y$-morphism. 
\vskip0.2in

\noindent {\em Case (ii).\ }$\tau = \tau_{u_2}$ such that $u_2 \sim w_1$ : Similar to the case (i), by Equations in Lemma \ref{lem simple computation},
\[\tau(S_{u_1}) = S_{u_1}, \tau(S_{w_1}) = \left( S_{u_2} \oplus S_{w_1}, x_{u_1,w_1} \in \shom_{\Fuk}^1(S_{u_2}, S_{w_1}) \right).\]
Thus, $G_{j,i}$ consists of a linear map from $V_i$ to $V_j$ and morphisms in
\[\shom^1_{\Fuk}(S_{u_1},S_{u_2}) \text{  and  } \shom^1_{\Fuk}(S_{u_1}, S_{w_1}).\]
Since $\shom^1_{\Fuk}(S_{u_1},S_{u_2}) =0$ and $\shom^1_{\Fuk}(S_{u_1}, S_{w_1})$ is generated by $x_{u_1,w_1}$, $G_{j,i}$ cannot contain neither an identity morphism or a $y$-morphism. 
\vskip0.2in

\noindent {\em Case (iii).\ }$\tau = \tau_{u_3}$ such that $u_3 \nsim w_1$ :
As we did before, Lemma \ref{lem simple computation} implies that
\[\tau(S_{u_1}) = S_{u_1}, \tau(S_{w_1})= S_{w_1}.\]
Thus, $G_{j,i}$ consists of a linear map from $V_i$ to $V_j$ and a morphism in $\shom_{\Fuk}^1(S_{u_1},S_{w_1})$ that is generated by $x_{u_1,w_1}$. 
It proves that $G_{j,i}$ cannot contain neither an identity morphism or a $y$-morphism. 
\vskip0.2in

{\em Cases (iv, v, vi).\ }: The proof of cases (i), (ii), and (iii) also prove the cases (iv), (v), and (vi), respectively.
\qed

\subsection{Proof of Lemma \ref{lem z-morphism}}
\label{subsection proof of lemma 2}
Similar to the proof of Lemma \ref{lem x-morphism}, a proof by cases proves Lemma \ref{lem z-morphism}.
And, the proof for each case below is a corollary of Lemma \ref{lem simple computation}.

In Lemma \ref{lem z-morphism}, $f_{j,i} \in \shom_{\Fuk}^1(S_{v_i}[d_i], S_{v_j}[d_j])$ is a $z$-morphism.
Thus, $v_i = v_j$ and $d_j = d_i +N -1$. 
We note that $v_i = v_j \in V_+(T)$ or $V_-(T)$.
In Section \ref{subsection proof of lemma 2}, we prove Lemma \ref{lem z-morphism} only for the case of $v_i = v_j \in V_+$.
The same logic works for the case of $v_i= v_j \in V_-$.

Let $u_1 \in V_+$ be the positive vertex such that $u_1 = v_i = v_j$. 
Without loss of generality, we assume that $d_i= 0$, or equivalently, $d_j = N-1$.  

There exist four possible cases.
\begin{enumerate}
	\item[(i)] $\tau = \tau_{u_1}$.
	\item[(ii)] $\tau = \tau_{w_1}^{-1}$ such that $w_1 \sim u_1$.
	\item[(iii)] $\tau = \tau_{u_2}$ such that $u_2 \neq u_1$.
	\item[(iv)] $\tau = \tau_{w_2}^{-1}$ such that $w_2 \nsim u_1$.
\end{enumerate}
\vskip0.2in

\noindent {\em Case (i).\ }$\tau = \tau_{u_1}$: Lemma \ref{lem simple computation} induces that
\[\tau(S_{u_1}) = S_{u_1}[1-N], \tau(S_{u_1}[N-1]) = S_{u_1}.\]
Thus, $G_{j,i} : V_i \otimes \tau(X_i = S_{u_1}) \to V_j \otimes \tau(X_j = S_{u_1}[N-1])$ should consist of a linear map from $V_i$ to $V_j$ and a morphism in
\[\shom^1_{\Fuk}(S_{u_1}[1-N],S_{u_1}).\]
We note that $\shom^1_{\Fuk}(S_{u_1}[1-N],S_{u_1})$ is generated by $z_{u_1}$.
Thus, $G_{j,i}$ cannot contain neither an identity morphism or a $y$-morphism. 
\vskip0.2in

\noindent {\em Case (ii).\ }$\tau = \tau_{w_1}^{-1}$ such that $w_1 \sim u_1$: Again, by Lemma \ref{lem simple computation}, we have  
\[\tau(X_i = S_{u_1}) =  \left( S_{u_1} \oplus S_{w_1}, x_{u_1,w_1} \in \shom^1_{\Fuk}(S_{u_1}, S_{w_1}) \right), \text{  and  } \tau(X_j = S_{u_1}[N-1]) = \tau(S_{u_1})[N-1].\]
Thus, $G_{j,i}$ consists of a linear map from $V_i$ to $V_j$ and morphisms in
\[\shom^1_{\Fuk}(S_{u_1},S_{u_1}[N-1]), \shom^1_{\Fuk}(S_{u_1},S_{w_1}[N-1]), \shom^1_{\Fuk}(S_{w_1},S_{u_1}[N-1]), \text{  and  } \shom^1_{\Fuk}(S_{w_1},S_{w_1}[N-1]).\]
We note that $\shom^1_{\Fuk}(S_{u_1},S_{u_1}[N-1])$ is generated by $z_{u_1}$, $\shom^1_{\Fuk}(S_{u_1},S_{w_1}[N-1]) = \shom^1_{\Fuk}(S_{w_1},S_{u_1}[N-1]) = 0$, and $\shom^1_{\Fuk}(S_{w_1},S_{w_1}[N-1])$ is generated by $z_{w_1}$. 
Thus, $G_{j,i}$ cannot contain neither an identity morphism or a $y$-morphism. 
\vskip0.2in

\noindent {\em Cases (iii, iv).\ }$\tau= \tau_{u_2}$ or $\tau_{w_2}^{-1}$ such that $u_2 \neq u_1, w_1 \nsim u_1$: It is easy to observe from Lemma \ref{lem simple computation} that 
\[\tau(S_{u_1}) = S_{u_1}.\]
Thus, $G_{j,i}$ consists of a linear map and a morphism in
\[\shom_{\Fuk}^1(S_{u_1}, S_{u_1}[N-1]),\]
which is generated by $z_{u_1}$. 
It completes the proof. 
\qed

\subsection{Proof of Lemma \ref{lem the zero morphism}}
\label{subsection proof of lemma 3}
We note that $\tau$ is either $\tau_{u_1}$ for some $u_1 \in V_+$ or $\tau_{w_1}^{-1}$ for some $w_1 \in V_-$. 
We prove Lemma \ref{lem the zero morphism} for the case of $\tau = \tau_{u_1}$ with $u_1 \in V_+$ in the current subsection, and we skip the other case because the same proof works for the other case. 
And, as similar to Sections \ref{subsection proof of lemma 1} and \ref{subsection proof of lemma 2}, we prove Lemma \ref{lem the zero morphism} by a case-by-case method.

Since $f_{j,i} \in \shom_{\Fuk}^1\left(X_i=S_{v_i}[d_i], X_j=S_{v_j}[d_j]\right)$ is the zero morphism, we could not get any information about $v_i$ and $v_j$, differently from the proofs of Lemmas \ref{lem x-morphism} and \ref{lem z-morphism}.
Thus, we should consider the following four cases under the assumption that $d_i=0$:
\begin{enumerate}
	\item[(i)] $X_i = V_i \otimes S_{u_2}$ and $X_j = V_j \otimes S_{u_3}[k]$ with $u_2, u_3 \in V_+$. 
	\item[(ii)] $X_i = V_i \otimes S_{u_2}$ and $X_j = V_j \otimes S_{w_1}[k]$ with $u_2 \in V_+, w_1 \in V_-$. 
	\item[(iii)] $X_i = V_i \otimes S_{w_1}$ and $X_j = V_j \otimes S_{u_2}[k]$ with $u_2 \in V_+, w_1 \in V_-$. 
	\item[(iv)] $X_i = V_i \otimes S_{w_1}$ and $X_j = V_j \otimes S_{w_2}[k]$ with $w_1, w_2 \in V_-$.
\end{enumerate}

\vskip0.2in

\noindent {\em Case (i).\ }$X_i = V_i \otimes S_{u_2}$ and $X_j = V_j \otimes S_{u_3}[k]$ with $u_2, u_3 \in V_+$ : Case (i) can be divided into the following four sub-cases given in Table \ref{tab case 1}:
\begin{table}[h!]
	\caption{Sub-cases of case (i).}
	\label{tab case 1}
	\begin{tabular}{| l | c | c |}
		\hline
		& $u_2 = u_1$  & $u_2 \neq u_1$ \\ \hline
		$u_3 = u_1$  & a) & c) \\ \hline
		$u_3 \neq u_1$ & b) & d) \\ \hline
	\end{tabular}
\end{table}

Case (i)-a): As we did in the proofs of Lemmas \ref{lem x-morphism} and \ref{lem z-morphism}, we use Lemma \ref{lem simple computation}.
Then, we have 
\[G_{j,i} : V_i \otimes \tau(X_i = S_{u_2=u_1}) = V_i \otimes S_{u_1}[1-N] \to V_j \otimes \tau(X_j=S_{u_3 = u_1}[k]) = V_j \otimes S_{u_1}[k+1 -N].\]
Thus, if $G_{j,i}$ contains an identity morphism $e_v$, then $v$ should be $u_1$ and $k$ should be $-1$. 
However, $k \geq 0$ since $E$ is well-ordered. 
It means that $G_{j,i}$ cannot contain an identity. 

Case (i)-b): Similar to the above sub-case, we observe that 
\[G_{j,i} : V_i \otimes \tau(X_i = S_{u_2=u_1}) = V_i \otimes S_{u_1}[1-N] \to V_j \otimes \tau(X_j=S_{u_3}[k]) = V_j \otimes S_{u_3}[k].\]
Since $u_2 = u_1 \neq u_3$, $G_{j,i}$ cannot contain an identity morphism. 

Case (i)-c): Again, we observe that
\[G_{j,i} : V_i \otimes \tau(X_i = S_{u_2}) = V_i \otimes S_{u_2} \to V_j \otimes \tau(X_j=S_{u_3=u_1}[k]) = V_j \otimes S_{u_1}[k+1-N].\]
Since $u_1 = u_3 \neq u_2$, $G_{j,i}$ cannot contain an identity morphism.

Case (i)-d): We can see that 
\[G_{j,i} : V_i \otimes \tau(X_i = S_{u_2}) = V_i \otimes S_{u_2} \to V_j \otimes \tau(X_j=S_{u_3}[k]) = V_j \otimes S_{u_3}[k].\]
If $G_{j,i}$ contains an identity morphism, then $u_2 =u_3$ and $k =-1$. 
However, $k \geq 0$ since $E$ is well-ordered. 
It implies that $G_{j,i}$ does not contain an identity.
\vskip0.2in

\noindent {\em Case (ii).\ }$X_i = V_i \otimes S_{u_2}$ and $X_j = V_j \otimes S_{w_1}[k]$ with $u_2 \in V_+, w_1 \in V_-$ : We will consider sub-cases given in Table \ref{tab case 2}. 
\begin{table}[h!]
	\caption{Sub-cases of case (ii).}
	\label{tab case 2}
	\begin{tabular}{| l | c | c |}
		\hline
		& $u_2 = u_1$  & $u_2 \neq u_1$ \\ \hline
		$w_1 \sim u_1$  & a) & c) \\ \hline
		$w_1 \nsim u_1$ & b) & d) \\ \hline
	\end{tabular}
\end{table}

Case (ii)-a): Thanks to Lemma \ref{lem simple computation},
\[G_{j,i} : V_i \otimes \tau(X_i = S_{u_2=u_1}) = V_i \otimes S_{u_1}[1-N] \to V_j \otimes \tau(X_j=S_{w_1}[k]) = V_j \otimes \left(S_{u_1}[k] \oplus S_{w_1}[k], x_{u_1,w_1}\right).\]
Thus, $G_{j,i}$ consists of a linear map from $V_i$ to $V_j$, and morphisms in 
\[\shom_{\Fuk}^1(S_{u_1},S_{u_1}[k]) \text{  and  } \shom_{\Fuk}^1(S_{u_1}, S_{w_1}[k]).\]
If $G_{j,i}$ contains an identity morphism $e_v$, then $v = u_1$ and $k=-1$. 
As similar to the cases (i)-a), d), it contradicts to that $E$ is well-ordered. 
Thus, $G_{j,i}$ does not have an identity morphism.

Case (ii)-b): We observe that 
\[G_{j,i} : V_i \otimes \tau(X_i = S_{u_2=u_1}) = V_i \otimes S_{u_1}[1-N] \to V_j \otimes \tau(X_j=S_{w_1}[k]) = V_j \otimes S_{w_1}[k].\]
Since $u_1 \neq w_1$, there is no identity morphism in $G_{j,i}$. 

Case (ii)-c): We use Lemma \ref{lem simple computation} again, then we know
\[G_{j,i} : V_i \otimes \tau(X_i = S_{u_2}) = V_i \otimes S_{u_2} \to V_j \otimes \tau(X_j=S_{w_1}[k]) = V_j \otimes \left(S_{u_1}[k] \oplus S_{w_1}[k], x_{u_1,w_1}\right).\]
Since $u_2 \neq u_1, u_2 \neq w_1$, $G_{j,i}$ cannot have an identity morphism. 

Case (ii)-d): As usual, from Lemma \ref{lem simple computation},
\[G_{j,i} : V_i \otimes \tau(X_i = S_{u_2}) = V_i \otimes S_{u_2} \to V_j \otimes \tau(X_j=S_{w_1}[k]) = V_j \otimes S_{w_1}[k].\]
Since $u_2 \neq w_1$, $G_{j,i}$ do not contain an identity morphism. 
\vskip0.2in

\noindent {\em Case (iii).\ }$X_i = V_i \otimes S_{w_1}$ and $X_j = V_j \otimes S_{u_2}[k]$ with $u_2 \in V_+, w_1 \in V_-$ : Case (iii) is also divided into four sub-cases in Table \ref{tab case 3}.
\begin{table}[h!]
	\caption{Sub-cases of case (iii).}
	\label{tab case 3}
	\begin{tabular}{| l | c | c |}
		\hline
		& $w_1 \sim u_1$  & $w_1 \nsim u_1$ \\ \hline
		$u_2 = u_1$  & a) & c) \\ \hline
		$u_2 \neq u_1$ & b) & d) \\ \hline
	\end{tabular}
\end{table}

Case (iii)-a) : We have 
\[G_{j,i} : V_i \otimes \tau(X_i = S_{w_1}) = V_i \otimes \left(S_{u_1} \oplus S_{w_1}, x_{u_1,w_1}\right) \to V_j \otimes \tau(X_j=S_{u_2 = u_1}[k]) = V_j \otimes S_{u_1}[k+1 -N].\]
If $G_{j,i}$ contains an identity morphism $e_v$, $v$ should be $u_1$ and $k$ should be $N-2$.

We remark that, if $G_{j,i}$ is not a zero map, then there exists at least one nonzero chain of arrows in $E$ from $X_i$ to $X_j$. 
For a nonzero chain of arrows from $X_i$ to $X_j$, let $X_\ell$ be the last object of the chain before $X_j$.
In particular, we know that 
\[F_{j,\ell} = \psi_{j,\ell} \otimes f_{j,\ell}: X_\ell = V_\ell \otimes S_{v_\ell}[d_\ell] \to X_j = V_j \otimes S_{u_1}[N-2]\] 
is not a zero map. 
It implies that $f_{j,\ell}$ is either a $y$-morphism or a $z$-morphism. 

If $f_{j,\ell}$ is a $z$-morphism, then $d_\ell = -1 < d_i =0$. 
Since $E$ is well-ordered, $\ell < i$. 
It contradicts to that $X_\ell$ is a part of a nonzero chain of arrows from $X_i$ to $X_j$. 
Thus, $f_{j,i}$ should be a $y$-morphism.

If $f_{j,\ell}$ is a $y$-morphism, then $v_\ell \in V_-$ and $d_\ell = 0$. 
It implies that every chain of arrows from $X_i=V_i \otimes S_{v_i}$ to $X_\ell=V_\ell \otimes S_{v_\ell}$ cannot be a nonzero chain, since $v_i, v_\ell \in V_-(T)$.
It is a contradiction to the existence of nonzero chain of arrows from $X_i$ to $X_\ell$.
Thus, $G_{j,i}$ cannot contain an identity morphism. 

Case (iii)-b) : By applying Lemma \ref{lem simple computation},
\[G_{j,i} : V_i \otimes \tau(X_i = S_{w_1}) = V_i \otimes \left(S_{u_1} \oplus S_{w_1}, x_{u_1,w_1}\right) \to V_j \otimes \tau(X_j=S_{u_2}[k]) = V_j \otimes S_{u_2}[k],\]
and $u_2 \neq u_1$. 
Thus, $G_{j,i}$ cannot contain an identity morphism. 

Case (iii)-c) : Lemma \ref{lem simple computation} says that 
\[G_{j,i} : V_i \otimes \tau(X_i = S_{w_1}) = V_i \otimes S_{w_1} \to V_j \otimes \tau(X_j=S_{u_1}[k]) = V_j \otimes S_{u_1}[k+1 -N].\]
Thus, $G_{j,i}$ cannot have an identity morphism since $u_1 \neq w_1$.  

Case (iii)-d) : Again, 
\[G_{j,i} : V_i \otimes \tau(X_i = S_{w_1}) = V_i \otimes S_{w_1} \to V_j \otimes \tau(X_j=S_{u_2}[k]) = V_j \otimes S_{u_2}[k].\]
It proves that $G_{j,i}$ cannot have an identity morphism. 
\vskip0.2in

\noindent {\em Case (iv).\ }$X_i = V_i \otimes S_{w_1}$ and $X_j = V_j \otimes S_{w_2}[k]$ with $w_1, w_2 \in V_-$ : 
\begin{table}[h!]
	\caption{Sub-cases of case (iv).}
	\label{tab case 4}
	\begin{tabular}{| l | c | c |}
		\hline
		& $w_1 \sim u_1$  & $w_1 \nsim u_1$ \\ \hline
		$w_2 \sim u_1$  & a) & c) \\ \hline
		$u_2 \nsim u_1$ & b) & d) \\ \hline
	\end{tabular}
\end{table}

Case (iv)-a) : As we did before, we observe that 
\[G_{j,i} : V_i \otimes \tau(S_{w_1}) = V_i \otimes \left(S_{u_1} \oplus S_{w_1}, x_{u_1,w_1}\right) \to V_j \otimes \tau(S_{w_2}[k]) = V_j \otimes \left(S_{u_1} \oplus S_{w_2}, x_{u_1,w_2}\right).\]
If $G_{j,i}$ contains an identity $e_v$, then $v$ should be either $u_1$ or $w_1 =w_2$. 
Moreover, for both cases, $k=-1$, and it contradicts to that $E$ is well-ordered. 
 
Case (iv)-b) : Again, we apply Lemma \ref{lem simple computation}, and
\[G_{j,i} : V_i \otimes \tau(S_{w_1}) = V_i \otimes \left(S_{u_1} \oplus S_{w_1}, x_{u_1,w_1}\right) \to V_j \otimes \tau(S_{w_2}[k]) = V_j \otimes S_{w_2}[k].\]
If $G_{j,i}$ has an identity morphism, then the only possible case is that $w_1 =w_2$.
It is not possible since $w_1 \sim u_1 \nsim w_2$. 

Case (iv)-c) : As usual, we observe that
\[G_{j,i} : V_i \otimes \tau(S_{w_1}) = V_i \otimes S_{w_1} \to V_j \otimes \tau(S_{w_2}[k]) = V_j \otimes \left(S_{u_1} \oplus S_{w_2}, x_{u_1,w_2}\right).\]
If $G_{j,i}$ contains an identity morphism, it implies that $w_1 = w_2$, but it is not possible since $w_1 \nsim u_1 \sim w_2$. 

Case (iv)-d) : We again use Lemma \ref{lem simple computation} for the final sub-case, and 
\[G_{j,i} : V_i \otimes \tau(S_{w_1}) = V_i \otimes S_{w_1} \to V_j \otimes \tau(S_{w_2}[k]) = V_j \otimes S_{w_2}[k].\]
If we assume that $G_{j,i}$ contains an identity morphism, $w_1$ and $w_2$ should be the same and $k$ should be $-1$. 
It contradicts to that $E$ is well-ordered.
\qed

\section{The case of categorically carried-by}
\label{section the case of categorically carried-by}
As mentioned in the introduction section, if $\Phi: P_N(T) \to P_N(T)$ is of Penner type, then \cite{Lee19} showed that $\Phi$ satisfies a geometric stability. 
Roughly speaking, the stability means that if a closed Lagrangian $L$ satisfies a condition, then $\Phi^n(L)$ converges to a closed subset as $n \to \infty$, which is depending only on $\Phi$ and independent of $L$. 
We note that the condition is not stated in the current paper, but if $L$ satisfies the condition, then $L$ is categorically carried-by as an object of $\Fuk$.

From the result of \cite{Lee19}, it would be natural to expect that if $E \in \Fuk$ is categorically carried-by, then the asymptotic behavior of $\Phi^n(E)$ is independent of the choice of $E$. 
Motivated from this, we prove Theorem \ref{thm without y morphism} in Section \ref{section the case of categorically carried-by}, or equivalently, we prove steps $1$--$3$ of the sketch of proof of Theorem \ref{thm pseudo-Anosov}, given in Section \ref{subsection sketch of the proof}. 
\begin{thm}
	\label{thm without y morphism}
	Let $T$ be a given tree. 
	Let $\sigma$ be a stability condition in $\stab^\dagger(\Fuk)$. 
	If an autoequivalence $\Phi: \mathcal{D} \to \mathcal{D}$ is of Penner type and if a nonzero object $E \in \Fuk$ is categorically carried-by, then there exists a real number $\lambda_\Phi > 1$ such that 
	\[\lim_{n \to \infty} \tfrac{1}{n} m_\sigma(\Phi^n E) = \lambda_\Phi.\]
	Moreover, $\lambda_\Phi$ depends only on $\Phi$.
\end{thm}

\subsection{Step 1 of the sketch of proof of Theorem \ref{thm pseudo-Anosov}}
\label{subsection step 1 of the sketch of proof} 
We start the proof of Theorem \ref{thm without y morphism} by showing that $\Phi^n E$ is categorically carried-by for all $n \in \mathbb{N}$ if $E$ is categorically carried-by.
Our strategy is to show that if $E$ is categorically carried-by, then so is $\tau(E)$ where $\tau$ is an element of $\left\{\tau_u, \tau_w^{-1} | u \in V_+(T), w \in V_-(T)\right\}$. 
We recall our assumption that $\Phi$ is a product of the above Dehn twists and inverses of Dehn twists. 

The key point of proving that $\tau E$ is categorically carried-by is to show that as a twisted complex, $\tau(E)$ does not have a $y$-morphism., i.e., to prove Lemma \ref{lem y-morphism}.
\begin{lem}
	\label{lem y-morphism}
	Let $E$ be a categorically carried-by twisted complex, i.e.,
	\[E= \left(\bigoplus_i X_i = V_i \otimes S_{v_i}[d_i], F_{j,i} = \psi_{j,i} \otimes f_{j,i}\right)\]
	be a minimal, well-ordered twisted complex such that for all ${j,i}$, $f_{j,i}$ is not a $y$-morphism.
	Let $\tau$ be an element of $\left\{\tau_u, \tau_w^{-1} | u \in V_+(T), w \in V_-(T)\right\}$.
	Then, there exists morphisms $G_{j,i} : \tau(X_i) \to \tau(X_j)$ such that 
	\begin{itemize}
		\item[(A)] $\tau(E) \simeq \left(\bigoplus_i \tau(X_i), G_{j,i}\right)$, and
		\item[(B)] $G_{j,i}$ does not contain a $y$-morphism for all $i < j$.
	\end{itemize}
\end{lem}
\begin{proof}
	In this proof, we consider only the case that $\tau = \tau_{u_1}$ for some $u_1 \in V_+$.
	The other case, i.e., $\tau=\tau_w^{-1}$ for some $w \in V_-$, can be proven by the same proof. 
	
	The existence of $G_{j,i}$ satisfying (A) is guaranteed by Lemma \ref{lem induced functor}.
	Let us assume that there exists a pair $i<j$ such that $G_{j,i}$ contains a $y$-morphism. 
	We note that 
	\[G_{j,i}: \tau_{u_1}(V_i \otimes S_{v_i}[d_i]) = V_i \otimes \tau_{u_1}(S_{v_i}[d_i]) \to \tau_{u_1}(V_j \otimes S_{v_j}[d_j]) = V_j \otimes \tau_{u_1}(S_{v_j}[d_j]).\]
	Thus, $G_{j,i}$ consists of a nonzero linear map from $V_i$ to $V_j$, a $y$-morphism, and possibly other nonzero morphisms. 
	Thanks to Lemmas \ref{lem x-morphism} and \ref{lem z-morphism}, $f_{j,i}$ is not neither a $x$-morphism nor a $z$-morphism.
	Thus, $f_{j,i}$ should be the zero morphism. 

	For simplicity, let us assume that $d_i =0$. 
	Since 
	\[G_{j,i} : V_i \otimes \tau_{u_1}(S_{v_i}) \to V_j \otimes \tau_{u_1}(S_{v_j}[d_j])\]
	contains a $y$-morphism, there exists a negative vertex $w_1 \in V_(T)$ such that $\tau_{u_1}(S_{v_i})$ contains $S_w$. 
	Thus, $v_i = w_1 \in V_-(T)$. 
	Moreover, there exists a positive vertex $u_2 \in V_+(T)$ such that $\tau_{u_1}(S_{v_j}[d_j])$ contains $S_{u_2}[N-2]$ and the $G_{j,i}$ contains a $y_{w_1,u_2}$-morphism.
	We note that the existence of $y_{w_1,u_2}$-morphism implies that $w_1 \sim u_2$, i.e., $w_1$ and $u_2$ are connected by an edge in $T$. 
	
	From the above arguments, one can find the following three possible cases for $v_j$:
	\begin{enumerate}
		\item[(i)] $S_{v_j}[d_j] = S_{w_2}[N-2]$ for a $w_2 \in V_-$ such that $w_2 \sim u_1$. 
		\item[(ii)] $S_{v_j}[d_j] = S_{u_1}[2N-3]$.
		\item[(iii)] $S_{v_j}[d_j] = S_{u_2}[N-2]$ for some $u_2 \in V_+$ such that $u_2 \neq u_1$. 
	\end{enumerate} 
	In the rest of the proof, we analyze each case.
	\vskip0.2in
	
	\noindent {\em Case (i).\ }$S_{v_j}[d_j] = S_{w_2}[N-2]$ for a $w_2 \in V_-$ such that $w_2 \sim u_1$: 
	By Lemma \ref{lem simple computation}, 
	\[\tau_{u_1}(S_{w_2}[N-2]) = \left(S_{u_1}[N-2] \oplus S_{w_2}[N-2], x_{u_1,w_2}\right).\]
	Thus, if $G_{j,i}: V_i \otimes \tau_{u_1}(S_{w_1}) \to V_j \otimes \tau_{u_1}(S_{w_2}[N-2])$ has a $y$-morphism, the $y$-morphism should be $y_{w_1,u_1}$.
	It means that $w_1 \sim u_1$ and 
	\[\tau_{u_1}(S_{w_1}) = \left(S_{u_1} \oplus S_{w_1}, x_{u_1,w_2}\right).\]
	
	Let $E_0$ be the subcomplex 
	\[E_0 := \left(\bigoplus_{k=i+1}^{j-1} X_k, F_{m,\ell}\right)\]
	of $E$. 
	Then, the following is a part of a twisted complex equivalent to $\tau_{u_1}(E)$:
	\begin{equation}
	\label{eqn diagram 1}
	\begin{tikzcd}
		\big( V_i \otimes S_{u_1} \arrow[r, red, "Id_{V_i} \otimes x_{u_1,w_1}"] &[3em] V_i \otimes S_{w_1} \big) \ar[r] \arrow[rr, blue, bend left, "\psi \otimes y_{w_1,u_1}"] & \tau(E_0) \ar[r] & \big( V_j \otimes S_{u_1}[N-2] \arrow[red]{r}{Id_{V_j} \otimes x_{u_1,w_2}} &[3em] V_j \otimes S_{w_2}[N-2] \big).
	\end{tikzcd}
	\end{equation}
	We note that Lemma \ref{lem simple computation} gives the red arrows in \eqref{eqn diagram 1}. 
	Moreover, since we are assuming that there is a $y$-morphism, the blue arrow in \eqref{eqn diagram 1} is a tensor product of a linear map $\psi$ and the $y$-morphism. 
	
	In \eqref{eqn diagram 1}, we focus on the first red arrow and the blue arrow.
	Then, the chain of two arrows is nonzero since their composition is $\psi_{j,i} \otimes z_{u_1}$.
	We note that the sum of all chains of arrows from $V_i \otimes S_{u_1}$ to $V_j \otimes S_{u_1}[N-2]$ is the zero morphism since the differential map of a minimal twisted complex is the zero map.
	See Lemma \ref{lem zero differential for minimal}. 
	Thus, there exist morphisms $\alpha$ and $\beta$ (see red arrows in \eqref{eqn diagram 2}) satisfying the following:
	\begin{equation}
		\label{eqn diagram 2}
		\begin{tikzcd}
			\big( V_i \otimes S_{u_1} \arrow[r, "Id_{V_i} \otimes x_{u_1,w_1}"] \ar[rr, red, bend right, "\alpha= \alpha_1\otimes \alpha_2"] &[3em] V_i \otimes S_{w_1} \big) \ar[r] \arrow[rr, bend left, "\psi \otimes y_{w_1,u_1}"] & \tau(E_0) \ar[r, red, "\beta= \beta_1\otimes \beta_2"] &[3em] \big( V_j \otimes S_{u_1}[N-2] \arrow[r, "Id_{V_j} \otimes x_{u_1,w_2}"] &[3em] V_j \otimes S_{w_2}[N-2] \big).
		\end{tikzcd}
	\end{equation}	
	Let $\alpha_1$ and $\beta_1$ be the linear map parts of $\alpha$ and $\beta$, and let $\alpha_2$ and $\beta_2$ be one of $x, y, z$ or the identity morphisms. 
	Then, 
	\[(\alpha_2, \beta_2) = (z_{u_1}, e_{u_1}), (e_{u_1}, z_{u_1}), \text{  or  } (x_{u_1, w_3}, y_{w_3, u_1}) \text{  for a  } w_3 \sim u_1.\]
	It is because the composition of $\alpha_2$ and $\beta_2$ should be $z_{u_1}$. 
	
	We recall that $E$ has no identity morphisms and no $y$-morphisms. 
	Then, thanks to Lemmas \ref{lem x-morphism}--\ref{lem the zero morphism}, $\alpha_2$ and $\beta_2$ cannot be an identity morphism. 
	Thus, $(\alpha_2, \beta_2) =  (x_{u_1, w_3}, y_{w_3, u_1})$ for a $w_3 \sim u_1$.
	It implies that there exists $\ell \in (i,j)$ such that $X_\ell = V_\ell \otimes S_{w_3}$ and $\alpha$ is a part of $G_{\ell,i}$. 
	
	We recall that $E$ is well-ordered, $X_i = V_i \otimes S_{w_1}$, $X_\ell = V_\ell \otimes S_{w_3}$, and $w_1, w_3 \in V_-$. 
	Thus, for any $k \in [i, \ell]$, $X_k = V_k \otimes S_{v_k}[d_k]$ should satisfy that $v_k \in V_-, d_k =0$. 
	It implies that every arrow between $X_{\overline{K}}$ and $X_{k_2}$ with $\overline{K}, k_2 \in [i, \ell]$ is the zero morphism. 
	In particular, there exists no nonzero chain of arrows from $X_i$ to $X_\ell$. 
	
	It follows that $G_{\ell,i}: \tau_{u_1}(X_i) \to \tau_{u_1}(X_\ell)$ is the zero morphism. 
	Thus, $\alpha$ should be the zero morphism.
	It contradicts to that $\alpha_2 = x_{u_1,w_3}$. 
	Thus, Case (i) cannot happens. 
	\vskip0.2in
	
	\noindent {\em Case (ii).\ }$S_{v_j}[d_j] = S_{u_1}[2N-3]$:
	The arguments for the case (i) still work for the case (ii), and the case (ii) also cannot happen. 
	Since the argument is almost the same to the case (i), we skip the case (ii).
	\vskip0.2in
	
	\noindent {\em Case (iii).\ } $S_{v_j}[d_j] = S_{u_2}[N-2]$ for some $u_2 \in V_+$ such that $u_2 \neq u_1$: 
	For this case, we consider a subcomplex $E_0$ of $E$ given as 
	\[E_0 :=\left(\bigoplus_{k=i}^j X_k, F_{m,\ell} \right).\]
	It is trivial that $E_0$ is also minimal and well-ordered. 
	
	Let $\ell$ be the maximum integer in $[i,j]$ such that $X_\ell = V_\ell \otimes S_w$ for some $w \in V_-$. 
	We point out that since $X_i = V_i \otimes S_{w_1}$, $\ell \geq i$. 
	Then, because $E_0$ is well-ordered, it is easy to check that 
	\[E_0 \simeq \Big(
	\begin{tikzcd}[column sep=small]
		X_i \oplus \cdots \oplus X_\ell \ar[r] & X_{\ell+1} \ar[r] & \cdots \ar[r] & X_j
	\end{tikzcd} \Big).\]
	We omit all long arrows for convenience. 
	
	We define another subcomplex $E_1$ as follows: 
	Let 
	\[E_1 := \left(\bigoplus_{k=\ell +1}^{j-1} X_k, F_{m,\ell} \right).\]
	Then, up to equivalence, we can write $E_0$ as the following twisted complex with three terms:
	\[E_0 \simeq \Big(
	\begin{tikzcd}
		X_i \oplus \cdots \oplus X_\ell \ar[r, "\alpha"] \ar[rr, bend left, "\gamma"] & E_1 \ar[r, "\beta"] & X_j
	\end{tikzcd} \Big).\]
	In the rest of the proof, we will show that $\alpha, \beta$ and $\gamma$ should be the zero maps.
	
	First, $\alpha$ should be the zero map because 
	\begin{itemize}
		\item if $k \in [i, \ell]$, $X_k = V_k \otimes S_w$ for some $w \in V_-$, and
		\item if $k \in [\ell+1, j-1]$, then $X_k = V_k \otimes S_{v_k}[d_k]$ with $d_k \in [1,N-2]$. 
	\end{itemize}
	The second item holds since $E_0$ is well-ordered. 
	To be more precise, we note that $\alpha$ is determined by the arrows from $X_{\overline{K}}$ to $X_{k_2}$ with $\overline{K} \in [i, \ell], k_2 \in [\ell+1, j-1]$. 
	The arrow from $\overline{K}$ to $k_2$ is an element of 
	\[\mathrm{Mor}\left(V_{\overline{K}}, V_{k_2}\right) \otimes \shom_{\Fuk}^1\left(S_w, S_{v_{k_2}}[d_{k_2}]\right).\]
	Since 
	\[\shom_{\Fuk}^1\left(S_w, S_{v_{k_2}}[d_{k_2}]\right) = \shom_{\Fuk}^{1+d_{k_2}}\left(S_w, S_{v_{k_2}}\right) \text{  and  } 1+d_{k_2} \in [2, N-1],\]
	the nonzero morphism should be a power of a $y$-morphism. 
	Since $E_0$ does not have any $y$-morphism, there exists no nonzero arrow, and it concludes that $\alpha$ is the zero map.
	
	Similarly, one can observe that $\beta$ is the zero map. 
	To show that, we would like to point out that, for all $\ell < k < j$, $F_{j,k} =0$ because $X_k$ is shifted at least by $1$ and $X_j = V_j \otimes S_{u_2}[N-2]$ with $u_2 \in V_+$. 
	Note that if a morphism whose target is $S_{u_2}$ is not the zero, identity, or $y$-morphism, then it should be a $z$-morphism and graded by $N$. 
	Then, similar to $\alpha$, every map from $X_k$ with $k \in [\ell+1, j-1]$ to $X_j$ is the zero map, and it induces that $\beta =0$. 

	Finally, we can consider the arrow from $X_i \oplus \cdots \oplus X_\ell$ to $X_j$. 
	The degree arguments we used for $\alpha$ and $\beta$ also show that the arrow should be the zero morphism.
	Thus, 
	\begin{equation}
		\label{eqn 5}
		E_0 \simeq \Big(
		\begin{tikzcd}
			X_i \oplus \cdots \oplus X_\ell \ar[r, "0"] \ar[rr, bend left, "0"] & E_1 \ar[r, "0"] & X_j
		\end{tikzcd} \Big) = X_i \oplus \cdots \oplus X_\ell \oplus E_1 \oplus X_j.
	\end{equation}

	Now, when we apply Lemma \ref{lem induced functor} to the both sides of Equation \eqref{eqn 5}, we have
	\[\tau_{u_1}(E_0) \simeq \tau_{u_1}(X_i)\oplus \cdots \oplus \tau_{u_1}(X_\ell) \oplus \tau_{u_1}(E_1) \oplus \tau_{u_1}(X_j).\]
	On the right-handed side, $G_{j,i}$ should be the zero map.
	Thus, there exists $G_{j,i}$ satisfying the condition (B) of Lemma \ref{lem y-morphism}.
\end{proof}

Now, we recall that the twisted complex in Lemma \ref{lem y-morphism} (A), i.e., $tau E = \left(\bigoplus_{i \in I} \tau(X_i), (G_{j,i}) \right)$, is not a good twisted complex. 
In other words, the given twisted complex is not built from $\{S_v|v \in V(T)\}$, but it is built from $\{\tau(S_v)|v \in V(T)\}$.
However, it is easily converted to a good twisted complex since 
\begin{equation}
	\label{eqn PhiE}
	\begin{split}
		\tau(X_i) &= \tau(V_i \otimes S_{v_i}[d_i]) = V_i \otimes \tau(S_{v_i}[d_i]), \\
		\tau(S_{v_i}[d_i]) &= 
		\begin{cases}
			S_{v_i}[d_i], \\ S_u[d_i] \xrightarrow{x_{u,w}} S_w[d_i] \text{  for some  } u, w \in V(T). 
		\end{cases}
	\end{split}
\end{equation}
More precisely, one can get a good twisted complex by replacing $\tau(X_i)$ with 
\[V_i \otimes S_{v_i}[d_i] \text{  or  } \left(V_i \otimes S_u[d_i] \bigoplus V_i \otimes S_w[d_i], Id_{V_i} \otimes x_{u,w}\right).\] 
Moreover, by Lemmas \ref{lem x-morphism}-\ref{lem the zero morphism}, \ref{lem y-morphism} and Equation \eqref{eqn PhiE}, the good twisted complex equivalent to $\tau(E)$ has no identity morphisms and $y$-morphism as an arrow. 
Thus, it proves Lemma \ref{lem PhiE}.

\begin{lem}
	\label{lem PhiE}
	\mbox{}
	\begin{enumerate}
		\item If \[E =\left(\bigoplus_{i \in I} V_i \otimes S_{v_i}[d_i], (F_{j,i}=\psi_{j,i} \otimes f_{j,i}) \right)\]
		is a minimal twisted complex such that $f_{j,i}$ is not a $y$-morphism for all $i, j \in I$, then 
		\[\Phi E = \left(\bigoplus_{i \in I} V_i \otimes \Phi(S_{v_i}[d_i]), (G_{j,i}) \right)\]
		also has no identity morphisms and $y$-morphisms as arrows. 
		\item Let a nonzero object $E \in \Fuk$ be categorically carried-by, and let $\Phi$ be an autoequivalence of Penner type such that $\Phi$ is a product of 
		\[\left\{\tau_u, \tau_w^{-1} | u \in V_+(T), w \in V_-(T)\right\}.\]
		Then, $\Phi E$ is also categorically carried-by.
	\end{enumerate} 
\end{lem}
When one applies Lemma \ref{lem PhiE} for a categorically carried-by $E$, one can prove the step 1 of the sketch of proof of Theorem \ref{thm pseudo-Anosov}, described in Section \ref{subsection sketch of the proof}.

\subsection{Step 2 of the sketch of proof of Theorem \ref{thm pseudo-Anosov}}
\label{subsection step 2 of the sketch of proof}
Let $E$ be a nonzero object of $\Fuk$ such that $E$ is categorically carried-by.
Lemma \ref{lem PhiE} guarantees that $\Phi^n E$ also categorically carried-by if $\Phi$ is of Penner type. 

Let us recall that we fixed a stability condition $\sigma_0 \in \stab^\dagger(\Fuk)$ in Section \ref{subsection stability conditions on Fuk}.
In Section \ref{subsection step 2 of the sketch of proof}, we introduce a tool that helps us to keep track of $m_{\sigma_0}(\Phi^n E)$ as $n$ increases where $E$ is categorically carried-by.

First, we convert an arbitrary object $X \in \Fuk$ as a vector.
\begin{definition}
	\label{def vec}
	Let $X \in \Fuk$. 
	Then, {\bf $\boldsymbol{\vect(X)}$} means a real-valued vector with $|V(T)|$-many components, i.e., $\vect(X) \in \mathbb{R}^{|V(T)|}$, given as 
	\[\vect(X) = \left(\dim \lhom^*_{\wrapped}(L_v,X)\right)_{v \in V(T)}.\]
\end{definition}

We recall that $\sigma_0 = \left(Z_0, P_0\right)$ satisfies that
\begin{itemize}
	\item every semistable object of phase $\tfrac{1}{2}$ (resp.\ $1$) is a direct sum of $\left\{S_u | u \in V_+(T)\right\}$ (resp.\ $\left\{S_w | w \in V_-(T)\right\}$), and
	\item for every $v \in V(T)$, $|Z_0(S_v)| =1$. 
\end{itemize}

\begin{remark}
	\label{rmk another vector}
	Before going further, let us assign another vector $\vect_0(E)$ to a twisted complex $E$ as follows: 
	\[\vect_0(E) := \left(\dim \shom_{\wrapped}^*\left(L_v, E\right)\right)_{v \in V(T)}.\]
	We compare two vectors $\vect(E)$ and $\vect_0(E)$ in the remark, and the comparison will elucidate why we care the notion of categorically carried-by. 
	The biggest difference is that $\vect(E)$ is defined from $\lhom_{\wrapped}^*$, but $\vect_0(E)$ is defined from $\shom_{\wrapped}^*$.
	Since the first is an invariant of quasi-equivalence classes, $\vect(E)$ is an invariant of equivalence classes, but $\vect_0(E)$ is not. 
	On the other hand, $\vect_0(E)$ has the following advantage compared to $\vect(E)$: 
	For a twisted complex $E$, let $\Phi E$ denote the twisted complex obtained by applying Lemma \ref{lem induced functor} to $E$. 
	Then, one can easily find a matrix $M_\Phi$ such that 
	\begin{gather}
		\label{eqn advantage}
		\vect_0(\Phi E) = M_\Phi \vect_0(E).
	\end{gather}
	Thus, one can keep track of $\vect_0(\Phi^n E) = M_\Phi^n \vect_0(E)$ through linear algebra. 
	In the later part of this subsection, we will show that if $E$ is categorically carried-by, then $\vect(E) = \vect_0(E)$.
\end{remark}

If
\[E =\left(\bigoplus_{i \in I} V_i \otimes S_{v_i}[d_i], (F_{j,i}=\psi_{j,i} \otimes f_{j,i}) \right)\]
is a minimal twisted complex, Equation \eqref{eqn m_sigma(X)} implies that 
\begin{align}
	\label{eqn vector for minimal}
	\begin{split}
		m_{\sigma_0}(E) &= \sum_{i \in I} \dim V_i \\
		&= \sum_{v \in V(T)} \dim \shom^*_{\wrapped}(E,L_v) \\
		&= \sum_{v \in V(T)} \dim \lhom^*_{\wrapped}(E,L_v) (\because \text{  Lemma \ref{lem zero differential for minimal}}) \\
		&= \parallel \vect(E) \parallel_1,
	\end{split}
\end{align}
where $\parallel \cdot \parallel_1$ means the $L_1$-norm of a vector, i.e., the sum of the absolute values of all components. 

We note that by Lemma \ref{lem minimal twisted complex}, for any $E \in \Fuk$, there exists a minimal twisted complex equivalent to $E$. 
This, one can conclude that $m_{\sigma_0}(E) = \parallel \vect(E) \parallel_1$ for any $E \in \Fuk$. 
Thus, in order to keep track of $m_{\sigma_0}(\Phi^n E)$, it is enough to keep track of the changes on vectors $\vect(\Phi^n E)$. 

Now, we prove that $\vect\left(\tau(E)\right)$ can be obtained by multiplying a matrix to $\vect(E)$, if $E$ is categorically carried-by. 
To prove that, we define the following:
\begin{definition}
	\label{def matrix}
	For all $v \in V(T)$, let {\bf $\boldsymbol{M_v}$} be the $|V(T)|$-by-$|V(T)|$ matrix given by 
	\[M_v = Id + \overline{M}_v,\]
	where the $(v_1, v_2)$ component $m_{v_1,v_2}$ of $\overline{M}_v := \left(m_{v_1,v_2}\right)_{v_i \in V(T)}$ is given by 
	\[m_{v_1,v_2} = \begin{cases}
		1 \text{  if  } v_1 = v, v_2 \sim v, \\
		0 \text{  otherwise}.
	\end{cases}\] 
\end{definition} 

\begin{remark}
	\label{rmk determinant}
	From Definition \ref{def matrix}, one can easily check that the determinant of $M_v$ is $1$ for all $v \in V(T)$. 
\end{remark}

Let $E$ be a categorically carried-by object of $\Fuk$. 
Then, by definition, $E$ is equivalent to a minimal, well-ordered twisted complex such that there exists no $y$-morphism as an arrow.
For convenience, let $E = \left(\bigoplus_i X_i = V_i \otimes S_{v_i}[d_i], F_{j,i} = \psi_{j,i} \otimes f_{j,i}\right)$ denote such a twisted complex. 

Now, for any $\tau \in \left\{\tau_u, \tau_w^{-1} | u \in V_+(T), w \in V_-(T)\right\}$, by Lemma \ref{lem y-morphism}, $\tau(E)$ is equivalent to a twisted complex 
\[\tau(E) \simeq \left(\bigoplus_i \tau(X_i), G_{j,i}\right),\]
such that $G_{j,i}$ does not contain any identity or $y$-morphisms.
Thus, one can obtain a minimal, well-ordered twisted equivalent to $\tau(E)$ and having no $y$-morphism as an arrow, by rearranging the above twisted complex. 
For convenience, let $\tau(E)$ denote the twisted complex we obtain after the rearrangement. 

As one can see in Equation \eqref{eqn vector for minimal}, one can easily see that 
\begin{gather}
	\label{eqn vector for minimal 2}
	\vect\left(E\right) = \left(\dim\hom^*_{\wrapped}\left(E, L_v\right)\right)_{v \in V(T)}, \vect\left(\tau(E)\right) = \left(\dim\hom^*_{\wrapped}\left(\tau(E), L_v\right)\right)_{v \in V(T)},
\end{gather}
from the minimality of $E$ and $\tau(E)$. 
Moreover, from Lemma \ref{lem y-morphism}, we see that every component of the twisted complex $\tau(E)$ comes from $\tau(X_i)$ where $E =\left(\oplus_{i \in I} X_i, F_{j,i}\right)$. 
Thus, we have
\begin{gather}
	\label{eqn linear algebra for small hom}
	\left(\dim \hom^*_{\wrapped}\left(\tau(E), L_v\right)\right)_{v \in V(T)}  = \begin{cases} M_u \cdot \left(\dim\hom^*_{\wrapped}\left(E, L_v\right)\right)_{v \in V(T)} \text{  if  } \tau = \tau_u \text{  for some  } u \in V_+(T), \\
		M_w \cdot \left(\dim\hom^*_{\wrapped}\left(E, L_v\right)\right)_{v \in V(T)} \text{  if  } \tau = \tau^{-1}_w \text{  for some  } w \in V_-(T).
	\end{cases}
\end{gather} 

When one combines Equations \eqref{eqn vector for minimal 2} and \eqref{eqn linear algebra for small hom}, one can conclude that, if $E$ is categorically carried-by, then 
\[\vect(\tau(E)) =
	\begin{cases} M_u \cdot \vect(E) \text{  if  } \tau = \tau_u \text{  for some  } u \in V_+(T), \\
		M_w \cdot \vect(E) \text{  if  } \tau = \tau^{-1}_w \text{  for some  } w \in V_-(T).
\end{cases}\]

Similarly, the following theorem holds:
\begin{lem}
	\label{lem linear algebra}
	Let $\Phi: \Fuk \to \Fuk$ be of Penner type.
	Then, there exists a matrix $M_{\Phi}$ depending only on $\Phi$ such that, if an object $E \in \Fuk$ is a categorically carried-by, then
	\begin{gather}
		\label{eqn linear algebra}
		\vect(\Phi E) = M_\Phi \cdot \vect(E).
	\end{gather}
	In particular, 
	\[m_{\sigma_0}(\Phi E) = \parallel \vect(\Phi E)\parallel_1 = \parallel M_\Phi \cdot \vect(E) \parallel_1 = m_{\sigma_0}(E).\]
\end{lem}
\begin{proof}
	We note that $\Phi$ is a product of 
	\[\left\{\tau_u, \tau_w^{-1} | u \in V_+(T), w \in V_-(T)\right\}.\]
	Thus, when we apply the above argument in an iterative manner, we have a matrix $M_\Phi$ satisfying Equation \eqref{eqn linear algebra}, as a product of $\{M_v | v \in V(T)\}$.
\end{proof}

\begin{remark}
	\label{rmk computation of the matrix}
	\mbox{}
	\begin{enumerate}
		\item From the construction of $\Phi$, one can easily see that $(v_1,v_2)$-component of $M_\Phi$ is computed as 
		\[\dim \lhom_{\wrapped}^*\left(L_{v_2},\Phi S_{v_1}\right).\]
		Especially, every component of $M_\Phi$ is a nonnegative integer. 
		\item We would like to point out that $\vect(E)$ and $M_\Phi$ are dependent on the choice of sign conventions since their definition are depending on the fixed object $S_v$ or $L_v$ for $v \in V(T)$. 
		We also point out that the Penner type autoequivalence $\Phi$ is a product of $\{\tau_u, \tau_w^{-1} | u \in V_+(T), w \in V_-(T)\}$. 
		Thus, if we would like to study $\Phi^{-1}$, we need to choose the other sign convention and it changes the definitions of $\vect$ and $M_{\Phi^{-1}}$. 
	\end{enumerate}
\end{remark}

\subsection{Step 3 of the sketch of proof of Theorem \ref{thm pseudo-Anosov}}
\label{subsection step 3 of the sketch of proof}
In the subsection, we prove the step 3 of the sketch of proof of Theorem \ref{thm pseudo-Anosov}, given in Section \ref{subsection sketch of the proof}.
In other words, we prove Theorem \ref{thm without y morphism}.
The proof of Theorem \ref{thm without y morphism} consists of two parts.
In the first part, we show that there exists a real number $\lambda_\Phi$ depending only on $\Phi$ such that if $E \in \Fuk$ is categorically carried-by, then
\[\lim_{n \to \infty} \tfrac{1}{n}\log m_{\sigma_0}(\Phi^n E) = \lambda_\Phi.\]
The second part is to prove that $\lambda_\Phi$ is strictly bigger than $1$. 

\begin{proof}[Proof of Theorem \ref{thm without y morphism}, the first part]
	First, we observe that 
	\begin{gather}
		\label{eqn base} 
		\lim_{n \to \infty} \tfrac{1}{n}\log m_{\sigma_0}(\Phi^n E) = \lim_{n \to \infty} \tfrac{n}{n+K} \tfrac{1}{n}\log m_{\sigma_0}\left(\Phi^n(\Phi^K E)\right) = \lim_{n \to \infty} \tfrac{1}{n}\log m_{\sigma_0}\left(\Phi^n(\Phi^K E)\right),
	\end{gather}
	for any $K \in \mathbb{N}$.
	
	We also note that if $K \geq |V(T)|$, then 
	 \begin{enumerate}
		\item[(i)] Definition \ref{def Penner type}, (II) and Lemma \ref{lem simple computation} imply that, for any $v \in V(T)$, 
		\[\dim \lhom^*_{\wrapped}\left(L_v, \Phi^K E\right) \geq 1.\]
		In other words, every component of $\vect(\Phi^K E)$ is a positive integer. 
	\end{enumerate}	
	Now, we would like to replace $E$ with $\Phi^K E$ based on Equation \eqref{eqn base} if we needed, thus we can assume that every component of $\vect(E)$ is a positive integer. 
	
	Without loss of generality, let $E = \left(\bigoplus_{i \in I} X_i = V_i \otimes S_{v_i}[d_i], F_{j,i} = \psi_{j,i} \otimes f_{j,i}\right)$ be a minimal twisted complex such that $E$ does not have any $y$-morphism as an arrow. 
	Then, it is easy to check the following facts:
	\begin{enumerate}
		\item[(ii)] For any $i \in I$, $X_i = V_i \otimes S_{v_i}[d_i]$ can be seen as a twisted complex with only one component.
		Thus, $X_i$ is an object of $\Fuk$, which is categorically carried-by.
		By applying Lemmas \ref{lem PhiE} and \ref{lem linear algebra} iteratively, one can observe that for any $n \in \mathbb{N}$, $\Phi^n X_i$ is categorically carried-by and \[\vect(\Phi^n X_i) = M_\Phi^n \cdot \vect(X_i).\]
	 	Thus, one can apply Lemmas \ref{lem PhiE} and \ref{lem linear algebra} for all $X_i$. 
	 	\item[(iii)] Since $E$ is categorically carried-by, so does $\Phi^n E$ for any $n \in \mathbb{N}$.  
	 	Thus, by applying Lemma \ref{lem PhiE} iteratively, one has $\vect(\Phi^n E) = M_\Phi^n \cdot \vect(E)$. 
	 	Moreover, since $E$ is a minimal twisted complex, we have
	 	$\vect(E) = \sum_{i \in I} \vect(X_i)$. 
	 	Thus, for any $n \in \mathbb{N}$,
	 	\[\vect(\Phi^n E) = M_\Phi^n \cdot \vect(E) = \sum_{i \in I} M_\Phi^n \cdot \vect(X_i) = \sum_{i \in I}\vect(\Phi^m X_i).\]
	  	Moreover, $\parallel\vect(\Phi^n E) \parallel_1= \parallel\sum_{i \in I} \vect(\Phi^n X_i)\parallel_1$ because every component of $\vect(\Phi^n X_i) = M_\Phi^n \cdot \vect(X_i)$ is non-negative.
	 	See (i) and Remark \ref{rmk computation of the matrix}.
	 	\item[(iv)] From (iii), for any $i_0 \in I$, 
	 	\[\parallel\vect(\Phi^nX_{i_0}) \parallel_1 \leq \parallel\vect(\Phi^nE) \parallel_1= \parallel\sum_{i \in I} \vect(\Phi^nX_i)\parallel_1.\]
	 	By taking $\lim_{n \to \infty} \tfrac{1}{n} \log$, one has 
	 	\[\lim_{n \to \infty}  \tfrac{1}{n}\log\parallel\vect(\Phi^nX_{i_0}) \parallel_1 \leq \lim_{n \to \infty} \tfrac{1}{n} \log \parallel\vect(\Phi^nE) \parallel_1= \lim_{n \to \infty} \tfrac{1}{n} \log \parallel\sum_{i \in I} \vect(\Phi^nX_i)\parallel_1.\]
	 	\item[(v)] Finally, by (i), (iv),
	 	\begin{gather}
	 		\label{eqn equality}
	 		\lim_{n \to \infty} \tfrac{1}{n}\log \parallel\vect(\Phi^nE)\parallel_1 = \max \left\{\lim_{n \to \infty} \tfrac{1}{n}\log \parallel M_\Phi^n \cdot \vec{e}_v \parallel_1 \Bigg\vert v \in V(T)\right\}=:\lambda_\Phi,
	 	\end{gather}
		where $\vec{e}_v$ is a vector in $\mathbb{R}^{|V(T)|}$ such that for any $w \in V(T)$, $w$-component of $\vec{e}_v$ is $1$ if $w =v$ and $0$ if $w \neq v$.
	 \end{enumerate}
 
 	The constant $\lambda_\Phi$ in Equation \eqref{eqn equality} does not depend on the choice of $E_0$ and only depends on $\Phi$.
 	Thus, Equation \eqref{eqn equality} completes the first part of the proof of Theorem \ref{thm without y morphism}.
\end{proof}

In order to prove the second part of the proof, we state the {\em Perron--Frobenius} Theorem in matrix theory.
\begin{thm}[Perron-Frobenius Theorem \cite{Perron1907, Frobenius1912}]
	\label{thm Perron Frobenius}
	Let $M$ be a real square matrix with positive entries. 
	Then, $M$ has a simple largest real eigenvalue, i.e., the {\em spectral radius}, or also called the {\em Perron root} or the {\em Perron-Frobenius eigenvalue}.
	Moreover, the eigenvector corresponding to the unique largest eigenvalue can be chosen to have positive real components.
\end{thm}

\begin{proof}[Proof of Theorem \ref{thm without y morphism}, the second part]
	We would like to show that $\lambda_\Phi >1$ where $\lambda_\Phi$ is the number defined in Equation \eqref{eqn equality}.
	
	We note that for any $K \in \mathbb{N}$, $M_\Phi$ and $M_\Phi^K$ have the same eigenvectors. 
	And if $\lambda$ is an eigenvalue of $M_\Phi$, then $\lambda^K$ is an eigenvalue of $M_\Phi^K$. 
	Based on this, we will consider the eigenvectors and eigenvalues of $M_\phi^K$ instead of those of $M_\Phi$, for sufficiently large $K$.
	
	We also note that as mentioned in Remark \ref{rmk computation of the matrix}, the $(v_1,v_2)$-component of $M_\Phi^K$ is 
	\[\dim \lhom^*_{\wrapped}\left(L_{v_2}, \Phi^K S_{v_1}\right).\]
	Because of Definition \ref{def Penner type} (II), we know that every entries of $M_\Phi^K$ is a positive integer if $K \geq |V(T)|$.  
	Thus, one can apply Theorem \ref{thm Perron Frobenius} to $M_\Phi^K$.
	Let the unique largest real eigenvalue of $M_\Phi^N$ be called $\mu$, and let $\vec{v}_0$ denote an eigenvector of $M_\Phi^K$ corresponding to $\mu$ such that every entry of $\vec{v}_0$ is a positive real number. 
	
	Now, we observe the following facts:
	\begin{enumerate}
		\item[(i)] Let $\lambda_1, \dots, \lambda_{|V(T)|} \in \mathbb{C}$ be the eigenvalues of $M_\Phi$ such that 
		\[|\lambda_1| \geq |\lambda_2| \geq \dots \geq |\lambda_{|V(T)|}|.\]
		Then, $\lambda_i^K$ is an eigenvalue of $M_\Phi^K$ and $\lambda_1^K=\mu$.
		\item[(ii)] It is easy to check that the determinant of $M_\Phi$ is $1$ since that of $M_v$ is $1$ for all $v \in V(T)$. 
		See Remark \ref{rmk determinant}
		It implies that every column/row of $M_\Phi$ is not the zero vector.
		\item[(iii)] We note that $\vec{v}_0 \in \left(\mathbb{R}_{>0}\right)^{|V(T)|}$. 
		Because $M_\Phi$ is a real matrix with nonnegative entry and because of (ii), 
		\[\lambda_1^i \vec{v}_0 = M_\Phi^i \cdot \vec{v}_0 \in \left(\mathbb{R}_{>0}\right)^{|V(T)|} \text{  for all  } i \in \mathbb{N}.\]
		\item[(iv)] As mentioned (i), $\lambda_1$ is a $K^{\text{th}}$ root of $\mu$. 
		If $\lambda_1$ is not a real number, then (iii) cannot hold. 
		Thus, $\lambda_1$ is a positive real number.
		\item[(v)] From (ii), $\prod_{i=1}^{|V(T)|} \lambda_i =1$, 
		Thus, $\lambda_1 \leq 1$, then $|\lambda_i|=1$ for all $i$. 
		It is contradict to the uniqueness of $\mu$. 
		Thus, $\lambda_1 >1$.
	\end{enumerate}

	We note that the largest eigenvalue of $M_\Phi$ is a positive real number $\lambda_1$. 
	Since $\left\{\vec{e}_v | v \in V(T) \right\}$ is the standard basis for $\mathbb{R}^{V(T)}$, $\lambda_\Phi$ defined in Equation \eqref{eqn equality} equals to $\lambda_1$, i.e., $\lambda_\Phi = \lambda_1 >1$.
\end{proof}

\subsection{Lemma \ref{lem maximality}}
\label{subsection lemma maximality}
In this subsection, we prove Lemma \ref{lem maximality} which is necessarily in Section \ref{section proof of thm pseudo-Anosov}.
\begin{lem}
	\label{lem maximality}
	Let $T$ be a given tree, and let $\Phi$ be an autoequivalence of Penner type. 
	Then, for any nonzero object $E \in \Fuk$, 
	\[\lim_{n \to \infty} \tfrac{1}{n} \log m_{\sigma_0}(\Phi^n E) \leq \lambda_\Phi.\]
\end{lem}
\begin{proof}
	We recall that as mentioned in Equation \eqref{eqn advantage}, for any twisted complex $E \in \Fuk$, 
	\[ \left( \dim \shom^*_{\wrapped}(L_v, \Phi E) \right)_{v \in V(T)} = M_\Phi\cdot \left( \dim \shom^*_{\wrapped}(L_v, E) \right)_{v \in V(T)}, \]
	where $\Phi E$ means the twisted complex we obtained by applying Lemma \ref{lem induced functor} to $E$. 
	Moreover, if we assume that $E$ is minimal twisted complex, then as we saw in Equation \eqref{eqn vector for minimal}, 
	\[\left( \dim \shom^*_{\wrapped}(L_v, E) \right)_{v \in V(T)} = \vect(E).\]
	
	Now, we recall that every nonzero object $X \in \Fuk$ is equivalent to a minimal twisted complex $E$. 
	Thus, we have 
	\begin{align}
		\label{eqn spectral radius}
		\begin{split}
			\parallel \vect(\Phi X) \parallel_1 &= \parallel\left( \dim \lhom^*_{\wrapped}(\Phi X, L_v) \right)_{v \in V(T)}\parallel_1 \\
			&\leq \parallel\left( \dim \shom^*_{\wrapped}(\Phi E, L_v) \right)_{v \in V(T)}\parallel_1 \\
			&=\parallel M_\Phi\cdot \vect(E)\parallel_1 = \parallel M_\Phi\cdot \vect(X)\parallel_1.
		\end{split}
	\end{align}
	Thus, for any nonzero $X \in \Fuk$,
	\[\lim_{n \to \infty} \tfrac{1}{n} \log m_{\sigma_0}(\Phi^n X) = \lim_{n \to \infty} \tfrac{1}{n} \log \parallel \vect(\Phi^n X) \parallel_1 \leq \lim_{n \to \infty} \tfrac{1}{n} \log\parallel M^n_\Phi\cdot \vect(E)\parallel_1. \]
	
	Since the spectral radius of $M_\Phi$, i.e., $\lambda_\Phi$, is larger or equal to $\lim_{n \to \infty} \tfrac{1}{n} \log\parallel M^n_\Phi\cdot \vect(E)\parallel_1$, Lemma \ref{lem maximality} holds. 	
\end{proof}

\section{Proof of Theorem \ref{thm pseudo-Anosov}}
\label{section proof of thm pseudo-Anosov}
Let $\Phi$ be an autoequivalence of Penner type on $\Fuk$. 
Moreover, we will assume that $\Phi$ is a product of $\left\{\tau_u, \tau_w^{-1} | u \in V_+(T), w \in V_-(T)\right\}$ as mentioned in Remark \ref{rmk sign is not important 1}. 
The goal of this section is to complete the proof of Theorem \ref{thm pseudo-Anosov}. 

As a partial proof of Theorem \ref{thm pseudo-Anosov}, we have already proven Theorem \ref{thm without y morphism} and Lemma \ref{lem maximality}, i.e.,  there exists a nonzero real number $\lambda_\Phi>1$ such that, 
\begin{itemize}
	\item for any nonzero object $E \in \Fuk$, the exponential mass growth of $\Phi^n E$ as $n \to \infty$ is smaller than or equal to $\log \lambda_\Phi$, and 
	\item if $E$ is categorically carried-by, then the exponential mass growth is exactly $\log \lambda_\Phi$.
\end{itemize}
We note that Theorem \ref{thm without y morphism} and Lemma \ref{lem maximality} are not enough to prove Theorem \ref{thm pseudo-Anosov}.
To complete the proof, we should prove that even if $E$ is not categorically carried-by, the mass growth of $\Phi^n E$ is equal to $\lambda_\Phi$. 
In Section \ref{section proof of thm pseudo-Anosov}, we study the nonzero objects $E$ that are not categorically carried-by. 

Before starting the main part of this subsection, let us roughly describe the story-line of Section \ref{section proof of thm pseudo-Anosov} that corresponds to the steps $4$--$6$ of the sketch of proof of Theorem \ref{thm pseudo-Anosov}, given in Section \ref{subsection sketch of the proof}. 
Let us recall that in Section \ref{subsection the main idea}, we described our strategy for objects $E$ that are not categorically carried-by.
Our strategy is to show that, for sufficiently large $n \in \mathbb{N}$, $\Phi^n E$ is divided into two parts, one containing all $y$-morphisms and the other which has no {\em effect} of $y$-morphisms. 
See the right above of Remark \ref{rmk weak pseudo-Anosov}.
We give a precise condition that $\Phi^n E$ should satisfy.

\begin{definition}
	\label{definition partially carried-by}
	\mbox{}
	\begin{enumerate}
		\item For a good twisted complex $E= \left(\bigoplus_{i \in I} (X_i=V_i \otimes S_{v_i}[d_i]), F_{j,i} \right) \in \Fuk$, a {\bf $\boldsymbol{x}$-arrow} will refer to a morphism $F_{j,i} = \psi_{j,i}\otimes f_{j,i}$ from $X_i$ to $X_j$ for some $i,j \in I$ such that $f_{j,i}$ is a $x$-morphism. 
		Similarly, one can define a {\bf $\boldsymbol{y}$-arrow}, a {\bf $\boldsymbol{z}$-arrow}, and an {\bf identity-arrow}.
		We note that since $E$ is a good twisted complex, $f_{j,i}$ is one of $x, y, z$-morphisms, the identity morphism, or the zero morphism. 
		See Proposition \ref{prop good twisted complex}.
		\item A nonzero object $E \in \Fuk$ is {\bf partially carried-by} if 
		$E$ is equivalent to a minimal twisted complex 
		\begin{gather}
			\label{eqn partially carried-by}
			E \simeq \left(\bigoplus_{i=1}^n \left(X_i = V_i \otimes S_{v_i}[d_i]\right), F_{j,i} =\psi_{j,i} \otimes f_{j,i}\right),
		\end{gather}
		and there exists an integer $K \in \{1, \dots, n-1\}$ satisfying the following conditions:
		\begin{enumerate}[label=(\alph*)]
			\item If $j > K$, then $F_{j,i}$ is not a $y$-arrow for any $i < j$, or equivalently, if $F_{j,i}$ is a $y$-arrow, then, $j \leq K$. 
			\item If $E = \left(\bigoplus_{i=1}^n X_i, F_{j,i}\right)$ has a chain of arrows
			\begin{equation}
				\label{eqn partially carried-by condition 1}
				X_{i_0} \xrightarrow{\psi_{i_1,i_0}\otimes y_{v_{i_0}, v_{i_1}}} X_{i_1} \xrightarrow{\psi_{i_2,i_1} \otimes z_{v_{i_1}}} X_{i_2} \xrightarrow{\psi_{i_3,i_2} \otimes z_{v_{i_2}}} \dots \xrightarrow{\psi_{i_k,i_{k-1}} \otimes z_{v_{i_{k-1}} }} X_{i_{k}},
			\end{equation}
			such that $\psi_{i_k,i_{k-1}} \circ \dots \circ \psi_{i_1,i_0} \neq 0$, i.e., a nonzero chain of arrows from $X_{i_0}$ to $X_{i_k}$ of the above form, then, $i_k \leq K$. 
			\item If $E = \left(\bigoplus_{i=1}^n X_i, F_{j,i}\right)$ has a chain of arrows
			\begin{equation}
				\label{eqn partially carried-by condition 2}
				X_{i_0} \xrightarrow{\psi_{i_1,i_0}\otimes y_{v_{i_0}, v_{i_1}}} X_{i_1} \xrightarrow{\psi_{i_2,i_1} \otimes z_{v_{i_1}}} X_{i_2} \xrightarrow{\psi_{i_3,i_2} \otimes z_{v_{i_2}}} \dots \xrightarrow{\psi_{i_k,i_{k-1}} \otimes z_{v_{i_{k-1}} }} X_{i_{k}} \xrightarrow{\psi_{i_{k+1},i_k} \otimes x_{v_{i_k},v_{i_{k+1}}}} X_{i_{k+1}},
			\end{equation}
			such that $\psi_{i_{k+1},i_k} \circ \dots \circ \psi_{i_1,i_0} \neq 0$, i.e., a nonzero chain of arrows from $X_{i_0}$ to $X_{i_{k+1}}$ of the above form, then, $i_{k+1} \leq K$. 
		\end{enumerate} 
	\end{enumerate}
\end{definition}
In the later part of Section \ref{section proof of thm pseudo-Anosov}, if a nonzero object $E$ is partially carried-by, then we will assume that $E$ is a minimal twisted complex given in Equation \eqref{eqn partially carried-by}, and $K$ denotes the integer satisfying all conditions of Definition \ref{definition partially carried-by}, (2).

\begin{remark}
	\label{rmk vertices}
	We also note that in \eqref{eqn partially carried-by condition 1}, $y_{v_{i_0},v_{i_1}}$ exists only if $v_{i_0} \in V_-(T)$ and $v_{i_1} \in V_+(T)$. 
	Moreover, since there exists $z_{v_{i_k}}: S_{v_{i_j}} \to S_{v_{i_{j+1}}}$, $v_{i_j} = v_{i_{j+1}}$ for all $j \geq 1$.
	Thus, 
	\[v_{i_0} \in V_-(T), v_{i_1} = \dots = v_{i_k} \in V_+(T).\] 
	Similarly, in \eqref{eqn partially carried-by condition 2}, 
	\[v_{i_0} \in V_-(T), v_{i_1} = \dots = v_{i_k} \in V_+(T), v_{i_{k+1}} \in V_-(T).\] 
\end{remark}

Let us briefly explain why Definition \ref{definition partially carried-by} (2) is called partially carried-by. 
If $E = \left(\bigoplus_{i=1}^n X_i, F_{j,i}\right)$ is partially carried-by, then one can write $E$ as a cone of a morphism $F$ between two minimal twisted complexes 
\[A[-1] = \left( \bigoplus_{i=1}^K X_i[-1], F_{j,i}[-1]\right) \ \mathrm{and}\ B=\left(\bigoplus_{i=K+1}^n X_i, F_{j,i}\right).\]
Then, one can observe that $B$, a part of $E$, is categorically carried-by, and it is a reason why $E$ is called partially carried-by. Moreover, by the conditions (b) and (c), when one applies a Penner type autoequivalence to $E$, the $y$-arrows in $A$ do not affect on $\Phi(B)$. 

From the above argument, one can have an alternative definition of the notion of partially carried-by.
\begin{definition}
	\label{definition partially carried-by 2} 
	A nonzero object $E \in \Fuk$ is {\bf partially carried-by with a triple $\boldsymbol{(A,B,F)}$} if there exist two minimal twisted complexes
	\[A = \left( \bigoplus_{i=1}^K X_i, F_{j,i}\right), B=\left(\bigoplus_{i=K+1}^n X_i, F_{j,i}\right) \neq 0,\]
	and a closed morphism from $A$ to $B$
	\[F=\left(F_{j,i}\right)_{j=K+1, \dots, n}^{i=1, \dots, K} \in \shom_{\Fuk}^1\left(A= \bigoplus_{i=1}^K X_i, B=\bigoplus_{j=K+1}^n X_j\right),\]
	such that 
	\begin{itemize}
		\item $E = \left( \bigoplus_{i=1}^n X_i, F_{j,i}\right)$ is a minimal twisted complex, and 
		\item $E$ and $K$ satisfy the conditions (a)--(c) of Definition \ref{definition partially carried-by} (2).
	\end{itemize} 
\end{definition} 
We note that the alternative definition, Definition \ref{definition partially carried-by 2}, specifies which part of $E$ is categorically carried-by if $E$ is partially carried-by.
Because of the advantage, we will use Definition \ref{definition partially carried-by 2}, not Definition \ref{definition partially carried-by}, as definition of partially carried-by, in most of Section \ref{section proof of thm pseudo-Anosov}. 

In the first and second subsections of Section \ref{section proof of thm pseudo-Anosov}, we will show that if $E$ is partially carried-by with a triple $(A,B,F)$, then $\Phi E$ is partially carried-by with a triple $(\overline{A}, \overline{B}, G)$, and moreover, $\Phi B = \overline{B}$, i.e., the step 4 of the sketch of proof of Theorem \ref{thm pseudo-Anosov}.
We note that by definition, $B$ is categorically carried-by, and because of Lemma \ref{lem PhiE}, $\overline{B}=\Phi B$ is also categorically carried-by.
In Section \ref{subsection step 5}, we will prove that, for any nonzero object $E \in \Fuk$, there exists $K \in \mathbb{N}$ such that $\Phi^K E$ is partially carried-by, i.e., the step 5 of the sketch.
From those two facts, we will complete the proof of Theorem \ref{thm pseudo-Anosov} in Section \ref{subsection step 6}.

\subsection{Step 4 of the sketch of proof of Theorem \ref{thm pseudo-Anosov}}
\label{subsection step 4}
As mentioned above, the goal of the subsection is to prove that if $E$ is partially carried-by with a triple $(A,B,F)$, then $\Phi E$ is also partially carried by with a triple $(\overline{A}, \overline{B}, G)$ such that $\overline{B} = \Phi B$. 
To prove this, it is enough to prove Lemma \ref{lem partially carried-by preserved}, since $\Phi$ is a product of 
\[\tau \in \left\{\tau_u, \tau_w^{-1} | u \in V_+(T), w \in V_-(T)\right\}.\]
\begin{lem}
	\label{lem partially carried-by preserved}
	Let $\tau \in \left\{\tau_u, \tau_w^{-1} | u \in V_+(T), w \in V_-(T)\right\}$. 
	Let us assume that $E$ is partially carried-by with a triple $(A,B,F)$.
	Then, there exist minimal twisted complexes $\overline{A}, \overline{B}$, and  $G:\overline{A} \to \overline{B}$ such that 
	\begin{itemize}
		\item $\tau(A) \simeq \overline{A}, \tau(B) \simeq \overline{B}, \tau(E) \simeq \left(\overline{A} \oplus \overline{B}, G\right)$, and
		\item $\tau(E)$ is partially carried-by with a triple $(\overline{A}, \overline{B}, G)$.
	\end{itemize}
\end{lem}
Since we would like to prove Lemma \ref{lem partially carried-by preserved} in Sections \ref{subsection step 4} and \ref{subsection proof of Lemma partially carried-by preserved}, we let $\tau$ denote an autoequivalence in $\left\{\tau_u, \tau_w^{-1} | u \in V_+(T), w \in V_-(T)\right\}$ in the subsections. 

To prove Lemma \ref{lem partially carried-by preserved}, we need some preparations, for example, Definition \ref{definition simplified twisted complex}, Lemmas \ref{lem z}, \ref{lem noxyyx}, and \ref{lem y-arrow2}.
The preparations will be given in the present subsection, and the proof of Lemma \ref{lem partially carried-by preserved} will be given in the next subsection. 

\begin{definition}\label{definition simplified twisted complex}
	Let $E = \left(\bigoplus_{i \in I} \left(X_i= V_i \otimes S_{v_i}[d_i]\right), F_{j,i}\right)$ be a minimal twisted complex. 
	\begin{enumerate}
		\item For each $v\in V(T)$ and $d\in \Z$, we define the subset $\boldsymbol{I(E;v,d)} \subset I$, where $I$ is the index set of the twisted complex $E$, as
		\[ I(E;v,d) =  \{ i \in I | v_i = v, d_i =d\}.\]
		\item We say that $E$ is {\bf simplified} if, for every pair $(v, d) \in V(T) \times \mathbb{Z}$, 
		\[|I(E;v,d)| = 0 \text{  or  } 1.\] 
	\end{enumerate}
\end{definition}

We would like to point out that it is always possible to turn a minimal twisted complex into a simplified one without canceling any arrows, because of Lemma \ref{lem grading}. 
\begin{remark}
	We note that we could define the notion of simplified twisted complex in Section \ref{subsection twisted complexes in the Fukaya category of a plumbing space}. Then we could consider only simplified twisted complexes in earlier sections.
	It could make the arguments in previous sections simpler, but we did not do so to argue with more general twisted complexes.
\end{remark}

Let us assume that $E$ is partially carried-by with a triple $(A,B,F)$.
Then, one can turn the minimal twisted complexes $A$ and $B$ into simplified ones. 
We note that the notion of partially carried-by cares chains of arrows of specific forms. 
By assuming that $A$ and $B$ are simplified, one can minimize the number of chains of arrows in the specific forms.
Thus, without loss of generality, we assume that $A$ and $B$ are simplified, and we will prove Lemma \ref{lem partially carried-by preserved} under the assumption. 

Under the assumption, we will use the following notation: 

\begin{notation}
	\label{notation 1}
	\mbox{}
	\begin{itemize}
		\item In the present and the next subsection, $E$ denotes a {\em partially carried-by} twisted complex. 
		Similarly, $E_0$ denotes a good twisted complex which is not necessarily partially carried-by.
		\item A partially carried-by twisted complex $E$ has an index set $I=\{1, \dots, n\}$, and a good twisted complex $E_0$ has an index set $I_0 = \{1, \dots, n_0\}$. 
		\item As a twisted complex, $E$ is written as 
		\[E = \left(\bigoplus_{i \in I} (X_i = V_i \otimes S_{v_i}[d_i]), F_{j,i}\right).\]
		\item Moreover, $E$ is partially carried-by with a triple
		\[\left(A = \left(\bigoplus_{i=1}^K (X_i = V_i \otimes S_{v_i}[d_i]), F_{j,i}\right), B = \left(\bigoplus_{i=K+1}^n (X_i= S_{v_i}[d_i], F_{j,i})\right), F = (F_{j,i})_{i=1, \dots, K}^{j=K+1, \dots, n} \in \shom_{\Fuk}^1(A, B)\right).\]
		\item We also assume that $A$ and $B$ are simplified twisted complexes.
	\end{itemize}
\end{notation}

In order to prove Lemma \ref{lem partially carried-by preserved}, we need to construct minimal twisted complexes $\overline{A}$ and $\overline{B}$ satisfying some conditions. 
First, let us construct two minimal twisted complexes $\overline{A}$ and $\overline{B}$ in this subsection. 
Their properties corresponding to the conditions (a)--(c) of Definition \ref{definition partially carried-by} will be proven in Section \ref{subsection proof of Lemma partially carried-by preserved}.
\vskip0.2in

\noindent{Construction of $\overline{A}$ and $\overline{B}$:}
We recall that by Lemma \ref{lem induced functor}, if $E_0 = \left(\bigoplus_{i \in I_0}(X_i=V_i \otimes S_{v_i}[d_i]), F_{j,i}\right)$, $\tau(E_0)$ is quasi-isomorphic to
\begin{equation}\label{eq quasiiso1}
	\tau(E_0) \cong \left(\bigoplus_{i \in I_0}\tau(X_i)= \tau(V_i \otimes S_{v_i}[d_i]), G_{j,i} \right)
\end{equation}
for a proper $G_{j,i} \in \hom^1_{\Fuk}\left(\tau(X_i),\tau(X_j)\right)$.
Moreover, for each $i\in I_0$, as computed in Lemma \ref{lem simple computation}, $\tau(X_i)$ should be one of the following four:
\begin{equation}\label{eq quasiiso2}
	\tau(V_i \otimes S_{v_i}[d_i]) \simeq \begin{cases}
		V_i \otimes S_{v_i}[d_i], \\
		V_i \otimes S_v[d_i+1-N] \text{  for some  } v \in V_+(T), \\
		V_i \otimes S_v[d_i+N-1] \text{  for some  } v \in V_-(T), \\
		V_i \otimes S_u[d_i] \xrightarrow{id_{V_i} \otimes x_{u,w}} V_i \otimes  S_w[d_i] \text{  for some  } u \in V_+(T), w \in V_-(T).\\
	\end{cases}
\end{equation}

Combining Equations \eqref{eq quasiiso1} and \eqref{eq quasiiso2}, one can obtain a good twisted complex equivalent to $\tau(E_0)$, 
\begin{equation}\label{eq quasiiso3}
	\tau(E_0) \simeq  \left( \bigoplus_{j\in J_0} (Y_j = \overline{V}_j \otimes S_{\overline{v}_j}[\overline{d}_j]), G_{j,i} = \overline{\psi}_{j,i} \otimes g_{j,i}\right)
\end{equation}
such that for all $j \in J_0$, $\overline{V}_j = V_i$ for a proper $i \in I_0$, and for all $i, j$, $\overline{\psi}_{j,i}\in \mathrm{Mor}(\overline{V}_i,\overline{V}_j), g_{j,i} \in \shom_{\Fuk}^1\left(S_{\overline{v}_i}[\overline{d}_i],S_{\overline{v}_j}[\overline{d}_j]\right)$.
We note that $J_0$ is the index set of the good twisted complex given in Equation \eqref{eq quasiiso3} and $I_0$ is the index set of the twisted complex $E_0$, or the index set of the twisted complex $\tau(E_0)$ given in Equation \eqref{eq quasiiso1}.

We note that the twisted complex $\left(\oplus_{j \in J_0} Y_j , G_{j,i}\right)$ in Equation \eqref{eq quasiiso3} is not minimal. 
Thus, we should cancel the identity-arrows from $\tau(E_0)$ in order to make $\tau(E_0)$ minimal. 
However, before that, we first cancel $z_u$-arrows for $u \in V_+(T)$ by utilizing Lemma \ref{lem z}. 
We will explain the reason why we cancel $z_u$-arrows in Remark \ref{rmk canceling z arrow}. 

\begin{lem}
	\label{lem z}
	Suppose that a given good twisted complex
	\begin{equation*}
		E_0= \left(\bigoplus_{i =1}^{n_0}\left(X_i = V_i \otimes S_{v_i} [d_i]\right), F_{j,i}\right)
	\end{equation*}
	satisfies that there exists $1 \leq i_1< j_1<i_2<j_2 \leq n_0$ such that 
	\begin{itemize}
		\item $v_{i_1} = v_{i_2} \in U_+(T)$ and $v_{j_1} = v_{j_2} \in V_-(T)$, or equivalently, there exist $u \in V_+(T), w \in V_-(T)$ so that $u = v_{i_1}=v_{i_2}, w = v_{j_1}= v_{j_2}$, 
		\item $u \sim w$ so that there exists $x_{u,w} \in \shom_{\Fuk}^1\left(S_u,S_w\right)$, 
		\item for $k=1,2$, $V_{i_k} = V_{j_k}$,
		\item for $k=1,2$, $F_{j_k,i_k} = \mathrm{Id}_{V_k}\otimes x_{u,w}$, and
		\item $F_{i_2,i_1} = \psi_1 \otimes z_u, F_{j_2,j_1}=\psi_2 \otimes z_w$.  
	\end{itemize}
	In other words, $E_0$ contains the following:
	\[
	\begin{tikzcd}
			V_{i_1} \otimes S_u[d_{i_1}] \arrow[swap]{r}{\mathrm{Id}_{V_{i_1}}\otimes x_{u,w}}  \arrow[bend left= 15]{rr}{\psi_1 \otimes z_u}  & V_{j_1} \otimes S_w[d_{j_1}=d_{i_1}] \arrow{r} \arrow[swap, bend right= 15]{rr}{\psi_2 \otimes z_w} & V_{i_2} \otimes S_u[d_{i_2} = d_{i_1}+N-1] \arrow{r}{\mathrm{Id}_{V_{i_2}} \otimes x_{u,w}} & V_{j_2} \otimes S_w[d_{j_2}=d_{i_1}+N-1] 
	\end{tikzcd}.\]
	Now, let $G$ be a collection of morphisms given as 
	\[G = \left(G_{j,i} \in \shom_{\Fuk}^1(V_i \otimes S_{v_i}[d_i], V_j \otimes S_{v_j}[d_j]\right)_{1 \leq i \leq j \leq n_0}, G_{j,i} = \begin{cases}
		a \psi \otimes z_u \text{  if  } (j,i) = (i_2,i_1), \\
		-a \psi \otimes z_w \text{  if  } (j,i) = (j_2, i_2),  \\
		0 \text{  otherwise},  
	\end{cases}\]
	for a linear map $\psi : V_{i_1} = V_{j_1} \to V_{i_2}=V_{j_2}$ and for an $a \in \mathbb{k}$.
	Then, the following two twisted complexes are equivalent:
	\[\left(\bigoplus_{i =1}^{n_0} V_i \otimes S_{v_i} [d_i], F_{j,i}\right) \simeq \left(\bigoplus_{i =1}^{n_0} V_i \otimes S_{v_i} [d_i], F_{j,i}+G_{j,i}\right).\]
\end{lem}
\begin{proof}
	To prove Lemma \ref{lem z}, it is sufficient to show that a morphism
	\[\begin{pmatrix} z_u & 0 \\ 0 & -z_w \end{pmatrix} 
	\in \shom_{\Fuk}^1 \left( C_1 \otimes S_u[d] \xrightarrow{\mathrm{Id_{C_1}} \otimes x_{u,w}} C_1\otimes S_w[d], C_2 \otimes S_u[d+N-1] \xrightarrow{\mathrm{Id}_{C_2} \otimes x_{u,w}} C_2 \otimes S_w[d+N-1]\right)\]
	is $\mu^1_{\Fuk}$-exact. 
	And, it is easy to observe that 
	\[
	\mu^1_{\Fuk} \begin{pmatrix}
			0 & y_{w,u} \\
			0 & 0
		\end{pmatrix} = \begin{pmatrix} z_u & 0 \\ 0 & -z_w \end{pmatrix}.
	\]
\end{proof}

\begin{remark}
	\label{rmk canceling z arrow}
	\mbox{}
	\begin{enumerate}
		\item Let us recall that the condition (b) of Definition \ref{definition partially carried-by} (2) is related to chains of arrows $(F_{i_2,i_1}, \dots, F_{i_k,i_{k-1}})$ such that $F_{i_2,i_1}$ is an $y$-arrow and $F_{i_m,i_{m-1}}$ is an $z_u$-arrow for some $u \in V_+(T)$, i.e., starting with a $y$-arrow, and all other arrows are $z_u$-arrows.
		Similarly, the condition (c) of Definition \ref{definition partially carried-by} (2) is related to chains of arrows starting with a $y$-arrow, ending with a $x$-arrow, and all other arrows are $z_u$-arrow. 
		Thus, if we can {\em cancel} $z_u$-arrows, then we can reduce the number of chains of arrows which we need to consider to prove Lemma \ref{lem partially carried-by preserved}. 
		It is the reason why we want to cancel $z_u$-arrows, and Lemma \ref{lem z} provides a way to cancel $z_u$-arrows by adding extra $z_w$-arrows. 
		\item Let us assume that a good twisted complex $E_0 = \left(\bigoplus_{i \in I_0} X_i, F_{j,i}\right)$ has two indices $i,j \in I$ satisfying
		\[X_i = V_i \otimes S_w[d] \ \mathrm{and}\ X_j = V_j \otimes S_w[d+N-1], \text{  and  } w \in V_-(T).\]
		If $u \sim w$, one can observe that 
		\[\tau_u(X_i) = V_i \otimes S_u[d] \xrightarrow{\mathrm{Id}_{V_1} \otimes x_{u,w}} V_i \otimes S_w[d] \ \mathrm{and}\ \tau_u(X_j) = V_j \otimes S_u[d+N-1] \xrightarrow{\mathrm{Id}_{V_2} \otimes x_{u,w}} V_i \otimes S_w[d+N-1].\]
		Thus, $\tau_u(E_0)$ satisfies the conditions of Lemma \ref{lem z}. 
		One can also easily observe that if
		\[X_i = V_i \otimes S_u[d], X_j = V_j \otimes S_u[d+N-1], u \in V_+(T),\]
		and if $w \in V_-(T)$ satisfies $u \sim w$, then $\tau_w^{-1}(E_0)$ also satisfies the conditions of Lemma \ref{lem z}.
	\end{enumerate}
\end{remark}

For convenience, by abusing notation, we let $\tau(E_0) \simeq \left(\bigoplus_{j \in J_0} Y_j, G_{j,i}\right)$ denote the twisted complex obtained by canceling as many $z_u$-arrows from $\tau(E_0)$ in Equation \eqref{eq quasiiso3} as possible. 
Now, we construct a {\em minimal} twisted complex equivalent to $\tau(E_0)$ from the twisted complex $\left(\bigoplus_{j \in J_0} Y_j, G_{j,i}\right)$ by applying the argument given in the proof of Lemma \ref{lem minimal twisted complex}, or equivalently, by {\em canceling the identity-arrows}. 
More precisely, for every pair $i < j \in J$ such that $g_{j,i}$ is the identity morphism, one can replace $\overline{V}_i$ with $\overline{V}_i / \mathrm{Ker} \overline{\psi}_{j,i}$ and $\overline{V}_j$ with $\overline{V}_j / \mathrm{Im} \overline{\psi}_{j,i}$. 
Then, one obtains a minimal twisted complex equivalent to $\tau(E_0)$. 

Now, let us consider a nonzero object $E$ which is partially carried-by with a triple $(A,B,F)$. 
As mentioned in Notation \ref{notation 1}, we use the following notations:
\begin{gather*}
	E= \left(\bigoplus_{i=1}^n (X_i = V_i \otimes S_{v_i}[d_i]), (F_{j,i}:X_i \to X_j)_{1 \leq i < j \leq n}\right), \\
	\left(A= \left(\bigoplus_{i=1}^K X_i, (F_{j,i})_{1 \leq i < j \leq K}\right), B =\left(\bigoplus_{i=K+1}^n X_i, (F_{j,i})_{K+1 \leq i < j \leq n}\right), F= \left(F_{j,i}\right)_{i=1, \dots, K}^{j = K+1, \dots, n}\right).
\end{gather*} 

Let $\overline{A}$ and $\overline{B}$ denote minimal twisted complexes equivalent to $\tau(A)$ and $\tau(B)$, which are obtained by applying the above arguments to $A$ and $B$ respectively. 
In other words, we first apply Lemma \ref{lem induced functor} to $\tau(A)$ and $\tau(B)$, then cancel as many $z_u$-arrows as possible, and then cancel the identity arrows by applying Lemma \ref{lem minimal twisted complex}. 
For convenience, we set 
\[\overline{A}= \left(\bigoplus_{i=1}^{\overline{K}} (Y_i = \overline{V}_i \otimes S_{\overline{v}_i}[\overline{d}_i]), (G_{j,i})_{1 \leq i < j \leq \overline{K}}\right), \overline{B} =\left(\bigoplus_{i=\overline{K}+1}^{\overline{n}} (Y_i = \overline{V}_i \otimes S_{\overline{v}_i}[\overline{d}_i]), (G_{j,i})_{\overline{K}+1 \leq i < j \leq \overline{n}}\right).\]

Since $E \simeq \left(A \oplus B, F\right)$, there exists a collection of morphisms 
\[G = \left(G_{j,i}\right)_{i=1, \dots, \overline{K}}^{j= \overline{K}+1, \dots, \overline{n}},\]
satisfying that $\tau(E)  \simeq \left(\overline{A} \oplus \overline{B}, G\right)$.
Or equivalently, $\tau(E)$ is equivalent to a twisted complex 
\begin{equation*}
	\tau(E) \simeq \left(\bigoplus_{k =1}^{\overline{n}}( Y_k = \overline{V}_k \otimes S_{\overline{v}_k}[\overline{d}_k]), \left(G_{j,i}\right)_{1 \leq i < j \leq \overline{n}}\right).
\end{equation*}

From the above construction, one can check that for any $i \leq \overline{K} <j$, $G_{j,i}$ cannot be an identity arrow. 
It is because, for any $j \geq K$, $F_{j,i}$ cannot be a $y$-arrow. 
Then, applying Lemmas \ref{lem x-morphism}--\ref{lem the zero morphism}, there exists no identity arrows from $\tau(X_i)$ to $\tau(X_j)$ for $j \geq K$. 
Then, when one constructs minimal twisted complexes $\overline{A}$ and $\overline{B}$ by canceling identity-arrows from $\tau(E) = \left(\bigoplus_{i\in I}X_i, F_{j,i}\right)$ (after canceling $z_u$-arrows), the procedure does not generate an extra identity-arrow.
Thus, we have a minimal twisted complex equivalent to $\tau(E_0)$, denoted as follows by abusing notations:
\begin{gather}
	\label{eqn final one} 
	\tau(E) \simeq \left(\bigoplus_{k =1}^{\overline{n}}( Y_k = \overline{V}_k \otimes S_{\overline{v}_k}[\overline{d}_k]), \left(G_{j,i}\right)_{1 \leq i < j \leq \overline{n}}\right).
\end{gather}

The above minimal twisted complex in Equation \eqref{eqn final one} is a minimal twisted complex equivalent to $\tau(E)$ with a triple 
\[\left(\overline{A}, \overline{B}, \left\{G_{j,i}\right\}_{1 \leq i \leq \overline{K}}^{\overline{K}+1 \leq j \leq \overline{n}}\right).\]
Thus, for proving Lemma \ref{lem partially carried-by preserved}, it is enough to show that the twisted complex in Equation \eqref{eqn final one} and $\overline{K}$ satisfy the conditions (a)--(c) of Definition \ref{definition partially carried-by} (2). 

Since the above construction of $\overline{A}$, $\overline{B}$, and the twisted complex in Equation \eqref{eqn final one}, we summarize the notations before going further. 
\begin{notation}
	\label{notation 2}
	\mbox{}
	\begin{itemize}
		\item We will use $\overline{A}$ to indicate the minimal twisted complex such that $\overline{A} \simeq \tau(A)$, which is obtained by canceling $z_u$-arrows and identity-arrows. 
		Similarly, $\overline{B}$ will denote the minimal twisted complex equivalent to $\tau(B)$ obtained by the same way, and $G: \overline{A} \to \overline{B}$ will denote the morphism such that 
		\[\tau(E) \simeq \left(\overline{A} \oplus \overline{B}, G\right).\]
		\item The minimal twisted complexes $\overline{A}$ and $\overline{B}$ will be written as 
		\[\overline{A}= \left(\bigoplus_{i=1}^{\overline{K}} (Y_i = \overline{V}_i \otimes S_{\overline{v}_k}[\overline{d}_k]), (G_{j,i})_{1 \leq i < j \leq \overline{K}}\right), \overline{B}= \left(\bigoplus_{i=\overline{K} +1}^{\overline{n}} (Y_i = \overline{V}_i \otimes S_{\overline{v}_k}[\overline{d}_k]), (G_{j,i})_{\overline{K}+1 \leq i < j \leq \overline{n}}\right).\]
		\item Moreover, $\tau(E)$ is equivalent to the following twisted complex:
		\[\tau(E) = \left(\bigoplus_{j \in J} (Y_j = \overline{V}_j \otimes S_{\overline{v}_j}[\overline{d}_j]), (G_{j,i} = \overline{\psi}_{j,i} \otimes g_{j,i})_{1 \leq i < j \leq \overline{n}}\right).\]
		Especially, we note that $J= \{1, \dots, \overline{n}\}$ is the index set of the above twisted complex. 
	\end{itemize}
\end{notation}

\begin{remark}
	\label{rmk originate}
	We note that from the above argument, one can observe that $Y_j$ {\em originates} from $X_i$, in the sense that $Y_j$ is a part or a quotient of a part of $\tau(X_i)$. 
	In the rest of this section, we will use a word ``{\em originated}" in this viewpoint. 
\end{remark}

Again, we note that, in order to prove Lemma \ref{lem partially carried-by preserved}, it is enough to show that $\tau(E)$ is partially carried-by with a triple $(\overline{A}, \overline{B}, G)$. 
In the rest of Section \ref{subsection step 4}, we state and prove Lemmas \ref{lem noxyyx} and \ref{lem y-arrow2} that we need to prove Lemma \ref{lem partially carried-by preserved}. 
The proof of Lemma \ref{lem partially carried-by preserved} will be given in the next subsection.

\begin{lem}
	\label{lem noxyyx}
	Let $E_0=\left(\bigoplus_{i =1}^{n_0} X_i, F_{j,i}\right)$ be a simplified twisted complex.
	\begin{enumerate}
		\item If there exist $i<j<k \in \{1, \dots, n_0\}$ satisfying that 
		\begin{gather}
			\label{eq (a)}
			X_i = V_i \otimes S_u[d_i] \xrightarrow{F_{j,i} = \psi_{j,i} \otimes x_{u,w}} X_j = V_j \otimes S_w[d_i] \xrightarrow{F_{k,j} = \psi_{k,j} \otimes y_{w,u}} X_k  = V_k \otimes S_u[d_i+N-2],
		\end{gather}
		then the composition of two linear maps $\psi_{k,j} \circ \psi_{j,i}$ should be the zero map.
		\item If there exist $i<j<k \in \{1, \dots, n_0\}$ satisfying that 
		\begin{gather}
			\label{eq (b)}
			X_i = V_i \otimes S_w[d_i] \xrightarrow{F_{j,i} = \psi_{j,i} \otimes y_{w,u}} X_j = V_j \otimes S_u[d_i+N-2] \xrightarrow{F_{k,j} = \psi_{k,j} \otimes x_{u,w}} X_k  = V_k \otimes S_w[d_i+N-2],
		\end{gather}
		then the composition of two linear maps $\psi_{k,j} \circ \psi_{j,i}$ should be the zero map.
	\end{enumerate}
\end{lem}
\begin{proof}
	We recall that, in \eqref{eq (a)}, since $F_{j,i}$ is an $x$-arrow, thus, $v_i = u \in V_+(T)$ and $v_j =w \in V_-(T)$.
	Moreover, since the degree of $x_{u,w}$ is $1$, $d_j - d_i=0$. 
	Thus, for simplicity, we just write that $d_i = d$ and $d_j =d$ in the quiver given in \eqref{eq (a)}. 
	The quiver in \eqref{eq (b)} is also obtained in the same way.
	
	We prove Lemma \ref{lem noxyyx} (1) only, and (2) can be proven by the same proof. 
	To prove (1), we first recall that the definition of {\em simplified} twisted complex requires the twisted complex to be minimal. 
	Then, by Lemma \ref{lem minimal twisted complex} (1), we can assume without loss of generality that the twisted complex is well-ordered.
	In other words, for a given twisted complex $E_0=\left(\bigoplus_{i =1}^{n_0} X_i = V_i \otimes S_{v_i}[d_i], F_{j,i}\right)$, we assume that 
	\begin{itemize}
		\item $d_i \leq d_{i+1}$ for all $i$, and 
		\item if $d_i = d_{i+1}$ and $v_i \in V_-(T)$, then $v_{i+1} \in V_-(T)$.
	\end{itemize}

	Let $i < j < k$ be a triple satisfying \eqref{eq (a)}. 
	For convenience, we will assume that $d_i=0$ in the rest of the proof. 
	Moreover, we rearrange the index set $I_0 = \{1, \dots, n_0\}$ of $E_0$ without changing the arrows $F_{j,i}$ so that after the rearrangement, the ordered index set satisfies that 
	\begin{itemize}
		\item if $i'$ is the index such that $X_{i'} = V_{i'} \otimes S_{u'}[d_{i'}]$ with $u' \in V_+(T)$ and $d_{i'} \leq 0$, then $i' \leq i$, 
		\item if $k'$ is the index such that $X_{k'} = V_{k'} \otimes S_{u'}[d_{k'}]$ with $u' \in V_+(T)$ and $d_{k'} \geq N-2$, then $k' \geq k$, 
		\item there exists a natural number $A \in \mathbb{N}$ such that the arrow $F_{j',i}$ is a $x$-arrow if and only if $0 < j'-i  \leq A$, and
		\item especially, $j= i+1$.
	\end{itemize}

	After rearranging the index set, one can observe the following:
	\begin{enumerate}
		\item[(i)] For any $k'$ such that $i+A < k' \leq k$, $F_{k',i}$ should be the zero arrow because of the degree reason. 
		\item[(ii)] For any pair $(j',k')$ such that  $i+q \leq j' \leq i+b < k' <k$, $F_{k',j'}$ should be the zero arrow because of the degree reason. 
	\end{enumerate}
	From (i) and (ii), one can ensure that every nonzero chain of arrows from $X_i$ to $X_k$ is the form of 
	\[X_i = V_i \otimes S_u \xrightarrow{F_{j',i} = \psi_{j',i} \otimes x_{u,w'}} X_{j'} = V_{j'} \otimes S_{w'} \xrightarrow{F_{k,j'} = \psi_{k,j'} \otimes y_{w',u}} X_k  = V_k \otimes S_u[N-2],\]
	with $i+1 \leq j' \leq i +A$.
	Moreover, from the generalized Maurer-Cartan equation, we have 
	\begin{gather}
		\label{eqn generalized Maurer-Cartan 1}
		\sum_{j' = i+1}^{i+A} (-1)^{*_{j'}} \left(\psi_{k,j'} \circ \psi_{j',i}\right) \otimes \mu_{\Fuk}^2\left(y_{w',u},x_{u,w'}\right) =  \sum_{j' = i+1}^{i+A} (-1)^{*_{j'}}\left(\psi_{k,j'} \circ \psi_{j',i}\right) \otimes z_u = dF_{k,i} = 0.
	\end{gather}
	We note that for simplicity, we omit the details on the sign $(-1)^{*_{j'}}$.
	
	From Equation \eqref{eqn generalized Maurer-Cartan 1}, we have 
	\[\sum_{j' = i+1}^{i+A} (-1)^{*_{j'}} \left(\psi_{k,j'} \circ \psi_{j',i}\right) =0.\]
	In order to prove (1), it is enough to prove that 
	\begin{gather}
		\label{eqn generalized Maurer-Cartan 2}
		\sum_{j' = i+2}^{i+A} (-1)^{*_{j'}} \left(\psi_{k,j'} \circ \psi_{j',i}\right) =0,
	\end{gather}
	since we are assuming that $j=i+1$.
	To show the above equality, we consider a twisted complex $E_1$ with an index set 
	\[I_1 := \left\{i < i+1 < \dots < i+A < k\right\},\]
	such that 
	\[E_1 = \left(\bigoplus_{a \in I_1} X_a, F_{b,a}\right).\]
	It is easy to check that $E_1$ is a twisted complex, i.e., $E_1$ satisfies the generalized Maurer-Cartan equation. 
	
	Now, we apply the autoequivalence $\tau_w$ to the new twisted complex $E_1$. 
	Then, Lemma \ref{lem induced functor} gives us 
	\[\tau_w(E_1) \simeq \left(\bigoplus_{a \in I_1} \tau_w(X_a), G_{b,a}\right).\]
	And, it is easy to check that the following hold:
\begin{itemize}
	\item $\tau_w(X_i = V_i \otimes S_u) = V_i \otimes S_w [-N+2] \xrightarrow{id_{V_i} \otimes y_{w,u}} V_i \otimes S_u$.
	\item $\tau_w(X_j = X_{i+1}) = V_j \otimes S_w [-N+1]$.
	\item For any $a$ such that $i+2 \leq a \leq i+A$, $\tau_w(X_a = V_a \otimes S_{w'}) = V_a \otimes S_{w'}$. We note that since $E_0$ is simplified, $w' \neq w$. 
	\item $\tau_w(X_k = V_k \otimes S_u[N-2]) = V_k \otimes S_w  \xrightarrow{id_{V_k} \otimes y_{w,u}} V_k \otimes S_u[N-2]$.
\end{itemize}

In the rest of proof, we would like to analyze the nonzero chains of arrows starting at $V_i \otimes S_u$ contained in $\tau_w(X_i)$ and ending at $V_i \otimes S_u[N-1]$ contained in $\tau_w(X_k)$. 
It is easy to observe that, by degree reason, the all possible nonzero chains of arrows are of the following form:
\[V_i \otimes S_u \xrightarrow{G_{j',i} = \psi_{j',i} \otimes x_{u,w'}} V_{j'} \otimes S_{w'} \xrightarrow{G_{k,j'} = \psi_{k,j'} \otimes y_{w',u}} V_k \otimes S_u[N-2],\]
with $i+2 \leq j' \leq i+A$.
We note that one can compute $G_{j',i}$ directly by using the argument in \cite[Section (3m)]{Seidel08}, since the only nonzero chain of arrows from $X_i$ to $X_{j'}$ in $E_1$ is the arrow $F_{j',i}$.

Now, since $\tau_w(E_1)$ should satisfy the Maurer-Cartan equation, we have 
\[\sum_{j' = i+2}^{i+A} (-1)^{*_{j'}} \left(\psi_{k,j'} \circ \psi_{j',i}\right) \otimes \mu_{\Fuk}^2\left(y_{w',u},x_{u,w'}\right) =  \sum_{j' = i+2}^{i+A} (-1)^{*_{j'}}\left(\psi_{k,j'} \circ \psi_{j',i}\right) \otimes z_u = 0.\]
	In other words, Equation \eqref{eqn generalized Maurer-Cartan 2} holds. 
	It completes the proof.
\end{proof}

\begin{lem}
	\label{lem y-arrow2}
	Let $E_0= (\bigoplus_{i =1}^{n_0} X_i, F_{j,i})$ be a simplified twisted complex, and let $\tau$ be either $\tau_u$ for some $u\in V_+(T)$ or $\tau_w^{-1}$ for some $w\in V_-(T)$.
	Moreover, let 
	\[\left(\bigoplus_{j \in J} (Y_j = \overline{V}_j \otimes S_{\overline{v}_j}[\overline{d}_j]),G_{j,i}\right)\]
	be the minimal twisted complex obtained by canceling $z_u$-arrows and the identity-arrows of $\tau(E_0)$.
	If there exists a $y$-arrow $G_{j_2,j_1}$ of $\left(\bigoplus_{j\in J} Y_j, G_{j,i}\right)$, i.e., there exist $u \in V_+(T), w \in V_-(T)$ such that 
	\[Y_{j_1} = \overline{V}_{j_1} \otimes S_w[\overline{d}_{j_1}] \xrightarrow{G_{j_2, j_1} = \overline{\psi}_{j_2,j_1} \otimes y_{w,u}} Y_{j_2} = \overline{V}_{j_2} \otimes S_u[\overline{d}_{j_2} = \overline{d}_{j_1}+N-2],\]
	then, at least one of the following six items must hold:
	\begin{enumerate}
		\item $\tau = \tau_{u}$, and there exist $i_1<i_{1.5}<i_2 \in \{1, \dots, n_0\}$ such that 
		\begin{itemize}
			\item $Y_{j_1}$ (resp.\ $Y_{j_2}$) originates from $X_{i_1}$ (resp.\ $X_{i_2}$) in the sense of Remark \ref{rmk originate}, and
			\item there exists a nonzero chain of arrows 
			\[X_{i_1} =V_{i_1} \otimes S_w[\overline{d}_{j_1}] \xrightarrow{\psi_{i_{1.5},i_1} \otimes y_{w,u}} X_{i_{1.5}} =V_{i_{1.5}} \otimes S_u[\overline{d}_{j_1}+N-2] \xrightarrow{\psi_{i_2,i_{1.5}} \otimes z_u} X_{i_2} = V_{i_2} \otimes S_u[\overline{d}_{j_1} + 2N-3].\] 
		\end{itemize}
		\item  $\tau = \tau_{u}$, and there exist $i_1<i_{1.5}<i_2 \in \{1, \dots, n_0\}$ and $w' \in V_-(T)\setminus \{w\}$ such that 
		\begin{itemize}
			\item $Y_{j_1}$ (resp.\ $Y_{j_2}$) originates from $X_{i_1}$ (resp.\ $X_{i_2}$) in the sense of Remark \ref{rmk originate}, and
			\item there exists a nonzero chain of arrows 
			\[X_{i_1} =V_{i_1} \otimes S_w[\overline{d}_{j_1}] \xrightarrow{\psi_{i_{1.5},i_1} \otimes y_{w,u}} X_{i_{1.5}} =V_{i_{1.5}} \otimes S_u[\overline{d}_{j_1}+N-2] \xrightarrow{\psi_{i_2,i_{1.5}} \otimes x_{u,w'}} X_{i_2} = V_{i_2} \otimes S_{w'}[\overline{d}_{j_1} + N-2].\] 
		\end{itemize}
		\item $\tau = \tau_{u'}$ for a $u' \neq V_+(T) \setminus \{u\}$, and there exist $i_1<i_2 \in \{1, \dots, n_0\}$ such that 
		\begin{itemize}
			\item $Y_{j_1}$ (resp.\ $Y_{j_2}$) originates from $X_{i_1}$ (resp.\ $X_{i_2}$) in the sense of Remark \ref{rmk originate}, and
			\item there exists a nonzero chain of arrows 
			\[X_{i_1} =V_{i_1} \otimes S_w[\overline{d}_{j_1}] \xrightarrow{\psi_{i_2,i_1} \otimes y_{w,u}} X_{i_2} =V_{i_2} \otimes S_u[\overline{d}_{j_1}+N-2].\] 
		\end{itemize}
		\item $\tau = \tau_w^{-1}$, and there exist $i_1<i_{1.5}<i_2 \in \{1, \dots, n_0\}$ such that 
		\begin{itemize}
			\item $Y_{j_1}$ (resp.\ $Y_{j_2}$) originates from $X_{i_1}$ (resp.\ $X_{i_2}$) in the sense of Remark \ref{rmk originate}, and
			\item there exists a nonzero chain of arrows 
			\[X_{i_1} =V_{i_1} \otimes S_w[\overline{d}_{j_1}-(N-1)] \xrightarrow{\psi_{i_{1.5},i_1} \otimes z_w} X_{i_{1.5}} =V_{i_{1.5}} \otimes S_w[\overline{d}_{j_1}] \xrightarrow{\psi_{i_2,i_{1.5}} \otimes y_{w,u}} X_{i_2} = V_{i_2} \otimes S_u[\overline{d}_{j_1}+N-2].\] 
		\end{itemize}
		\item $\tau = \tau_{w^{-1}}$, and there exist $i_1<i_{1.5}<i_2 \in \{1, \dots, n_0\}$ and $u' \in V_+(T) \setminus \{u\}$ such that 
		\begin{itemize}
			\item $Y_{j_1}$ (resp.\ $Y_{j_2}$) originates from $X_{i_1}$ (resp.\ $X_{i_2}$) in the sense of Remark \ref{rmk originate}, and
			\item there exists a nonzero chain of arrows 
			\[X_{i_1} =V_{i_1} \otimes S_{u'}[\overline{d}_{j_1}] \xrightarrow{\psi_{i_{1.5},i_1} \otimes x_{u',w}} X_{i_{1.5}} =V_{i_{1.5}} \otimes S_w[\overline{d}_{j_1}] \xrightarrow{\psi_{i_2,i_{1.5}} \otimes y_{w,u}} X_{i_2} = V_{i_2} \otimes S_u[\overline{d}_{j_1} + N-2].\] 
		\end{itemize}		
		\item $\tau = \tau_{w'}$ for a $w' \in V_-(T) \setminus \{w\}$, and there exist $i_1<i_2 \in \{1, \dots, n_0\}$ such that 
		\begin{itemize}
			\item $Y_{j_1}$ (resp.\ $Y_{j_2}$) originates from $X_{i_1}$ (resp.\ $X_{i_2}$) in the sense of Remark \ref{rmk originate}, and
			\item there exists a nonzero chain of arrows 
			\[X_{i_1} =V_{i_1} \otimes S_w[\overline{d}_{j_1}] \xrightarrow{\psi_{i_2,i_1} \otimes y_{w,u}} X_{i_2} =V_{i_2} \otimes S_u[\overline{d}_{j_1}+N-2].\] 
		\end{itemize}
	\end{enumerate}
\end{lem}
\begin{proof}
	We will show that if $\tau = \tau_u$, then either the first item (1) or the second item (2) holds. 
	And, we also show that if $\tau = \tau_{u'}$ for a $u' \in V_+(T)\setminus \{u\}$, then the third item (3) holds.
	We omit the other cases, i.e., the cases of $\tau = \tau_w^{-1}$ or $\tau = \tau_{w'}^{-1}$ for a $w' \in V_-(T) \setminus \{w\}$, because the same logic works, and it will show that (4)--(6) hold for those cases. 
	
	Assume $\tau = \tau_u$. Since $\tau_u(X_{i_1})$ is assumed to contain $Y_{j_1} = \overline{V}_{j_1} \otimes S_w [\overline{d}_{j_1}]$, Lemma \ref{lem simple computation} says that $X_{i_1}$ is necessarily given by $V_{i_1} \otimes S_w[\overline{d}_{j_1}]$.
	Moreover, since $\tau_u(X_{i_2})$ contains $Y_{j_2} = \overline{V}_{j_2} \otimes S_u [\overline{d}_{j_1}+N-2]$, 
	\[X_{i_2} = \begin{cases}
		V_{i_2} \otimes S_u[\overline{d}_{j_1}+2N-3], \\
		V_{i_2} \otimes S_{w_0}[\overline{d}_{j_1} +N-2] \text{  with  } w_0 \sim u.
	\end{cases}\]
	
	Let us first assume that $X_{i_2} = V_{i_2} \otimes S_u[\overline{d}_{j_1}+2N-3]$.
	Since the arrow $G_{j_2,j_1}$ is an $y$-arrow, there exists at least one nonzero chain of arrow from $X_{i_1}$ to $X_{i_2}$.
	When one considers the degree, there exist the following possible nonzero chains of arrows:
	\begin{enumerate}
		\item[\textcircled{1}] For any $N \geq 3$, 
		\[X_{i_1} =V_{i_1} \otimes S_w[\overline{d}_{j_1}] \xrightarrow{\psi_{i_{1.5},i_1} \otimes y_{w,u}} X_{i_{1.5}} =V_{i_{1.5}} \otimes S_u[\overline{d}_{j_1}+N-2] \xrightarrow{\psi_{i_2,i_{1.5}} \otimes z_u} X_{i_2} = V_{i_2} \otimes S_u[\overline{d}_{j_1} + 2N-3].\] 
		\item[\textcircled{2}] For any $N \geq 3$, 
		\[X_{i_1} =V_{i_1} \otimes S_w[\overline{d}_{j_1}] \xrightarrow{\psi_{i_{1.5},i_1} \otimes z_w} X_{i_{1.5}} =V_{i_{1.5}} \otimes S_w[\overline{d}_{j_1}+N-1] \xrightarrow{\psi_{i_2,i_{1.5}} \otimes y_{w,u}} X_{i_2} = V_{i_2} \otimes S_u[\overline{d}_{j_1} + 2N-3].\]
		\item[\textcircled{3}] For $N=3$, the chain of arrow consisting of three $y$-arrows and two $x$-arrows.  
	\end{enumerate}
	To prove Lemma \ref{lem y-arrow2}, we need to show that \textcircled{1} must happen.
	
	We note that \textcircled{3} cannot happen.
	If \textcircled{3} happens, then it contradicts to Lemma \ref{lem noxyyx}. 
	
	Now, let us assume that \textcircled{2} happens, but \textcircled{1} does not happen. 
	Then, we can observe that the second arrow in \textcircled{2}, i.e., 
	\[X_{i_{1.5}} =V_{i_{1.5}} \otimes S_w[\overline{d}_{j_1}+N-1] \xrightarrow{\psi_{i_2,i_{1.5}} \otimes y_{w,u}} X_{i_2} = V_{i_2} \otimes S_u[\overline{d}_{j_1} + 2N-3].\]
	is the only nonzero chain of arrows from $X_{i_{1.5}}$ to $X_{i_2}$ since $E_0$ is simplified and because of Lemma \ref{lem noxyyx}.
	Then, when one applies Lemma \ref{lem induced functor} for computing $\tau(E_0) = \tau_u(E_0)$, the arrows between $\tau(X_{i_{1.5}})$ to $\tau(X_{i_2})$ are determined from the unique nonzero arrows and $\tau_u$.
	As the result, we have 
	\[\begin{tikzcd}
		\big(V_{i_{1.5}} \otimes S_u[\overline{d}_{j_1} +N-1] \arrow[r, "id_{V_{i_{1.5}}} \otimes x_{u,w}"'] \arrow[rr, bend left = 15, "\psi_{i_2,i_{1.5}} \otimes e_u"]  & V_{i_{1.5}} \otimes S_w [\overline{d}_{j_1} + N-1]\big)  \arrow[r, "0\text{-arrow}"'] & V_{i_2} \otimes S_u [\overline{d}_{j_1}+N-2].
	\end{tikzcd}\]
	We note that $\tau_u(X_{i_{1.5}})$ is given in the inside of parentheses.
	Thus, by canceling the identity arrow $\psi_{i_2,i_{1.5}} \otimes e_u$, one observes that $\overline{V}_{j_2} = V_{i_2} / \operatorname{Im}(\psi_{i_2, i_{1.5}})$.
	Moreover, the nonzero chain of arrows given in \textcircled{2} cannot induce a $y$-arrow from $Y_{j_1}$ to $Y_{j_2}$. 
	It contradicts to the existence of $y$-arrow from $Y_{j_1}$ to $Y_{j_2}$.
	Thus, \textcircled{1} must happen. 
	
	Now, let us assume the second possible case for $\tau = \tau_u$, i.e., 
	\[X_{i_2} = V_{i_2} \otimes S_{w_0}[\overline{d}_{j_1} +N-2] \text{  with  } w_0 \sim u.\]
	Because of the degree reason, the only possible nonzero chain of arrows from $X_{i_1}$ to $X_{i_2}$ is the one given in the second item (2) of Lemma \ref{lem y-arrow2}. 
	Thus, (2) holds for the second case.
	
	Finally, if $\tau = \tau_{u'}$ with $u' \in V_+(T) \setminus \{u\}$, then we have 
	\[X_{i_1} = V_{i_1} \otimes S_w[\overline{d}_{j_1}], X_{i_2} = V_{i_2} \otimes S_u[\overline{d}_{j_2} +N-2].\]
	It is because $\tau(X_{i_1})$ contains $Y_{j_1} = \overline{V}_{j_1} \otimes S_w[\overline{d}_{j_1}]$ and $\tau(X_{i_2})$ contains $Y_{j_2} = \overline{V}_{j_2} \otimes S_u[\overline{d}+N-2]$. 
	And, the only possible nonzero chain of arrows from $X_{i_1}$ to $X_{i_2}$ is the one given in Lemma \ref{lem y-arrow2}, (3). 
	Thus, (3) holds if $\tau = \tau_{u'}$ with $u' \in V_+(T) \setminus \{u\}$.
\end{proof}

Before proving Lemma \ref{lem partially carried-by preserved}, let us remark the meaning of Lemma \ref{lem y-arrow2}.
\begin{remark}
	\label{rmk meaning of Lemma}
	Similar to Lemma \ref{lem y-arrow2}, let $E_0= (\bigoplus_{i =1}^{n_0} X_i, F_{j,i})$ be a simplified twisted complex, $\tau$ be either $\tau_u$ for $u \in V_+(T)$ or $\tau_w^{-1}$ for $w \in V_-(T)$, and $\tau(E_0)$ denote the minimal twisted complex 
	\[\left(\bigoplus_{j \in J} (Y_j = \overline{V}_j \otimes S_{\overline{v}_j}[\overline{d}_j]),G_{j,i}\right),\]
	which is obtained by canceling $z_u$ and identity-arrows. 
	From Lemma \ref{lem y-arrow2}, one can conclude that the number of $y$-arrows we need to generate $\Phi E_0$ is smaller or equal to the number of $y$-arrow we need to generate $E_0$. 
	See the main idea section, Section \ref{subsection the main idea}, especially, the arguments right above of Remark \ref{rmk weak pseudo-Anosov}. 
\end{remark}

In order to state Remark \ref{rmk meaning of Lemma} formally, we define the following:
\begin{definition}
	\label{def y-rank}
	Let $E$ be an arbitrary object of $\Fuk$ and, let $E_0 = \left(\bigoplus_{i=1}^{n_0} X_i = V_i \otimes S_{v_i}[d_i], F_{j,i}=\psi_{j,i} \otimes f_{j,i}\right)$ be a minimal twisted complex equivalent to $E$. 
	The {\bf $\boldsymbol{y}$-rank of $\boldsymbol{E}$} denote the nonnegative integer defined as 
	\[r(E) := \sum_{j,i \in I \text{  such that  } F_{j,i} \text{  is an $y$-arrow}} \mathrm{dim} \operatorname{Im} \psi_{j,i}.\]
\end{definition}

\begin{lem}
	\label{lem y-rank}
	In the same setting of Remark \ref{rmk meaning of Lemma}, 
	\[r(E_0) \geq r(\Phi E_0).\]
\end{lem}
\begin{proof}
	If there exists a $y$-arrow 
	\[Y_{j_1} = \overline{V}_{j_1} \otimes S_w[\overline{d}_{j_1}] \xrightarrow{G_{j_2, j_1} = \overline{\psi}_{j_2,j_1} \otimes y_{w,u}} Y_{j_2} = \overline{V}_{j_2} \otimes S_x[\overline{d}_{j_2} = \overline{d}_{j_1}+N-2],\]
	then by Lemma \ref{lem y-arrow2}, there exists a pair $j_1', j_2'$ such that 
	\[X_{j_1'} =V_{j_1'} \otimes S_w[\overline{d}_{j_1}] \xrightarrow{F_{j_2',j_1'} = \psi_{j_2',j_1'} \otimes y_{w,u}} X_{j_2'} = V_{j_2'} \otimes S_u[\overline{d}_{j_1}+N-2],\]
	such that 
	\[\dim \operatorname{Im} \psi_{j_2',j_1'} \geq \dim \operatorname{Im} \overline{\psi}_{j_2,j_1}.\]
	
	Moreover, if $G_{j_2,j_1}$ and $G_{j_4,j_3}$ are two different $y$-arrows, i.e., two pairs $(j_1,j_2)$ and $(j_3,j_4)$ are different, then $(j_1',j_2')$ and $(j_3',j_4')$ are different since $E_0$ is simplified.
	Thus, we have 
	\[r(E_0)= \sum_{j,i \in I \text{  such that  } F_{j,i} \text{  is an $y$-arrow}} \mathrm{dim} \operatorname{Im} \psi_{j,i} \geq r(\Phi E_0) := \sum_{j,i \in I \text{  such that  } G_{j,i} \text{  is an $y$-arrow}} \mathrm{dim} \operatorname{Im} \overline{\psi}_{j,i}.\]
\end{proof}

\subsection{Proof of Lemma \ref{lem partially carried-by preserved}}
\label{subsection proof of Lemma partially carried-by preserved}

Now we are ready to prove Lemma \ref{lem partially carried-by preserved}.
\begin{proof}[Proof of Lemma \ref{lem partially carried-by preserved}]
	We will prove Lemma \ref{lem partially carried-by preserved} for the case of $\tau = \tau_{u_0}$ for a $u_0 \in V_+(T)$. 
	Since the other case, i.e., the case of $\tau= \tau_{w_0}^{-1}$ for a $w_0 \in V_-(T)$, can be proven in a similar way, we omit the case. 
	
	We also note that in this proof, we will use the notations given in Notations \ref{notation 1} and \ref{notation 2}.
	In other words, if a nonzero object $E \in \Fuk$ is partially carried-by with a triple $(A,B,F)$, then we write them as follows:
	\begin{gather*}
		A=\left(\bigoplus_{i=1}^{K} (X_i=V_i \otimes S_{v_i}[d_i]), F_{j,i}\right), B=\left(\bigoplus_{i=K+1}^n(X_i=V_i \otimes S_{v_i}[d_i]),F_{j,i}\right),\\
		E=\left(A \oplus B, F\right) = \left(\bigoplus_{i=1}^n X_i,F_{j,i}\right).
	\end{gather*}
	In the previous subsection, we constructed two minimal twisted complexes $\overline{A} \simeq \tau(A)$ and $\overline{B} \simeq \tau(B)$, which we write as follows:
	\begin{gather*}
		\overline{A} =\left(\bigoplus_{j=1}^{\overline{K}} (Y_j=\overline{V}_j \otimes S_{\overline{v}_j}[\overline{d}_j]), G_{j,i}\right), \overline{B} =\left(\bigoplus_{j=\overline{K} +1}^{\overline{n}} (Y_j=\overline{V}_j \otimes S_{\overline{v}_j}[\overline{d}_j]), G_{j,i}\right).
	\end{gather*}
	Moreover, we obtained a minimal twisted complex equivalent to $\tau(E)$,
	\[\tau(E) \simeq \left(\overline{A} \oplus \overline{B}, G\right) = \left(\bigoplus_{j=1}^{\overline{n}}Y_j, G_{j,i}\right).\]
	
	In the proof, we prove that $\tau(E)$ is partially carried-by with the triple $(\overline{A}, \overline{B}, G)$. 
	In other words, we need to prove the following three claims corresponding to the conditions (a)--(c) of Definition \ref{definition partially carried-by}:
	\begin{enumerate}
		\item[(a)] If $j_2 > \overline{K}$, then $G_{j_2,j_1}$ cannot be an $y$-arrow. Or equivalently, if there exists an $y$-arrow $G_{j_2,j_1}$, then $j_2 \leq \overline{K}$. 
		\item[(b)] If there exists a nonzero chain of arrows $\left(G_{j_1,j_0}, \dots, G_{j_k,j_{k-1}}\right)$ such that
		\begin{itemize}
			\item $G_{j_1,j_0}$ is a $y$-arrow, 
			\item $\overline{\psi}_{j_k,j_{k-1}} \circ \dots \circ \overline{\psi}_{j_1,j_0} \neq 0$, and
			\item for all $1 \leq m <k$,  $G_{j_{m+1},j_m}$ is a $z$-arrow,
		\end{itemize}
		then $j_k \leq \overline{K}$, or equivalently, $Y_{j_k}$ is a part of $\overline{A}$. 
		\item[(c)] If there exists a nonzero chain of arrows $\left(G_{j_1,j_0}, \dots, G_{j_{k+1},j_k}\right)$ such that 
		\begin{itemize}
			\item $G_{j_1,j_0}$ is a $y$-arrow,
			\item for all $1 \leq m <k$, $G_{j_{m+1},j_m}$ is a $z$-arrow, 
			\item $\overline{\psi}_{j_{k+1},j_k} \circ \dots \circ \overline{\psi}_{j_1,j_0} \neq 0$, and
			\item $G_{j_{k+1},j_k}$ is an $x$-arrow,
		\end{itemize}   
		then $j_{k+1} \leq \overline{K}$, or equivalently, $Y_{j_{k+1}}$ is a part of $\overline{A}$. 
	\end{enumerate}	
	\vskip0.2in 
	
	\noindent{\em Proof of (a)}: 
	Let us assume that there exists a $y$-arrow $G_{j_2,j_1}$. 
	We note that because of the construction of $\overline{A}$ and $\overline{B}$, there exist $i_1, i_2 \in \{1, \dots, n\}$ such that $Y_{j_1}$ and $Y_{j_2}$ originate from $X_{i_1}$ and $X_{i_2}$ respectively, in the sense of Remark \ref{rmk originate}.
	Then, thanks to Lemma \ref{lem y-arrow2} and Definition \ref{definition partially carried-by}, one can observe that $i_2 \leq K$, or equivalently, $X_{i_2}$ is a part of $A$. 
	Since $Y_{j_2}$ originates from $X_{i_2}$, it means that $Y_{j_2}$ is a part of $\overline{A}$, i.e., $j_2 \leq \overline{K}$.  
	It completes the proof of (a).
	\vskip0.2in
	
	\noindent{\em Proof of (b)}:
	Let us assume that there exists a nonzero chain of arrows satisfying the conditions of (b), i.e., 
	\[Y_{j_0} =\overline{V}_{j_0} \otimes S_w [\overline{d}_{j_0}] \xrightarrow{\overline{\psi}_{j_1,j_0} \otimes y_{w,u}} Y_{j_1} =\overline{V}_{j_1} \otimes S_u [\overline{d}_{j_1}] \xrightarrow{\overline{\psi}_{j_2,j_1} \otimes z_u} \dots \xrightarrow{\overline{\psi}_{j_k,j_{k-1}} \otimes z_u} Y_{j_k} =\overline{V}_{j_k} \otimes S_u [\overline{d}_{j_k}].\]
	We note that $w \in V_-(T)$ and $u \in V_+(T)$ as mentioned in Remark \ref{rmk vertices}.
	For simplicity, we will assume that $\overline{d}_{j_0}=0$. 
	Then, one can easily see that $\overline{d}_{j_m} = m(N-1)-1$ because the degrees of $y$- and $z$-morphisms are $N-1$ and $N$, respectively. 
	
	Now, we prove (b) by an induction on the length of the above chain of arrows, i.e., $k$. 
	For the base step, let us assume that $k=1$, i.e., we have a chain of arrows
	\[Y_{j_0} =\overline{V}_{j_0} \otimes S_w  \xrightarrow{\overline{\psi}_{j_1,j_0} \otimes y_{w,u}} Y_{j_1} =\overline{V}_{j_1} \otimes S_u [N-2].\]
	If $Y_{j_0}$ (resp.\ $Y_{j_1}$) originates from $X_{i_0}$ (resp.\ $X_{i_1}$) in the sense of Remark \ref{rmk originate}, one can easily see that 
	\[X_{i_0} = V_{i_0} \otimes S_w \ \mathrm{and}\
	X_{i_1} = \begin{cases} 
		V_{i_1} \otimes S_u[2N-3] \text{  with  } u_0 =u, \\
		V_{i_1} \otimes S_{w_1} [N-2] \text{  with  } u_0 = u, w_1 \sim u,\\
		V_{i_1} \otimes S_u[N-2] \text{  if  } u_0 \neq u.
	\end{cases}\] 
	
	For each of the three possible $X_{i_1}$, there must exist at least one nonzero chain of arrows from $X_{i_0}$ to $X_{i_1}$. 
	One can easily see that for any cases, $i_1 \leq K$, or equivalently, $X_{i_1}$ is a part of $A$. 
	Since $A$ is simplified, one can apply Lemma \ref{lem y-arrow2} to the $y$-arrow from $Y_{j_0}$ to $Y_{j_1}$ in $\overline{A}$. 
	Then, Lemma \ref{lem y-arrow2} completes the proof of (b) for the case of $k=1$. 
	It completes the proof of the base step. 
	
	We assume the induction hypothesis. 
	Moreover, we also assume that there exists no $j > j_0$ such that for some $2 \leq m \leq k$, there exits a nonzero chain of arrows 
	\begin{gather}
		\label{eq assumption}
		Y_j =\overline{V}_j \otimes S_{w'} [\overline{d}_j] \xrightarrow{\overline{\psi}_{j,j_m} \otimes y_{w',u}} Y_{j_m} =\overline{V}_{j_m} \otimes S_u [\overline{d}_{j_m}] \xrightarrow{\overline{\psi}_{j_{m+1},j_m} \otimes z_u} \dots \xrightarrow{\overline{\psi}_{j_k,j_{k-1}} \otimes z_u} Y_{j_k} =\overline{V}_{j_k} \otimes S_u [\overline{d}_{j_k}].
	\end{gather}
	If such a $j$ exits, then we can replace the original chain of arrows with the above shorter one.
	Then, the induction hypothesis proves that $Y_{j_k}$ is a part of $\overline{A}$. 
	Thus, without loss of generality, we assume that there exists no $j >j_0$ satisfying the above assumption. 
	
	We recall that
	\begin{gather}
		\label{eq Y_j} 
		Y_{j_0} = \overline{V}_{j_0} \otimes S_w, Y_{j_m} = \overline{V}_{j_m} \otimes S_u [m(N-1)-1] \text{  for  all  } m = 1, \dots, k.
	\end{gather}
	See Remark \ref{rmk vertices}.
	Let $Y_{j_m}$ originate from $X_{i_m}$, i.e., $\tau(X_{i_m})$ contains $Y_{j_m}$ even after canceling identity-arrows.
	Since $\tau= \tau_{u_0}$ for some $u_0 \in V_+(T)$ by our assumption, Equations in \eqref{eq Y_j} imply that there exist following two possibilities:
	\begin{enumerate}
		\item[Case (i).] $u_0 =u, X_{i_0} = V_{i_0} \otimes S_w$, and for all $m =1, \dots, k$,
		\[X_{i_m} = \begin{cases}
			V_{i_m} \otimes S_u[(m+1)(N-1)-1], \\
			V_{i_m} \otimes S_{w_m}[m(N-1)-1] \text{  for  a  } w_m \in V_-(T) \text{  such that  } u \sim w_m.
		\end{cases} \]
		\item[Case (ii).] $u_0 \neq u, X_{i_0} = V_{i_0} \otimes S_w$, and $X_{i_m} = V_{i_m} \otimes S_u[m(N-1)-1]$ for all $m =1, \dots, k$. 
	\end{enumerate}
	
	Let us consider the case (i) first.
	For $m \geq 1$, there exist four possible sub-cases given in Table \ref{tab subcases}.
	\begin{table}[h!]
		\caption{Sub-cases of case (i).}
		\label{tab subcases}
		\begin{tabular}{| l | c | c |}
			\hline
			& $X_{i_m} = V_{i_m} \otimes S_u[(m+1)(N-1)-1]$  & $X_{i_m} = V_{i_m} \otimes S_{w_m}[m(N-1)-1]$ \\ \hline
			$X_{i_{m+1}} = V_{i_{m+1}} \otimes S_u[(m+2)(N-1)-1)]$  & \textcircled{1} & \textcircled{3} \\ \hline
			$X_{i_{m+1}} = V_{i_{m+1}} \otimes S_{w_{m+1}}[(m+1)(N-1)-1)]$ & \textcircled{2} & \textcircled{4} \\ \hline
		\end{tabular}
	\end{table}
	
	We prove the case (i) by contradiction. 
	Thus, we assume that $Y_{j_k}$ is not a part of $\overline{A}$.
	Then, we would like to point out that sub-cases \textcircled{3} or \textcircled{4} must happen. 
	If not, one can observe that 
	\begin{gather*}
		X_{i_m} = V_{i_m} \otimes S_u[(m+1)(N-1)-1] \text{  for all  } m = 1, \dots, k-1, \\
		X_{i_k} = V_{i_k} \otimes S_u[(k+1)(N-1)-1] \text{  or  } V_{i_k} \otimes S_{w_k}[k(N-1)-1].
	\end{gather*}
	Moreover, because of the degree reason, the only possible nonzero chain of arrows from $X_{i_m}$ to $X_{i_{m+1}}$ for $m=1, \dots, k-1$ is 
	\begin{itemize}
		\item a chain consisting of one $z_u$-arrow for $m =1, \dots, k-2$, 
		\item a chain consisting of one $z_u$-arrow if $m=k-1$ and $X_{i_{m+1}} = X_{i_k} = V_{i_k} \otimes S_u[(k+1)(N-1)-1]$, and 
		\item a chain consisting of one $x_{u,w_k}$-arrow if $m=k-1$ and $X_{i_{m+1}} = x_{i_k} = V_{i_k} \otimes S_{w_k}[k(N-1)-1]$.
	\end{itemize}
	Similarly, the only possible nonzero chain of arrows from $X_{i_0}$ to $X_{i_1}$ is a $y_{w,u}$-arrow. 
	
	We point out that the nonzero arrow $G_{j_{m+1},j_m}$ forces to have a nonzero chain of arrows from $X_{i_m}$ to $X_{i_{m+1}}$. 
	Thus, we obtain a chain of arrows 
	\begin{gather}
		\label{eq nonzero chain}
		X_{i_0} \xrightarrow{\psi_{i_1,i_0} \otimes y_{w,u}} X_{i_1} \xrightarrow{\psi_{i_2,i_1} \otimes z_u} \dots \xrightarrow{\psi_{i_k,i_{k-1}} \otimes \left(z_u \text{  or  } x_{u,w_k}\right)} X_{i_k}.
	\end{gather}
	
	Finally, one can see that $\overline{V}_{j_m}$ is a quotient of $V_{i_m}$ and $\psi_{i_{m+1},i_m}:V_{i_m} \to V_{i_{m+1}}$ induces a linear map $\overline{\psi}_{j_{m+1},j_m}:\overline{V}_{j_m} \to \overline{V}_{j_{m+1}}$ for all $m =0, \dots, k-1$. 
	Since $\left(G_{j_1,j_0}, \dots, G_{j_k,j_{k-1}}\right)$ satisfies that $\overline{\psi}_{j_1,j_0} \circ \dots \circ \overline{\psi}_{j_k,j_{k-1}}$ is not zero, one can conclude that $\psi_{i_1,i_0} \circ \dots \circ \psi_{i_k,i_{k-1}}$ is also nonzero.
	It implies that $i_k \leq K$, or equivalently, $X_{i_k}$ is a part of $A$, because $E$ is partially carried-by with a triple $(A,B,F)$. 
	Finally, we would like to point out that $Y_{j_k}$ is a part of $\overline{A}$ since $X_{i_k}$ is a part of $A$. 
	In other words, $j_k \leq \overline{K}$, and it gives us a contradiction. 
	
	Let us assume that there exists $m$ such that \textcircled{4} happens. 
	Since $G_{j_{m+1},j_m}$ is not a nonzero arrow, there exists a nonzero chain of arrows from $X_{i_m} = V_{i_m} \otimes S_{w_m}[m(N-1)-1]$ to $X_{i_{m+1}} = V_{i_{m+1}} \otimes S_{w_{m+1}}[(m+1)(N-1)-1)]$.
	We point out that $w_m, w_{m+1} \in V_-(T)$.
	Then, every nonzero arrow from $X_{i_m}$ to $X_{i_{m+1}}$ should either consist of 
	\begin{itemize}
		\item two $y$-arrows and one $x$-arrow with $N=3$, or
		\item one $z$-arrow with $N \geq 3$.
	\end{itemize}
	The first possibility, a chain consisting of two $y$-arrows and one $x$-arrow, is not possible because of Lemma \ref{lem noxyyx}. 
	More precisely, if the first one happens, then there should exist $u' \in V_+(T), w' \in V_-(T)$ such that two consecutive arrows in the nonzero chain of arrows are the form of 
	\[X_a \xrightarrow{x_{u',w'}\text{-arrow}} X_b \xrightarrow{y_{w',u'}\text{-arrow}} X_c, \text{  or  } X_a \xrightarrow{y_{w',u'}\text{-arrow}} X_b \xrightarrow{x_{u',w'}\text{-arrow}} X_c.\]
	We note that because of $y_{w',u'}$ and because $E$ is partially carried-by, $a <b <c \leq K$. 
	Thus, $X_a, X_b$, and $X_c$ are parts of a {\em simplified} twisted complex $A$. 
	Then, it contradicts to Lemma \ref{lem noxyyx}. 
	Thus, $w_m$ and $w_{m+1}$ should agree with each other, and the nonzero arrow from $X_{i_m}$ to $X_{i_{m+1}}$ should be one $z_{w_m=w_{m+1}}$-arrow. 

	Now, we recall that
	\begin{gather*}
		\tau(X_{i_m}) = V_{i_m} \otimes S_u [m(N-1)-1] \to V_{i_m} \otimes S_{w_m}[m(N-1)-1], \\
		\tau(X_{i_{m+1}}) = V_{i_{m+1}} \otimes S_u [(m+1)(N-1)-1] \to V_{i_{m+1}} \otimes S_{w_m}[(m+1)(N-1)-1].
	\end{gather*}
	The $z_u$-arrow $G_{j_{m+1},j_m}$ is obtained from the $V_{i_m} \otimes S_u[m(N-1)-1]$ part of $\tau(X_{i_m})$ to the $V_{i_{m+1}} \otimes S_u[(m+1)(N-1)-1]$ part of $\tau(X_{i_{m+1}})$. 
	
	In order to analyze how to obtain $G_{j_{m+1},j_m}$ from the above two parts, we review the construction of $\tau(E) \simeq \left(\overline{A} \oplus \overline{B}, G\right)$.
	At the first, we use Lemma \ref{lem induced functor} for computing $\tau(E)$, then we cancel as many $z_u$-arrows as possible. 
	Thanks to Lemma \ref{lem z}, the price of canceling $z_u$-arrows is to add $z_w$-arrows. 
	Thus, there exists no $z_u$-arrow from the $(V_{i_m} \otimes S_u [m(N-1)-1])$ part of $\tau(X_{i_m})$ to the $(V_{i_{m+1}} \otimes S_u [(m+1)(N-1)-1])$ part of $\tau(X_{i_{m+1}})$. 
	It is contradicts to that $G_{j_{m+1},j_m}$ is a $z_u$-arrow. 
	Thus, the sub-case \textcircled{4} could not happen.
	
	Now, let us assume that there exists $m \geq 1$ such that the sub-case \textcircled{3} happens, i.e., 
	\[X_{i_m} = V_{i_m} \otimes S_{w_m}[m(N-1)-1] \text{  with  } w_m \sim u, X_{i_{m+1}}= V_{i_{m+1}} \otimes S_u[(m+2)(N-1)-1)].\]
	Similar to the above cases, the existence of a nonzero arrow $G_{j_{m+1},j_m}$ guarantees the existence of nonzero chain of arrows from $X_{i_m}$ to $X_{i_{m+1}}$. 
	Considering the degrees, every nonzero chain of arrows consists of either  
	\begin{itemize}
		\item three $y$-arrows and two $x$-arrows where $N=4$,
		\item four $y$-arrows and three $x$-arrows where $N=3$, or
		\item one $z$-arrow, two $y$-arrows, and one $x$-arrow where $N=3$. 
	\end{itemize}
	One can easily check that the first two items do not happen because of Lemma \ref{lem noxyyx}, as similar to the sub-case \textcircled{4}.
	
	To complete the proof, let us assume that $N =3$ and there exists a nonzero chain of arrows from $X_{i_m}$ to $X_{i_{m+1}}$ consisting of one $x$-arrow, two $y$-arrow, and one $z$-arrow. 
	Every such nonzero chain of arrows is one of the following four types:
	\begin{gather*}
		\text{Type 1:  } X_{i_m}\xrightarrow{\psi_a \otimes z_{w_m}} X_a  \xrightarrow{\psi_b \otimes y_{w_m,u'}} X_b  \xrightarrow{\psi_c \otimes x_{u',w'}} X_c \xrightarrow{\psi_d \otimes y_{w', u}} X_{i_{m+1}}, \\
		\text{Type 2:  } X_{i_m}\xrightarrow{\psi_a \otimes y_{w_m,u'}} X_a  \xrightarrow{\psi_b \otimes z_{u'}} X_b  \xrightarrow{\psi_c \otimes x_{u',w'}} X_c \xrightarrow{\psi_d \otimes y_{w', u}} X_{i_{m+1}}, \\
		\text{Type 3:  } X_{i_m}\xrightarrow{\psi_a \otimes y_{w_m,u'}} X_a  \xrightarrow{\psi_b \otimes x_{u',w'}} X_b  \xrightarrow{\psi_c \otimes z_{w'}} X_c \xrightarrow{\psi_d \otimes y_{w', u}} X_{i_{m+1}}, \\
		\text{Type 4:  } X_{i_m}\xrightarrow{\psi_a \otimes y_{w_m,u'}} X_a  \xrightarrow{\psi_b \otimes x_{u',w'}} X_b  \xrightarrow{\psi_c \otimes y_{w',u}} X_c \xrightarrow{\psi_d \otimes z_u} X_{i_{m+1}}.
	\end{gather*}
	
	We note that chains of arrows types $1$ and $4$ cannot be nonzero, because they contain three consecutive $y_{w_m,u'}, x_{u',w'}, y_{w',u}$-arrows.
	To be more precise, let us note that since $w_m \sim u$ and since $T$ is a tree, either $u' =u$ or $w_m = w'$.
	Then, it contradicts to Lemma \ref{lem noxyyx}. 
	
	We also note that any nonzero chain of arrows of type $3$ cannot {\em contribute} to the chain of arrows 
	\begin{gather}
		\label{eq target chain of arrows}
		Y_{j_{m-1}} \xrightarrow{\overline{\psi}_{j_m,j_{m-1}} \otimes \left(z_u\text{ or  } y_{w,u}\right)} Y_{j_m} \xrightarrow{\overline{\psi}_{j_{m+1},j_m} \otimes z_u} Y_{j_{m+1}},
	\end{gather}
	in the following sense: 
	When one applies $\tau = \tau_u$, a nonzero chain of arrow of type $3$ can contribute to the chain of arrows in \eqref{eq target chain of arrows} only by giving a $z$-arrow from $Y_{j_m}$ to $Y_{j_{m+1}}$.
	Let $\overline{\psi}_{\text{type 3}} \otimes z_u : Y_{j_m} \to Y_{j_{m+1}}$ denote the $z$-arrow obtained from the nonzero chain of arrows of type 3. 
	However, we can show that the composition of two linear maps $\overline{\psi}_{j_m, j_{m-1}}$ and $\overline{\psi}_{\text{type 3}}$ is zero.
	
	Before showing that a nonzero chain of arrows of type 3 cannot have any contribution, let us remark some facts.
	\begin{itemize}
		\item If there exists a nonzero chain of arrows of type 3, one can easily check that $w_m \neq w'$.
		If not, $\psi_b \circ \psi_a$ should be zero because of Lemma \ref{lem noxyyx}.
		Moreover, since $w' \sim u \sim w_m$ and $w' \sim u' \sim w_m$, considering that $T$ is a tree, we have $u = u'$.
		\item On the other hand, the existence of $y$-arrow from $X_c$ to $X_{i_{m+1}}$ implies that $i_{m+1} \leq K$. 
		Thus, for all $i \leq i_{m+1} \leq K$, $X_i$ is a part of a {\em simplified} twisted complex $A$. 
		\item If $m=1$, then $X_{i_{m-1}}= X_{i_0} = V_{i_0} \otimes S_w$.
		Moreover, since there exists a nonzero chain of arrows from $X_{i_0}$ to $X_{i_1} = V_{i_1} \otimes S_{w_1} [N-2 =1]$, one can easily check that $w_1 \neq w$.
		Moreover, the only possible nonzero chain of arrows from $X_{i_0}$ to $X_{i_1}$ should be  of the following form:
		\[X_{i_0} \xrightarrow{\psi_{i_{0.5},i_0}\otimes y_{w,u_1}} X_{i_{0.5}} = V_{i_{0.5}} \otimes S_{u_1}[1] \xrightarrow{\psi_{i_1,i_{0.5}} \otimes x_{u_1,w_1}} X_{i_1},\]
		with $w \sim u_1 \sim w_1$.
		Since $w \sim u_1\sim w_1$ and $w \neq w_1$, $u_1$ should be the same as $u$. 
		Because $A$ is simplified as mentioned in the second fact, the above is the only possible chain of arrows from $X_{i_0}$ to $X_{i_1}$, which can contribute on $\overline{\psi}_{j_1,j_0}$.
		Thus, $\overline{\psi}_{j_1,j_0}$ is determined by $\psi_{i_1,i_{0.5}} \circ \psi_{i_{0.5},i_0}$. 
		\item If $m\geq 2$, then the pair $\left(X_{i_{m-1}, X_{i_m}}\right)$ should satisfy \textcircled{2} or \textcircled{4}. 
		Since \textcircled{4} does not happen, \textcircled{2} should happen between $X_{i_{m-1}}$ and $X_{i_m}$ and $X_{i_{m-1}} = V_{i_{m-1}} \otimes S_u[m(N-1)-1 = 2m-1]$.
		Then, from the degree argument, we can conclude that the only nonzero chain of arrows between them should be 
		\[X_{i_{m-1}} \xrightarrow{\psi_{i_m,i_{m-1}} \otimes x_{u,w_m}} X_{i_m}.\]
		In this case, $\psi_{i_m,i_{m-1}}$ solely determines $\overline{\psi}_{j_m,j_{m-1}}$. 
	\end{itemize} 
	
	Let us consider the case of $m=1$, first. 
	From the first- and third-listed items, we have the following chain of arrows:
	\[X_{i_0} \xrightarrow{\psi_{i_{0.5},i_0}\otimes y_{w,u}} X_{i_{0.5}} \xrightarrow{\psi_{i_1,i_{0.5}} \otimes x_{u,w_1}} X_{i_1} \xrightarrow {\psi_a \otimes y_{w_1,u}} X_a.\]
	Then, by Lemma \ref{lem noxyyx}, $\psi_a \circ \psi_{i_1,i_{0.5}} =0$. 
	Since $\overline{\psi}_{type 3}$ is determined by $\psi_d \circ \psi_c \circ \psi_b \circ \psi_a$ and $\overline{\psi}_{j_1,j_0}$ is determined by $\psi_{i_1,i_{0.5}} \circ \psi_{i_{0.5},i_0}$, one can easily check that $\overline{\psi}_{type 3} \circ \overline{\psi}_{j_1,j_0} =0$. 
	
	Let us consider the case of $m\geq 2$. 
	In the fourth-listed item, we observed that the only nonzero chain of arrows between $X_{i_m-1}$ and $X_{i_m}$ is the $x_{u,w_m}$-arrow. 
	Then, we have 
	\[X_{i_{m-1}} \xrightarrow{\psi_{i_m,i_{m-1}} \otimes x_{u,w_m}} X_{i_m} \xrightarrow {\psi_a \otimes y_{w_m,u}} X_a,\]
	and by Lemma \ref{lem noxyyx}, $\psi_a \circ \psi_{i_m,i_{m-1}} =0$. 
	As same as the case of $m=1$, $\overline{\psi}_{type 3} \circ \overline{\psi}_{j_m,j_{m-1}} =0$. 
	
	We note that $\overline{\psi}_{j_{m+1},j_m} \circ \overline{\psi}_{j_m,j_{m-1}} \neq 0$ because there exists a chain of arrows 
	\[Y_{j_0} \xrightarrow{\overline{\psi}_{j_1,j_0} \otimes y_{w,u}} Y_{j_1} \xrightarrow{\overline{\psi}_{j_2,j_1} \otimes z_u} \dots \xrightarrow{\overline{\psi}_{j_k,j_{k-1}} \otimes z_u} Y_{j_k},\]
	satisfying the conditions of (b). 
	Thus, there exists a nonzero chain of arrows of type $2$.
	
	However, if there exists a nonzero chain of arrows of type $2$, then 
	\[X_c \to X_{i_{m+1}} \to \dots \to X_{i_k}\]
	induces a nonzero chain of arrows in \eqref{eq assumption}. 
	Thus, it contradicts to the assumption that there exists no $j$ satisfying \eqref{eq assumption}. 
	It proves that \textcircled{3} cannot happen even if $N=3$.
	And it completes the proof of the case (i) of (b).
	
	To prove the part (b) for the case (ii), let us recall that, in the case (ii),
	\begin{itemize}
		\item $\tau = \tau_{u_0}, u_0 \neq u, X_{i_0} = V_{i_0} \otimes S_w$, and $X_{i_m} = V_{i_m} \otimes S_u[m(N-1)-1]$ for all $m =1, \dots, k$. 
	\end{itemize}
	Because of the degree reason, the only possible nonzero chain of arrows from $X_{i_0}$ to $X_{i_1}$ consists of one $y_{w,u}$-arrow. 
	Similarly, by the same degree reason, possible nonzero chains of arrows from $X_{i_m}$ to $X_{i_{m+1}}$, for $m = 1, \dots, k-1$, are of the one of the following form: 
	\begin{itemize}
		\item A chain of arrows consisting of one $z_u$-arrow.
		\item When $N=3$, a chain of arrows consisting of two $y$-arrows and two $x$-arrows.
	\end{itemize}
	We observe that by Lemma \ref{lem noxyyx}, the second-listed chain of arrows cannot exist.
	
	From the above argument, we conclude that there exists a chain of arrows from $X_{i_0}$ to $X_{i_k}$
	\[X_{i_0} \xrightarrow{\psi_{i_1,i_0} \otimes y_{w,u}} X_{i_1} \xrightarrow{\psi_{i_2,i_1} \otimes z_u} \dots \xrightarrow{\psi_{i_k,i_{k-1}} \otimes z_u} X_{i_k}.\]
	Moreover, since $\psi_{i_m,i_{m-1}}$ solely determines $\overline{\psi}_{j_m,j_{m-1}}$ for all $m = 1, \dots, k$, the condition of (b), i.e., $\overline{\psi}_{j_k,j_{k-1}} \circ \dots \circ \overline{\psi}_{j_1,j_0} \neq 0$ implies that $\psi_{i_k,i_{k-1}} \circ \dots \circ \psi_{i_1,i_0} \neq 0$.
	
	We recall that $E$ is partially carried-by.
	It implies that $i_k \leq K$, or equivalently, $X_{i_k}$ is a part of $A$. 
	Then, $Y_{j_k}$, which is a part of $\tau(X_{i_k})$, must be a part of $\overline{A}$. 
	It proves (b), case (ii).
	\vskip0.2in
	
	\noindent{Proof of (c)}:
	Under the same assumption as the proof of (b), let us assume that
	\begin{gather*}
		Y_{j_0} = \overline{V}_{j_0} \otimes S_w\xrightarrow{\overline{\psi}_{j_1,j_0} \otimes y_{w,u}} Y_{j_1}=\overline{V}_{j_1} \otimes S_u[N-2] \xrightarrow{\overline{\psi}_{j_2,j_1} \otimes z_u} \dots \xrightarrow{\overline{\psi}_{j_k,j_{k-1}} \otimes z_u} 
		Y_{j_k}=\overline{V}_{j_k} \otimes S_u[k(N-1)-1]  \\ \xrightarrow{\overline{\psi}_{j_{k+1},j_k} \otimes x_{u, w'}} Y_{j_{k+1}}=\overline{V}_{j_{k+1}} \otimes S_{w'}[k(N-1)-1]
	\end{gather*}
	is a chain of arrows satisfying the conditions of (c).
	As we did in the proof of (b), we say that $Y_{j_m}$ is obtained from $\tau(X_{i_m})$. 
	
	We recall that $\tau = \tau_{u_0}$ in the proof. 
	Thus, $X_{i_0}$ and $X_{i_{k+1}}$ should satisfy
	\[X_{i_0} = V_{i_0} \otimes S_w, X_{i_{k+1}} = V_{i_{k+1}} \otimes S_{w'}[k(N-1)-1].\]
	For $X_{i_m}$ with $m = 1, \dots, k$, similar to (b), we need to consider the following two cases:
	\begin{enumerate}
		\item[Case (i).] $u_0 =u, X_{i_0} = V_{i_0} \otimes S_w$, and for all $m =1, \dots, k$,
		\[X_{i_m} = \begin{cases}
			V_{i_m} \otimes S_u[(m+1)(N-1)-1], \\
			V_{i_m} \otimes S_{w_m}[m(N-1)-1] \text{  for  a  } w_m \in V_-(T) \text{  such that  } u \sim w_m.
		\end{cases} \]
		\item[Case (ii).] $u_0 \neq u, X_{i_0} = V_{i_0} \otimes S_w$, and $X_{i_m} = V_{i_m} \otimes S_u[m(N-1)-1]$ for all $m =1, \dots, k$. 
	\end{enumerate}
	
	First, let us consider the case (i). 
	We note that since $\tau= \tau_u$, $i_k$ and $i_{k+1}$ could be the same, or equivalently, $Y_{j_k}$ and $Y_{j_{k+1}}$ are obtained from the same $X_{i_k}$. 
	If $i_k = i_{k+1}$, since $i_k \leq K$ by the proof of (b), it implies that $Y_{j_{k+1}}$ is a part of $\overline{A}$. 
	In other words, (c) holds for the case. 
	Thus, we assume that $i_k \neq i_{k+1}$ in order to complete the proof. 
	
	We note that in the case (i), $X_{i_k}$ is either $V_{i_k} \otimes S_u[(k+1)(N-1)-1]$ or $V_{i_k} \otimes S_{w_k}[k(N-1)-1]$. 
	Since $X_{i_{k+1}} = V_{i_{k+1}} \otimes S_{w'}[k(N-1)-1]$, for both cases of $X_{i_k} = V_{i_k} \otimes S_u[(k+1)(N-1)-1]$ or $V_{i_k} \otimes S_{w_k}[k(N-1)-1]$, 
	one can observe that there exists no possible nonzero chain of arrows from $X_{i_k}$ to $X_{i_{k+1}}$ by the degree reason.
	It is a contradict to the existence of nonzero $x$-arrow from $Y_{j_k}$ to $Y_{j_{k+1}}$, and thus the assumption $i_k \neq i_{k+1}$ cannot be true. 
	It completes the proof of (c) for the case (i). 
	
	Let us assume the case (ii), especially, $u_0 \neq u$. 
	Then, from the proof of the case (ii) of (b), we have a nonzero chain of arrows 
	\[X_{i_0} \xrightarrow{\psi_{i_1,i_0} \otimes y_{w,u}} X_{i_1} \xrightarrow{\psi_{i_2,i_1} \otimes z_u} \dots \xrightarrow{\psi_{i_k,i_{k-1}} \otimes z_u} X_{i_k},\]
	i.e., $\psi_{i_k,i_{k-1}} \circ \dots \circ \psi_{i_1,i_0} \neq 0$.
	Moreover, since $X_{i_k} =  V_{i_k} \otimes S_u[k(N-1)-1]$ and $X_{i_{k+1}} = V_{i_{k+1}} \otimes S_{w'}[k(N-1)-1]$, the only possible nonzero chain of arrows from $X_{i_k}$ to $X_{i_{k+1}}$ consists of one $x_{u,w'}$-arrow by the degree reason. 
	It extends the above chain of arrows to 
	\[X_{i_0} \xrightarrow{\psi_{i_1,i_0} \otimes y_{w,u}} X_{i_1} \xrightarrow{\psi_{i_2,i_1} \otimes z_u} \dots \xrightarrow{\psi_{i_{k+1},i_k} \otimes z_u} X_{i_{k+1}}.\]
	Moreover, similar to the proof of (b), case (ii), we can conclude that $\psi_{i_{k+1},i_k} \circ \dots \circ \psi_{i_1,i_0} \neq 0$.
	It implies that $X_{i_{k+1}}$ is a part of $A$, and thus, $Y_{j_{k+1}}$ is a part of $\overline{A}$. 
	It completes the proof of (c), the case (ii).
\end{proof}

\subsection{Step 5 of the sketch of proof of Theorem \ref{thm pseudo-Anosov}}
\label{subsection step 5}
Now, we prove Step 5 of Section \ref{subsection sketch of the proof}.
In other words, we prove Lemma \ref{lem partially carried-by}

\begin{lem}
	\label{lem partially carried-by}
	Let $T$ be a tree and $\phi$ be an autoequivalence of Penner type. 
	Moreover, we assume that $\Phi$ is a product of 
	\[\left\{\tau_u, \tau_w^{-1} | u \in V_+(T), w \in V_-(T)\right\}.\]
	Then, for any object $E_0 \in \Fuk$, there exists a natural number $N \in \mathbb{N}$ such that for any $n \geq N$, $\phi^n E_0$ is partially carried-by. 
\end{lem}
\begin{proof}[Sketch of proof]
	We first give a sketch of proof, then the full proof will be given after proving Lemmas \ref{lem no outside effect}--\ref{lem effect on C_3} whose necessity will be justified by the sketch. 
	
	Let $E_0$ be an object of $\Fuk$.
	We recall that by Lemma \ref{lem y-rank}, there exists a natural number $N$ such that for any $n \geq N$, 
	\[r(\Phi^n E_0) = r(\Phi^N E_0).\]
	Thus, it is enough to prove Lemma \ref{lem partially carried-by} for $E_0$ such that 
	\[r(E_0) = r\left(\Phi^n E_0\right),\]
	for any $n \in \mathbb{N}$. 
	And, in the sketch, we let $\Phi^n E_0$ denote the minimal, simplified twisted complex equivalent to $\Phi^n E_0$. 
	
	We note that in order to prove that an object $E$ is partially carried-by, one needs to consider chains of arrows of the following two forms: 
\begin{align*}
	&X_{i_0} \xrightarrow{\psi_{i_1,i_0}\otimes y_{v_{i_0}, v_{i_1}}} X_{i_1} \xrightarrow{\psi_{i_2,i_1} \otimes z_{v_{i_1}}} X_{i_2} \xrightarrow{\psi_{i_3,i_2} \otimes z_{v_{i_2}}} \dots \xrightarrow{\psi_{i_k,i_{k-1}} \otimes z_{v_{i_{k-1}} }} X_{i_{k}},\\
&	X_{i_0} \xrightarrow{\psi_{i_1,i_0}\otimes y_{v_{i_0}, v_{i_1}}} X_{i_1} \xrightarrow{\psi_{i_2,i_1} \otimes z_{v_{i_1}}} X_{i_2} \xrightarrow{\psi_{i_3,i_2} \otimes z_{v_{i_2}}} \dots \xrightarrow{\psi_{i_k,i_{k-1}} \otimes z_{v_{i_{k-1}} }} X_{i_{k}} \xrightarrow{\psi_{i_{k+1},i_k} \otimes x_{v_{i_k},v_{i_{k+1}}}} X_{i_{k+1}},
\end{align*}
i.e., one form (the upper one) given in \eqref{eqn partially carried-by condition 1}, and the other form (the lower form) given in \eqref{eqn partially carried-by condition 2}. 
For convenience, in the rest of Section \ref{subsection step 5}, let the first and second forms mean the forms in \eqref{eqn partially carried-by condition 1} and \eqref{eqn partially carried-by condition 2}, respectively.

We first note that both of the above two forms start with a $y$-arrow, and we also note that because of the assumption that $r(E_0) = r\left(\Phi^n E_0\right)$, there exists one to one relations between the sets of $y$-arrows. 
From those two facts we noted, we will show that for a sufficiently large $N \in \mathbb{N}$, every $y$-arrow in $\Phi^N E$ admits a chain of arrows of the second form, starting with the given $y$-arrow. 
It is the main part of the proof and we will prove the main part with Lemmas \ref{lem no outside effect}--\ref{lem effect on C_3}.

After that, we consider the set of $x$-arrows that are the last arrows of chains of arrows of the second form.
Then, there exists the last one in the set, and we divide $\Phi^N E_0$ into two parts: one that comes before the last $x$-arrow, and the other one that comes after the last $x$-arrow. 
If $A$ and $B$ denote the parts coming before and after the last $x$-arrow, then $\Phi^N E_0$ is partially carried-by with a triple $\left(A, B, F\right)$. 
It will complete the proof. 
\end{proof}

Our strategy is to observe the effect of $\tau \in \left\{\tau_u, \tau_w^{-1} | u \in V_+(T), w \in V_-(T)\right\}$ on chains of arrows of the first and second forms. 
To do that, let us recall that by Lemma \ref{lem induced functor}, when we apply a functor $\tau$ to a twisted complex $E = \left(\oplus_{i \in I} X_i, F_{j,i}\right)$, we have a twisted complex
\[\tau(E) \simeq \left(\oplus_{i \in I} \tau(X_i), G_{j,i}\right).\]
And, $G_{j,i}$ is determined by $\tau$ and nonzero chain of arrows from $X_i$ to $X_j$. 
Based on this, we prove Lemma \ref{lem no outside effect}.

\begin{lem}
	\label{lem no outside effect}
	Let $E_0 = \left(\bigoplus_{i =1}^{n_0}X_i, F_{j,i}\right)$ be a minimal, simplified twisted complex, $\tau \in \left\{\tau_u, \tau_w^{-1} | u \in V_+(T), w \in V_-(T)\right\}$, and let 
	\[\tau(E_0) =\left(\oplus_{i=1}^{n_0} \tau(X_i), G_{j,i}\right)\]
	be the twisted complex obtained from Lemma \ref{lem induced functor}.
	Let us assume that there exist nonzero chains of arrows of the first and second forms in $E_0$, i.e., chains of arrows satisfying either one of the following three conditions: 
	\begin{enumerate}
		\item[(a)] The nonzero chain of arrows of the first form, i.e., 
		\begin{gather*}
			C_1 = X_{i_0} \xrightarrow{\psi_{i_1,i_0}\otimes y_{w,u}} X_{i_1} \xrightarrow{\psi_{i_2,i_1} \otimes z_u} X_{i_2} \xrightarrow{\psi_{i_3,i_2} \otimes z_u} \dots \xrightarrow{\psi_{i_k,i_{k-1}} \otimes z_u} X_{i_{k}}.
		\end{gather*}
		\item[(b)] The nonzero chain of arrows of the second form having at least one $z$-arrow, i.e., for a $k \geq 2$,
		\begin{gather*}
			C_2 = X_{i_0} \xrightarrow{\psi_{i_1,i_0}\otimes y_{w,u}} X_{i_1} \xrightarrow{\psi_{i_2,i_1} \otimes z_u} X_{i_2} \xrightarrow{\psi_{i_3,i_2} \otimes z_u} \dots \xrightarrow{\psi_{i_k,i_{k-1}} \otimes z_u} X_{i_{k}} \xrightarrow{\psi_{i_{k+1},i_k} \otimes x_{u,w'}} X_{i_{k+1}}.
		\end{gather*}
		\item[(c)] The nonzero chain of arrows of the second form having no $z$-arrow, i.e., 
		\begin{gather*}
			C_3 = X_{i_0} \xrightarrow{\psi_{i_1,i_0}\otimes y_{w,u}} X_{i_1} \xrightarrow{\psi_{i_2,i_1} \otimes x_{u,w'}} X_{i_2},	
		\end{gather*}
	\end{enumerate}
	Then, for any $i_m, i_\ell$, there exists no chain of arrows outside of $C_i$ contributing to $G_{i_\ell, i_m}$.
	Moreover, for the cases of (a) or (b), if $\ell - m \geq 2$, then $G_{i_\ell,i_m}=0$. 
\end{lem}
\begin{proof}
	Let us first consider the case (a), i.e., the twisted complex $C_1$ of the first form. 
	One can easily see that 
	\[X_{i_0} = V_{i_0} \otimes S_w[d_{i_0}]\ \mathrm{and}\ X_{i_s} = V_{i_s} \otimes S_u[d_{i_0} + s(N-1)-1].\]
	See Remark \ref{rmk vertices}.
	For convenience, we assume that $d_{i_0} = 0$ without loss of generality.
	
	We fix a pair of indices $(i_m < i_\ell)$, and we will show that there exists no chain of arrows contributing to $G_{i_\ell, i_m}$ outside of $C_1$. 
	If $m \geq 1$, then we have 
	\[X_{i_m} = V_{i_m} \otimes S_u[m(N-1)-1] \ \mathrm{and}\ X_{i_\ell} = V_{i_\ell} \otimes S_u[\ell(N-1)-1].\]
	
	Let $\left(F_{j_1, j_0 = i_m}, F_{j_2, j_1}, \dots, F_{j_s = i_\ell, j_{s-1}}\right)$ be a chain of arrows of $E_0$, which is contributing to $G_{i_\ell, i_m}$. 
	Equivalently, 
	\[\tau^s\left(F_{j_s, j_{s-1}}, \dots, F_{j_2, j_1}, F_{j_1, j_0}\right) = \psi_{j_s,j_{s-1}} \circ \dots \circ \psi_{j_1,j_0} \cdot \tau^s \left(f_{j_s, j_{s-1}}, \dots, f_{j_1, j_0}\right) \neq 0.\]
	Here, we are considering $\tau$ as an $\mathcal{A}_\infty$-functor, thus $\tau^s$ is a multilinear map of order $\left(1-s\right)$.
	For more details, see \cite[Chapters (1d) and (3m)]{Seidel08}.
	Thus, we have 
	\[\tau^s \left(f_{j_s, j_{s-1}}, \dots, f_{j_1, j_0}\right) \in \shom_{\wrapped}^{1-s+\sum_{i=1}^s|f_{j_i,j_{i-1}}|} \left(S_u[m(N-1)-1], S_u[\ell(N-1)-1]\right).\]

	We note that by the definition of twisted complex, every arrow is of degree $1$. 
	Then, we have 
	\begin{gather*}
		0 \neq \tau^s \left(f_{j_s, j_{s-1}}, \dots, f_{j_1, j_0}\right) \in \shom_{\wrapped}^{1 +(\ell-m)(N-1)}\left(S_u,S_u\right).
	\end{gather*}

	From Lemma \ref{lem grading}, we can check that 
	\[\shom_{\wrapped}^{1 +(\ell-m)(N-1)}\left(S_u,S_u\right) \neq 0 \Longleftrightarrow 1 +(\ell-m)(N-1) = 0 \text{  or  } N.\]
	Thus, $\ell -m =1$. 
	Moreover, from the fact that $E_0$ is minimal, simplified, one can easily see that there exists only one nonzero chain of arrows from $X_{i_m}$ to $X_{i_\ell}$, which is a part of $C_1$. 
	For more details on it, see the degree arguments in the proof of Lemma \ref{lem y-arrow2}.
	We also note that it proves that if $\ell - m \geq 2$, then $G_{i_\ell, i_m} =0$. 
	
	In order to complete the case of $C_1$, let us assume that $i_m=0$. 
	Similar to before, if a chain of arrow $\left(F_{j_1, j_0 = i_m}, F_{j_2, j_1}, \dots, F_{j_s = i_\ell, j_{s-1}}\right)$ contributes to $G_{i_\ell, i_m}$, we have 
	\begin{align*}
		0 \neq \tau^s \left(f_{j_s, j_{s-1}}, \dots, f_{j_1, j_0}\right) \in \shom_{\wrapped}^1 \left(S_w, S_u[\ell(N-1)-1]\right) = \shom_{\wrapped}^{\ell(N-1)}\left(S_w,S_u\right).
	\end{align*}

	Again, Lemma \ref{lem grading} induces that if $\ell \geq 2$, then $G_{i_\ell,i_m=0} =0$.
	Thus, $\ell =1$. 
	Then, one can easily see that the chain of arrows should be the one consisting $F_{i_1,i_0}$, which is a part of $C_1$. 
	
	Now, we consider the case (b), i.e., the twisted complex $C_2$ of the second form having at least one $z$-arrow.
	Similar to the first case, we fix a pair of indices $(i_m < i_\ell)$, then will show that the chains of arrows contributing to $G_{i_\ell,i_m}$ are should be parts of $C_2$. 
	
	It is easy to observe that if $\ell \leq k$, then the proof of the first case is enough.
	Thus, let us assume that $\ell =k+1$, and accordingly, $X_{i_\ell} = V_{i_{k+1}} \otimes S_{w'}[d_0 + k(N-1)-1]$.
	Then, if we assume that $\left(F_{j_1, j_0 = i_m}, F_{j_2, j_1}, \dots, F_{j_s = i_\ell, j_{s-1}}\right)$ is a chain of arrows contributing to $G_{i_\ell, i_m}$, 
	\begin{gather*}
		0 \neq \tau^s \left(f_{j_s, j_{s-1}}, \dots, f_{j_1, j_0}\right) \in \begin{cases}
			\shom_{\wrapped}^{1 +(k -m)(N-1)}\left(S_u, S_{w'}\right) \text{  if  } m \geq 1, \\
			\shom_{\wrapped}^{1 +k(N-1)}\left(S_w, S_{w'}\right) \text{  if  } m =0.
		\end{cases}
	\end{gather*}
	
	From Lemma \ref{lem grading}, one can observe that if $m \geq 1$, $k=m$ because
	\[\shom_{\wrapped}^{1 +(k -m)(N-1)}\left(S_u, S_{w'}\right) \neq 0.\]
	Then, the given pair of indices is $(i_m = i_k, i_\ell = i_{k+1})$. 
	It is easy to check that the only nonzero chain of arrows from $X_{i_k}$ to $X_{i_{k+1}}$ is the one given in $C_2$.
	If $m =0$, then $w = w'$ and $k=1$ because 
	\[\shom_{\wrapped}^{1 +k(N-1)}\left(S_w, S_{w'}\right) \neq 0.\]
	Similar to the above case, it is easy to check that Lemma \ref{lem no outside effect} holds for the case (b). 
	
	Finally, let us consider the case (c), i.e., the chain of arrows $C_2$ of the second form having no $z$-arrow.
	Because the chain of arrows consists of only two arrows, we only need to consider three arrows $G_{i_1,i_0}, G_{i_2,i_1}$, and $G_{i_2,i_0}$. 
	From the degree argument we used above, we can easily see that all of $G_{i_\ell,i_m}$ for $0\leq m < \ell \leq 2$ are determined only by arrows in $C_3$.
\end{proof}

Let $\tau \in \left\{\tau_u, \tau_w^{-1} | u \in V_+(T), w \in V_-(T)\right\}$, and let $C_1, C_2$, and $C_3$ be the chains of arrows given in Lemma \ref{lem no outside effect} (a)--(c), respectively.
Then, it is easy to check that $C_i$ forms a twisted complex.
Now, by applying Lemma \ref{lem induced functor} to $C_i$, one can compute a twisted complex $\tau(C_i)$. 
Then, thanks to Lemma \ref{lem no outside effect}, $\tau(C_i)$ appears as a part of $\tau(E)$ if $E$ contains a chain of arrows $C_i$. 

The following Lemmas \ref{lem effect on C_1}--\ref{lem effect on C_3} summarize the effects of $\tau$ on $C_1, C_2$, and $C_3$. 

\begin{lem}
	\label{lem effect on C_1} 
	Let $C_1$ be a good twisted complex of the first form, i.e.,
	\begin{gather*}
		C_1 = X_{i_0} \xrightarrow{\psi_{i_1,i_0}\otimes y_{w,u}} X_{i_1} \xrightarrow{\psi_{i_2,i_1} \otimes z_u} X_{i_2} \xrightarrow{\psi_{i_3,i_2} \otimes z_u} \dots \xrightarrow{\psi_{i_k,i_{k-1}} \otimes z_u} X_{i_{k}},
	\end{gather*}
	where $X_{i_0} = V_{i_0} \otimes S_w[d_{i_0}], X_{i_s} = V_{i_s} \otimes S_u [d_{i_0} + s (N-1)-1]$ for all $1 \leq s \leq k$.
	Then, the following hold:
	\begin{enumerate}
		\item If $\tau= \tau_u$ and $k=1$, then $\tau(C_1)$ is quasi-isomorphic to
		\[ 
		\begin{tikzcd}[arrow style=math font,cells={nodes={text height=2ex,text depth=0.75ex}}]
			\big( V_{i_0} \otimes S_u[d_{i_0}] \arrow{r}[swap]{\mathrm{Id} \otimes x_{u,w}} \arrow[bend left =15]{rr}{\psi_{i_1,i_0} \otimes e_u} & V_{i_0} \otimes S_w[d_{i_0}] \big) \arrow{r}[swap]{0} & V_{i_1}\otimes S_u[d_{i_0}-1].
		\end{tikzcd}
		\]
		\item If $\tau=\tau_u$ and $k\geq 2$, then $\tau(C_1)$ is quasi-isomorphic to 
		\begin{equation*}
			\begin{tikzcd}[arrow style=math font,cells={nodes={text height=2ex,text depth=0.75ex}}]
			\big( V_{i_0} \otimes S_u[d_{i_0}] \arrow{r}[swap]{\mathrm{Id}\ \otimes x_{u,w}} \arrow[bend left=15]{rr}{\psi_{i_1,i_0} \otimes e_u}& V_{i_0} \otimes S_w[d_{i_0}] \big) \arrow{r}{0} \arrow[bend right =15]{rr}{(\psi_{i_2,i_1} \circ \psi_{i_1,i_0}) \otimes y_{w,u}} & V_{i_1}\otimes S_u[d_{i_0}-1] \arrow{r}{\psi_{i_3,i_2} \otimes z_u} &  V_{i_2} \otimes S_u[d_{i_0}+(N-1)-1] \arrow{r}{\psi_{i_4,i_3} \otimes z_u} & ~  \\
			& \dots \arrow{r}{\psi_{i_k,i_{k-1}} \otimes z_u} & V_{i_k} \otimes S_u[d_{i_0}+(k-1)(N-1)-1]. & & 
			\end{tikzcd}
		\end{equation*}
		\item If $\tau = \tau_{u'}$ such that $u \neq u' \in V_+(T), u' \sim w$, $\tau(C_1)$ is quasi-isomorphic to
		\begin{gather*}
			\big( V_{i_0} \otimes S_{u'}[d_{i_0}] \xrightarrow{\mathrm{Id}\ \otimes x_{u',w}} V_{i_0} \otimes S_w[d_{i_0}] \big) \xrightarrow{\psi_{i_1,i_0} \otimes y_{w,u}}  V_{i_1}\otimes S_u[d_{i_0}+(N-1)-1] \xrightarrow{\psi_{i_2,i_1} \otimes z_u} \\ 
			\dots \xrightarrow{\psi_{i_k,i_{k-1}} \otimes z_u}  V_{i_k} \otimes S_u[d_{i_0}+k(N-1)-1].
		\end{gather*}
		\item If $\tau = \tau_{u'}$ such that $u \neq u' \in V_+(T), u' \nsim w$, $\tau(C_1)$ is quasi-isomorphic to
		\begin{gather*}
			V_{i_0} \otimes S_w[d_{i_0}] \xrightarrow{\psi_{i_1,i_0} \otimes y_{w,u}}  V_{i_1}\otimes S_u[d_{i_0}+(N-1)-1] \xrightarrow{\psi_{i_2,i_1} \otimes z_u} \dots \xrightarrow{\psi_{i_k,i_{k-1}} \otimes z_u}  V_{i_k} \otimes S_u[d_{i_0}+k(N-1)-1].
		\end{gather*}	
		\item If $\tau = \tau_w^{-1}$, then $\tau(C_1)$ is quasi-isomorphic to 
		\[
		\begin{tikzcd}[arrow style=math font,cells={nodes={text height=2ex,text depth=0.75ex}},column sep = 1em]
			V_{i_0} \otimes S_w[d_{i_0}+N-1] \arrow{r}{0} \arrow[bend left=15]{rr}{\psi_{i_1,i_0} \otimes e_w} &\big( V_{i_1}\otimes S_u[d_{i_0}+(N-1)-1] \arrow{r}{\mathrm{Id}\ \otimes x_{u,w}}  \arrow[bend right=15]{rr}{ \psi_{i_2,i_1} \otimes z_u}  & V_{i_1}\otimes S_w[d_{i_0}+(N-1)-1]\big) \arrow{r}{0} & ~ \\
			\big( V_{i_2} \otimes 	S_u[d_{i_0}+2(N-1)-1] \arrow{r}{\mathrm{Id}\ \otimes x_{u,w}} & V_{i_2}\otimes S_w[d_{i_0}+2(N-1)-1] \big) \arrow{r}{\psi_{i_3,i_2} \otimes z_u}& \dots \arrow{r}{\psi_{i_k,i_{k-1}}\otimes z_u}& ~ \\ 
			\big(V_{i_k} \otimes S_u[d_{i_0}+k(N-1)-1]\big) \arrow{r}{\mathrm{Id}\ \otimes x_{u,w}}& V_{i_k} \otimes S_w[d_{i_0}+k(N-1)-1] \big).& & & 
		\end{tikzcd}
		\]
		\item If $\tau = \tau_{w'}^{-1}$ such that $w \neq w' \in V_-(T), w' \sim u$, then $\tau(C_1)$ is quasi-isomorphic to
		\begin{gather*}
			V_{i_0} \otimes S_w[d_{i_0}]  \xrightarrow{\psi_{i_1,i_0} \otimes y_{w,u}} \big( V_{i_1}\otimes S_u[d_{i_0}+(N-1)-1] \xrightarrow{\mathrm{Id}\ \otimes x_{u,w'}} V_{i_1}\otimes S_{w'}[d_{i_0}+(N-1)-1]\big) \xrightarrow{ \psi_{i_2,i_1} \otimes z_u} \dots  \\
			\xrightarrow{\psi_{i_k,i_{k-1}}\otimes z_u} \big(V_{i_k} \otimes S_u[d_{i_0}+k(N-1)-1] \xrightarrow{\mathrm{Id}\ \otimes x_{u,w'}} V_{i_k} \otimes S_{w'}[d_{i_0}+k(N-1)-1] \big).
		\end{gather*}
		\item If $\tau = \tau_{w'}^{-1}$ such that $w \neq w' \in V_-(T), w' \sim u$, then $\tau(C_1)$ is quasi-isomorphic to
		\begin{gather*}
			V_{i_0} \otimes S_w[d_{i_0}]  \xrightarrow{\psi_{i_1,i_0} \otimes y_{w,u}} V_{i_1}\otimes S_u[d_{i_0}+(N-1)-1] \xrightarrow{ \psi_{i_2,i_1} \otimes z_u} \dots  \xrightarrow{\psi_{i_k,i_{k-1}}\otimes z_u} V_{i_k} \otimes S_u[d_{i_0}+k(N-1)-1].
		\end{gather*}
	\end{enumerate}
\end{lem}

\begin{proof}
	Lemma \ref{lem effect on C_1} is a simple corollary of Lemmas \ref{lem induced functor} and \ref{lem simple computation}.
\end{proof}

We note that $C_2$ in Lemma \ref{lem no outside effect} (b) is a chain of arrows starting with a $y$-arrow, continuing with $k$-many $z$-arrows, and ending with an $x$-arrow. 
Thus, $C_2$ can be obtained by adding one nonzero arrow $F_{i_{k+1},i_k} : X_{i_k} \to X_{i_{k+1}}$ to $C_1$ in Lemma \ref{lem no outside effect} (a). 
Thus, $\tau(C_2)$ can be obtained by adding $\tau\left(X_{i_k} \oplus X_{i_{k+1}}, F_{i_{k+1},i_k}\right)$ to $\tau(C_1)$ computed in Lemma \ref{lem effect on C_1}.
See Lemmas \ref{lem effect on C_2} for the computation of $\tau(C_2)$. 

\begin{lem}
	\label{lem effect on C_2}
	Let $C_2$ be a good twisted complex of the second form, i.e.,  	
	\[C_2 = X_{i_0} \xrightarrow{\psi_{i_1,i_0}\otimes y_{w,u}} X_{i_1} \xrightarrow{\psi_{i_2,i_1} \otimes z_u} X_{i_2} \xrightarrow{\psi_{i_3,i_2} \otimes z_u} \dots \xrightarrow{\psi_{i_k,i_{k-1}} \otimes z_u} X_{i_{k}} \xrightarrow{\psi_{i_{k+1},i_k} \otimes x_{u,w'}} X_{i_{k+1}},\]
	where 
	\[X_{i_0} = V_{i_0} \otimes S_w[d_{i_0}], X_{i_s} = V_{i_s} \otimes S_u [d_{i_0} + s (N-1)-1] \text{  for all  } 1 \leq s \leq k, \text{  and  } X_{i_{k+1}} = V_{i_{k+1}} \otimes S_{w'}[d_{i_0}+k(N-1)-1].\]
	If $k \geq 1$, $C_2$ is obtained by adding $X_{i_k}\xrightarrow{F_{i_{k+1}},i_k} X_{i_{k+1}}$ to $C_1$ and $\tau(C_1)$ is computed in Lemma \ref{lem effect on C_1}.
	Then, $\tau(C_2)$ is obtained by adding $\tau\left(X_{i_k}\oplus X_{i_{k+1}}, F_{i_{k+1},i_k}\right)$ to $\tau(C_1)$, which is equivalent to the following:
	\begin{enumerate}
		\item If $\tau = \tau_u$, 
		\[V_{i_k}\otimes S_u[d_{i_0}+(k-1)(N-1)-1] \xrightarrow{\psi_{i_{k+1},i_k} \otimes z_u} V_{i_{k+1}} \otimes S_u[d_{i_0}+k(N-1)-1] \xrightarrow{Id \otimes x_{u,w'}} V_{i_{k+1}} \otimes S_{w'}[d_{i_0}+k(N-1)-1].\]
		\item If $\tau = \tau_{u'}$ such that $u \neq u' \sim w'$, 
		\[\begin{tikzcd}[arrow style=math font,cells={nodes={text height=2ex,text depth=0.75ex}},column sep = 1em]
			V_{i_k}\otimes S_u[d_{i_0}+k(N-1)-1] \arrow{r}{0} \arrow[bend left=15]{rr}{\psi_{i_{k+1},i_k} \otimes x_{u,w'}} & V_{i_{k+1}} \otimes S_{u'}[d_{i_0}+k(N-1)-1] \arrow{r}{Id \otimes x_{u',w'}} & V_{i_{k+1}} \otimes S_{w'}[d_{i_0}+k(N-1)-1].
		\end{tikzcd}\]
		\item If $\tau = \tau_{u'}$ such that $u \neq u' \nsim w'$, 
		\[V_{i_{k+1}} \otimes S_u[d_{i_0}+k(N-1)-1] \xrightarrow{Id \otimes x_{u,w'}} V_{i_{k+1}} \otimes S_{w'}[d_{i_0}+k(N-1)-1].\]
		\item If $\tau = \tau_{w'}$, 
		\[V_{i_k} \otimes S_u[d_{i_0}+k(N-1)-1] \xrightarrow{Id \otimes x_{u,w'}} V_{i_k} \otimes S_{w'}[d_{i_0}+k(N-1)-1] \xrightarrow{\psi_{i_{k+1},i_k} \otimes z_{w'}} V_{i_{k+1}} \otimes S_{w'}[d_{i_0}+(k+1)(N-1)-1].\]
		\item If $\tau=\tau_{w''}$ such that $w' \neq w'' \sim u$, 
		\[\begin{tikzcd}[arrow style=math font,cells={nodes={text height=2ex,text depth=0.75ex}},column sep = 1em]
			V_{i_k}\otimes S_u[d_{i_0}+k(N-1)-1] \arrow{r}{Id\otimes_{u,w''}} \arrow[bend left=15]{rr}{\psi_{i_{k+1},i_k} \otimes x_{u,w'}} & V_{i_{k+1}} \otimes S_{w''}[d_{i_0}+k(N-1)-1] \arrow{r}{0} & V_{i_{k+1}} \otimes S_{w'}[d_{i_0}+k(N-1)-1].
		\end{tikzcd}\]
		\item If $\tau = \tau_{w''}$ such that $w' \neq w'' \nsim u$, 
		\[V_{i_{k+1}} \otimes S_u[d_{i_0}+k(N-1)-1] \xrightarrow{Id \otimes x_{u,w'}} V_{i_{k+1}} \otimes S_{w'}[d_{i_0}+k(N-1)-1].\]
	\end{enumerate}
	\end{lem}
\begin{proof}
	It is a simple corollary of Lemmas \ref{lem induced functor}, \ref{lem simple computation}, and \ref{lem no outside effect}.
\end{proof}

Now we analyze the effect of $\tau \in \left\{\tau_u, \tau_w^{-1}\right\}$ on $C_3$ that is a chain of arrows given in Lemma \ref{lem no outside effect} (c). 
Since $C_3$ consists of only two arrows, we can prove Lemma \ref{lem effect on C_3} with simple computations. 

\begin{lem}
	\label{lem effect on C_3}
	Let $C_3$ be a good twisted complex of the following form:
	\[C_3 = \left(V_{i_0} \otimes S_w[d_{i_0}] \xrightarrow{\psi_{i_1,i_0}\otimes y_{w,u}} V_{i_1} \otimes S_u[d_{i_0}+N-2]\xrightarrow{\psi_{i_2,i_1} \otimes x_{u,w'}} V_{i_2} \otimes S_{w'}[d_{i_0}+N-2]\right).\]
	Then, $\tau(C_2)$ is equivalent to the following:
	\begin{enumerate}
		\item If $\tau = \tau_u$, 
		\[\begin{tikzcd}
			V_{i_0} \otimes S_u [d_{i_0}] \arrow[swap]{r}{Id \otimes x_{u,w}} \arrow[bend left=15]{rr}{\psi_{i_1,i_0} \otimes e_u} & V_{i_0} \otimes S_w[d_{i_0}] \arrow{r}{0} \arrow[swap, bend right=15]{rr}{(\psi_{i_2,i_1}\circ \psi_{i_1,i_0}) \otimes y_{u,w}} & V_{i_1} \otimes S_u[d_{i_0}-1] \arrow{r}{\psi_{i_2,i_1} \otimes z_u} & V_{i_2}\otimes S_u[d_{i_0} + N-2] \arrow{r}{Id \otimes x_{u,w'}} & V_{i_2} \otimes S_{w'}[d_{i_0} + N-2]
		\end{tikzcd}\]
		\item If $\tau = \tau_{u'}$ such that $u \neq u' \sim w$, 
		\[V_{i_0} \otimes S_{u'}[d_{i_0}] \xrightarrow{Id \otimes x_{u',w}} V_{i_0} \otimes S_w[d_{i_0}] \xrightarrow{\psi_{i_1,i_0} \otimes y_{u,w}} V_{i_1} \otimes S_u[d_{i_0}+N-2] \xrightarrow{\psi_{i_2,i_1} \otimes x_{u,w'}} V_{i_2} \otimes S_{w'}[d_{i_0}+N-2].\]
		\item If $\tau = \tau_{u'}$ such that $u \neq u' \sim w'$, 
		\[\begin{tikzcd}
			V_{i_0} \otimes S_w[d_{i_0}] \arrow{r}{\psi_{i_1,i_0} \otimes y_{u,w}} & V_{i_1} \otimes S_u[d_{i_0}+N-2] \arrow{r}{0} \arrow[bend left=15]{rr}{\psi_{i_2,i_1} \otimes x_{u,w'}} &  V_{i_2} \otimes S_{u'}[d_{i_0}+N-2] \arrow{r}{Id \otimes x_{u',w'}} & V_{i_2} \otimes S_{w'}[d_{i_0}+N-2].
		\end{tikzcd}\]
		\item If $\tau = \tau_{u'}$ such that $u \neq u' \nsim w, w'$,
		\[V_{i_0} \otimes S_w[d_{i_0}] \xrightarrow{\psi_{i_1,i_0}\otimes y_{w,u}} V_{i_1} \otimes S_u[d_{i_0}+N-2]\xrightarrow{\psi_{i_2,i_1} \otimes x_{u,w'}} V_{i_2} \otimes S_{w'}[d_{i_0}+N-2].\]
		\item If $\tau = \tau_w^{-1}$, 
		\[\begin{tikzcd}
			V_{i_0} \otimes S_w[d_{i_0} +N-1] \arrow{r}{0} \arrow[bend left=15]{rr}{\psi_{i_1,i_0} \otimes e_w} & V_{i_1}\otimes S_u[d_{i_0} + N-2] \arrow{r}{Id \otimes x_{u,w}} \arrow[bend right=15]{rr}{\psi_{i_2,i_1} \otimes x_{u,w'}} &V_{i_1} \otimes S_w[d_{i_0}+N-2] \arrow{r}{0} & V_{i_2} \otimes S_{w'} [d_{i_0} +N-2].
		\end{tikzcd}\] 
		\item If $\tau = \tau_{w'}^{-1}$,
		\[V_{i_0} \otimes S_w[d_{i_0} +N-1] \xrightarrow{\psi_{i_1,i_0}\otimes y_{w,u}} V_{i_1}\otimes S_u[d_{i_0} + N-2] \xrightarrow{Id \otimes x_{u,w'}} V_{i_1} \otimes S_{w'}[d_{i_0}+N-2] \xrightarrow{\psi_{i_2,i_1} \otimes z_{w'}} V_{i_2} \otimes S_{w'} [d_{i_0} +2N-3].\]
		\item If $\tau = \tau_{w''}^{-1}$ such that $w, w' \neq w'' \sim u$, 
		\[\begin{tikzcd}
			V_{i_0} \otimes S_w[d_{i_0}] \arrow{r}{\psi_{i_1,i_0} \otimes y_{w,u}}  & V_{i_1}\otimes S_u[d_{i_0} + N-2] \arrow[swap]{r}{Id \otimes x_{u,w''}} \arrow[bend left=15]{rr}{\psi_{i_2,i_1} \otimes x_{u,w'}} &V_{i_1} \otimes S_{w''}[d_{i_0}+N-2] \arrow{r}{0} & V_{i_2} \otimes S_{w'} [d_{i_0} +N-2].
		\end{tikzcd}\] 
		\item If $\tau = \tau_{w''}^{-1}$ such that $w, w' \neq w'' \nsim u$, 
		\[V_{i_0} \otimes S_w[d_{i_0}] \xrightarrow{\psi_{i_1,i_0}\otimes y_{w,u}} V_{i_1} \otimes S_u[d_{i_0}+N-2]\xrightarrow{\psi_{i_2,i_1} \otimes x_{u,w'}} V_{i_2} \otimes S_{w'}[d_{i_0}+N-2].\]
	\end{enumerate}
\end{lem}
\begin{proof}
	First, we would like to note that since $C_3$ is a {\em nonzero} chain of arrows, we have $w \neq w'$. 
	See Lemma \ref{lem noxyyx}. 
	Since $T$ is a tree, one can conclude that $u$ is the unique vertex such that $w \sim u \sim w'$. 
	Thus, for example in (2), $u = u' \sim w$ implies that $ u' \nsim w'$. 
	Remembering these facts, Lemma \ref{lem effect on C_3} is a direct result from Lemmas \ref{lem induced functor}, \ref{lem simple computation}, and the definition of spherical twists $\tau$. 
\end{proof}

\begin{remark}
	\label{rmk the other direction}
	We would like to point out that in Lemmas \ref{lem effect on C_1}--\ref{lem effect on C_3}, we considered a nonzero chain of arrows starting with a $y_{w,u}$-arrow, continuing with $z_u$-arrows, and ending with a $z_u$- or $x_{u,w'}$-arrow. 
	If we would consider a nonzero chain of arrows ending with a $y_{w,u}$-arrow, having $z_w$-arrows in the middle, and starting with a $z_w$- or $x_{u',w}$-arrow, then we could have a similar result. 
\end{remark}

We are ready to prove the main result of the subsection.
\begin{proof}[Proof of Lemma \ref{lem partially carried-by}]
	Let $\Phi$ be an autoequivalence of Penner type. 
	As explained in the sketch, we prove that for a nonzero object $E_0$ such that $r(E_0) = r(\Phi^n E_0)$ for any $n \geq 1$, there exists $N \in \mathbb{N}$ such that $\Phi^N E_0$ is partially carried-by. 
	If $r(E_0) = 0$, then $E_0$ is categorically carried-by.
	Moreover, by Lemma \ref{lem PhiE}, $\Phi^n E_0$ is also categorically carried-by, and thus, partially carried-by for any $n \in \mathbb{N}$.
	Thus, we assume that $r(E_0) \geq 1$. 
	
	Now, for a simplified twisted complex $E$, we would like to define two sets, 
	\begin{itemize}
		\item a set of $y$-arrows of $E$, and 
		\item a set of $y$-arrows having a nonzero chain of arrows of the second form starting with the $y$-arrow. 	
	\end{itemize}
	To do that, we note that for a simplified twisted complex $E = \left(\bigoplus_{i \in I} V_i \otimes S_{v_i}[d_i], F_{j,i}\right)$, every $y$-arrow $F_{j,i}$ is a form of 
	\[V_i \otimes S_w [d] \xrightarrow{F_{j,i} = \psi_{j,i} \otimes y_{w,u}} V_j \otimes S_u [d+N-2].\] 
	Thus, one can define the set of $y$-arrows in $E$ as the following set:
	\[Y(E):= \left\{\left(w,u,d\right) \Big\vert \exists \text{  a $y$-arrow $F_{j,i}$ of the form  } V_i \otimes S_w [d] \xrightarrow{\psi_{j,i} \otimes y_{w,u}} V_j \otimes S_u [d+N-2]\right\}.\]
	Similarly, we define a subset of $Y(E)$ as follows:
	\begin{align*}
		Y_x(E):= \Big\{\left(w,u,d\right) \Big\vert & \exists \text{  a $y$-arrow $F_{j,i}$ of the form  } V_i \otimes S_w [d] \xrightarrow{\psi_{j,i} \otimes y_{w,u}} V_j \otimes S_u [d+N-2], \text{  and}\\
		& \text{there exists a nonzero chain of arrows of the second form starting with $F_{j,i}$.}\Big\}.
	\end{align*}
	
	For convenience, let $\Phi^n E_0$ denote the minimal, simplified twisted complex equivalent to $\Phi^n E_0$. 
	The assumption that $r(E_0) = r(\Phi^nE_0)$ for any $n \in \mathbb{N}$ and Lemma \ref{lem y-arrow2} imply that $Y(E_0) = Y_(\Phi^n E_0)$ for any $n \in \mathbb{N}$. 
	Now, we would like to prove the following three: 
	\begin{enumerate}
		\item[(i)] First, we would like to show that for any $n \geq 1$, 
		\[Y_x(\Phi^n E_0) \subset Y_x(\Phi^{n+1}E_0).\]
		\item[(ii)] Second, we want to show that for sufficiently large $n \in \mathbb{N}$, 
		\[Y_x(\Phi^n E_0) = Y(\Phi^n E_0).\]
		\item[(iii)] For a sufficiently large $N \in \mathbb{N}$, let $\Phi^N E_0 = \left(\bigoplus_i X_i, F_{j,i}\right)$. 
		We would like to consider the set of $x$-arrows which are the last arrow of a nonzero chain of arrows of the second form. 
		Let an $x$-arrow $F_{j_0,i_0}$ be the last $x$-arrow among the above mentioned set of $x$-arrows.
		Then, we set 
		\[B= \left(\bigoplus_{\text{  There exists a nonzero chain of arrows from $X_{j_0}$ to $X_i$.}} X_i, F_{j,i}\right).\]
		The last step we would like to show is that there exists a sufficiently large $N$ such that $B$ is nonzero. 
	\end{enumerate}

	We note that (i) is necessarily to prove (ii) and (ii) is necessarily to prove (iii). 
	If (iii) holds, one can easily find a minimal twisted complex $A$ and $F: A \to B$ such that the given minimal twisted complex $\Phi^N E_0$ satisfies that 
	\[\Phi^N E_0 = \left(A \oplus B, F\right),\]
	from the definition of $B$. 
	More precisely, since $B$ is defined to be a twisted complex consisting of factors of $\Phi^N E_0$ such that
	\begin{itemize}
		\item the factors appear after the {\em last} $x$-arrow of $\Phi^N E_0$ where the last is in the sense of (iii), and
		\item the factors are connected to the last $x$-arrow by a nonzero chain of arrows,  
	\end{itemize}
	 the other factors of $\Phi^N E_0$ gives a minimal twisted complex $A$ and $F: A \to B$ such that $\Phi^N E_0 = \left(A \oplus B, F\right)$. 
	 
	 Now, it is trivial to observe that $\Phi^N E_0$ is partially carried-by with a triple $\left(A, B, F\right)$. 
	 Thus, it is enough to prove that (iii).
	 
	 \noindent{\em Proof of (i):}
	 Let us assume that $(w,u,d) \in Y_x(\Phi^n E_0)$. 
	 Then, there exists a nonzero chain of arrows in $\Phi^n E_0$ of the second form such that the first arrow is the arrow corresponding to $(w,u,d)$. 
	 Let $C$ denote the chain of arrows.
	 
	 We recall that $\Phi$ is a product of spherical twists along $v \in V_+(T)$ and the inverses of spherical twists along $v \in V_-(T)$. 
	 Thus, by applying Lemmas \ref{lem no outside effect}--\ref{lem effect on C_3} and Remark \ref{rmk the other direction}, one can analyze the effect of $\Phi$ on $C$. 
	 Then, one can easily see that $(w,u,d) \in Y_x(\Phi^{n+1} E_0)$. 
	 
	 For more rigorous proof, one can list all possible cases of the chain of arrows $C$ and the product representation of $\Phi$. 
	 Then, one can prove (i) by using the idea of proof by cases. 
	 Since the proof of each case is a simple corollary of Lemmas \ref{lem no outside effect}--\ref{lem effect on C_3}, we omit the details. 
	 
	 \noindent{\em Proof of (ii) and (iii):}
	 One can similarly prove (ii) (resp.\ (iii)) from (i) (resp.\ (ii)) by using the idea of proof by cases
\end{proof}

\subsection{Step 6 of the sketch of proof of Theorem \ref{thm pseudo-Anosov}}
\label{subsection step 6}
We prove Theorem \ref{thm pseudo-Anosov} in this subsection. 

\begin{proof}[Proof of Theorem \ref{thm pseudo-Anosov}]
We recall the statement of Theorem \ref{thm pseudo-Anosov}. 
In the statement, we consider a tree $T$ and the associated triangulated category $\Fuk$. 
We consider an autoequivalence $\Phi: \Fuk \to \Fuk$ which is of Penner type.
Moreover, we assume that $\Phi$ is a product of autoequivalences in the following set, as mentioned in Remark \ref{rmk sign is not important 1}:
\[\left\{\tau_u, \tau_w^{-1} | u \in V_+(T), w \in V_-(T)\right\}.\]
And, let $\sigma$ denote an arbitrary stability condition contained in the fixed connected component $\stab^\dagger(\Fuk)$. 
Our goal is to show that there exists a real number $\lambda_\Phi$ such that for any nonzero object $E \in \Fuk$, 
\[\lim_{n \to \infty} \tfrac{1}{n} m_\sigma(\Phi^n E) = \log \lambda_\Phi.\]

In this setting, the following are what we already know:
\begin{itemize}
\item By Proposition \ref{prop same connected component}, it is enough to consider a fixed stability condition, instead of an arbitrary stability condition in $\stab^\dagger(\Fuk)$. 
We fixed $\sigma_0$ in Definition \ref{def fixed component of the space of stability conditions}. 
\item By Theorem \ref{thm without y morphism}, there exists a real number $\lambda_\Phi$ such that for any nonzero, categorically carried-by $E \in \Fuk$,
\[\lim_{n\to \infty} \tfrac{1}{n} m_{\sigma_0}(\Phi^n E) = \log \lambda_\Phi.\]
\item Lemma \ref{lem maximality} shows that for any nonzero object $E \in \Fuk$, 
\[\lim_{n\to \infty} \tfrac{1}{n} m_{\sigma_0}(\Phi^n E) \leq \log \lambda_\Phi.\]
\end{itemize}

Thanks to the above three facts, it is enough to show that for any nonzero $E \in \Fuk$, 
\begin{gather}
\label{eqn final inequality}
\lim_{n\to \infty} \tfrac{1}{n} m_{\sigma_0}(\Phi^n E) \geq \log \lambda_\Phi.
\end{gather}
In this proof, we will prove the following claim:
\begin{enumerate}
\item[Claim:] If $E$ is partially carried-by, then inequality \eqref{eqn final inequality} holds. 
\end{enumerate}
If one assumes the claim, since every nonzero object $E \in \Fuk$ has a natural number $M$ such that $\Phi^M E$ is partially carried-by because of Lemma \ref{lem partially carried-by},
\[\lim_{n \to \infty} \tfrac{1}{n} m_{\sigma_0}(\Phi^n E) = \lim_{n \to \infty} \tfrac{1}{n+M} m_{\sigma_0}\left(\Phi^n (\Phi^M E)\right) \geq \log \lambda_\Phi.\]

To prove the claim, let us recall that if $E$ is partially carried-by, then there exists a triple $(A,B,F)$ such that $E$ is partially carried-by with the triple $(A,B,F)$, see Definitions \ref{definition partially carried-by} and \ref{definition partially carried-by 2}.
Thus, $E \simeq \left(A \oplus B, F\right)$.
Moreover, since the twisted complex $\left(A \oplus B, F\right)$ can be seen as a minimal twisted complex, 
\begin{gather}
\label{eqn partial mass}
m_{\sigma_0}(E) \geq m_{\sigma_0}(B).
\end{gather}

Let us remark that $\Phi$ is a product of autoequivalences in $\left\{\tau_u, \tau_w^{-1} | u \in V_+(T), w \in V_-(T)\right\}$.
Then, by applying Lemma \ref{lem partially carried-by preserved} repeatedly, one can prove that, for any $n \in \mathbb{N}$, there exists a triple $(A_n,B_n,F_n)$ such that 
\begin{itemize}
\item $\Phi^n E$ is partially carried-by with a triple $(A_n, B_n, F_n)$, and 
\item $\Phi^n A \simeq A_n, \Phi^n B \simeq B_n$.
\end{itemize} 

When one applies inequality \eqref{eqn partial mass} to $\Phi^n E$ that is partially carried-by with a triple $(A_n, B_n, F_n)$, one obtains 
\[m_{\sigma_0}(\Phi^n E) \geq m_{\sigma_0}(B_n).\]
Moreover, since the mass function $m_{\sigma_0}$ is a function of quasi-equivalent classes, 
\[m_{\sigma_0}(B_n) = m_{\sigma_0}(\Phi^n B).\]
Let us recall that $B$ is categorically carried-by. 
Then, the above arguments and Theorem \ref{thm without y morphism} conclude that  
\[\lim_{n \to \infty} \tfrac{1}{n}m_{\sigma_0}(\Phi^n E) \geq \lim_{n \to \infty} \tfrac{1}{n} m_{\sigma_0}(B_n) =\lim_{n \to \infty} \tfrac{1}{n} m_{\sigma_0}(\Phi^n B) \geq \log \lambda_\Phi.\]
It completes the proof.
\end{proof}

\subsection{Stretching factor and entropy}
\label{subsection stretching factor and the categorical entropy}
In this subsection, we discuss the relation between the stretch factor and the entropy of a Penner type autoequivalence. 
We note that the entropy is defined in \cite{Dimitrov-Haiden-Katzarkov-Kontsevich14} as follows:
\begin{definition}
\label{def categorical entropy}
Let $\mathcal{D}$ be a triangulated category equipped with a split-generator $G$, and let $F:\mathcal{D}\to \mathcal{D}$ be an exact endofunctor.
\begin{enumerate}
\item Let $E_1, E_2$ be objects of $\mathcal{D}$. The {\bf complexity of $\boldsymbol{E_2}$ relative to $\boldsymbol{E_1}$} is the function $\delta_t(E_1,E_2):\mathbb{R} \to [0,\infty]$ given by 
\[\delta_t(E_1,E_2)=\inf\left\{\sum_{i=1}^k e^{n_i t} \Big\vert \begin{tikzcd}
0 \ar[r] & A_1 \ar[d]\ar[r] & A_2 \ar[d] \ar[r] & \dots \ar[r] & A_{k-1} \ar[r] \ar[d] & E_2\oplus E_2' \ar[d]\\
& E_1[n_1] \ar[ul, dashed]         & E_1[n_2] \ar[ul, dashed]       & \dots        & E_1[n_{k-1}] \ar[ul, dashed]       & E_1[n_k] \ar[ul, dashed]
\end{tikzcd} \right\}.\]
\item The {\bf entropy of $\boldsymbol{F}$} is the function $h_t(F): \mathbb{R} \to [-\infty,\infty]$ of $t$ given by 
\[h_t(F)=\lim_{n \to \infty} \frac{1}{n} \log \delta_t(G, F^nG).\]
\end{enumerate}
\end{definition}

\begin{thm}
\label{thm stretch factor and entropy}
Let $T$ be a tree, and let $\Phi : \Fuk \to \Fuk$ be an autoequivalence of Penner type. 
If $\lambda_\Phi$ is the stretch factor of $\Phi_\Phi$, then the following hold:
\begin{enumerate}
\item $h_0(\Phi) = \log \lambda_\Phi.$
\item The stretching factor of $\Phi^{-1}$ is also $\lambda_\Phi$. 
\end{enumerate}
\end{thm}
\begin{proof}
We note that in the proof of Theorem \ref{thm without y morphism}, the stretch factor $\lambda_\Phi$ is given as the largest eigenvalue of $M_\Phi$. 
We recall that by the arguments given after Remark \ref{rmk determinant}, every $(v,w)$-entry of $M_\Phi$ is 
\[\dim \lhom^*_{\wrapped}\left(L_v, \Phi(S_w)\right).\]

The entropy of a Penner type autoequivalence is studied in \cite[Section 5]{Bae-Choa-Jeong-Karabas-Lee22}. 
Especially, by \cite[Theorem 5.8]{Bae-Choa-Jeong-Karabas-Lee22}, it is known that $h_0(\Phi)$ is the logarithm of the largest eigenvalue of a $A_\Phi$ such that
\begin{itemize}
\item $A_\Phi$ is a $|V(T)|$-by-$|V(T)|$ matrix, and
\item if we use $V(T)$ as an index set for columns/rows of $A_\Phi$, the $(v,w)$-component of $A_\Phi$ is 
\[\dim \lhom_{\wrapped}^*\left(L_v, \Phi(S_w)\right).\]
\end{itemize}

Thus, one can observe that $M_\Phi = A_\Phi$. 
It completes the proof the first item. 

To prove the second item, we note that by \cite[Theorem 4.14]{Bae-Choa-Jeong-Karabas-Lee22}, we have $h_0(\Phi) = h_0(\Phi^{-1})$. 
We also note that since $\Phi^{-1}$ is also of Penner type, one can apply the first item to $\Phi^{-1}$.
When one applies the first item to $\Phi^{-1}$, one has 
\[\log \left(\text{the stretching factor of  } \Phi^{-1}\right) = h_0(\Phi^{-}) = h_0(\Phi) = \log \lambda_\Phi.\]
\end{proof}

\section{Strong pseudo-Anosov and phase stretching factors}
\label{section strong pseudo-Anosov}
Our next goal is to show that every Penner type autoequivalence $\Phi$ is {\em strong pseudo-Anosov}.
The notion of strong pseudo-Anosov is defined in Definition \ref{def strong pesudo-Anosov}. 
In other words, we would like to show that the asymptotic behavior of the maximal/minimal phases of $\Phi^n E$ does not depend on the choice of $E$. 
Moreover, we also would like to give an {\em algorithmic computation of the positive/negative phase stretching factors} of $\Phi$.

In the first subsection, we prove that $\Phi$ is strong pseudo-Anosov.
And in the following subsections, we present an algorithm computing the phase stretching factors of $\Phi$.
The last subsection proves Lemma \ref{lem shifting numbers of inverse} that we need in Section \ref{section translation of Penner type autoequivalences}.

\subsection{Strong pseudo-Anosov}
\label{subsection strong pseudo-Anosov}
In this subsection, we prove that every Penner type autoequivalence is strong pseudo-Anosov.

First, let us recall the setting and notation of the current paper.
Let $T$ be a tree and let $\Phi$ be an autoequivalence of Penner type on $\Fuk$. 
We note that there exists a collection of spherical objects $\{S_v \in \Fuk\}_{v \in V(T)}$ generating $\Fuk$, and there exists a collection of generating objects $\{L_v \in \wrapped\}_{v \in V(T)}$, which is dual to $\{S_v\}_{v \in V(T)}$. 
As usual, we assume that $\Phi$ is a product of $\{\tau_u, \tau_w^{-1} | u \in V_+(T), w \in V_-(T)\}$.

By definition, the maximal/minimal phase filtrations associated to $\Phi$ care the growth ratio of $\phi^\pm_\sigma(\Phi^n E)$, where $\sigma \in \stab^\dagger(\Fuk)$. 
We note that the notion of pseudo-Anosov cares the exponential growth of $m_\sigma(\Phi^n E)$.
To compute the exponential growth, we used a linear-algebraic tool.
More precisely, we defined $M_\Phi$ in Section \ref{section the case of categorically carried-by}, and then the spectral radius of $M_\Phi$ is turned out to be the exponential growth of $m_\sigma(\Phi^n E)$.
In order to use the similar idea, we define the following matrix for $\Phi$ and a vector for $E \in \Fuk$ whose components are Laurent polynomials. 
\begin{definition}
	\label{def Laurent-poly matrix}
	Let $\Phi: \Fuk \to \Fuk$ be an autoequivalence of Penner type, and let $\{S_v\}_{v \in V(T)}$ and $\{L_v\}_{v\in V(T)}$ be collection of generators of $\Fuk$ and $\wrapped$ respectively, as defined in Section \ref{section settings and the main idea}.
	\begin{enumerate}
		\item The {\bf Laurent-polynomial matrix associated to $\boldsymbol{\Phi}$} is $|V(T)|$-by-$|V(T)|$ matrix whose $(v_i, v_j)$-entry is defined as 
		\[\sum_{k \in \mathbb{Z}} \Big(\dim \lhom_{\wrapped}^{-k}\left(L_{v_i}, \Phi S_{v_j}\right)\Big) t^k.\] 
		For simplicity, we denote the Laurent-polynomial matrix corresponding to $\Phi$ as $\boldsymbol{M_{\Phi}(t)}$. 
		\item For a good twisted complex $E$, let $\boldsymbol{V(t;E)}$ be the vector having $|V(T)|$-many Laurent-polynomial-components defined as follows:
		\[\vect(t;E):= \left(\sum_{k \in \mathbb{Z}} \Big(\dim \lhom_{\wrapped}^{-k}\left(L_{v_i}, E\right)\Big) t^k\right).\]
		\item For a good twisted complex $E$, let $\boldsymbol{\vect_h(t;E)}$ be the vector having $|V(T)|$-many Laurent-polynomial-components defined as follows:
		\[\vect_h(t;E):= \left(\sum_{k \in \mathbb{Z}} \Big(\dim \shom_{\wrapped}^{-k}\left(L_{v_i}, E\right)\Big) t^k\right).\]
		\item For convenience, we define the following notation: 
		If $M$ (resp.\ $V$) is a $|V(T)|$-by-$|V(T)|$ matrix (resp.\ a vector with $|V(T)|$-many components), we denote the $(v_i,v_j)$-component of $M$ (resp.\ $v_i$-component of $V$) by $\boldsymbol{M_{v_i,v_j}}$ (resp.\ $\boldsymbol{V_{v_i}}$).
	\end{enumerate}
\end{definition}

We remark some properties of $M_{\Phi}(t), \vect(t;E)$, and $\vect_h(t;E)$.
\begin{remark}
	\label{rmk properties of Laurent-poly matrix}
	\mbox{}
	\begin{enumerate}
		\item First of all, we note that every component of $M_\Phi(t), \vect(t;E)$, and $\vect_h(t;E)$ is either zero or a Laurent polynomial with positive integer coefficients.
		Moreover, from the definition, for any $v \in V(T), k \in \mathbb{Z}$, and any good twisted complex $E$, one can easily see that the coefficient of $t^k$ of $\vect(t;E)_v$ is the same as the number of $S_v[k]$-factor in a {\em minimal} twisted complex equivalent to $E$. 
		To be more precise, for a good twisted complex $E = \left(\bigoplus_{i \in I} V_i \otimes S_{v_i}[d_i], F_{j,i}\right)$, the number of $S_v[k]$-factors in $E$ means the number 
		\[\sum_{i \in I \text{  such that  } v_i = v, d_i =k} \dim V_i.\]
		\item We would like to point out that $\vect(t;E)$ is an invariant of quasi-equivalence classes since we are taking $\lhom^*_{\wrapped}$ in Definition \ref{def Laurent-poly matrix}. 
		On the other hand, $\vect_h(t;E)$ is an invariant of twisted complexes, but not quasi-equivalent classes. 
		\item Since $\lhom_{\wrapped}^*$ is the cohomology of $\shom_{\wrapped}^*$, for any $v \in V(T), k \in \mathbb{Z}$, and good twisted complex $E \in \Fuk$, 
		\[\dim \lhom_{\wrapped}^{-k}(L_v, E) \leq \dim \shom_{\wrapped}^{-k}(L_v, E).\]
		Thus, the coefficient of $t^k$-term of $\vect(t;E)_v$ is smaller or equal to the coefficient of $t^k$-term of $\vect_h(t;E)_v$. 
		And, especially, if $E$ is a minimal twisted complex, then by Lemma \ref{lem zero differential for minimal}, 
		\[\vect(t;E) = \vect_h(t;E).\]
		\item For a good twisted complex $E$, as we did in the previous sections, let $\Phi E$ be a good twisted complex obtained by applying Lemma \ref{lem induced functor}.
		We note that since $S_v$ is categorically carried-by for any $v \in V(T)$, $\Phi S_v$ is also categorically carried-by and minimal. 
		Then, from Lemma \ref{lem induced functor}, one can observe that for any twisted complex $E$,
		\[\vect_h(t;\Phi E) = M_\Phi(t) \cdot \vect_h(t;E).\]
		\item When we plug $t=1$ in $M_\Phi(t)$ and $\vect(t;E)$, we can see $M_\Phi(1) = M_\Phi, \vect(1;E) = \vect(E)$, where $M_\Phi$ and $\vect(E)$ are the matrix and the vector defined in Section \ref{section the case of categorically carried-by}.
		\item Similar to Remark \ref{rmk computation of the matrix} (2), we would like to remark that $\vect(t;E), \vect_h(t;E)$ and $M_\Phi(t)$ are dependent on the choice of sign convention. 
		It is important in Section \ref{subsection lemma shifting numbers of inverse}.
		\end{enumerate}
\end{remark}

Now, we fix a stability condition $\sigma_\star \in \stab^\dagger(\Fuk)$, and we will compute the growth of $\phi^\pm_{\sigma_\star}(\Phi^n E)$ as $n \to \infty$. 
\begin{definition}
	\label{def sigma star}
	Let $\boldsymbol{\sigma_\star= \left(Z_\star, P_\star\right)}$ be a fixed stability condition in $\stab^\dagger(\Fuk)$ constructed as follows:
	We note that by Proposition \ref{prop construction of a stability condition}, in order to define a stability condition, it is enough to give a bounded $t$-structure and a stability function on the heart satisfying the Harder-Narasimhan property. 
	The bounded $t$-structure we consider here is $\mathbb{T}$ defined in the proof of Lemma \ref{lem existence of a stability condition}. 
	In the proof, we observed that the heart $\mathcal{A}$ consists of the zero object and objects equivalent to a twisted complex consisting of $\{S_v | v \in V(T)\}$. 
	We set 
	\[Z_\star(S_v) = \imagin,\]
	for all $v \in V(T)$. 
	Then, every nonzero object of $\mathcal{A}$ is a semistable object with phase $\tfrac{1}{2}$ and $Z_\star$ satisfies the Harder-Narasimhan property. 
\end{definition}

Lemma \ref{lem shifting number} expresses the growth of the maximal/minimal phases of $\Phi^n E$ with respect to $\sigma_\star$, i.e., $\phi^\pm_{\sigma_\star}(\Phi^n E)$, in terms of the Laurent polynomial associated to $\Phi$.
\begin{lem}
	\label{lem shifting number}
	Let $E$ be a nonzero object of $\Fuk$. 
	Then, the following equations hold:
	\begin{align}
		\label{eqn minimal}
		&\liminf_{n \to \infty} \tfrac{1}{n} \phi^-_{\sigma_\star}(\Phi^n E) = \liminf_{n\to \infty} \tfrac{1}{n} \min_{v_i, v_j \in V(T)} \left\{k \Big\vert \text{  the coefficient of $t^k$-term in $\left(M_\Phi(t)\right)^n_{v_i,v_j}$ is not zero}\right\},\\
		\label{eqn maximal}
		&\limsup_{n \to \infty} \tfrac{1}{n} \phi^+_{\sigma_\star}(\Phi^n E) = \limsup_{n\to \infty} \tfrac{1}{n} \max_{v_i, v_j \in V(T)} \left\{k \Big\vert \text{  the coefficient of $t^k$-term in $\left(M_\Phi(t)\right)^n_{v_i,v_j}$ is not zero}\right\}.
	\end{align}
\end{lem}
\begin{proof}
	We only prove Equation \eqref{eqn minimal} since Equation \eqref{eqn maximal} can be proven in the same way. 
	
	The proof of Equation \eqref{eqn minimal} consists of following three steps: 
	\begin{enumerate}
		\item[(i)] First, we show that for any good twisted complex $E \in \Fuk$, 
		\begin{gather*}
			\liminf_{n \to \infty} \tfrac{1}{n} \phi^-_{\sigma_\star}(\Phi^n E) \geq \liminf_{n\to \infty} \tfrac{1}{n} \min_{v_i, v_j \in V(T)} \left\{k \Big\vert \text{  the coefficient of $t^k$-term in $\left(M_\Phi(t)\right)^n_{v_i,v_j}$ is not zero}\right\}.
		\end{gather*}
		And, if $E$ is minimal, then the equality holds. 
		\item[(ii)] Second, we show that for any {\em partially carried-by} twisted complex $E \in \Fuk$, 
		\[\liminf_{n \to \infty} \tfrac{1}{n} \phi^-_{\sigma_\star}(\Phi^n E) \leq \liminf_{n\to \infty} \tfrac{1}{n} \min_{v_i, v_j \in V(T)} \left\{k \Big\vert \text{  the coefficient of $t^k$-term in $\left(M_\Phi(t)\right)^n_{v_i,v_j}$ is not zero}\right\}.\]
		\item[(iii)] By combining (i), (ii), and Lemma \ref{lem partially carried-by}, we complete the proof of Equation \eqref{eqn minimal}.
	\end{enumerate}

	\noindent {\em Proof of (i).}
	From Remark \ref{rmk properties of Laurent-poly matrix} (1) and (3), we can observe that for any good twisted complex $E$,
	\begin{align*}
		\phi_{\sigma_\star}^-(E) &= \min_{v \in V(T)} \left\{k \Big\vert \text{  the coefficient of $t^k$-term in $\vect(t;E)_v$ is not zero}\right\} +\tfrac{1}{2} \\
		&\geq \min_{v \in V(T)} \left\{k \Big\vert \text{  the coefficient of $t^k$-term in $\vect_h(t;E)_v$ is not zero}\right\} +\tfrac{1}{2}.
	\end{align*}
	Moreover, if $E$ is minimal, then Remark \ref{rmk properties of Laurent-poly matrix} (3) implies that the equality holds. 
	
	Now, we put $\Phi^n E$ in the position of $E$ of the above inequality, then thanks to Remark \ref{rmk properties of Laurent-poly matrix} (4), we have 
	\begin{align*}
		\phi_{\sigma_\star}^-(\Phi^n E) &\geq \min_{v \in V(T)} \left\{k \Big\vert \text{  the coefficient of $t^k$-term in $\vect_h(t;\Phi^n E)_v$ is not zero}\right\} +\tfrac{1}{2} \\
		& = \min_{v \in V(T)} \left\{k \Big\vert \text{  the coefficient of $t^k$-term in $\left(M_\Phi^n \cdot \vect_h(t;E)\right)_v$ is not zero}\right\} +\tfrac{1}{2}.
	\end{align*}
	And if $m \leq n$, we have 
	\begin{gather*}
		\phi_{\sigma_\star}^-(\Phi^n E) \geq \min_{v \in V(T)} \left\{k \Big\vert \text{  the coefficient of $t^k$-term in $\left(M_\Phi^{n-m} \cdot \vect_h(t;\Phi^m E)\right)_v$ is not zero}\right\} +\tfrac{1}{2}.
	\end{gather*}
	Again, the equality holds if $E$ is minimal. 
	
	Now, we recall that for any nonzero good twisted complex $E \in \Fuk$, there exists $m \in \mathbb{N}$, such that every component of $\vect_h(t;E)$ is nonzero.
	Definition \ref{def Penner type} $(II)$ implies the above fact.
	Thus, one can fix a constant $m$ such that $\vect_h(t;\Phi^m E)$ is a constant vector whose every component is a nonzero Laurent polynomial with nonnegative coefficients. 
	It implies that 
	\begin{align*}
		\liminf_{n \to \infty} \tfrac{1}{n} \phi_{\sigma_\star}^-(\Phi^n E) &\geq \liminf_{n \to \infty} \tfrac{1}{n} \min_{v \in V(T)} \left\{k \Big\vert \text{  the coefficient of $t^k$-term in $\left(M_\Phi^{n-m} \cdot \vect_h(t;\Phi^m E)\right)_v$ is not zero}\right\}\\
		& = \liminf_{n\to \infty} \tfrac{1}{n} \min_{v_i, v_j \in V(T)} \left\{k \Big\vert \text{  the coefficient of $t^k$-term in $\left(M_\Phi(t)\right)^n_{v_i,v_j}$ is not zero}\right\}.
	\end{align*}
	Moreover, the equality holds if $E$ is minimal. 
	
	\noindent {\em Proof of (ii).}
	Let us assume that $E$ is a partially carried-by twisted complex with a triple $(A,B,F)$.
	It implies that $E$ is equivalent to a minimal twisted complex $\left(A \oplus B, F\right)$.
	Thus, we have 
	\[\phi^-_{\sigma_\star}(E) = \phi^-_{\sigma_\star}\left(A \oplus B, F\right).\]
	Especially, since $B$ and $\left(A \oplus B, F\right)$ are minimal, we have 
	\begin{gather}
		\label{eqn lem minimal shifting 1}
		\phi^-_{\sigma_\star}(E) \leq \phi^-_{\sigma_\star}(B).
	\end{gather}
		
	We recall that by Lemma \ref{lem partially carried-by preserved}, for any $n \in \mathbb{N}$, there exists a triple $\left(A_n,B_n,F_n\right)$ such that 
	\begin{itemize}
		\item $\Phi^n E$ is partially carried-by with a triple $\left(A_n,B_n,F_n\right)$,
		\item $\Phi^n A \simeq A_n$ and $\Phi^n B \simeq B_n$, and 
		\item for any $n \in \mathbb{N}$, $B_n$ is categorically carried-by.
	\end{itemize}
	Then, by applying inequality in \eqref{eqn lem minimal shifting 1} to $\Phi^n E$, we have 
	\[\phi^-_{\sigma_\star}(\Phi^n E) \leq \phi^-_{\sigma_\star}(\Phi^n B).\]
	
	Now, by taking $\liminf_{n \to \infty} \tfrac{1}{n}$ of the above inequality, we conclude
	\begin{align*}
		\liminf_{n \to \infty} \tfrac{1}{n} \phi^-_{\sigma_\star}(\Phi^n E) &\leq \liminf_{n \to \infty} \tfrac{1}{n} \phi^-_{\sigma_\star}(\Phi^n B)\\
		&= \liminf_{n\to \infty} \tfrac{1}{n} \min_{v_i, v_j \in V(T)} \left\{k \Big\vert \text{  the coefficient of $t^k$-term in $\left(M_\Phi(t)\right)^n_{v_i,v_j}$ is not zero}\right\}.
	\end{align*}
	The part (i) implies the equality in the second line.
	And, it completes the proof of (ii).
	
	\noindent {\em Proof of (iii).}
	Let us recall that by Lemma \ref{lem partially carried-by}, for any nonzero $E$, there exists $m \in \mathbb{N}$ such that $\Phi^m E$ is partially carried-by. 
	Let us fix such a $m \in \mathbb{N}$, then we have 
	\begin{align*}
		\liminf_{n \to \infty} \tfrac{1}{n} \phi^-_{\sigma_\star}(\Phi^n E) &= \liminf_{n \to \infty} \tfrac{1}{n} \phi^-_{\sigma_\star}(\Phi^{n-m} E) \\
		&\leq \liminf_{n\to \infty} \tfrac{1}{n} \min_{v_i, v_j \in V(T)} \left\{k \Big\vert \text{  the coefficient of $t^k$-term in $\left(M_\Phi(t)\right)^n_{v_i,v_j}$ is not zero}\right\}.
	\end{align*}
	The inequality in second line holds because of (ii).
	
	Then, the above inequality and (i) complete the proof of inequality \eqref{eqn minimal}.
\end{proof}

One can see that the right-handed sides of Equations \eqref{eqn minimal} and \eqref{eqn maximal} do not have a nonzero object $E$ that appears in the left-handed sides.
In other words, Lemma \ref{lem shifting number} says that for every nonzero object $E$, $\phi^\pm_{\sigma_\star}\left(\Phi^n E\right)$ has the same asymptotic behavior as $n \to \infty$.
Thus, together with Lemma \ref{lem independence of the choice of stability conditions}, Lemma \ref{lem shifting number} proves that $\Phi$ is a strong pseudo-Anosov, i.e., Theorem \ref{thm strong pseudo-Anosov}. 
\begin{thm}
	\label{thm strong pseudo-Anosov}
	If $\Phi$ is an autoequivalence of Penner type, then $\Phi$ is strong pseudo-Anosov.
\end{thm}
\begin{proof}
	We recall that by definition, $\Phi$ is strong pseudo-Anosov if and only if $\Phi$ is pseudo-Anosov and the maximal phase/minimal phase filtrations associated to $\Phi$ have only one step. 
	Since Theorem \ref{thm pseudo-Anosov} proves that $\Phi$ is pseudo-Anosov, it is enough to show that the maximal/minimal phase filtrations have only one step. 
	
	Thanks to Definition \ref{def strong pesudo-Anosov} and Lemma \ref{lem independence of the choice of stability conditions}, the maximal/minimal phase filtrations can be defined as
	\[\left\{\mathcal{P}_\lambda(\Phi):=\mathcal{P}_{\sigma_\star,\lambda}(\Phi)\right\}_{\lambda \in \mathbb{R} >0},\]
	where 
	\[\mathcal{P}^+_{\sigma_\star, \lambda}(\Phi) := \left\{E \in \Fuk \Big\vert \limsup_{n \to \infty} \tfrac{1}{n} \phi^+_{\sigma_\star} (\Phi^n E) \leq \lambda \right\}, \mathcal{P}^-_{\sigma_\star, \lambda}(\Phi) := \left\{E \in \Fuk \Big\vert \liminf_{n \to \infty} \tfrac{1}{n} \phi^-_{\sigma_\star} (\Phi^n E) \geq - \lambda \right\}.\] 
	Then, Lemma \ref{lem shifting number} completes the proof.
\end{proof}
The next goal is to compute the positive/negative phase stretching factors of a Penner type $\Phi$. 
Lemma \ref{lem shifting number} implies that the phase stretching factors can be computed by tracking the change of $\left(M_\Phi(t)\right)^n$ as $n$ increases. 
In order to keep tracking of $\left(M_\Phi(t)\right)^n$, we study some linear algebra over {\em Puiseux series} in Section \ref{subsection linear algebra over Puiseux series}.
After that, one can see that the phase stretching factors are computed by simple linear algebra, see Theorem \ref{thm shifting number}. 

\subsection{Linear algebra over Puiseux series}
\label{subsection linear algebra over Puiseux series}
First, we define the notions of Puiseux series and their {\em valuations}.
\begin{definition}
	\label{def Puiseux series}
	\mbox{}
	\begin{enumerate}
		\item Let $\mathbb{k}$ be a field. 
		A {\bf Puiseux series} with coefficients in $\mathbb{k}$ is a formal Laurent series 
		\[f(t)= \sum_{k \geq m} a_k t^{\tfrac{k}{r}},\]
		with $r \in \mathbb{N}, m \in \mathbb{Z}, a_k \in \mathbb{k}$ for all $k \geq \mathbb{Z}_{\geq m}$, and $a_m \neq 0$. 
		\item Let $f(t)$ be a Puiseux series. 
		Then, the {\bf valuation of a Puiseux series $\boldsymbol{f(t)}$} is the smallest exponent of $f(t)$.
		More precisely, if 
		\[f(t)= \sum_{k \geq m} a_k t^{\tfrac{k}{r}},\]
		with $r \in \mathbb{N}, m \in \mathbb{Z}, a_k \in \mathbb{k}$ for all $k \geq \mathbb{Z}_{\geq m}$, and $a_m \neq 0$, then the valuation of $f(t)$ is $\tfrac{m}{r}$.
		We denote the valuation of $f(t)$ by $\boldsymbol{\upsilon\left(f(t)\right)}$.
	\end{enumerate}
\end{definition}
In the paper, we consider Puiseux series only over $\mathbb{C}$. 

We would like to mention some well-known properties of Puiseux series.
\begin{itemize}
	\item By definition, every Laurent series is a Puiseux series. 
	\item Similar to the case of Laurent series, a Puiseux series $f(t)$ (over $\mathbb{C}$) has the {\bf radius of convergence}, i.e., the maximal number $R$ satisfying that if $z_0 \in \mathbb{C}$ satisfies $|z_0| < R$, then $f(z_0)$ converges. 
	We note that the radius of convergence can be zero. 
	If a Puiseux series has a positive radius of convergence, then the Puiseux series is called {\bf convergent}.
	\item It is also well-known that the collection of Puiseux series admits a field structure, when the operations are defined naturally. 
	Similarly, the collection of all {\em convergent} Puiseux series admits a field structure. 
\end{itemize}

We go back to our main purpose.
Our main purpose is to keep track of the changes on $\left(M_\Phi(t)\right)^n$ as $n$ increases, in order to compute the positive/negative phase stretching factors of a Penner type $\Phi$. 
Thanks to Lemma \ref{lem shifting number}, we can write the phase stretching factors in terms of valuations of components of $\left(M_\Phi(t)\right)^n$ and $\left(M_{\Phi}(\tfrac{1}{t})\right)^n$, as follows: 
\begin{align*}
	&\text{The negative phase stretching factor of  } \Phi = \liminf_{n\to \infty} \tfrac{1}{n} \min_{v_i, v_j \in V(T)} \left\{\upsilon\left(\left(M_\Phi(t)^n\right)_{v_i,v_j}\right)\right\}, \\
	&\text{The positive phase stretching factor of  } \Phi = - \liminf_{n\to \infty} \tfrac{1}{n} \min_{v_i, v_j \in V(T)} \left\{\upsilon\left(\left(M_\Phi(\tfrac{1}{t})^n\right)_{v_i,v_j}\right)\right\}.
\end{align*}
We note that by definition of Puiseux series, the lowest degree of a Puiseux series $f(t)$ is well-defined, but the highest degree is not well-defined. 
Thus, in the second one of the above equations, we used $-\liminf_{n \to \infty}$ and $M_\Phi(\tfrac{1}{t})$, instead of $\limsup_{n \to \infty}$ and $M_\Phi(t)$.

For convenience, we give names for the right-handed sides of the above equations. 
\begin{definition}
	\label{def shifting number}
	For a Penner type $\Phi$, let the {\bf negative} (resp.\ {\bf positive}) {\bf shifting number}, denoted by $\boldsymbol{\tau^-(\Phi)}$ (resp.\ $\boldsymbol{\tau^+(\Phi)}$), be the number defined as 
	\begin{align*}
		&\tau^-(\Phi) := \liminf_{n\to \infty} \tfrac{1}{n} \min_{v_i, v_j \in V(T)} \left\{\upsilon\left(\left(M_\Phi(t)\right)^n_{v_i,v_j}\right)\right\}, \\
		&\tau^+(\Phi) := -\liminf_{n\to \infty} \tfrac{1}{n} \min_{v_i, v_j \in V(T)} \left\{\upsilon\left(\left(M_\Phi(\tfrac{1}{t})\right)^n_{v_i,v_j}\right)\right\}.
	\end{align*}
\end{definition}
We note that Lemma \ref{lem shifting number} shows that the negative and positive shifting numbers are the same as the negative and positive phase stretching factors. 
Thus, we compute the shifting numbers in the rest of Section \ref{section strong pseudo-Anosov}. 

\begin{remark}
	\label{rmk general shifting number} 
	We note that by Definitions and the fact that $\left(M_\Phi(t)\right)^n = M_{\Phi^n}(t)$, the following hold:
	\begin{gather}
		\label{eqn shifting number}
		\begin{split}
			&\tau^-(\Phi) = \liminf_{n \to \infty} \tfrac{1}{n} \min \left\{k \Big\vert \lhom_{\wrapped}^{-k}\left(L:=\oplus_{v\in V(T)}L_v, \Phi^n\left(S:=\oplus_{v \in V(T)}S_v\right)\right) \neq 0 \right\}, \\
			&\tau^+(\Phi) = \limsup_{n \to \infty} \tfrac{1}{n} \max \left\{k \Big\vert \lhom_{\wrapped}^{-k}\left(L:=\oplus_{v\in V(T)}L_v, \Phi^n\left(S:=\oplus_{v \in V(T)}S_v\right)\right) \neq 0 \right\}.
		\end{split}
	\end{gather}
	We also note that in \cite{Fan-Filip23}, the notion of the positive/negative shifting numbers are defined by Fan and Filip, on the saturated, finite type triangulated categories. 
	Equation \eqref{eqn shifting number} can be seen as the generalization of the shifting number defined in \cite{Fan-Filip23}, to the triangulated categories which are not necessarily to be saturated or of finite type, but satisfies a kind of duality. 
	We use the generalized version, since in symplectic topology, one cannot expect usually that a wrapped Fukaya category is of finite type. 
\end{remark}

If one can find the Jordan normal forms of the matrices $M_\Phi(t)$ and $M_\Phi(\tfrac{1}{t})$, then it reduces the difficulty of computing $\tau^\pm(\Phi)$. 
The classical {\em Newton-Puiseux} theorem guarantees the existence of Jordan normal forms of $M_\Phi(t)$ and $M_\Phi(\tfrac{1}{t})$.
The following theorem is a modification of Newton-Puiseux theorem for our purpose.
\begin{thm}[Newton-Puiseux theorem]
	\label{thm Newton Puiseux thm}
	The field of Puiseux series and the field of convergent Puiseux series with complex coefficients are both algebraically closed. 
\end{thm}
See, for example, \cite[Thereom 3.1]{Walker78}, or \cite{Eichler66, vandenDries-Ribenboim84, Brieskorn-Knorrer86}, for more details about Newton-Puiseux theorem.

We would like to point out that a Laurent-polynomial can be seen as a convergent Puiseux series. 
Thus, $M_\Phi(t)$ (resp.\ $M_\Phi(\tfrac{1}{t})$) is a matrix over an algebraically closed field and there exist a Jordan normal form of it. 
Then, Lemma \ref{lem eigenvalue} is easy to prove.
\begin{lem}
	\label{lem eigenvalue}
	Let $\Phi$ be an autoequivalence of Penner type. 
	Then, 
	\begin{gather*}
		\tau^-(\Phi) = \min \{\upsilon\left(\lambda_i(t)\right) \vert \lambda_i(t) \text{  is a eigenvalue of  } M_\Phi(t)\},\\
		\tau^+(\Phi) = -\min \{\upsilon\left(\mu_i(t)\right) \vert \mu_i(t) \text{  is a eigenvalue of  } M_\Phi\left(\tfrac{1}{t}\right)\}.
	\end{gather*}
\end{lem}
\begin{proof}
	We prove only the first equality since the second one can be proven in the same way. 
	
	For convenience, let $\left\{\lambda_1(t), \dots, \lambda_k(t)\right\}$ be the set of the eigenvalues of $M_\Phi(t)$. 
	We first observe the following facts: 
	\begin{enumerate}
		\item[(i)] First of all, $\lambda_i(t)$ is a convergent Puiseux series because of Theorem \ref{thm Newton Puiseux thm}. 
		\item[(ii)] There exists $\lambda_i(t)$ such that $\upsilon\left(\lambda_i(t)\right) <0$. 
		If not, one can easily see that $\tau^-(\Phi) \geq 0$. 
		However, it is easy to check that $\tau^-(\Phi) \leq -(N-1)$ from Equation \eqref{eqn shifting number} and a simple computation. 
		\item[(iii)] For any $t_0 \in \mathbb{R}_{>0}$, $M_\Phi(t_0)$ has the Perron-Frobenius eigenvalue, i.e., a simple, real eigenvalue having the maximal absolute value among all the eigenvalues of $M_\Phi(t_0)$. 
		Thus, there exists a distinguished simple eigenvalue $\lambda_1(t)$ of $M_\Phi(t)$ such that for a sufficiently small $t_0 \in \mathbb{R}_{>0}$, $\lambda_1(t_0)$ is the Perron-Frobenius eigenvalue of $M_\Phi(t_0)$. 
		\item[(iv)] Moreover, 
		\[\upsilon\left(\lambda_1(t)\right) = \min \{\upsilon\left(\lambda_i(t)\right) \vert \lambda_i(t) \text{  is an eigen-value of  } M_\Phi(t)\} <0. \]
	\end{enumerate}

	As mentioned above, $M_\Phi(t)$ has the Jordan normal form.
	In other words, there exists an invertible matrix $U$ such that 
	\[M_\Phi(t) = U \cdot \begin{bmatrix}
		J_1(T) & & \\
		 & \ddots & \\
		 &  & J_k(t)
	\end{bmatrix} \cdot U^{-1},\]
	where $J_i(t)$ is the Jordan block for the eigenvalue $\lambda_i(t)$, i.e., 
	\[J_i(t) = \begin{bmatrix}
		\lambda_i(t) & 1 & & \\
		& \lambda_i(t) & 1 & &\\
		& & \ddots & \ddots & \\
		& & & \lambda_i(t) & 1 \\
		& & & & \lambda_i(t)
	\end{bmatrix}.\]

	Now, we would like to point out that $U$ is a fixed matrix and the Jordan blocks $J_i(t)$ are upper triangular matrices. 
	Thus, from (i)--(iv), one can easily see that the minimal valuation of the $(v_i,v_j)$-component of $M_\Phi(t)^n$ is 
	\[n \cdot \upsilon\left(\lambda_1(t)\right) + c,\]
	where $c$ is a constant determined by $U$. 
	Thus, the first equality of Lemma \ref{lem eigenvalue} holds. 
\end{proof}

\begin{remark}
	\label{rmk limit not liminf}
	Lemma \ref{lem eigenvalue} implies that $\tau^\pm(\Phi)$ is a {\em limit} of a sequence, i.e.,  
	\begin{align*}
		&\tau^-(\Phi) := \lim_{n\to \infty} \tfrac{1}{n} \min_{v_i, v_j \in V(T)} \left\{\upsilon\left(\left(M_\Phi(t)\right)^n_{v_i,v_j}\right)\right\}, \\
		&\tau^+(\Phi) := -\lim_{n\to \infty} \tfrac{1}{n} \min_{v_i, v_j \in V(T)} \left\{\upsilon\left(\left(M_\Phi(\tfrac{1}{t})\right)^n_{v_i,v_j}\right)\right\}.
	\end{align*}
\end{remark}

Lemma \ref{lem eigenvalue} gives an algorithm computing the positive and negative shifting numbers of $\Phi$. 
The first step of the algorithm is to compute $M_\Phi(t)$ and $M_\Phi(\tfrac{1}{t})$, and the second step is to find the eigenvalues of $M_\Phi(t)$ and $M_\Phi(\tfrac{1}{t})$.
Finally, the algorithm ends by finding the minimal valuations of the eigenvalues. 
However, the algorithm does not seem {\em practical}, because the second step could be a hard step to do if $\Phi$ is not simple. 

In the next subsection, we will show that the negative (resp.\ positive) shifting number of $\Phi$ is the same as the minimal valuation (resp.\ the negative of the minimal valuation) among the diagonal components of $M_\Phi(t)$ (resp.\ $M_\Phi(\tfrac{1}{t})$).

\subsection{Shifting number}
\label{subsection shifting number}
In the subsection, we prove the claim stated at the end of Section \ref{subsection linear algebra over Puiseux series}. 
To state the claim more formally, we define the following notation:
\begin{definition}
	\label{def diagonal component}
	 For a Penner type autoequivalence $\Phi$, let $\boldsymbol{s^-(\Phi)}$ and $\boldsymbol{s^+(\Phi)}$ denote the number defined as 
	\[s^-(\Phi):=\min_{v \in V(T)} \left\{\upsilon\left(\left(M_\Phi(t)\right)_{v,v}\right)\right\}, s^+(\Phi):=-\min_{v \in V(T)} \left\{\upsilon\left(\left(M_\Phi(\tfrac{1}{t})\right)_{v,v}\right)\right\}.\]
\end{definition}

With the above notation, we can write the main claim of Section \ref{subsection shifting number} formally. 
\begin{thm}
	\label{thm shifting number}
	Using the above notation, if $\Phi$ is an autoequivalence of Penner type, then the following hold: 
	\[\tau^\pm(\Phi) = s^\pm(\Phi).\]
\end{thm}

If Theorem \ref{thm shifting number} holds, then it is enough to find the matrix $M_\Phi(t)$ in order to compute the shifting numbers of $\Phi$. 
Since $\Phi$ is a finite product of $\left\{\tau_u, \tau_w^{-1} | u \in V_+(T), w \in V_-(T)\right\}$, $M_\Phi(t)$ is also a finite product of 
\[\left\{M_{\tau_u}(t), M_{\tau_w^{-1}}(t) | u \in V_+(T), w \in V_-(T)\right\}.\]
Thus, Theorem \ref{thm shifting number} gives a practical way of computing the shifting numbers of every Penner type $\Phi$.

We prove Lemmas \ref{lem sufficiently large} and \ref{lem 2} in order to prove Theorem \ref{thm shifting number}. 
\begin{lem}
	\label{lem sufficiently large} 
	Let $\Phi$ be a Penner type autoequivalence. 
	If $n \in \mathbb{N}$ is sufficiently large, then 
	\[s^-(\Phi^n) = \tau^-(\Phi^n) = n \cdot \tau^-(\Phi), s^+(\Phi^n) = \tau^+(\Phi^n) = n \cdot \tau^+(\Phi).\]
\end{lem}
\begin{proof}
	In the proof of Lemma \ref{lem sufficiently large}, we use the following notation for convenience: 
	Since $M_\Phi(t)$ is $|V(t)|$-by-$|V(T)|$ matrix, the number of diagonal components of the Jordan normal form is $|V(t)|$. 
	Let $\{\lambda_v(t) | v \in V(T)\}$ be the set of diagonal components of the Jordan normal form of $M_\Phi(t)$.
	Then, we can simply write 
	\[tr\left(M_\Phi(t)\right) = \sum_{v \in V(T)} \lambda_v(t).\]

	We prove Lemma \ref{lem sufficiently large} only for the negative shifting number.
	The same logic proves Lemma \ref{lem sufficiently large} for the positive shifting number.
	
	First of all, since $\tau^-(\Phi^n)$ is the lowest valuation among that of the eigenvalues of $M_{\Phi^n}(t) = \left(M_\Phi(t)\right)^n$, we have 
	\[\tau^-(\Phi^n) = \min_{v \in V(T)} \upsilon\left(\left(\lambda_v(t)\right)^n\right) = n \cdot \min_{v\in V(T)}\upsilon\left(\lambda_v(t)\right) = n \cdot \tau^-(\Phi).\]
	This proves the second equality of Lemma \ref{lem sufficiently large}. 
	
	We recall that every components of $M_\Phi(t)$ is zero or a Laurent polynomial with positive integer coefficient. 
	Thus, $s^-(\Phi)$ is equal to the valuation of the trace of $M_\Phi(t)$. 
	Moreover, since the trace of $M_\Phi(t)$ agrees with the sum of all eigenvalue, $s^-(\Phi)$ is the valuation of the sum of all eigenvalues. 
	If $n$ is a natural number, the above arguments hold for $\Phi^n$, and we have
	\[s^-(\Phi^n) = \upsilon\big(\mathrm{tr}\left(M_{\Phi^n}(t)\right)\big) = \upsilon\big(\mathrm{tr}\left(M_\Phi(t)^n\right)\big) = \upsilon \left(\sum_{v \in V(T)} \left(\lambda_v(t)\right)^n\right).\]
	
	If $s^-(\Phi^n) \neq \tau^-(\Phi^n)$, then we can conclude that 
	\begin{gather}
		\label{eqn sufficiently large 1}
		\upsilon \left(\sum_{v \in V(T)} \left(\lambda_v(t)\right)^n\right) \neq n \cdot \min_{v\in V(T)}\upsilon\left(\lambda_v(t)\right).
	\end{gather}
	
	Let $\alpha_v$ be the coefficient of the lowest degree term in $\lambda_v(t)$. 
	If Equation \eqref{eqn sufficiently large 1} holds for some $n \in \mathbb{N}$, it means that 
	\begin{gather}
		\label{eqn sufficiently large 2}
		\sum_{v \text{  such that  } \upsilon\left(\lambda_v(t)\right) = \tau^-(\Phi)} \alpha_v^n =0.
	\end{gather}
	Since $\alpha_v$ is a fixed complex number, there exist only finitely many $n \in \mathbb{N}$ such that Equation \eqref{eqn sufficiently large 2} holds. 
	Thus, for sufficiently large $n \in \mathbb{N}$, 
	\[s^-(\Phi^n) = \tau^-(\Phi^n) = n \cdot \tau^-(\Phi).\]
\end{proof}

\begin{lem}
	\label{lem 2}
	Let $\Phi$ be a Penner type autoequivalence. 
	Then, the following hold:
	\[s^-(\Phi^2) = 2s^-(\Phi), s^+(\Phi^2) = 2 s^+(\Phi).\]
\end{lem}
\begin{proof}
	We prove Lemma \ref{lem 2} only for the negative shifting number, and the case of positive shifting number can be proven in the same way. 
	
	We recall that $s^-(\Phi)$ (resp.\ $s^-(\Phi^2)$) is defined to be the smallest valuations among the diagonal components of $M_\Phi(t)$ (resp.\ $M_{\Phi^2}(t)=\left(M_\Phi(t)\right)^2$).
	Since every component of $M_\Phi(t)$ is either zero or a Laurent polynomial with positive integer coefficient, we have 
	\[s^-(\Phi^2) \leq 2 s^-(\Phi).\]

	Now, let us assume that $s^-(\Phi^2) \neq 2 s^-(\Phi)$, and we will find a contraction.
	The assumption implies that there exists $v_1 \neq v_2 \in V(T)$ such that 
	\begin{gather}
		\label{eqn 2 1}
		\upsilon\left(\left(M_\Phi(t)\right)_{v_1,v_2}\right) + \upsilon\left(\left(M_\Phi(t)\right)_{v_2,v_1}\right) \lneq 2 s^-(\Phi).
	\end{gather}
	
	We recall that by definition, $\upsilon\left(\left(M_\Phi(t)\right)_{v,v'}\right)$ is equal to 
	\[\min \left\{d | \text{  the twisted complex  } \Phi(S_{v'}) \text{  contains  } S_v[d] \right\}.\]
	See Remark \ref{rmk properties of Laurent-poly matrix} (1). 
	Thus, if we set $a := \upsilon\left(\left(M_\Phi(t)\right)_{v_1,v_2}\right), b := \upsilon\left(\left(M_\Phi(t)\right)_{v_2,v_1}\right)$ for convenience, $\Phi(S_{v_2})$ (resp.\ $\Phi(S_{v_1})$) has $S_{v_1}[a]$ (resp.\ $S_{v_2}[b]$).
	
	We would like to point out that $\Phi(S_v)$ is computed by iteratively applying Lemmas \ref{lem induced functor} and \ref{lem simple computation}. 
	Thus, if we write 
	\[\Phi = \alpha_\ell \circ \dots \circ \alpha_1 \text{  with  } \alpha_i \in \left\{\tau_u, \tau_w^{-1} | u \in V_+(T), w \in V_-(T)\right\},\]
	then, $\Phi(S_{v'})$ contains $S_v[d]$ if and only if there exist a pair of sequences 
	\[\{v_{i_0} = v', v_{i_1}, \dots, v_{i_\ell} = v | v_{i_k} \in V(T)\}, \left\{a_0 = 0, \dots, a_\ell =d | a_i \in \mathbb{Z}\right\},\]
	satisfying that
	\begin{enumerate}
		\item[(I)] $\alpha_{k+1}\left(S_{v_{i_k}}[a_k]\right)$ contains $S_{v_{i_{k+1}}}[a_{k+1}]$ for all $k =0, \dots, \ell-1$.
	\end{enumerate}
	We also would like to point out that the following (II) should holds since $\alpha_{k+1}\left(S_{v_{i_k}}[a_k]\right)$ contains $S_{v_{i_{k+1}}}[a_{k+1}]$. 
	\begin{enumerate}
		\item[(II)] $v_{i_k}$ and $v_{i_{k+1}}$ must be either the same or adjacent vertices in the tree $T$. 
	\end{enumerate}
	
	Since $\Phi(S_{v_2})$ (resp.\ $\Phi(S_{v_1})$) has $S_{v_1}[a]$ (resp.\ $S_{v_2}[b]$), there exist sequences 
	\begin{gather*}
		\{v_{i_0} = v_2, v_{i_1}, \dots, v_{i_\ell} = v_1 | v_i \in V(T)\}, \left\{a_0 = 0, \dots, a_\ell =a | a_i \in \mathbb{Z}\right\}, \\
		\{v_{j_0} = v_1, v_{j_1}, \dots, v_{j_\ell} = v_2 | v_j \in V(T)\}, \left\{b_0 = 0, \dots, b_\ell =b | b_j \in \mathbb{Z}\right\},
	\end{gather*}
	satisfying (I) (and, thus (II)).
	Then, from (II), one has two paths on a tree $T$, one starting at $v_{i_0}=v_2$ and ending at $v_{i_\ell} = v_2$, and the other starting at $v_{j_0} = v_2$ and ending at $v_{j_\ell}=v_1$ from the sequences of vertices $\{v_i\}$ and $\{v_j\}$.
	
	Since $T$ is a tree, these two paths have to meet in the middle at least once. 
	In other words, there exists at least one integer $k \in (1, \ell)$ such that either
	\begin{enumerate}
		\item[(i)] $v_{i_k} = v_{j_k}$, or 
		\item[(ii)] $v_{i_k} = v_{j_{k+1}} \sim v_{j_k} = v_{i_{k+1}}$. 
	\end{enumerate}
	
	One can observe that (ii) does not happen. 
	If (ii) happens, then $\sigma_k(S_{v_{i_k}})$ contains $S_{v_{i_{k+1}}}$ factor by (I). 
	Since $v_{i_k} \neq v_{i_{k+1}}$, $\sigma_k$ should be $\tau_v^\pm$ satisfying that $v \sim v_{i_k}$. 
	Similarly, we can observe that $\alpha_k = \tau_v^\pm$ such that $v \sim v_{j_k}$. 
	Since $v_{i_k} \sim v_{j_k}$ by (ii), it is not possible to find $v$ satisfying $v_{i_k} \sim v \sim v_{j_k}$.
	Thus, (ii) could not happen, and (i) must happen. 
	
	Now, one can construct the following two pairs of sequences satisfying (I):
	\begin{gather*}
		\{v_{i_0} = v_2, v_{i_1}, \dots, v_{i_k} = v_{j_k}, v_{j_{k+1}}, \dots, v_{j_\ell} = v_2\}, \left\{a_0 = 0, \dots, a_k, a_k+(b_{k+1}-b_k), \dots, a_k + (b_\ell-b_k)  = b + a_k - b_k\right\}, \\
		\{v_{j_0} = v_1, v_{j_1}, \dots, v_{j_k} = v_{i_k}, v_{i_{k+1}}, \dots, v_{i_\ell} = v_1\}, \left\{b_0 = 0, \dots, b_k, b_k+(a_{k+1}-a_k), \dots, b_k + (a_\ell-a_k) = a + b_k - a_k\right\}.
	\end{gather*}

	The existence of the above sequences implies that $\Phi(S_{v_2})$ contains $S_{v_2}[b + a_k - b_k]$ factor and $\Phi(S_{v_1})$ contains $S_{v_1}[a+ b_k - a_k]$ factor. 
	Thus, $s^-(\Phi)$ should satisfy 
	\[s^-(\Phi) \leq a+b_k-a_k, b+a_k-b_k,\] 
	and
	\[2s^-(\Phi) \leq (a+b_k-a_k)+(b+a_k-b_k) = a+b = \upsilon\left(\left(M_\Phi(t)\right)_{v_1,v_2}\right) + \upsilon\left(\left(M_\Phi(t)\right)_{v_2,v_1}\right).\]
	It contradicts to the inequality \eqref{eqn 2 1}.
\end{proof}

\begin{proof}[Proof of Theorem \ref{thm shifting number}]
	By applying Lemma \ref{lem 2} multiple times, we have 
	\[2^K s^\pm(\Phi) = s^\pm(\Phi^{2^K}),\]
	for any $K \in \mathbb{N}$.
	Then, by Lemma \ref{lem sufficiently large}, for sufficiently large $K$, we have 
	\[2^K s^\pm(\Phi) = \tau^\pm(\Phi^{2^K})=2^K \tau^\pm(\Phi).\]
	Thus, when one divides both sides of the above equation, one has the equality $s^\pm(\Phi) = \tau^\pm(\Phi)$. 
\end{proof}

\subsection{Lemma \ref{lem shifting numbers of inverse}}
\label{subsection lemma shifting numbers of inverse}
Before ending Section \ref{section strong pseudo-Anosov}, we prove an easy Lemma that we need in Section \ref{subsection proof of lemma phase}.
\begin{lem}
	\label{lem shifting numbers of inverse}
	If $\Phi$ is of Penner type, then 
	\[\tau^+(\Phi) = -\tau^-(\Phi^{-1})\ \mathrm{and}\ \tau^-(\Phi) = - \tau^-(\Phi^{-1}).\]
\end{lem}

Before proving Lemma \ref{lem shifting numbers of inverse}, we would like to discuss a subtlety in the notation $\tau^\pm(\Phi^{-1})$.
We recall that in Sections \ref{section strong pseudo-Anosov} and before, $\Phi$ is assumed to be a product of 
\[\{\tau_u, \tau_w^{-1} | u \in V_+(T), w \in V_-(T)\}.\]
It is based on a fixed sign convention, as mentioned in Remarks \ref{rmk sign conventions}, \ref{rmk sign is not important 0}, etc. 
Moreover, many notions we defined above, for example, $M_\Phi(t), \tau^\pm(\Phi)$, are dependent on the fixed sign convention, since the fixed generating set $\{S_v\}_{v \in V(T)}$ of $\Fuk$ appears in their definitions. 
See Remark \ref{rmk sign is not important 2}. 
It means that, for $\Phi^{-1}$, $M_{\Phi^{-1}}(t), \tau^\pm(\Phi^{-1})$, and something similar to theses, are defined based on the other sign convention. 

In spite of the subtlety, the proof of Lemma \ref{lem shifting numbers of inverse} is simple.

\begin{proof}[Proof of Lemma \ref{lem shifting numbers of inverse}]
	Let $\Phi = \alpha_\ell \circ \dots \circ \alpha_1$ with $\alpha_i \in \{\tau_u, \tau_w^{-1} | u \in V_+(T), w \in V_-(T)\}$. 
	In the proof of Lemma \ref{lem 2}, we showed that for any $v\in V(T)$, the Laurent polynomial $\left(M_\Phi(t)\right)_{v,v}$ is determined by the set of pairs of sequences 
	\[\{v_{i_0} = v, v_{i_1}, \dots, v_{i_\ell} = v | v_i \in V(T)\}\ \mathrm{and}\ \left\{a_0 = 0, \dots, a_\ell | a_i \in \mathbb{Z}\right\},\]
	satisfying the condition (I) given in the proof of Lemma \ref{lem 2}. 
	Moreover, since $\tau^\pm(\Phi) = s^\pm(\Phi)$, the maximum (resp.\ minimal) of $a_\ell$ is $\tau^+(\Phi)$ (resp.\ $\tau^-(\Phi)$).
	
	For each pair of sequence $\left(\{v_i\}, \{a_i\}\right)$ satisfying (I), one can consider a pair of ``reversing'' sequences, defined as follows:
	\[\{v_{j_0}=v, v_{j_1} = v_{i_{\ell-1}}, \dots, v_{j_\ell} = v_{i_0} =v\}\ \mathrm{and}\ \{b_0=0, b_1 = a_{\ell-1} - a_\ell, \dots, b_\ell = a_0 - a_\ell=- a_\ell\}.\]
	Then, one can observe that the pair of ``reversing" sequences satisfies 
	\begin{enumerate}
		\item[(I')] $\alpha_{\ell-(k+1)}^{-1}\left(R_{v_{j_k}}[b_k]\right)$ contains $R_{v_{j_{k+1}}}[b_{k+1}]$ for all $k =0, \dots, \ell-1$,
	\end{enumerate}
	where $R_v$, $v\in V(T)$ is defined by
		\[R_v = \begin{cases}
		S_v \text{  if  } v \in V_+(T), \\ S_v[2-N] \text{  if  } v \in V_-(T).
	\end{cases}\]
	See Remark \ref{rmk sign is not important 2} for more details on the definition of $R_v$.
	
	Because the maximum (resp.\ minimum) of such $b_\ell$ is $\tau^+(\Phi^{-1})$ (resp.\ $\tau^-(\Phi^{-1})$) by the same logic as before, we have 
	\[\tau^+(\Phi) = -\tau^-(\Phi^{-1}), \tau^-(\Phi) = - \tau^+(\Phi^{-1}).\]
\end{proof}

\section{Translation length of Penner type autoequivalences}
\label{section translation of Penner type autoequivalences}
Let us recall that $\stab^\dagger(\Fuk)$ is a metric space. 
See Definition \ref{def metric} for the metric function. 
It is well-known that the metric function on $\stab^\dagger(\Fuk)$ induces a metric on the quotient space $\stab^\dagger(\Fuk)/\mathbb{C}$.

\begin{definition}
	\label{def induced metric}
	On the quotient space $\stab(\Fuk)/\mathbb{C}$, a distance function $\boldsymbol{d_B}$ is given by 
	\[d_B\left([\sigma_1], [\sigma_2]\right) := \inf_{z \in \mathbb{C}} \left\{d(\sigma_1, \sigma_2 \cdot z)\right\}.\]
\end{definition}

As seen in Section \ref{section actions on the space of stability conditions}, a Penner type autoequivalence $\Phi$ induces an left action on the metric space $\stab^\dagger(\Fuk)$. 
Moreover, since the left $\Phi$-action on $\stab^\dagger(\Fuk)$ commutes with the right $\mathbb{C}$-action, it induces an action on the quotient space  $\stab^\dagger(\Fuk)/\mathbb{C}$. 
For convenience, let 
\[\Phi_\stab: \stab^\dagger(\Fuk) \to \stab^\dagger(\Fuk), \Phi_\mathbb{C}: \stab^\dagger(\Fuk)/\mathbb{C} \to \stab^\dagger(\Fuk)/\mathbb{C} \]
denote the induced actions on $\stab(\Fuk)$ and $\stab(\Fuk)/\mathbb{C}$. 

In this section, we study the two actions $\Phi_\stab$ and $\Phi_\mathbb{C}$ and their translation lengths.
For the reader's convenience, we recall Definition of translation length.
\begin{definition}[ = Definition \ref{def hyperbolic action}]
	\label{def translation length}
	Let $(X,d)$ be a metric space, and let $f: X \to X$ be an isometry. 
	\begin{enumerate}
		\item The {\bf translation length} of $f$ is defined by 
		\[\ell(f):=\inf_{x \in X} \left\{d\left(x,f(x)\right)\right\}.\]
		\item The isometry $f$ is said to be {\bf hyperbolic} if $\ell(f)>0$ and there exists $x \in X$ such that $\ell(f) = d \left(x, f(x)\right)$.
	\end{enumerate}
\end{definition}

Since $\Phi_\stab$ and $\Phi_\mathbb{C}$ are isometries, one can define their translation lengths.
One of the main theorems of Section \ref{section translation of Penner type autoequivalences} is to show that the translation lengths of $\Phi_\stab$ and $\Phi_\mathbb{C}$ are positive and bounded below by the log of the stretching factor and shifting numbers of $\Phi$. 
\begin{thm}
	\label{thm lowerbound}
	Let $T$ be a tree and $\Phi:\Fuk \to \Fuk$ be an autoequivalence of Penner type. 
	Let $\lambda_\Phi$ denote the stretching factor of $\Phi$, and let $\tau^\pm(\Phi)$ denote the shifting numbers of $\Phi$ as defined in Definition \ref{def shifting number}. 
	We note that in Section \ref{section strong pseudo-Anosov}, we proved that $\tau^\pm(\Phi)$ equals the positive/negative phase stretching factor of $\Phi$. 
	\begin{enumerate}
		\item The translation length of $\Phi_\stab$ is larger or equal to 
		\[\max \left\{\tau^+(\Phi), -\tau^-(\Phi), \log \lambda_\Phi \right\}.\]
		Especially, the translation length is positive.
		\item The translation length of $\Phi_\mathbb{C}$ is larger or equal to 
		\[\max \left\{\tfrac{1}{2} \left(\tau^+(\Phi) - \tau^-(\Phi)\right), \log \lambda_\Phi\right\}.\]
		Especially, the translation length is positive.
	\end{enumerate}
\end{thm}
\begin{remark}
	\label{rmk negative shifting number}
	We note that from the Definition \ref{def Penner type} and Theorem \ref{thm shifting number}, one can easily check that $\tau^-(\Phi) <0$. 
\end{remark}

Since $\Phi_\stab$ and $\Phi_\mathbb{C}$ have positive translation lengths, it would be natural to ask whether the actions are hyperbolic.
Another main theorem of the current section is to show that the actions are hyperbolic under assumptions. 
See Section \ref{subsection hyperbolic actions of Penner type autoequivalences on stab under an assumption} for more details.

\begin{remark}
	\label{rmk sharp lowerbound}
	For the case of hyperbolic actions, we observe that the translation length is exactly the same as the lower bound given in Theorem \ref{thm lowerbound}. 
	In other words, the lower bound is sharp under the assumption. 
	We will mention later that the assumptions are added because our tools are not robust to study the {\em one-time} behavior of the induced actions. 
	See Remark \ref{rmk meaning of the assumption}. 
	Thus, we conjecture that the induced actions are hyperbolic without any assumptions and the lower-bound in Theorem \ref{thm lowerbound} is sharp.
\end{remark}

\subsection{Lower bound of translation length}
\label{subsection lower bound of translation length}
In Section \ref{subsection lower bound of translation length}, we prove Theorem \ref{thm lowerbound}, under the assumptions that Lemmas \ref{lem lower bound for phase} and \ref{lem lower bound for mass} hold. 
These two lemmas will be proven in Sections \ref{subsection proof of lemma phase} and \ref{subsection proof of lemma mass}, respectively.
\begin{lem}
	\label{lem lower bound for phase}
	Let $T$ be a tree, and let $\Phi$ be an autoequivalence of Penner type. 
	Then, for every $\sigma \in \stab^\dagger(\Fuk)$,
	\[\sup_{0 \neq E \in \Fuk} \left\{\phi^\pm_\sigma(\Phi E) - \phi^\pm_\sigma(E) \right\} \geq \tau^+(\Phi), \inf_{0 \neq E \in \Fuk} \left\{\phi^\pm_\sigma(\Phi E) - \phi^\pm_\sigma(E) \right\} \leq \tau^-(\Phi).\]
\end{lem}

\begin{lem}
	\label{lem lower bound for mass}
	Let $T$ be a tree, and let $\Phi$ be an autoequivalence of Penner type.
	Let $\lambda_\Phi$ denote the stretching factor of $\Phi$. 
	Then, for every $\sigma \in \stab^\dagger(\Fuk)$,
	\[\sup_{0 \neq E \in \Fuk}\left\{\log m_\sigma(\Phi E) - \log m_\sigma (E)\right\} \geq \log \lambda_\Phi, \inf_{0 \neq E \in \Fuk}\left\{\log m_\sigma(\Phi E) - \log m_\sigma (E)\right\} \leq -\log \lambda_\Phi. \]
\end{lem}

\begin{proof}[Proof of Theorem \ref{thm lowerbound}]
	The first item of Theorem \ref{thm lowerbound} can be restated as the following inequality:
	For any $\sigma \in \stab^\dagger(\Fuk)$, 
	\begin{gather}
		\label{eqn lowerbound 1} 
		d(\sigma, \Phi \cdot \sigma) \geq \max \left\{\tau^+(\Phi), -\tau^-(\Phi), \log \lambda_\Phi \right\}.
	\end{gather}
	
	To prove \eqref{eqn lowerbound 1}, we recall that for any $\sigma = \left(Z_\sigma,P_\sigma = \left\{P_\sigma(\phi)\right\}_{\phi \in \mathbb{R}}\right) \in \stab(\Fuk)$, 
	\[\Phi \cdot \sigma = \left(Z_\sigma \cdot \Phi^{-1}, \left\{\Phi\left(P_\sigma(\phi)\right)\right\}_{\phi \in \mathbb{R}}\right).\]
	Thus,
	\begin{align}
		\label{eqn lowerbound 3}
		\begin{split}
			d\left(\sigma, \Phi \cdot \sigma\right) &= \sup_{0 \neq E \in \Fuk} \left\{|\phi^+_\sigma(E) - \phi^+_{\Phi \cdot \sigma}(E)|, |\phi^-_\sigma(E) - \phi^-_{\Phi \cdot \sigma}(E)|, | \log\frac{m_\sigma(E)}{m_{\Phi \cdot \sigma}(E)}|\right\} \\ 
			&= \sup_{0 \neq E \in \Fuk} \left\{|\phi^+_\sigma(E) - \phi^+_\sigma(\Phi^{-1}E)|, |\phi^-_\sigma(E) - \phi^-_\sigma(\Phi^{-1}E)|, | \log\frac{m_\sigma(E)}{m_\sigma(\Phi^{-1}E)}|\right\} \\
			&= 	\sup_{0 \neq E \in \Fuk} \left\{|\phi^+_\sigma(\Phi E) - \phi^+_\sigma(E)|, |\phi^-_\sigma(\Phi E) - \phi^-_\sigma(E)|, | \log\frac{m_\sigma(\Phi E)}{m_\sigma(E)}|\right\}.
		\end{split}
	\end{align}
	
	By \eqref{eqn lowerbound 3}, it is enough to prove 
	\[\sup_{0 \neq E \in \Fuk}\left\{|\phi^\pm_\sigma(\Phi E) - \phi^\pm_\sigma(E)|\right\} \geq \max \left\{\tau^+(\Phi), -\tau^-(\Phi)\right\}, \sup_{0 \neq E \in \Fuk}\left\{|\log\frac{m_\sigma(\Phi E)}{m_\sigma(E)}|\right\} \geq \log \lambda_\Phi.\]
	The above inequalities hold because of the following computations: 
	\begin{align*}
		\sup_{0 \neq E \in \Fuk}\left\{|\phi^\pm_\sigma(\Phi E) - \phi^\pm_\sigma(E)|\right\} &= \max \left\{\sup_{0 \neq E \in \Fuk}\left\{\phi^\pm_\sigma(\Phi E) - \phi^\pm_\sigma(E)\right\}, -\left(\inf_{0 \neq E \in \Fuk}\left\{\phi^\pm_\sigma(\Phi E) - \phi^\pm_\sigma(E)\right\}\right)\right\} \\
		&\geq \max \left\{\tau^+(\Phi), - \tau^-(\Phi)\right\} \text{  (by Lemma \ref{lem lower bound for phase})},
	\end{align*}
	and 
	\begin{align*}
		\sup_{0 \neq E \in \Fuk}\left\{|\log\frac{m_\sigma(\Phi E)}{m_\sigma(E)}|\right\} &= \max \left\{\sup_{0 \neq E \in \Fuk}\left\{\log m_\sigma(\Phi E) - \log m_\sigma(E)\right\}, -\left(\inf_{0 \neq E \in \Fuk}\left\{\log m_\sigma(\Phi E) - \log m_\sigma(E)\right\}\right)\right\} \\
		&\geq \log \lambda_\Phi \text{  (by Lemma \ref{lem lower bound for mass})}.
	\end{align*}

	As similar to the first item, the second item of Theorem \ref{thm lowerbound} can be restated as the following inequality: 
	\begin{gather}
		\label{eqn lowerbound 2} 
		d_B([\sigma], [\Phi \cdot \sigma]) \geq \max \left\{\tfrac{1}{2} \left(\tau^+(\Phi) - \tau^-(\Phi)\right), \log \lambda_\Phi\right\}.
	\end{gather}
	By definition of $d_B$, we have 
	\begin{gather}
		\label{eqn lowerbound 4}
		d_B\left([\sigma], [\Phi \cdot \sigma]\right) = \inf_{x, y \in \mathbb{R}} \sup_{0 \neq E \in \Fuk}  \left\{|\phi^+_\sigma(\Phi E) - \phi^+_\sigma(E)-x|, |\phi^-_\sigma(\Phi E) - \phi^-_\sigma(E)-x|, |\log\frac{m_\sigma(\Phi E)}{m_\sigma(E)}-y|\right\}.
	\end{gather}
	Then, similar to the proof of the first item, it is enough to prove 
	\begin{gather*}
		\inf_{x \in \mathbb{R}} \sup_{0 \neq E \in \Fuk}  \left\{|\phi^\pm_\sigma(\Phi E) - \phi^\pm_\sigma(E)-x|\right\} \geq \tfrac{1}{2} \left(\tau^+(\Phi) + \tau^-(\Phi)\right), \\
		\inf_{y \in \mathbb{R}} \sup_{0 \neq E \in \Fuk} \left\{|\log\frac{m_\sigma(\Phi E)}{m_\sigma(E)}-y|\right\} \geq \log \lambda_\Phi.
	\end{gather*} 
	These two inequalities hold because of Lemmas \ref{lem lower bound for phase}, \ref{lem lower bound for mass}, and the following equations:
	\begin{gather*}
		\inf_{x \in \mathbb{R}} \sup_{0 \neq E \in \Fuk}  \left\{|\phi^\pm_\sigma(\Phi E) - \phi^\pm_\sigma(E)-x|\right\} = \frac{1}{2}\left( \sup_{0 \neq E \in \Fuk}\left\{\phi^\pm_\sigma(\Phi E) - \phi^\pm_\sigma(E)\right\}-\inf_{0 \neq E \in \Fuk}\left\{\phi^\pm_\sigma(\Phi E) - \phi^\pm_\sigma(E)\right\}\right), \\
		\inf_{y \in \mathbb{R}} \sup_{0 \neq E \in \Fuk} \left\{|\log\frac{m_\sigma(\Phi E)}{m_\sigma(E)}-y|\right\} = \frac{1}{2}\left( \sup_{0 \neq E \in \Fuk}\left\{\log\frac{m_\sigma(\Phi E)}{m_\sigma(E)}\right\}-\inf_{0 \neq E \in \Fuk}\left\{\log\frac{m_\sigma(\Phi E)}{m_\sigma(E)}\right\}\right).
	\end{gather*}
\end{proof}

\subsection{Proof of Lemma \ref{lem lower bound for phase}}
\label{subsection proof of lemma phase}
In Section \ref{subsection proof of lemma phase}, we prove Lemma \ref{lem lower bound for phase}.
To do that, we first introduce Definition \ref{def sigma entropy}.

\begin{definition}
	\label{def sigma entropy}
	Let $T$ be a tree. 
	\begin{enumerate}
		\item Let $\sigma = (Z_\sigma, P_\sigma) \in \stab(\Fuk)$, and let $E$ be an object of $\Fuk$. 
		Then, there exists a Harder-Narasimhan filtration of $E$ with respect to $\sigma$. 
		Let the Harder-Narasimhan filtration be 
		\[\begin{tikzcd}
			0 =E_0 \arrow{r}  & E_1 \arrow{d}\arrow{r} & E_2 \arrow{d} \ar[r] & \dots \ar[r] & E_{m-1} \ar[r] \ar[d] & E \ar[d]\\
			& A_1 \arrow[ul, dashed]         & A_2 \arrow[ul, dashed]       & \dots        & A_{m-1} \ar[ul, dashed]       & A_m \ar[ul, dashed].
		\end{tikzcd} \]
		The {\bf mass function of $\boldsymbol{E}$ with respect to $\boldsymbol{\sigma}$} is the function of the $t$-variable $m_{\sigma,t}(E): \mathbb{R}\to \mathbb{R}_{\geq 0}$ given by 
		\[m_{\sigma, t}(E) := \sum_{k=1}^m|Z_\sigma(A_k)|e^{\phi_\sigma(A_k)t}.\]
		\item Let $\Phi: \Fuk \to \Fuk$ be an endofunctor of $\Fuk$. 
		The {\bf mass growth function of $\boldsymbol{F}$ with respect to $\boldsymbol{\sigma}$} is the function $h_{\sigma,t}: \mathbb{R} \to [-\infty, \infty)$ in variable $t$ given by 
		\[h_{\sigma,t}(\Phi) := \sup_{E \in \Fuk}\left\{\limsup_{n \to \infty} \frac{\log m_{\sigma,t}\left(\Phi^n E\right)}{n}\right\}.\]
	\end{enumerate}
\end{definition}

One can find a similarity of Definitions \ref{def categorical entropy} and \ref{def sigma entropy}. 
As one can expect from the similarity, the notions defined in Definitions \ref{def categorical entropy} and \ref{def sigma entropy} are strongly related. 
For example, see \cite{Ikeda21}. 
The following Theorem is one of the results in \cite{Ikeda21}. 

\begin{thm}[Theorem 1.3 of \cite{Ikeda21}]
	\label{thm Ikeda}
	Let $G$ be a split-generator of a triangulated category $\mathcal{D}$ and $\Phi: \mathcal{D} \to \mathcal{D}$ be an endofunctor. 
	If a connected component $\stab^\dagger(\mathcal{D}) \subset \stab(\mathcal{D})$ contains an algebraic stability condition, then for any $\sigma \in \stab^\dagger(\mathcal{D})$, we have 
	\[h_t(\Phi) = h_{\sigma,t}(\Phi) = \lim_{n \to \infty} \tfrac{1}{n} \log \left(m_{\sigma, t}(\Phi^nG)\right).\]
\end{thm}

We do not introduce the notion of algebraic stability condition appeared in Theorem \ref{thm Ikeda}. 
The definition is following:
\begin{definition}[Definitions 2.7 and 2.11 of \cite{Ikeda21}]
	\label{def algebraic stability condition}
	A stability condition $\sigma$ is called {\bf algebraic}, if its heart is a finite length abelian category with finitely many isomorphism classes of simple objects.
\end{definition}

It is easy to check that $\sigma_\star$ defined in Definition \ref{def sigma star} is an algebraic stability condition.
Thus, for any $\sigma \in \stab^\dagger(\Fuk)$, one can apply Theorem \ref{thm Ikeda}, and as an application of Theorem \ref{thm Ikeda}, we prove Lemma \ref{lem lower bound for phase}.
\begin{proof}[Proof of Lemma \ref{lem lower bound for phase}]
	Let $L:=\oplus_{v \in V(T)}L_v$ and $S:= \oplus_{v \in V(T)}S_v$.
	We would like to remark the following facts: 
	\begin{enumerate}
		\item[(i)] By Remarks \ref{rmk general shifting number} and \ref{rmk limit not liminf}, the following hold:
		\begin{align*}
			&\tau^-(\Phi) = \lim_{n \to \infty} \tfrac{1}{n} \min \left\{k \Big\vert \lhom_{\wrapped}^{-k}\left(L,\Phi^n S\right) \neq 0 \right\}, \\
			&\tau^+(\Phi) = \lim_{n \to \infty} \tfrac{1}{n} \max \left\{k \Big\vert \lhom_{\wrapped}^{-k}\left(L,\Phi^n S\right) \neq 0 \right\}.
		\end{align*}
		\item[(ii)] It is known by \cite[Theorem 4.15]{Bae-Choa-Jeong-Karabas-Lee22} that 
		\[h_t(\Phi) := \lim_{n \to \infty} \tfrac{1}{n} \log \sum_{k \in \mathbb{Z}} \dim \lhom_{\wrapped}^{-k}\left(L, \Phi^n S\right) e^{kt}.\]
		\item[(iii)] We also note that for any $n \in \mathbb{N}$, 
		\begin{align*}
			e^{t \cdot \min \left\{k | \lhom_{\wrapped}^{-k}\left(L,\Phi^n S\right) \neq 0 \right\}}  \leq \sum_{k \in \mathbb{Z}}\dim \lhom_{\wrapped}^{-k}(L,\Phi^nS) e^{kt},\\
			e^{t \cdot \max \left\{k | \lhom_{\wrapped}^{-k}\left(L,\Phi^n S\right) \neq 0 \right\}} \leq \sum_{k \in \mathbb{Z}}\dim \lhom_{\wrapped}^{-k}(L,\Phi^nS) e^{kt}.
		\end{align*}
	\end{enumerate}
	
	By taking $\lim_{n \to \infty} \tfrac{1}{n} \log$ of the inequalities in (iii), one has
	\begin{gather}
		\label{eqn lem 1 1} t \cdot \tau^\pm(\Phi) \leq h_t(\Phi).
	\end{gather}

	Now, let us assume that $t >0$ (resp.\ $t<0$).
	Then, for any $n \in \mathbb{N}$, we have 
	\[m_{\sigma,t}(\Phi^n S)\leq m_{\sigma,0}(\Phi^n S) \cdot e^{\phi^+_{\sigma}(\Phi^n S)t} \text{  $\Big($resp.\  } m_{\sigma,t}(\Phi^n S)\leq m_{\sigma,0}(\Phi^n S) \cdot e^{\phi^-_{\sigma}(\Phi^n S)t}\Big),\]
	by definition of $m_{\sigma, t}(\Phi^n S)$. 
	When one takes $\lim_{n \to \infty}\tfrac{1}{n}\log$ to the both sides of the above inequalities, by Theorem \ref{thm Ikeda}, one has,
	\begin{gather}
		\label{eqn lem 1 2}
		\begin{split}
		\text{for  } t>0, h_{\sigma, t}(\Phi) \leq \log \lambda_\Phi + t\cdot \lim_{n \to \infty}\tfrac{1}{n} \phi^+_\sigma(\Phi^n S), \\
		\text{for  } t<0, h_{\sigma, t}(\Phi) \leq \log \lambda_\Phi + t\cdot \lim_{n \to \infty}\tfrac{1}{n} \phi^-_\sigma(\Phi^n S).
		\end{split}
	\end{gather}
	where $\lambda_\Phi$ is the stretching factor of $\Phi$. 
	We note that $m_{\sigma,0}(\Phi^n S)$ equals $m_\sigma(\Phi^n S)$ defined in Definition \ref{def stability mass function}. 
	Thus, we have 
	\[\lim_{n \to \infty} \tfrac{1}{n} m_{\sigma,0}(\Phi^n S) =\lim_{n \to \infty} \tfrac{1}{n} m_\sigma(\Phi^n S) = \log \lambda_\Phi. \]
	
	On the other hand, one has,
	\begin{align}	
		\label{eqn lem 1 3}
		\begin{split}
			\lim_{n \to \infty} \tfrac{1}{n} \phi^+_\sigma(\Phi^n S) &= \lim_{n \to \infty} \left(\tfrac{1}{n} \left(\sum_{k=1}^n \phi^+_\sigma(\Phi^k S) - \phi^+_\sigma(\Phi^{k-1} S)\right) + 	\tfrac{1}{n} \phi^+_\sigma(S)\right) \\
			&\leq \sup_{0 \neq E \in \Fuk} \left\{\phi^+_\sigma(\Phi E) - \phi^+_\sigma(E)\right\},\\
			\lim_{n \to \infty} \tfrac{1}{n} \phi^-_\sigma(\Phi^n S) &= \lim_{n \to \infty} \left(\tfrac{1}{n} \left(\sum_{k=1}^n \phi^-_\sigma(\Phi^k S) - \phi^-_\sigma(\Phi^{k-1} S)\right) + 	\tfrac{1}{n} \phi^-_\sigma(S)\right) \\
			&\geq \inf_{0 \neq E \in \Fuk} \left\{\phi^-_\sigma(\Phi E) - \phi^-_\sigma(E)\right\}.
		\end{split}
	\end{align}
	
	By combining inequalities in \eqref{eqn lem 1 2} and \eqref{eqn lem 1 3}, 
	\begin{align}
		\label{eqn lem 1 4} 
		\begin{split}
			\text{for  } t>0, h_{\sigma,t}(\Phi) \leq \log \lambda_\Phi + t\cdot\sup_{0 \neq E \in \Fuk} \left\{\phi^+_\sigma(\Phi E) - \phi^+_\sigma(E)\right\},\\
			\text{for  } t<0, h_{\sigma,t}(\Phi) \leq \log \lambda_\Phi + t\cdot\inf_{0 \neq E \in \Fuk} \left\{\phi^-_\sigma(\Phi E) - \phi^-_\sigma(E)\right\}.
		\end{split}
	\end{align}

	Now, inequalities in \eqref{eqn lem 1 1}, \eqref{eqn lem 1 4}, and Theorem \ref{thm Ikeda} conclude that 
	\begin{align*}
		\text{for  } t>0, t \cdot \tau^+(\Phi) \leq h_t(\Phi) = h_{\sigma,t}(\Phi) \leq \log \lambda_\Phi + t\cdot\sup_{0 \neq E \in \Fuk} \left\{\phi^+_\sigma(\Phi E) - \phi^+_\sigma(E)\right\}, \\
		\text{for  } t<0, t \cdot \tau^-(\Phi) \leq h_t(\Phi) = h_{\sigma,t}(\Phi) \leq \log \lambda_\Phi + t\cdot\inf_{0 \neq E \in \Fuk} \left\{\phi^-_\sigma(\Phi E) - \phi^-_\sigma(E)\right\}.
	\end{align*}	
	When one divides by $t$ and take the limit as $t \to \pm \infty$, one can conclude that 
	\begin{gather}
		\label{eqn lem 1 5}
		\tau^+(\Phi) \leq \sup_{0\neq E \in \Fuk}\left\{\phi^+_\sigma(\Phi E) - \phi^+_\sigma(E)\right\}, \tau^-(\Phi) \geq \inf_{0\neq E \in \Fuk}\left\{\phi^-_\sigma(\Phi E) - \phi^-_\sigma(E)\right\}.
	\end{gather}

	Now, thanks to Lemma \ref{lem shifting numbers of inverse} and \eqref{eqn lem 1 5}, 
	\begin{align}
		\label{eqn lem 1 6}
		\begin{split}
			\inf_{0\neq E \in \Fuk} \left\{\phi^+_\sigma(\Phi E) - \phi^+_\sigma(E)\right\} &= -\sup_{0\neq E \in \Fuk} \left\{\phi^+_\sigma(E) - \phi^+_\sigma(\Phi E) \right\} \\
			&= -\sup_{0\neq E \in \Fuk} \left\{\phi^+_\sigma(\Phi^{-1} E) - \phi^+_\sigma(E) \right\} \\
			& \leq - \tau^+(\Phi^{-1}) =\tau^-(\Phi). \\
			\sup_{0 \neq E \in \Fuk} \left\{\phi^-_\sigma(\Phi E) - \phi^-_\sigma(E)\right\} &= -\inf_{0\neq E \in \Fuk} \left\{\phi^-_\sigma(E) - \phi^-_\sigma(\Phi E) \right\} \\
			&= -\inf_{0\neq E \in \Fuk} \left\{\phi^-_\sigma(\Phi^{-1} E) - \phi^-_\sigma(E) \right\} \\
			& \geq - \tau^-(\Phi^{-1}) =\tau^+(\Phi). 
		\end{split}
	\end{align}
	
	Then, \eqref{eqn lem 1 5} and \eqref{eqn lem 1 6} complete the proof.
\end{proof}

\subsection{Proof of Lemma \ref{lem lower bound for mass}}
\label{subsection proof of lemma mass}
Lemma \ref{lem lower bound for mass} is easier to prove than Lemma \ref{lem lower bound for phase}.
\begin{proof}[Proof of Lemma \ref{lem lower bound for mass}]
	Let $E$ be a nonzero object of $\Fuk$. 
	Then, for any $n \in \mathbb{N}$ and $\sigma \in \stab^\dagger(\Fuk)$,
	\begin{align}
		\label{eqn mass 1}
		\begin{split}
			\tfrac{1}{n}\left(\log m_\sigma(\Phi^n E) - \log m_\sigma(E)\right) &= \tfrac{1}{n}\sum_{k=1}^n \left(\log m_\sigma(\Phi^k E) - \log m_\sigma(\Phi^{k-1} E)\right)\\
			&\leq \sup_{0 \neq F \in \Fuk} \left\{\log m_\sigma(\Phi F) - \log m_\sigma(F)\right\}.
		\end{split}
	\end{align}
	
	By taking $\limsup_{n \to \infty}$ on both sides of \eqref{eqn mass 1}, one has 
	\[\log \lambda_\Phi = \limsup_{n \to \infty} \log m_\sigma(\Phi^n E) = \limsup_{n \to \infty} \tfrac{1}{n}\left(\log m_\sigma(\Phi^n E) - \log m_\sigma(E)\right) \leq \sup_{0 \neq F \in \Fuk} \left\{\log m_\sigma(\Phi F) - \log m_\sigma(F)\right\}.\]
	We note that the first equality, i.e., $\log \lambda_\Phi = \limsup_{n \to \infty} \log m_\sigma(\Phi^n E)$, holds because of Theorem \ref{thm pseudo-Anosov}, Proposition \ref{prop same connected component}, and Definition \ref{def growth filtration}.
	
	The above proves the first inequality of Lemma \ref{lem lower bound for mass}, i.e., 
	\[\sup_{0 \neq F \in \Fuk} \left\{\log m_\sigma(\Phi F) - \log m_\sigma(F)\right\} \geq \log \lambda_\Phi.\]
	To prove the second inequality of Lemma \ref{lem lower bound for mass}, we note that one can apply the above arguments to another Penner type autoequivalence $\Phi^{-1}$.
	Then, one has 
	\[\sup_{0 \neq E \in \Fuk}\left\{\log m_\sigma(\Phi^{-1} E) - \log m_\sigma(E)\right\} \geq \log \lambda_\Phi.\]
	We note that by Theorem \ref{thm stretch factor and entropy}, $\Phi$ and $\Phi^{-1}$ have the same stretching factor $\lambda_\Phi = \lambda_{\Phi^{-1}}$. 
	Thus,
	\begin{gather*}
		\inf_{0 \neq E \in \Fuk}\left\{\log m_\sigma(\Phi E) - \log m_\sigma(E)\right\} = - \sup_{0 \neq E \in \Fuk} \left\{\log m_\sigma(E) - \log m_\sigma(\Phi E)\right\} \\ 
		= - \sup_{0 \neq E \in \Fuk}\left\{\log m_\sigma(\Phi^{-1} E) - \log m_\sigma(E)\right\} \leq - \log \lambda_\Phi.
	\end{gather*}
	It completes the proof.
\end{proof}

\subsection{Hyperbolic actions of Penner type autoequivalences on $\stab(\Fuk)$ under an assumption}
\label{subsection hyperbolic actions of Penner type autoequivalences on stab under an assumption}
Given our demonstration that each autoequivalence of Penner type induces actions with positive translation lengths on $\stab^\dagger(\Fuk)$ and $\stab^\dagger(\Fuk)/\mathbb{C}$, a natural inquiry emerges: are these induced actions hyperbolic?
More directly, one might inquire whether there exists an element $\sigma \in \stab^\dagger(\Fuk)$ such that the distances $d(\sigma, \Phi \sigma)$ or $d_B([\sigma], [\Phi \sigma])$ are equal to the translation length of these induced actions.

While this question arises inherently, its nature diverges from the primary focus of our present work. 
Our current research centers on elucidating the asymptotic behavior of Penner type autoequivalences, or equivalently, delving into the consequences of iterating $\Phi$.
On the other hand, the question pertains to the one-time behavior of $\Phi$.
Due to this disparity in focus, addressing this question through the techniques employed in our current paper does not appear to be a straightforward challenge.

In spite of the difficulty mentioned above, one can show  through the tools used in the paper that if $\Phi$ satisfies a certain condition that will be introduced in Theorem \ref{thm hyperbolic action}, then it induces a hyperbolic action.
For the convenience of stating the condition, we define the following notation: 
For a Penner type autoequivalence $\Phi$,
\[h(\Phi):=- \min_{v_i, v_j \in V(T)} \left\{ \upsilon\left(M_\Phi\left(\tfrac{1}{t}\right)_{v_i, v_j}\right)\right\} \ \mathrm{and}\ \ell(\Phi) := \min_{v_i, v_j \in V(T)} \left\{ \upsilon\left(M_\Phi(t)_{v_i, v_j}\right)\right\}.\]
In other words, $h(\Phi)$ (resp.\ $\ell(\Phi)$) means the highest (resp.\ lowest) degree among degrees of components of $M_\Phi(t)$.
  
\begin{remark}
	Similar to Remarks \ref{rmk computation of the matrix} and \ref{rmk properties of Laurent-poly matrix}, we would like to remark that $h(\Phi)$ and $\ell(\Phi)$ are dependent on the choice of a sign convention.
\end{remark}

\begin{thm}
	\label{thm hyperbolic action}
	Let $T$ be a tree and let $\Phi: \Fuk \to \Fuk$ be a Penner type autoequivalence with the stretching factor $\lambda_\Phi$. 
	Let conditions (A) and (B) mean the following:
	\begin{enumerate}
		\item[(A)] $\log \lambda_\Phi \geq \max \left\{ h(\Phi), - \ell(\Phi)\right\}+N-2$.
		\item[(B)] $\log \lambda_\Phi \geq \tfrac{1}{2}\left(h(\Phi) - \ell(\Phi) + N-2\right)$.
	\end{enumerate}
	Then, the following hold:
	\begin{enumerate}
		\item If $\Phi$ satisfies the condition (A), then $\Phi$ induces a hyperbolic action on $\stab^\dagger(\Fuk)$.
		Especially, $d\left(\sigma_\star, \Phi \cdot \sigma_\star\right)$ agrees with the translation length. 
		\item If $\Phi$ satisfies the condition (B), then $\Phi$ induces a hyperbolic action on $\stab^\dagger(\Fuk)/\mathbb{C}$. 
		Especially, $d_B\left([\sigma_\star],[\Phi \cdot \sigma_\star]\right)$ agrees with the translation length.
	\end{enumerate}
\end{thm}

\begin{remark}
	\label{rmk meaning of the assumption}
	Before proving Theorem \ref{thm hyperbolic action}, we would like to discuss the meaning of the statement of Theorem \ref{thm hyperbolic action}.
	\begin{enumerate}
		\item It is important to acknowledge that the tools employed throughout this paper, especially the matrix $M_\Phi(t)$, were carefully designed to be well-suited for the fixed stability condition $\sigma_\star$. 
		However, it is essential to recognize that the choice of this fixed stability condition, denoted as $\sigma_\star$, is not canonical. 
		Consequently, the applicability of our tools may be somewhat constrained due to this variability in the choice of $\sigma_\star$. 
		Nonetheless, when it comes to studying the {\em asymptotic} behaviors of the induced actions, our tools are proven to be robust. 
		This is thanks to the valuable support provided by Proposition \ref{prop independence of the choice of stability conditions} and Lemma \ref{lem independence of the choice of stability conditions}. 
		However, when we delve into the examination of the {\em one-time} behaviors of the actions induced by $\Phi$ through the application of our tools, we encounter a limitation. 
		The tools do not yield precise measurements of the changes in maximal/minimal phases. 
		In order to overcome the problem, we have introduced the assumptions (A) and (B) in Theorem \ref{thm hyperbolic action}.
		\item We would like to emphasize that the definition of {\em hyperbolic action} is inherently reliant on the underlying metric structure. 
		Consequently, by opting for a different metric structure on $\stab(\Fuk)$ and $\stab^\dagger(\Fuk)$, we have the flexibility to relax the conditions (A) and (B) of Theorem \ref{thm hyperbolic action}. 
		As an illustrative example, let us consider the introduction of a positive real number, denoted as $c$. 
		In this context, we define a distance function, $d_c$, on $\stab(\Fuk)$ as follows:
		\[d_c(\sigma_1, \sigma_2) := \sup_{0 \neq E \in \mathcal{D}} \left\{c|\phi^-_{\sigma_2}(E) - \phi^-_{\sigma_1}(E)|, c|\phi^+_{\sigma_2}(E)- \phi^+_{\sigma_1}(E)|, |\log \frac{m_{\sigma_2}(E)}{m_{\sigma_1}(E)}|\right\}.\]
		Now, let us introduce (A') as a modification of (A).
		\begin{enumerate}
			\item[(A')] $\log \lambda_\Phi \geq c\left(|\tau^\pm(\Phi)|+N-2\right)$,
		\end{enumerate}
		With these adjustments, we can establish that for $\Phi$ satisfying (A'), it induces a hyperbolic action, but this time on the metric space ``$\left(\stab^\dagger(\Fuk), d_c\right)$."
	\end{enumerate}
\end{remark}

\begin{proof}[Proof of Theorem \ref{thm hyperbolic action}]
	We prove $(1)$ only and $(2)$ can be proven in a similar way. 
	
	We start the proof by pointing out that 
	\[h(\Phi) \geq \tau^+(\Phi) \geq 0 \geq \tau^-(\Phi) \geq \ell(\Phi).\] 
	Thus, if (A) holds, then 
	\[\max \left\{ \tau^+(\Phi), -\tau^-(\Phi), \log \lambda_\Phi\right\} = \log \lambda_\Phi.\]

	We also recall that by Theorem \ref{thm lowerbound}, for any $\sigma \in \stab^\dagger(\Fuk)$, 
	\[d\left(\sigma, \Phi \cdot \sigma\right) \geq \log \lambda_\Phi = \max \left\{\tau^+(\Phi), -\tau^-(\Phi), \log \lambda_\Phi \right\}.\]
	Thus, it is enough to show that 
	\begin{gather}
		\label{eqn hyperbolic goal}
		d(\sigma_\star, \Phi \cdot \sigma_\star) \leq \log \lambda_\Phi.
	\end{gather}
	
	We would like to recall Equations \eqref{eqn lowerbound 3}, i.e., for any $\sigma \in \stab(\Fuk)$,
	\begin{gather*}
		d(\sigma, \Phi \cdot \sigma) = \sup_{0 \neq E \in \Fuk} \left\{|\phi^+_\sigma(\Phi E) - \phi^+_\sigma(E)|, |\phi^-_\sigma(\Phi E) - \phi^-_\sigma(E)|, | \log\frac{m_\sigma(\Phi E)}{m_\sigma(E)}|\right\}.
	\end{gather*}
	Thus, in order to prove Equation \eqref{eqn hyperbolic goal}, it is enough to show the following: For any nonzero object $E \in \Fuk$, 
	\begin{enumerate}
		\item[(i)] $\ell(\Phi) \leq \phi^-_{\sigma_\star}(\Phi E) - \phi^-_{\sigma_\star}(E) \leq h(\Phi)+(N-2).$
		\item[(ii)] $\ell(\Phi) - (N-2) \leq \phi^+_{\sigma_\star}(\Phi E) - \phi^+_{\sigma_\star}(E) \leq h(\Phi).$
		\item[(iii)] $-\log \lambda_\Phi \leq \log m_{\sigma_\star}(\Phi E) - \log m_{\sigma_\star}(E) \leq \log \lambda_\Phi$.
	\end{enumerate}

	\noindent {\em Proof of (i) and (ii):}
	We first note that by the definition of $M_\Phi(t)$, one can easily observe that for any nonzero $E \in \Fuk$,
	\begin{gather}
		\label{eqn hyperbolic 1}
		\begin{split}
			\phi^-_{\sigma_\star}(\Phi E) - \phi^-_{\sigma_\star}(E) \geq \min_{v_i, v_j \in V(T)} \left\{ \upsilon\left(M_\Phi(t)_{v_i,v_j}\right) \right\} = \ell(\Phi), \\
		\phi^+_{\sigma_\star}(\Phi E) - \phi^+_{\sigma_\star}(E) \leq -\min_{v_i, v_j \in V(T)} \left\{ \upsilon\left(M_\Phi(\tfrac{1}{t})_{v_i,v_j}\right) \right\} = h(\Phi).
		\end{split}
	\end{gather}
	It proves a half of (i) and (ii). 
	
	We would like to apply the inequalities in \eqref{eqn hyperbolic 1} to $\Phi^{-1}$. 
	To do that, we should remark that the inverse $\Phi^{-1}$ requires us to consider a new fixed stability condition $\sigma_\star'$ using the other sign convention different from that used to define the stability condition $\sigma_\star$. 
	We note that the sign convention means a way decomposing $V(T)$ into positive and negative vertices, thus there exist two different possible sign conventions. 
	For more details, see Remarks \ref{rmk sign conventions}--\ref{rmk sign is not important 0}, \ref{rmk sign is not important 1}, and \ref{rmk sign is not important 2}.
	
	Let $\sigma_\star' = (Z'_\star, P'_\star)$ denote the new fixed stability condition for the other sign convention.
	More precisely, $\sigma_\star'$ is the stability condition such that 
	\begin{itemize}
		\item every semi-simple object with phase in $(0,1]$ is a direct sum of $\{R_v | v \in V(T)\}$, where 
			\[R_v = \begin{cases}
			S_v \text{  if  } v \in V_+(T), \\ S_v[2-N] \text{  if  } v \in V_-(T),
			\end{cases}\]
		\item for every $v \in V(T)$, $Z'_\star(R_v) = \sqrt{-1}$. 
	\end{itemize}
	
	Since the inverse $\Phi^{-1}$ behaves well with the new spherical objects $R_v$'s, we define a new matrix $M'_{\Phi^{-1}}(t)$ following the idea of $M_{\Phi}(t)$ (Definition \ref{def Laurent-poly matrix})) accordingly:
	\[ M'_{\Phi^{-1}} (t)_{v_i,v_j} =\sum_{k\in \mathbb{Z}} \left( \dim \mathrm{Hom}^{-k}_{\mathcal{W}_T} (L'_{v_i}, \Phi^{-1} R_{v_j}) \right), \]
	where $\{L'_{v}\}_{v\in V(T)}$ are objects of $\mathcal{W}_T$ dual to $\{R_v\}_{v\in V(T)}$.
	In other words, $\{L'_v\}_{v \in V(T)}$ and $\{R_v\}_{v \in V(T)}$ can replace $\{L_v\}_{v \in V(T)}$ and $\{S_v\}_{v \in V(T)}$ in Lemma \ref{lem wrapped}.
	Moreover, we define
	\[h'(\Phi^{-1}):=- \min_{v_i, v_j \in V(T)} \left\{ \upsilon\left(M'_{\Phi^{-1}}\left(\tfrac{1}{t}\right)_{v_i, v_j}\right)\right\} \ \mathrm{and}\ \ell'(\Phi^{-1}) := \min_{v_i, v_j \in V(T)} \left\{ \upsilon\left(M'_{\Phi^{-1}}(t)_{v_i, v_j}\right)\right\}\]
	analogously to $h(\Phi)$ and $\ell(\Phi)$.
		
	Then, for every nonzero object $E \in \Fuk$, we have 
	\begin{gather}
		\label{eqn hyperbolic 2}
		\begin{split}
			\phi^-_{\sigma'_\star}(\Phi^{-1} E) - \phi^-_{\sigma'_\star}(E) \geq \min_{v_i, v_j \in V(T)} \left\{ \upsilon\left(M_{\Phi^{-1}}(t)_{v_i,v_j}\right) \right\} = \ell'(\Phi^{-1}), \\
			\phi^+_{\sigma'_\star}(\Phi^{-1} E) - \phi^+_{\sigma'_\star}(E) \leq -\min_{v_i, v_j \in V(T)} \left\{ \upsilon\left(M_{\Phi^{-1}}(\tfrac{1}{t})_{v_i,v_j}\right) \right\} = h'(\Phi^{-1}).
		\end{split}
	\end{gather}
	
	Now, a slight modification of the proof of Lemma \ref{lem shifting numbers of inverse} proves that
	\[\ell'(\Phi^{-1}) = - h(\Phi) \ \mathrm{and}\ -h'(\Phi^{-1}) = \ell(\Phi).\] 
	Then, since the inequalities in \eqref{eqn hyperbolic 2} hold for every nonzero $E \in \Fuk$, one can deduce the following:
	\begin{gather}
		\label{eqn hyperbolic 3}
		\phi^-_{\sigma'_\star}(\Phi E) - \phi^-_{\sigma'_\star}(E) \leq h(\Phi)\ \mathrm{and}\ \phi^+_{\sigma'_\star}(\Phi E) - \phi^+_{\sigma'_\star}(E) \geq \ell(\Phi).
	\end{gather}
	
	We would like to point out that, because of definitions of $\sigma_\star$ and $\sigma'_\star$, 
	\[\phi_{\sigma'_\star}(S_u) = \phi_{\sigma_\star}(S_u) \text{  for all  } u \in V_+(T), \phi_{\sigma'_\star}(S_w) = \phi_{\sigma_\star}(S_w) + N-2.\]
	Thus, for any nonzero $E \in \Fuk$, 
	\[\phi^\pm_{\sigma_\star}(E) \leq \phi^\pm_{\sigma'_\star}(E) \leq \phi^\pm_{\sigma_\star}(E)+ N-2.\]
	It allows us to modify inequalities in \eqref{eqn hyperbolic 3}, and we obtain
	\[\phi^-_{\sigma_\star}(\Phi E) - \phi^-_{\sigma_\star}(E) \leq h(\Phi) + (N-2)\ \mathrm{and}\ \phi^+_{\sigma_\star}(\Phi E) - \phi^+_{\sigma_\star}(E) \geq \ell(\Phi) -(N-2).\]
	It is the left half of (i) and (ii). 
	
	\noindent {\em Proof of (iii):} We remark the following facts: 
	\begin{itemize}
		\item From the construction of fixed stability conditions $\sigma_\star, \sigma'_\star, \sigma_0$, one can easily see that, for any nonzero $E \in \Fuk$, 
		\[m_{\sigma_\star}(E) = m_{\sigma'_\star}(E) = m_{\sigma_0}(E).\] 
		\item As mentioned in Equation \eqref{eqn vector for minimal}, for any object $E \in \Fuk$, 
		\[m_{\sigma_0} (E) = \parallel \vect(E) \parallel_1.\]
		\item From Equation \eqref{eqn spectral radius}, 
		\[\parallel \vect(\Phi E) \parallel_1 \leq \parallel M_\Phi \cdot \vect(E) \parallel_1.\]
		\item Since $\lambda_\Phi$ is the spectral radius of $M_\Phi$, we have 
		\[\parallel M_\Phi \cdot \vect(E) \parallel_1 \leq \lambda_\Phi \cdot \parallel \vect(E) \parallel_1.\]
	\end{itemize}
	From the above facts, we conclude that 
	\begin{gather*}
		\log m_{\sigma_\star} (\Phi E) - \log m_{\sigma_\star} (E) \leq \log \lambda_\Phi.
	\end{gather*}
	
	Similarly, for $\Phi^{-1}$, one has 
	\[\log m_{\sigma'_\star} (\Phi^{-1} E) - \log m_{\sigma'_\star} (E) \leq \log \lambda_{\Phi^{-1}} = \log \lambda_\Phi.\]
	Since it holds for every nonzero $E \in \Fuk$, by replacing $E$ (resp.\ $\sigma'_\star$) with $\Phi E$ (resp.\ $\sigma_\star$), one has 
	\[- \log \lambda_\Phi \leq \log m_{\sigma_\star} (\Phi E) - \log m_{\sigma_\star} (E) \leq \log \lambda_\Phi.\]
\end{proof}

\section{Examples}
\label{section examples}
In the last section, we provide illustrative examples. In
Section \ref{subsection computations of stretching factors} we give computations of the stretching factor and shifting numbers of a provided $\Phi$. Additionally, in Section \ref{subsection example of pseudo-Anosov but not strong pseudo-Anosov}, we give an example of pseudo-Anosov, but not strong pseudo-Anosov autoequivalence. 

\subsection{Computations of stretching factors} 
\label{subsection computations of stretching factors}
In this subsection, we will give computations of stretching factors and shifting numbers where Penner type autoequivalences are provided. 
We consider three autoequivalences $\Phi_1, \Phi_2$ and $\Phi_3$.

The first two examples $\Phi_1$ and $\Phi_2$ are autoequivalences on $\Fuk$, where the tree $T$ is the Dynkin diagram of type $D_5$ and $N \geq 3$ is a fixed integer. 
Because our tree $T$ is $D_5$, we can choose a sign convention having three positive vertices, labeled by $V_+(T) = \{u_1, u_2, u_3\}$, and two negative vertices, labeled by $V_-(T) = \{w_1, w_2\}$, see Figure \ref{figure D_5} (a).

\begin{figure}[h]
	\centering
	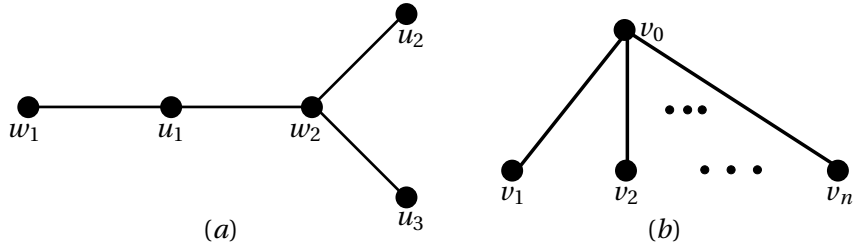		
	\caption{(a) is the Dynkin tree of type $D_5$. We consider the tree in (a) for the first two examples $\Phi_1$ and $\Phi_2$. (b) is the tree with one positive vertex and $n$-many negative vertices. We consider the tree in (b) for the third example $\Phi_3$.}
	\label{figure D_5}
\end{figure}

For a vertex $v \in V(T)$, let $\tau_v$ denote the spherical twist along $S_v$. 
Our examples are 
\[\Phi_1 := \tau_{u_1} \circ \tau_{u_2} \circ \tau_{u_3}\circ \tau_{w_1}^{-1} \circ \tau_{w_2}^{-1} \text{  and  } \Phi_2 = \tau_{u_3}^3 \circ \tau_{u_2}^3 \circ \tau_{w_2}^{-1} \circ \tau_{u_1}^3 \circ \tau_{w_1}^{-1}.\] 
Then, Definition \ref{def Laurent-poly matrix} (1) and Lemma \ref{lem simple computation} give us the Laurent-polynomial matrices associated to $\tau_{u_i}$ and $\tau_{w_j}^{-1}$. 
For convenience, let us assume that $V(T) = \left\{u_1, u_2, u_3, w_1, w_2\right\}$ is an ordered set, i.e., $u_1$ is the first vertex, $u_2$ is the second, and so on.
With the given order, the Laurent-polynomial matrices are 
\begin{gather*}
	M_{\tau_{u_1}}(t) = \begin{bmatrix}
		t^{-(N-1)} & 0 & 0 & 1 & 1 \\
		0 & 1 & 0 & 0 & 0 \\
		0 & 0 & 1 & 0 & 0 \\
		0 & 0 & 0 & 1 & 0 \\
		0 & 0 & 0 & 0 & 1 \\
	\end{bmatrix}, 
	M_{\tau_{u_2}}(t) = \begin{bmatrix}
		1 & 0 & 0 & 0 & 0 \\
		0 & t^{-(N-1)} & 0 & 0 & 1 \\
		0 & 0 & 1 & 0 & 0 \\
		0 & 0 & 0 & 1 & 0 \\
		0 & 0 & 0 & 0 & 1 \\
	\end{bmatrix}, 
	M_{\tau_{u_2}}(t) = \begin{bmatrix}
	1 & 0 & 0 & 0 & 0 \\
	0 & 1 & 0 & 0 & 0 \\
	0 & 0 & t^{-(N-1)} & 0 & 1 \\
	0 & 0 & 0 & 1 & 0 \\
	0 & 0 & 0 & 0 & 1 \\
	\end{bmatrix}, \\
	M_{\tau_{w_1}^{-1}}(t) = \begin{bmatrix}
		1 & 0 & 0 & 0 & 0 \\
		0 & 1 & 0 & 0 & 0 \\
		0 & 0 & 1 & 0 & 0 \\
		1 & 0 & 0 & t^{N-1} & 0 \\
		0 & 0 & 0 & 0 & 1 \\
	\end{bmatrix},
	M_{\tau_{w_2}^{-1}}(t) = \begin{bmatrix}
		1 & 0 & 0 & 0 & 0 \\
		0 & 1 & 0 & 0 & 0 \\
		0 & 0 & 1 & 0 & 0 \\
		0 & 0 & 0 & 1 & 0 \\
		1 & 1 & 1 & 0 & t^{N-1} \\
	\end{bmatrix}.
\end{gather*}

For the first example $\Phi_1$, by multiplying the above five matrices, one can compute $M_{\Phi_1}(t)$.
More precisely, 
\begin{align*}
	M_{\Phi_1}(t) :=& M_{\tau_{u_1}}(t) \circ M_{\tau_{u_2}}(t) \circ M_{\tau_{u_3}}(t) \circ M_{\tau_{w_1}^{-1}}(t) \circ M_{\tau_{w_2}^{-1}}(t) \\ 
	 			=& \begin{bmatrix}
	 				t^{-(N-1)}+2 & 1 & 1 & t^{N-1} & t^{N-1} \\
	 				1 & t^{-(N-1)}+1 & 1 & 0 & t^{N-1} \\
	 				1 & 1 & t^{-(N-1)}+1 & 0 & t^{N-1} \\
	 				1 & 0 & 0 & t^{N-1} & 0 \\
	 				1 & 1 & 1 & 0 & t^{N-1} \\
	 			\end{bmatrix}.
\end{align*}
Moreover, by putting $t=1$, one has 
\begin{gather*}
	M_{\Phi_1}(1) = \begin{bmatrix}
		3 & 1 & 1 & 1 & 1 \\
		1 & 2 & 1 & 0 & 1 \\
		1 & 1 & 2 & 0 & 1 \\
		1 & 0 & 0 & 1 & 0 \\
		1 & 1 & 1 & 0 & 1 \\
	\end{bmatrix}.
\end{gather*}

From $M_{\Phi_1}(1)$, one can compute the spectral radius and it is the stretching factor of $\Phi_1$. 
Thus, if we denote the stretching factor by $\lambda_{\Phi_1}$, 
\[\lambda_{\Phi_1} = 2 + \tfrac{1}{\sqrt{2}} + \sqrt{\tfrac{1}{2}\left(7 + 4 \sqrt{2}\right)}.\]
Moreover, by Theorems \ref{thm strong pseudo-Anosov} and \ref{thm shifting number}, the positive/negative phase stretching factors, or equivalently shifting numbers, of $\Phi_1$ are
\begin{gather*}
	\tau^-(\Phi_1) = \text{  the minimal valuation among the diagonal components of  } M_{\Phi_1}(t) = -(N-1), \\
	\tau^+(\Phi_1) = -\text{  the minimal valuation among the diagonal components of  } M_{\Phi_1}(\tfrac{1}{t}) = N-1.
\end{gather*}

Moreover $h(\Phi_1)$ and $\ell(\Phi_1)$ are also easily computed and $h(\Phi_1) = N-1, \ell(\Phi_1) = -(N-1)$. 
We note that $5 < \lambda_{\Phi_1} < 6$, and thus $\log \lambda_{\Phi_1} <2$. 
Thus, for any $N\geq 3$, $\Phi_1$ does not satisfy the conditions (A) and (B) of Theorem \ref{thm hyperbolic action}. 
Theorem \ref{thm hyperbolic action} does not answer a question asking the hyperbolicity of the action induced from $\Phi_1$.
However, as mentioned in Remark \ref{rmk meaning of the assumption}, after modifying the metric structures of $\stab^\dagger(\Fuk)$ and $\stab^\dagger(\Fuk)/\mathbb{C}$, one can see that $\Phi_1$ acts hyperbolically.

One can do similar computations for $\Phi_2$ and the results of computations are the following: 
For convenience, let us define a notation 
\[f_n(t) := 1 + t^1 + \dots + t^n, T := t^{N-1}.\]
Then, one has 
\begin{gather*}
	M_{\Phi_2}(t) = M_{\tau_{u_2}}(t)^3 \circ M_{\tau_{u_3}}(t)^3 \circ M_{\tau_{w_1}^{-1}}(t) \circ M_{\tau_{u_1}}(t)^3 \circ M_{\tau_{w_1}^{-1}}(t) \\ 
	= \begin{bmatrix}
		f_3(T^{-1}) & 0 & 0 & T \cdot f_2(T^{-1}) & f_2(T^{-1}) \\
		\left(T^{-3}+f_2(T^{-1})\right) \cdot f_2(T^{-1}) & f_3(T^{-1}) & f_2(T^{-1}) & T \cdot \left(f_2(T^{-1})\right)^2 & \left(T+f_2(T^{-1})\right) \cdot f_2(T^{-1}) \\
		\left(T^{-3}+f_2(T^{-1})\right) \cdot f_2(T^{-1}) & f_2(T^{-1}) & f_3(T^{-1}) & T \cdot \left(f_2(T^{-1})\right)^2 & \left(T+f_2(T^{-1})\right) \cdot f_2(T^{-1}) \\
		1 & 0 & 0 & T & 0 \\
		f_3(T^{-1}) & 1 & 1 & T \cdot f_2(T^{-1}) & T + f_2(T^{-1}) \\
	\end{bmatrix}, \\
	M_{\Phi_2}(1) =\begin{bmatrix}
		4 & 0 & 0 & 3 & 3 \\
		12 & 4 & 3 & 9 & 12 \\
		12 & 3 & 4 & 9 & 12 \\
		1 & 0 & 0 & 1 & 0 \\
		4 & 1 & 1 & 3 & 4 \\
	\end{bmatrix}.
\end{gather*}
Thus, the stretching factor and shifting numbers of $\Phi_2$ are 
\[\lambda_{\Phi_2} = 4 + \tfrac{3}{\sqrt{2}} + \sqrt{\tfrac{3}{2} \left(13 + 8 \sqrt{2}\right)}\ \mathrm{and}\ \tau^-(\Phi_2) = -3(N-1), \tau^+(\Phi_2) = N-1. \]

Moreover, as similar to the case of $\Phi_1$, $\Phi_2$ does not satisfy the conditions of Theorem \ref{thm hyperbolic action}.
However, similar to $\Phi_1$, on the modified metric spaces, $\Phi_2$ induces hyperbolic actions. 
Remark \ref{rmk meaning of the assumption} explains how to modify the metric structures in order to induce hyperbolic actions. 

For the last example, we consider the tree given in Figure \ref{figure D_5} (b), having one positive vertex $v_0$ and $n$-many negative vertices $v_1, \dots, v_n$. 
Let $\tau_k$ denote the spherical twist along $S_{v_k}$.
The last example we will consider is 
\[\Phi_3 = \tau_0 \circ \tau_1^{-1} \circ \dots \tau_n^{-1}.\]

One can easily check that the Laurent-polynomial matrix associated to $\tau_k$ is given as, for $k=0$, 
\[M_{\tau_0} = \begin{bmatrix}
	t^{-(N-1)} & 1 & \dots & 1 \\
	0 & 1 & \dots & 0 \\
	\vdots & \vdots & \ddots & \vdots \\
	0 & 0 & \dots & 1
\end{bmatrix},
\]
and for $k \geq 1$, 
\[M_{\tau_k^{-1}} = \begin{bmatrix}
	1 & 0 & \cdots & 0 & 0 & 0 & \cdots & 0 & 0 &\\
	0 & 1 & \cdots & 0 & 0 & 0 & \cdots & 0 & 0 & \\
	\vdots & \vdots & \ddots & \vdots & \vdots & \vdots & \ddots & \vdots & \vdots \\
	0 & 0 & \cdots & 1 & 0 & 0 & \cdots & 0 & 0\\
	1 & 0 & \cdots & 0 & t^{N-1} & 0 & \cdots & 0 & 0\\
	0 & 0 & \cdots & 0 & 0 & 1 & \cdots & 0 & 0\\
	\vdots & \vdots & \ddots & \vdots & \vdots & \vdots & \ddots & \vdots & \vdots\\
	0 & 0 & \cdots & 0 & 0 & 0 & \cdots & 1 & 0\\
	0 & 0 & \cdots & 0 & 0 & 0 & \cdots & 0 & 1\\
\end{bmatrix},\]
where $t^{N-1}$ is the $\left(k+1, k+1\right)$-component.

Then, a simple matrix multiplication gives us 
\[M_{\Phi_3} = \begin{bmatrix}
	t^{-(N-1)}+n & t^{N-1} & \cdots & t^{N-1} & t^{N-1} & t^{N-1} & \cdots & t^{N-1} & t^{N-1} &\\
	1 & t^{N-1} & \cdots & 0 & 0 & 0 & \cdots & 0 & 0 & \\
	\vdots & \vdots & \ddots & \vdots & \vdots & \vdots & \ddots & \vdots & \vdots \\
	1 & 0 & \cdots & t^{N-1} & 0 & 0 & \cdots & 0 & 0\\
	1 & 0 & \cdots & 0 & t^{N-1} & 0 & \cdots & 0 & 0\\
	1 & 0 & \cdots & 0 & 0 & t^{N-1} & \cdots & 0 & 0\\
	\vdots & \vdots & \ddots & \vdots & \vdots & \vdots & \ddots & \vdots & \vdots\\
	1 & 0 & \cdots & 0 & 0 & 0 & \cdots & t^{N-1} & 0\\
	1 & 0 & \cdots & 0 & 0 & 0 & \cdots & 0 & t^{N-1}\\
\end{bmatrix}.\]
Thus, the stretching factor and the shifting numbers are 
\[\lambda_{\Phi_3} = \tfrac{1}{2}\left(n +2 + \sqrt{n^2 + 4n}\right), \tau^-(\Phi_3)=-(N-1), \tau^+(\Phi_3) = N-1.\]

From the above computation, for a fixed $N \geq 3$, one can see that if $n$ is sufficiently large, then $\Phi_3$ satisfies the conditions of Theorem \ref{thm hyperbolic action}.
Thus, for such a $n$, $\Phi_3$ induces hyperbolic actions.

\subsection{Example of pseudo-Anosov but not strong pseudo-Anosov}
\label{subsection example of pseudo-Anosov but not strong pseudo-Anosov}
Now, we describe an example of autoequivalence that is pseudo-Anosov, but not strong pseudo-Anosov. 
To construct an example, we will first consider a tree $T$ of Dynkin type $A_2$, i.e., a tree having two vertices connected by one edge. 
Or equivalently, we consider the tree given in Figure \ref{figure D_5} (b), where $n = 1$. 
Thus, we will use the notation used in Section \ref{subsection computations of stretching factors}, for example, $\tau_{v_0}$ and $\tau_{v_1}$ will denote the spherical twists. 

In order to define $\Fuk$ for the tree $T$ of Dynkin type $A_2$, we need to fix an integer $N \geq 3$. 
In the subsection, we consider two different integers $3 \leq N_1 \lneq N_2$. 
And, let $\mathcal{C}_i$ denote the category associated to our tree $T$ and the fixed integer $N_i$ for $i =1 ,2$. 

For each $i =1 ,2$, one can define an autoequivalence $\Phi_i = \tau_{v_0} \circ \tau_{v_1}^{-1} : \mathcal{C}_i \to \mathcal{C}_i$. 
Then, by the computation given in Section \ref{subsection computations of stretching factors}, the stretching factor and the shifting numbers of $\Phi_i$ are given as 
\[\lambda_{\Phi_1} = \lambda_{\Phi_2} = \tfrac{1}{2}\left(3 + \sqrt{5}\right), \tau^\pm(\Phi_i) = \pm (N_i-1).\]

Now, let us define another category $\mathcal{C}$ from $\mathcal{C}_1$ and $\mathcal{C}_2$ as follows: 
\begin{itemize}
	\item The set of objects of $\mathcal{C}$ is the disjoint union of the sets of objects of $\mathcal{C}_i$, i.e., 
	\[\mathrm{Ob}(\mathcal{C}) = \mathrm{Ob}(\mathcal{C}_1) \sqcup \mathrm{Ob}(\mathcal{C}_2).\]
	\item Let $X$ and $Y$ be two objects of $\mathcal{C}$. 
	If $X$ and $Y$ are objects of $\mathcal{C}_i$ for $i =1$ or $2$, then the morphism space from $X$ to $Y$ in $\mathcal{C}$ is the same as the morphism space from $X$ to $Y$ in $\mathcal{C}_i$. 
	If not, the morphism space from $X$ to $Y$ is the zero space. 
\end{itemize}
Similarly, we define our example $\Phi : \mathcal{C} \to \mathcal{C}$ as follows:
\begin{itemize}
	\item When we restrict $\Phi$ on $\mathcal{C}_i$, then it agrees with $\Phi_i$, i.e., 
	\[\Phi|_{\mathcal{C}_1} = \Phi_1, \Phi|_{\mathcal{C}_2} = \Phi_2.\]
\end{itemize}

Then, one can easily see that $\Phi$ is a pseudo-Anosov autoequivalence with the stretching factor 
\[\lambda_\Phi = \lambda_{\Phi_1} = \lambda_{\Phi_2} = \tfrac{1}{2}\left(3 + \sqrt{5}\right).\]
However, $\Phi$ is not strong pseudo-Anosov, since the maximal and minimal phase filtrations associated to $\Phi$ have two steps 
\[0 \subsetneq \mathcal{P}^\pm_{N_1-1} (\Phi) \subsetneq \mathcal{P}^\pm_{N_2-1} (\Phi) = \mathcal{C}.\]

We end the paper by pointing out that we artificially construct the category $\mathcal{C}$ for providing an example, and that one could say that $\mathcal{C}$ lacks a sense of naturalness. 
For example, $\mathcal{C}$ is decomposable into $\mathcal{C}_1$ and $\mathcal{C}_2$. 
As mentioned in Remark \ref{rmk the shifting number 1}, we remain uncertain whether the notions of pseudo-Anosov and strong pseudo-Anosov coincide under a reasonable assumption, for example, indecomposable underlying category.

\bibliographystyle{amsalpha}
\bibliography{pA_autoequivalence}
\end{document}

%% file: traintrack.pdf_tex
\begingroup%
  \makeatletter%
  \providecommand\color[2][]{%
    \errmessage{(Inkscape) Color is used for the text in Inkscape, but the package 'color.sty' is not loaded}%
    \renewcommand\color[2][]{}%
  }%
  \providecommand\transparent[1]{%
    \errmessage{(Inkscape) Transparency is used (non-zero) for the text in Inkscape, but the package 'transparent.sty' is not loaded}%
    \renewcommand\transparent[1]{}%
  }%
  \providecommand\rotatebox[2]{#2}%
  \newcommand*\fsize{\dimexpr\f@size pt\relax}%
  \newcommand*\lineheight[1]{\fontsize{\fsize}{#1\fsize}\selectfont}%
  \ifx\svgwidth\undefined%
    \setlength{\unitlength}{340.15748031bp}%
    \ifx\svgscale\undefined%
      \relax%
    \else%
      \setlength{\unitlength}{\unitlength * \real{\svgscale}}%
    \fi%
  \else%
    \setlength{\unitlength}{\svgwidth}%
  \fi%
  \global\let\svgwidth\undefined%
  \global\let\svgscale\undefined%
  \makeatother%
  \begin{picture}(1,0.61666667)%
    \lineheight{1}%
    \setlength\tabcolsep{0pt}%
    \put(0,0){\includegraphics[width=\unitlength,page=1]{traintrack.pdf}}%
    \put(0.05694931,0.08749532){\color[rgb]{0,0,0}\makebox(0,0)[lt]{\lineheight{1.25}\smash{\begin{tabular}[t]{l}$5$\end{tabular}}}}%
    \put(0.37543533,0.22616912){\color[rgb]{0,0,0}\makebox(0,0)[lt]{\lineheight{1.25}\smash{\begin{tabular}[t]{l}$2$\end{tabular}}}}%
    \put(0.38824925,0.09117798){\color[rgb]{0,0,0}\makebox(0,0)[lt]{\lineheight{1.25}\smash{\begin{tabular}[t]{l}$3$\end{tabular}}}}%
    \put(0,0){\includegraphics[width=\unitlength,page=2]{traintrack.pdf}}%
    \put(0.18522427,0.35057165){\makebox(0,0)[lt]{\lineheight{1.25}\smash{\begin{tabular}[t]{l}$(a)$\end{tabular}}}}%
    \put(0.18522427,0.0086713){\makebox(0,0)[lt]{\lineheight{1.25}\smash{\begin{tabular}[t]{l}$(c)$\end{tabular}}}}%
    \put(0.72571846,0.0086713){\makebox(0,0)[lt]{\lineheight{1.25}\smash{\begin{tabular}[t]{l}$(d)$\end{tabular}}}}%
    \put(0.72571846,0.35057165){\makebox(0,0)[lt]{\lineheight{1.25}\smash{\begin{tabular}[t]{l}$(b)$\end{tabular}}}}%
  \end{picture}%
\endgroup%

%% file: D_5.pdf_tex
\begingroup%
  \makeatletter%
  \providecommand\color[2][]{%
    \errmessage{(Inkscape) Color is used for the text in Inkscape, but the package 'color.sty' is not loaded}%
    \renewcommand\color[2][]{}%
  }%
  \providecommand\transparent[1]{%
    \errmessage{(Inkscape) Transparency is used (non-zero) for the text in Inkscape, but the package 'transparent.sty' is not loaded}%
    \renewcommand\transparent[1]{}%
  }%
  \providecommand\rotatebox[2]{#2}%
  \newcommand*\fsize{\dimexpr\f@size pt\relax}%
  \newcommand*\lineheight[1]{\fontsize{\fsize}{#1\fsize}\selectfont}%
  \ifx\svgwidth\undefined%
    \setlength{\unitlength}{328.81889764bp}%
    \ifx\svgscale\undefined%
      \relax%
    \else%
      \setlength{\unitlength}{\unitlength * \real{\svgscale}}%
    \fi%
  \else%
    \setlength{\unitlength}{\svgwidth}%
  \fi%
  \global\let\svgwidth\undefined%
  \global\let\svgscale\undefined%
  \makeatother%
  \begin{picture}(1,0.29310345)%
    \lineheight{1}%
    \setlength\tabcolsep{0pt}%
    \put(0,0){\includegraphics[width=\unitlength,page=1]{D_5.pdf}}%
    \put(-0.00176903,0.13943404){\makebox(0,0)[lt]{\lineheight{1.25}\smash{\begin{tabular}[t]{l}$w_1$\end{tabular}}}}%
    \put(0.32048166,0.13943587){\makebox(0,0)[lt]{\lineheight{1.25}\smash{\begin{tabular}[t]{l}$w_2$\end{tabular}}}}%
    \put(0.16960427,0.13943587){\makebox(0,0)[lt]{\lineheight{1.25}\smash{\begin{tabular}[t]{l}$u_1$\end{tabular}}}}%
    \put(0.44463121,0.24725755){\makebox(0,0)[lt]{\lineheight{1.25}\smash{\begin{tabular}[t]{l}$u_2$\end{tabular}}}}%
    \put(0.44463072,0.03837566){\makebox(0,0)[lt]{\lineheight{1.25}\smash{\begin{tabular}[t]{l}$u_3$\end{tabular}}}}%
    \put(0.22362372,0.02348397){\makebox(0,0)[lt]{\lineheight{1.25}\smash{\begin{tabular}[t]{l}$(a)$\end{tabular}}}}%
    \put(0.73230504,0.02348509){\makebox(0,0)[lt]{\lineheight{1.25}\smash{\begin{tabular}[t]{l}$(b)$\end{tabular}}}}%
    \put(0,0){\includegraphics[width=\unitlength,page=2]{D_5.pdf}}%
    \put(0.72199145,0.25485062){\makebox(0,0)[lt]{\lineheight{1.25}\smash{\begin{tabular}[t]{l}$v_0$\end{tabular}}}}%
    \put(0,0){\includegraphics[width=\unitlength,page=3]{D_5.pdf}}%
    \put(0.56214537,0.06721628){\makebox(0,0)[lt]{\lineheight{1.25}\smash{\begin{tabular}[t]{l}$v_1$\end{tabular}}}}%
    \put(0,0){\includegraphics[width=\unitlength,page=4]{D_5.pdf}}%
    \put(0.69258585,0.06721385){\makebox(0,0)[lt]{\lineheight{1.25}\smash{\begin{tabular}[t]{l}$v_2$\end{tabular}}}}%
    \put(0,0){\includegraphics[width=\unitlength,page=5]{D_5.pdf}}%
    \put(0.9354302,0.06721385){\makebox(0,0)[lt]{\lineheight{1.25}\smash{\begin{tabular}[t]{l}$v_n$\end{tabular}}}}%
    \put(0,0){\includegraphics[width=\unitlength,page=6]{D_5.pdf}}%
  \end{picture}%
\endgroup%